\documentclass[11pt]{amsart} 
\usepackage{amsthm,amsbsy,amsmath,amssymb,amscd,amsfonts,array,mathrsfs,verbatim,enumerate,xypic,enumitem,color}
\usepackage[all]{xy}
\xyoption{arc} 

\newtheorem{thm}{Theorem}[subsection]
\newtheorem{prop}[thm]{Proposition}     
\newtheorem{lem}[thm]{Lemma}
\newtheorem{cor}[thm]{Corollary}

\theoremstyle{definition}

\newtheorem{defn}[thm]{Definition}

\newtheorem{example}[thm]{Example} 
\newtheorem{rem}[thm]{Remark}

\DeclareFontFamily{OT1}{rsfs}{}
\DeclareFontShape{OT1}{rsfs}{n}{it}{<-> rsfs10}{}
\DeclareMathAlphabet{\curly}{OT1}{rsfs}{n}{it}

\makeatletter \@addtoreset{equation}{subsection} \makeatother  
\makeatletter \@addtoreset{thm}{subsection} \makeatother  

\makeatletter \@addtoreset{equation}{subsection} \makeatother  
\makeatletter \@addtoreset{thm}{subsection} \makeatother  

\newcommand{\Pointsclosed}{Lemma~11}    
\newcommand{\Inverselimits}{Theorem~22} 
\newcommand{\Zerolocus}{Lemma~15}       
\newcommand{\Biginduced}{Lemma~18}      

\newcommand{\Esp}{{\bf Esp}} 
\newcommand{\Univ}{{\bf Univ}} 
\newcommand{\Alg}{{\bf Alg}} 
\newcommand{\An}{{\bf An}} 
\newcommand{\AS}{{\bf AS}} 
\newcommand{\LAS}{{\bf LAS}} 
\newcommand{\Ab}{{\bf Ab}} 
\newcommand{\A}{{\bf Kat}} 
\newcommand{\Sets}{{\bf Sets}} 
\newcommand{\Log}{{\bf Log}} 
\newcommand{\LogLRS}{{\bf LogLRS}} 
\newcommand{\PreLog}{{\bf PreLog}} 
\newcommand{\Mon}{{\bf Mon}} 
\newcommand{\Mod}{{\bf Mod}} 
\newcommand{\DS}{{\bf DS}}   
\newcommand{\LRS}{{\bf LRS}}  
\newcommand{\RS}{{\bf RS}}  
\newcommand{\LDS}{{\bf LDS}} 
\newcommand{\PRS}{{\bf PRS}} 
\newcommand{\fLDS}{{\bf fLDS}} 
\newcommand{\PLDS}{{\bf PLDS}} 
\newcommand{\LogSch}{{\bf LogSch}} 
\newcommand{\LogEsp}{{\bf LogEsp}} 
\newcommand{\IntLogEsp}{{\bf IntLogEsp}} 
\newcommand{\IntLDS}{{\bf IntLDS}} 
\newcommand{\IntPLDS}{{\bf IntPLDS}} 
\newcommand{\CohLogEsp}{{\bf CohLogEsp}} 
\newcommand{\fLogEsp}{{\bf FineLogEsp}} 
\newcommand{\FineLogEsp}{{\bf FineLogEsp}} 
\newcommand{\Sch}{{\bf Sch}} 
\newcommand{\Open}{{\bf Open}} 
\newcommand{\LMS}{{\bf LMS}} 
\newcommand{\fLMS}{{\bf FineLMS}} 
\newcommand{\PMS}{{\bf PMS}} 
\newcommand{\SMS}{{\bf SMS}} 
\newcommand{\fSMS}{{\bf FineSMS}} 
\newcommand{\MS}{{\bf MS}}   
\newcommand{\Top}{{\bf Top}}   
\newcommand{\Fans}{{\bf Fans}} 
\newcommand{\SFans}{{\bf SFans}} 

\newcommand{\Z}{{\bf Z}} 

\newcommand{\KN}{{\bf KN}}   

\renewcommand{\AA}{\mathbb{A}} 
\newcommand{\NN}{\mathbb{N}} 
\newcommand{\ZZ}{\mathbb{Z}} 
\newcommand{\RR}{\mathbb{R}} 
\newcommand{\QQ}{\mathbb{Q}} 
\newcommand{\PP}{\mathbb{P}}
\newcommand{\CC}{\mathbb{C}}
\newcommand{\GG}{\mathbb{G}}

\newcommand{\p}{\mathfrak{p}}
\newcommand{\q}{\mathfrak{q}}
\newcommand{\m}{\mathfrak{m}}

\newcommand{\C}{\curly{C}}  

\newcommand{\boundary}{\Delta}

\newcommand{\M}{\mathcal{M}} 
\newcommand{\N}{\mathcal{N}}
\renewcommand{\O}{\mathcal{O}} 
\newcommand{\str}{\mathcal{A}} 

\renewcommand{\u}{\underline}

\newcommand{\ov}{\overline}
\newcommand{\into}{\hookrightarrow}
\newcommand{\be}{\begin{eqnarray*}}
\newcommand{\ee}{\end{eqnarray*}}

\newcommand{\lms}[1]{ {#1}^\dagger }

\newcommand{\bne}[1]{\begin{eqnarray} \label{#1} }
\newcommand{\ene}{\end{eqnarray}}
\newcommand{\xym}{\xymatrix}
\newcommand{\bp}{\begin{pmatrix}}
\newcommand{\ep}{\end{pmatrix}}
\newcommand{\slot}{ \hspace{0.05in} {\rm \_} \hspace{0.05in} } 
\newcommand{\dirlim}{\displaystyle \lim_{ \longrightarrow } \,} 

\newcommand{\invlim}{\displaystyle \lim_{ \longleftarrow } \,} 

\newcommand{\sms}[1]{\ov{#1}}

\newcommand{\Hom}{\operatorname{Hom}}   
\newcommand{\Bilin}{\operatorname{Bilin}}   
\newcommand{\Ext}{\operatorname{Ext}}

\newcommand{\Aut}{\operatorname{Aut}}

\newcommand{\gp}{\operatorname{gp}}   
  
\newcommand{\Bl}{\operatorname{Bl}}
\newcommand{\sign}{\operatorname{sign}}

\newcommand{\Ker}{\operatorname{Ker}}
\newcommand{\Cok}{\operatorname{Cok}}

\newcommand{\Sym}{\operatorname{Sym}}    
\newcommand{\sat}{\operatorname{sat}}

\newcommand{\Id}{\operatorname{Id}}
\newcommand{\Spec}{\operatorname{Spec}}
\newcommand{\Proj}{\operatorname{Proj}}

\newcommand{\GL}{\operatorname{GL}}     

\numberwithin{equation}{subsection}

\begin{document}

\author{W.~D.~Gillam}
\address{Department of Mathematics, Bogazici University, Istanbul}
\email{william.gillam@boun.edu.tr}

\author{S.~Molcho}
\address{Department of Mathematics, University of Colorado Boulder, Boulder, CO}
\email{samouil.molcho@colorado.edu}

\date{\today}
\title[Log differentiable spaces]{Log differentiable spaces \\ and manifolds with corners}

\begin{abstract}  We develop a general theory of \emph{log spaces}, in which one can make sense of the basic notions of logarithmic geometry, in the sense of Fontaine-Illusie-Kato.  Many of our general constructions with log spaces are new, even in the algebraic setting.  In the differentiable setting, our theory yields a framework for treating manifolds with corners generalizing recent work of Kottke-Melrose.  We give a treatment of the theory of \emph{fans}, which are to monoids as schemes are to rings.  By adapting similar results from logarithmic algebraic geometry, we prove a general result on resolution of toric singularities which can be used to resolve singularities of a wide class of ``log smooth" spaces.  \end{abstract}

\maketitle

\arraycolsep=2pt 
\def\arraystretch{1.2}

\section*{Introduction}  It is generally agreed that a \emph{manifold with corners} is a locally ringed space over $\RR$ locally isomorphic to an open subset of $\RR_{\geq 0}^n$ (for various $n$) with its sheaf of smooth functions.  The definition of a \emph{morphism} of manifolds with corners, however, is debatable.  One might simply say that a morphism is just a morphism of locally ringed spaces over $\RR$---a ``smooth morphism" in the usual sense of differential geometry.   The problem with this is that manifolds with corners carry various additional structures---log tangent bundles, for example---which are not functorial with respect to such morphisms.  D.~Joyce \cite{Joyce} has proposed a theory of manifolds with corners which is at the opposite extreme in the sense that he allows very few morphisms.  For example, $x \mapsto x^2$ is not a valid Joyce endomorphism of $\RR_{\geq 0}.$

Kottke and Melrose \cite{KM} proposed a notion of morphism which is less restrictive than Joyce's, but controlled enough to allow the usual constructions to become functorial.  It turns out that the Kottke-Melrose formalism for manifolds with corners becomes very natural in the language of logarithmic geometry, as developed by Fontaine, Illusie, Kato, and others.  Unfortunately, the machinery of logarithmic geometry has largely been confined to algebraic geometry, where it has been extremely useful for various purposes.  One major goal of this work is to import this theory into the setting of differential geometry.  With a certain amount of formalism in place, it turns out that we can carry most of the usual algebro-geometric constructions over to our differential-geometric setting.

There are various points in the story where we will leave the setting of \emph{smooth manifolds}, as we have already done in a sense by introducing manifolds with corners.  At some point, the appearance of even more singular spaces will be inevitable.  For example, unlike the situation with smooth manifolds, the fibered product of transverse (in an appropriate ``log" sense) maps of manifolds with corners will not yield a manifold with corners, nor even a topological manifold with boundary.  We treat these singular spaces by working systematically in the category of \emph{differentiable spaces}, reviewed in \S\ref{section:differentiablespaces}.  This category will have a certain appeal to those with a background in algebraic geometry, but might be offputting to differential geometers.  All that really matters is that we have some category containing smooth manifolds as a full subcategory in which we can speak intelligibly about fibered products and so forth.

In fact, we work quite generally with an arbitrary category of \emph{log spaces}.  The basic input here is a reasonable category of ``spaces" (schemes, analytic spaces, or differentiable spaces, say), together with a distinguished monoid object $\AA^1$.  Our theory of \emph{positive log differentiable spaces} is recovered as the special case where ``spaces" are differentiable spaces and $\AA^1 = \RR_+$ is the monoid of non-negative real numbers under multiplication.  When possible, we have tried to phrase our log-geometric constructions in this general setting.  One such construction which seems to be new, even in the algebraic setting, is our notion of \emph{boundary} (\S\ref{section:boundary}).  In the differential setting it specializes, in particular, to Joyce's notion of the boundary of a manifold of corners, which, for the local model $\RR_+^n$ is the disjoint union of the coordinate hyperplanes.  For a toric variety, the boundary is the disjoint union of the toric divisors.  The boundary construction is necessary for a complete understanding of the Kottke-Melrose notion of ``$b$-map," in that the natural maps of log spaces we consider are actually the so-called ``interior $b$-maps" of their setup.  The boundary construction (and its relative variants) can be useful in the algebraic theory of gluing and degeneration \cite{logflatness}.

Monoids are analogous to rings.  A monoid $P$ has a prime spectrum $\Spec P$, which is a topological space equipped with a sheaf of monoids $\M_P$ (\S\ref{section:Spec}).  A \emph{fan} is then defined to be a topological space equipped with a sheaf of monoids which is locally isomorphic to $\Spec P$, for varying $P$, so that fans are to monoids what schemes are to rings.   There is a theory of modules over monoids (\S\ref{section:modules}) and a corresponding theory of coherent sheaves on fans, developed to the extent that we need it in \S\ref{section:monoidalspacesII}.  Just as the category of schemes has certain advantages over the category of varieties, the category of fans has advantages over the category of ``classical fans" familiar from toric varieties: one can form finite inverse limits and coproducts, for example.  The category of fans lies over any category of log spaces (\S\ref{section:realizationoffans}), so that fans serve as a combinatorial testing ground for all constructions in log geometry, much as the classical theory of toric varieties provides a testing ground for the general theory of schemes.  For example, in future work \cite{GM}, we intend to make use of fans as a testing ground for various Morse-theoretic constructions.

Some parts of this paper are specifically devoted to the differential setting.  We develop our theory of \emph{log smooth maps} only in this context, mostly to avoid setting up a cumbersome axiomatic theory of smooth maps in an arbitrary category of spaces.  Our definition of a log smooth map (\S\ref{section:logsmoothness}) is in terms of a ``chart criterion," like the one in \cite{Kat1}.  It takes a significant amount of work to show that such maps are stable under composition and base change.  In algebraic geometry, one can avoid this work by falling back on the notion of ``formally log smooth," which is not available in our differentiable setting.  (We discuss this point further in \S\ref{section:scholium}.)  Our arguments could be used to directly work out a theory of log smoothness in the algebraic setting without using the notion of formal log smoothness, and might therefore be of some independent interest.

Another advantage of our general approach is that one is led naturally to consider the relationships between various kinds of ``log spaces."  For example, we interpret the so-called ``Kato-Nakayama space" as a functor from log analytic spaces to positive log differentiable spaces (\S\ref{section:KNspace}).

In addition to the ``foundational" aspects of our work mentioned above, we also establish some new general results concerning ``resolution" of ``toric singularities," by which we mean ``resolution" of ``log smooth" spaces.  For now, let us agree that a ``log smooth" space is roughly a space with the sort of singularities that can appear in a toric variety.  We emphasize at this point that we work here in the general setting of log spaces, so that ``toric variety" has many possible interpretations.  Our results specialize to yield resolution theorems for fans, log analytic spaces, log schemes, and log differentiable spaces.

We close with a brief overview of our approach to resolution of singularities.  Given a log smooth space $X$, we want to produce a \emph{resolution} $Y \to X$ of $X$: a locally projective (roughly: ``proper") map from a smooth space to $X$ which is an isomorphism over the smooth part of $X$.  We remark here that any category of log spaces comes with a notion of ``locally projective morphism" because it comes with a notion of projective space $\PP^n$ since it recevies a functor from fans.   Our approach is to divide the resolution process into two parts: First, there is a combinatorial part, which always involves ``projectively subdividing cones," or, more abstractly, ``blowing up monoids."  Second, there is a ``geometric realization" step in which one passes from the combinatorial resolution to an actual map of (log) spaces.  For example, in the classical theory of toric varieties, there isn't much beyond the combinatorial step of correctly subdividing cones.  The geometric realization step is: apply the monoid algebra functor and the $\Spec$ functor.

When dealing with spaces that are not so rigidly combinatorial as toric varieties both steps become more involved.  There is a long history of such toric resolutions.  The first significant step beyond the classical case of toric varieties probably can be attributed to Mumford et al.\ in the theory of \emph{toroidal embeddings} \cite{KKMS}.  Significant generalizations were made by Kato \cite{Kat2}, using the language of log geometry.  This approach reaches its apex in work of Nizio\l\, \cite{Niz}, on which our approach is loosely based.  A certain amount of our work here is devoted to separating out the combinatorial and geometric realization steps in Nizio\l's work.  We then have to formulate the combinatorial step in a sufficiently general way that we can make sense of it in our general framework.  In fact, we formulate the combinatorial step purely as a problem of \emph{sharp monoidal spaces}---topological spaces equipped with a sheaf of sharp monoids (\S\ref{section:monoidalspacesI}).  We need a general result on functorial resolution of monoids (\S\ref{section:functorialresolution}), which, unfortunately, is proved through an appeal to functorial resolution for complex varieties, though it is a purely combinatorial statement.  For the geometric realization step, we generalize Kato's procedure from \cite[\S10]{Kat2}.  

\noindent {\bf Acknowledgements.}  The origins of the present paper can be traced to a talk given by Chris Kottke at Brown in the Fall of 2012 on his joint work \cite{KM} with Melrose.  Dan Abramovich and the present authors were in attendence.  We had been working on log geometry in other algebro-geometric contexts, thus it occurred naturally to us that the language of log geometry would provide a useful formalism for treating manifolds with corners and extending that theory to an appropriate class of more singular spaces.  The present work is the culmination of that effort.  We are especially grateful to Chris Kottke for explaining several points of \cite{KM} to us and for helping us work out and clarify our own point of view.  We also thank Abramovich for several helpful conversations.  Gillam was partially supported by an NSF Postdoctoral Fellowship.

\newpage

\setcounter{tocdepth}{2}
\tableofcontents

\newpage

\section{Monoids} \label{section:monoids}  For reasons of logical dependence we begin with a general discussion about monoids.  However, this section is used mainly as a sort a repository for general information about monoids needed elsewhere in the paper.  Most readers should probably just review the basic terminology of \S\ref{section:monoidbasics} and \S\ref{section:idealsandfaces} and then skip to \S\ref{section:differentiablespaces}.  The more delicate results on \emph{refinements} in \S\ref{section:monoidrefinements} will not be needed until \S\ref{section:geometricrealization}.

\subsection{Monoid basics} \label{section:monoidbasics} A \emph{monoid} is a set $P$ equipped with an associative, commutative binary operation $(p,q) \mapsto p+q$  (``addition") admitting a (necessarily unique and two-sided) identity element $0 \in P$.  A \emph{morphism} of monoids is a map between their underlying sets that commutes with addition and takes $0$ to $0$.  The category of monoids $\Mon$ has all direct and inverse limits.

An element $u \in P$ is called a \emph{unit} if there is a $v \in P$ so that $u + v =0$.  The units in $P$ form a submonoid denoted $P^* \subseteq P$.  If $P^* = \{ 0 \}$, then $P$ is called \emph{sharp}.  The monoid $\ov{P} := P / P^*$ is called the \emph{sharpening} of $P$.  The map $P \to \ov{P}$ is initial among maps from $P$ to a sharp monoid and $P \mapsto \ov{P}$ defines a functor from monoids to sharp monoids which is left adjoint to the inclusion.  A monoid homomorphism $h : Q \to P$ is called \emph{strict} iff $\ov{h} : \ov{Q} \to \ov{P}$ is an isomorphism.  A non-zero element $p$ of a sharp monoid is called \emph{irreducible} iff whenever $p=a+b$ either $a=0$ or $b=0$.  It can be shown that a fine, sharp monoid $P$ has only finitely many irreducible elements and that these elements generate $P$.

If $P^* = P$, then $P$ is called a \emph{group}.\footnote{All ``groups" considered in this paper are abelian.}  For any monoid $P$ there is a map of monoids $P \to P^{\gp}$ from $P$ to a group, through which any other map to a group factors uniquely.  That is: $P \mapsto P^{\rm gp}$ is left adjoint to the inclusion of groups into monoids.  More generally, if $S \subseteq P$ is any submonoid, then the \emph{localization} of $P$ at $S$, denoted $S^{-1} P$, is the monoid whose elements are equivalence classes $[p,s]$ of pairs $(p,s)$ where $p \in P$ and $s \in S$, where $(p,s) \sim (p',s')$ iff $p+s'+t = p'+s+t$ for some $t \in S$.  The natural map $P \to S^{-1}P$ given by $p \mapsto [p,0]$ is initial among maps from $P$ to a monoid taking elements of $S$ to units.  Let $h : Q \to P$ be a map of monoids, $T$ a submonoid of $Q$, $S$ its image under $h$.  Then we have a natural map $T^{-1}h : T^{-1}Q \to S^{-1} P$ and it is clear from the universal properties of the monoids in question that \bne{localizationpushout} & \xym@C+15pt{ Q \ar[r]^-h \ar[d] & P \ar[d] \\ T^{-1}Q \ar[r]^-{T^{-1}h} & S^{-1} P } \ene is a pushout diagram of monoids.

A monoid is called \emph{finitely generated} iff there is a surjective (on underlying sets) monoid homomorphism $\NN^r \to P$ for some finite $r$.  If $h : Q \to P$ is any surjection of finitely generated monoids, then it can be shown that the monoid \be E & := & \{ (q_1,q_2) \in Q \times Q : h(q_1)=h(q_2) \} \ee is also finitely generated.  It follows that every finitely generated monoid $P$ is \emph{finitely presented} in the sense that there is a coequalizer diagram of monoids $$ \NN^m \rightrightarrows \NN^n \to P $$ for some finite $m,n$ (often called a \emph{presentation of} $P$).  It also follows that any two presentations of $P$ receive maps from a third presentation of $P$.  

If $P$ is a finitely generated monoid, then $P^{\gp}$ is a finitely generated abelian group.  A monoid is called \emph{integral} iff $P \to P^{\gp}$ is injective.  For an arbitary monoid $P$, we let $P^{\rm int}$ denote the image of $P$ in $P^{\rm gp}$.  The map $P \to P^{\rm int}$ is initial among maps from $P$ to an integral monoid.  A monoid is called \emph{fine} iff it is finitely generated and integral.  We will mostly be interested in fine monoids.  The forgetful functor $\Mon \to \Sets$ has a left adjoint $X \mapsto \oplus_X \NN$ called the \emph{free monoid} functor.  A monoid in the essential image of this functor is called \emph{free}.  Any finitely generated free monoid is isomorphic to $\NN^n$ for some $n$.

\begin{lem} \label{lem:integral} Let $P$ be an integral monoid.  Then for any submonoid $Q \subseteq P$, $Q$ and $P/Q$ are integral. \end{lem}

\begin{proof} Exercise. \end{proof}

\begin{lem} \label{lem:chartproduction} Let $A$ be an abelian group, $P$ an integral monoid, $\pi: P \to \ov{P}$ the sharpening map.  Suppose $h : A \to P^{\rm gp}$ is a group homomorphism such that the image of $\pi^{\rm gp} h : A \to \ov{P}^{\rm gp}$ contains $\ov{P}$.  Set $Q := h^{-1}(P)$ so that $h$ restricts to a monoid homomorphism $h : Q \to P$.  Then $Q^* = h^{-1}(P^*)$ and $\ov{h} : \ov{Q} \to \ov{P}$ is an isomorphism.  If $A$ and $\ov{P}$ are finitely generated, then $Q$ is fine. \end{lem}

\begin{proof} The equality $Q^* = h^{-1}(P^*)$ is clear.  The map $\ov{h}$ is surjective by the assumption on the image of $\pi^{\rm gp} h$.  For injectivity, suppose $\ov{h}(q_1) = \ov{h}(q_2)$ for $q_1,q_2 \in Q$.  Then $h(q_1) = h(q_2)+u$ for some $u \in P^*$.  But then $h^{\rm gp}(q_1-q_2) = u$ implies that $q_1-q_2 \in Q^*$, which implies $q_1=q_2$ in $\ov{Q}$.  The monoid $Q$ is always integral since it is a submonoid of $A$.  When $A$ is finitely generated, so is the subgroup $h^{-1}(P^*)=Q^*$.  If $\ov{P}$ is also finitely generated, then it is easy to see that $Q$ is finitely generated using finite generation of $Q^*$ and $\ov{Q} \cong \ov{P}$. \end{proof}

A monoid $P$ is called \emph{saturated} iff it is integral and whenever $n q \in P$ for some $q \in P^{\gp}$ and some positive integer $n$, we have $q \in P \subseteq P^{\gp}$.  For every integral monoid $P$, we let \be P^{\rm sat} & := & \{ q \in P^{\rm gp} : nq \in P {\rm \; for \; some \; } n > 0 \}. \ee  We extend this definition to possibly non-integral $P$ by setting $P^{\rm sat} := (P^{\rm int})^{\rm sat}$.  The map $P \to P^{\rm sat}$ is initial among monoid homomorphisms from $P$ to a saturated monoid and is injective when $P$ is integral.

We take this opportunity to emphasize that the functors $P \mapsto P^{\rm gp}$, $P \mapsto P^{\rm int}$, $P \mapsto P^{\rm sat}$, and $P \mapsto \ov{P}$ are all left adjoints to (and retractions of) inclusions of various full subcategories of $\Mon$.  In particular, all of these functors preserve direct limits---we will use this fact often without further comment.

\begin{lem} \label{lem:saturated}  Suppose $P$ is a saturated monoid.  If $P$ is sharp, then $P^{\rm gp}$ is torsion-free.  For any submonoid $Q \subseteq P$, the quotient $P/Q$ is saturated.  In particular, $\ov{P}^{\rm gp}$ is torsion-free. \end{lem}  

\begin{proof} For the first statement, suppose $p_1-p_2 \in P^{\rm gp}$ is torsion for some $p_1,p_2 \in P$.  Then $n(p_1-p_2)=0$ in $P^{\rm gp}$ for some $n>0$.  But then $n(p_2-p_1)$ is also zero in $P^{\rm gp}$, so, in particular, $n(p_1-p_2), n(p_2-p_1) \in P \subseteq P^{\rm gp}$, hence $p_1-p_2, p_2-p_1 \in P$ because $P$ is saturated.  But then $p_1-p_2$ is a unit in $P$, so if $P$ is sharp, $p_1-p_2=0$.  

For the second statement, first note that $F$ is integral as it is a submonoid of the integral monoid $P$ (Lemma~\ref{lem:integral}).  

For the third statement, we first note that $P/Q$ is integral by Lemma~\ref{lem:integral}.  Two elements $a,b \in P$ have the same image in $P/Q$ iff there are $q_1,q_2 \in Q$ with $a+q_1 = b + q_2$ in $P$.  A typical element of $(P/Q)^{\rm gp} = P^{\rm gp} / Q^{\rm gp}$ can be written as an equivalence class $[p_1-p_2]$ for $p_1,p_2 \in P$.  Fix $n > 0$.  To say that $n[p_1-p_2] \in (P/Q)$ is to say that there is a $p \in P$ and $q_1,q_2 \in Q$ such that $n(p_1-p_2) + q_1 = q_2 + p$ in $P^{\rm gp}$.  Add $(n-1)q_1$ to both sides and rearrange to conclude that \be n(p_1-p_2+q_1) = q_2 + (n-1)q_1+p \ee in $P^{\rm gp}$.  The right hand side is in $P$ and $P$ is saturated, so $p_1-p_2+q_1 \in P$.  But $[p_1-p_2+q_1] = [p_1-p_2]$ in $(P/Q)^{\rm gp}$ and $[p_1-p_2+q_1] \in (P/Q)$, so $[p_1-p_2] \in (P/Q)$ as desired. \end{proof}

\begin{defn} \label{defn:fs} A monoid $P$ is called \emph{fs} iff $P$ is both fine and saturated. \end{defn}

If $P$ is a sharp fs monoid, then $P^{\rm gp} \cong \ZZ^r$ is a finitely generated free abelian group.  It can be shown that the rank of $P^{\rm gp}$ coincides with the Krull dimension of $P$, defined in terms of chains of prime ideals of $P$ (\S\ref{section:idealsandfaces}), though we do not need this result. 

\begin{lem} \label{lem:splitting} Let $P$ be an integral monoid, $A \subseteq P^*$ a sub\emph{group}.  The quotient $P/A$ is also integral.  If the exact sequence $$0 \to A \to P^{\rm gp} \to (P/A)^{\rm gp} \to 0$$ of abelian groups splits, then the quotient map $P \to P/A$ admits a section $s : P/A \to P$ and any such section yields an isomorphism of monoids $P \cong A \oplus P/A$. \end{lem}

\begin{proof} The quotient $P/A$ is integral by Lemma~\ref{lem:integral}.  Set $Q := P/A$.  If $s :  Q^{\rm gp} \to P^{\rm gp}$ is a splitting of the sequence of groups, then it is easy to see that $s(q) \in P$ for $q \in Q \subseteq Q^{\rm gp}$ using the fact that $A$ is a group.  Similarly using the fact that $A$ is a group, one sees that $(\subseteq , s) : A \oplus Q \to P$ is an isomorphism. \end{proof}

\begin{lem} \label{lem:strictfiniteness} Let $P$ be a finitely generated monoid, $h : Q \to P$ a strict, surjective monoid homomorphism.  There is a finitely generated submonoid $S \subseteq Q$ such that $h|S : S \to P$ is strict and surjective. \end{lem}

\begin{proof} Choose a finite set of generators $p_1,\dots,p_k$ for $P$ and lifts $t_1,\dots,t_k$ of the $p_i$ to $Q$ and let $T \subseteq Q$ be the submonoid generated by the $t_i$, so that $h|T : T \to P$ is surjective and hence $\ov{h|T} : \ov{T} \to \ov{P}$ is also surjective.  Let $E \subseteq \ov{T}^2$ be the equalizer of $\ov{h|T}\pi_1,\ov{h|T}\pi_2 : \ov{T}^2 \to \ov{P}$ so that $E$ contains the diagonal copy of $\ov{T}$ and this containment is an equality iff $\ov{h|T}$ is an isomorphism.  Since $\ov{T}^2$ and $\ov{P}$ are finitely generated, it follows from general finiteness results (a finite inverse limit of finitely generated monoids is finitely generated) that $E$ is finitely generated, so there are elements $a_1,\dots,a_n,b_1, \dots,b_n \in T$ such that $(\ov{a}_1,\ov{b}_1),\dots,(\ov{a}_n,\ov{b}_n)$ generate $E$.  Since $h$ is strict, there are units $u_i \in Q^*$ such that $b_i = a_i+u_i$ for $i=1,\dots,n$.  Let $S$ be the submonoid of $Q$ generated by the $t_i$ and the $\pm u_i$, so $h|S :S \to P$ is surjective and any element of $S$ can be written in the form $t+\sum_{i=1}^n m_i u_i$ for $t \in T$, $m_i \in \ZZ$.  Then the inclusion $j : T \into S$ clearly induces a surjection $\ov{j} : \ov{T} \to \ov{S}$ and $\ov{h|S}\ov{j} = \ov{h|T}$, so if we define $F$ to be the equalizer of $\ov{h|S}\pi_1,\ov{h|S}\pi_2 : \ov{S}^2 \to \ov{P}$, then we have a surjection $E \to F$ so any set of generators for $E$ maps to a set of generators for $F$.  But our construction of $S$ ensures that the generators of $E$ map to the diagonal copy of $\ov{S}$ in $F$, so that $F$ is equal to this diagonal copy of $\ov{S}$ and hence $\ov{h|S}$ is an isomorphism, so $S$ is as desired. \end{proof}

\subsection{Ideals and faces}  \label{section:idealsandfaces} Recall \cite[1.3]{Ogus}, \cite[5.1]{Kat2} that an \emph{ideal} $I$ of a monoid $P$ is a subset $I \subseteq P$ so that $I+P \subseteq I$.  An ideal $\p \subseteq P$ is called \emph{prime} if its complement $P \setminus \p$ is a submonoid of $P$.  The complement of a prime ideal is called a \emph{face}.  Equivalently, a \emph{face} $F$ of $P$ is a submonoid $F \subseteq P$ whose complement is an ideal (necessarily prime).  Yet another equivalent formulation not mentioning ideals: a submonoid $F \subseteq P$ is a face iff, for all $p,q \in P$, $p+q \in F$ implies $p,q \in F$.  The \emph{codimension} of a face $F \subseteq P$ is the rank of the abelian group $(P/F)^{\rm gp}$ (defined when this abelian group is finitely generated).  The next two lemmas are easy exercises with the definitions.

\begin{lem} \label{lem:quotientbyface} Let $P$ be a monoid, $F \subseteq P$ a face.  Then $F=P$ iff $P/F=0$. \end{lem}

\begin{lem} \label{lem:sharpening} Let $P$ be a monoid, $F \subseteq P$ a face.  The natural map $F^{\rm gp} \to F^{-1} P$ is an isomorphism onto $(F^{-1} P)^*$ and we have $\ov{ F^{-1} P } = P/F$.  In particular, $P/F$ is sharp.  \end{lem}

\begin{lem} \label{lem:fssplitting} Let $P$ be an fs monoid, $F \subseteq P$ a face.  Then $(P/F)^{\rm gp}$ is free and there is a (non-canonical) isomorphism of monoids $F^{-1} P \cong F^{\rm gp} \oplus P/F$ compatible with the projection $F^{-1} P \to \ov{F^{-1}P}=P/F$.  In particular, when $F=P^*$ we obtain a splitting $P = P^* \oplus \ov{P}$.  \end{lem}

\begin{proof} The monoid $P/F$ is a sharp fs monoid by Lemma~\ref{lem:sharpening} and Lemma~\ref{lem:saturated}, so the group $(P/F)^{\rm gp} = P^{\rm gp} / F^{\rm gp}$ is free (Lemma~\ref{lem:saturated}), hence the natural exact sequence $$0 \to F^{\rm gp} \to P^{\rm gp} \to (P/F)^{\rm gp} \to 0 $$ splits (non-canonically).  Note that $P^{\rm gp} = (F^{-1} P)^{\rm gp}$, so we obtain the desired splitting of monoids by applying Lemma~\ref{lem:splitting} to the subgroup $F^{\rm gp} = (F^{-1} P)^*$. \end{proof}

\begin{lem} \label{lem:faces} Let $P$ be a monoid, $F \subseteq P$ a face of $P$.  If $P$ is finitely generated (resp.\ saturated, integral, fine, fs), then so is $F$.  In fact, if $p_1,\dots,p_n$ generate $P$, then $\{ p_i : p_i \in F \}$ generates $F$.  \end{lem}

\begin{proof}  After possibly reordering we can assume $p_1,\dots,p_k \in F$ and $p_{k+1},\dots,p_n \notin F$.  Then we claim $p_1,\dots,p_k$ generate $F$.  Given $f \in F$, we can write $f = \sum_{i=1}^n a_i p_i$ for $a_i \in \NN$.  Since $F$ is a face we must have $a_i = 0$ for $i>k$ (otherwise $p_i +b \in F$ for some $i>k$ for some $b \in P$, which would imply $p_i \in F$).  Obviously a face of an integral monoid is integral because any submonoid of an integral monoid is integral.  Suppose $P$ is saturated.  To see that $F$ is saturated, suppose $n(f_1-f_2) \in F \subseteq F^{\rm gp}$ for some $f_1,f_2 \in F$ for some $n>0$.  Since $P$ is saturated this implies $f_1-f_2 \in P$.  But if $f_1-f_2$ where not in $F \subseteq P$, then $n(f_1-f_2)$ could not be in $F$ because $F$ is a face.  \end{proof}

Every monoid has a smallest prime ideal, $\emptyset$, and a largest prime ideal $\m_P := P \setminus P^*$ (which coincide iff $P^*=P$ is a group), so every monoid is ``local," but not every \emph{morphism} of monoids is \emph{local} in the sense of:

\begin{defn} \label{defn:local} A morphism $h : P \to Q$ of monoids is called \emph{local} iff $f(\m_P) \subseteq \m_Q$. \end{defn}

For example, every map out of a group is local, but a map \emph{to} a group is local iff its domain is a group.

For a prime ideal $\p \subseteq P$ with complementary face $F := P \setminus \p$, we call $P_{\p} := F^{-1} P$ the \emph{localization} of $P$ at $\p$.  If $h : P \to Q$ is a monoid homomorphism and $\q$ is a prime of $Q$ with complementary face $G$ and $\p = h^{-1}(\q)$, then $h$ induces a local morphism $h : F^{-1} P \to G^{-1} Q$ as in the case of rings.  The proofs of the next two lemmas are left as exercises for the reader.

\begin{lem} \label{lem:local1} A map of monoids $h : P \to Q$ is local iff its sharpening $\ov{h}$ is local iff $\ov{h}^{-1}(0)=\{ 0 \}$.  In particular the sharpening map itself is local. \end{lem}

\begin{lem} \label{lem:local} Suppose $h : P \to Q$ is a surjective local map of monoids.  Then for any monoid map $f : Q \to R$, $f$ is local iff $fh$ is local. \end{lem}

\begin{thm} \label{thm:directlimitlocal} Let $\{ P_i, f_{ij} : P_i \to P_j \}$ be a functor to monoids with direct limit $Q$.  Suppose the transition maps $f_{ij}$ are local.  Then the structure maps $P_i \to Q$ are local and $Q$ is also the direct limit of the $P_i$ in the category of monoids with \emph{local} maps as the morphisms. \end{thm} 

\begin{proof} By Lemma~\ref{lem:local1} and the fact that sharpening preserves direct limits, we reduce to the case where the $P_i$ are sharp in the first part of the theorem.  The assumption that the $f_{ij}$ are local then means $f_{ij}^{-1}(0)=\{ 0 \}$.  Let $P := \oplus_i P_i$.  Clearly $P$ is sharp.  As a matter of notation we write an element $p \in P$ as a formal sum $\sum_i p_i e_i$ where $e_i$ is a formal symbol and all but finitely many $p_i$ are zero.  It is straightforward to see that the direct limit $Q$ (in $\Mon$, as opposed to sharp monoids) can be explicitly presented as $Q = P/\cong$ where $\cong$ is the smallest equivalence relation on $P$ satisfying: \begin{enumerate} \item $\cong$ is \emph{monoidal} in the sense that $$ \forall \; p,p',q,q' \in P \; ( p \cong p' \; {\rm and} \; q \cong q' ) \implies (p+q) \cong (p'+q').  $$  Equivalently, (the ``graph" of) $\cong$ is a submonoid of $P \times P$.  \item $p_i e_i \cong f_{ij}(p_i)e_j$ for every transition map $f_{ij}$  and every $p_i \in P_i$. \end{enumerate} Since $P_i \to P$ is certainly local, it suffices to prove that $P \to Q$ is local.  Since $P$ is sharp and $P \to Q$ is surjective, $P \to Q$ local is equivalent to: for any $p \in P$, $p \cong 0$ in $P$ implies $p=0$ in $P$.  It follows easily from $f_{ij}^{-1}(0) = \{ 0 \}$ that the relation $\sim$ on $P$ defined by \bne{simdefn}  p \sim p' & := & \exists i \, \exists p_i \in P_i \; p= p_i e_i \, \; {\rm and} \, \; p' = f_{ij}(p_i)e_j \ene satisfies the following condition:

\noindent {\bf (*)} For any $p \in P$, if $p \sim 0$ or $0 \sim p$, then $p=0$.

\noindent \emph{Claim:}  If $\sim$ is any relation on any sharp monoid $P$ satisfying {\bf (*)}, then \begin{enumerate} \item \label{refclosure} the reflexive closure of $\sim$ also satisfies {\bf (*)}. \item \label{symclosure} the symmetric closure of $\sim$ also satisfies {\bf (*)}.  \item \label{monclosure} the relation $\simeq$ on $P$ defined by \be p \simeq p' & := & \exists q,q',r,r' \in P : q \sim q', \; r \sim r', \; p = q+r, \; p' = q'+r' \ee also satisfies {\bf (*)}. \item \label{tranclosure} the transitive closure $\simeq$ of $\sim$ also satisfies {\bf (*)}.  \end{enumerate}

The claim implies the (first part of) the theorem because $p \cong 0$ implies $p \simeq 0$ for some relation $\simeq$ obtained from the relation $\sim$ of \eqref{simdefn} by applying finitely many ``closure operations" of the type mentioned in the claim.  (The point is that $\cong$ is obtained from $\sim$ by repeatedly applying this sequence of closure operations a denumerably infinite number of times; one can do this a bit more efficiently but that is unimportant now.)  Parts \eqref{refclosure} and \eqref{symclosure} of the claim are trivial.  For \eqref{monclosure}:  If $\sim$ satisfies {\bf (*)} and $p \simeq 0$ (witnessed by $q,q',r,r'$ as in the definition of $\simeq$) then $0 = q'+r'$ in the sharp monoid $P$ implies $q'=r'=0$.  By {\bf (*)}, the condition $q \sim q'=0$ (resp.\ $r \sim r'=0$) then implies $q = 0$ (resp.\ $r=0$), so $p=q+r=0$ as desired.  For \eqref{tranclosure}, $p \simeq 0$ implies there is a sequence $$p = p_1  \sim p_2 \sim \cdots \sim p_{n-1} \sim p_n =0.$$ First apply {\bf (*)} to see that $p_{n-1}=0$, then that $p_{n-2}=0$, and so forth to find that $p=0$. 

The second statement of the theorem follows easily from Lemma~\ref{lem:local} using the fact that $P \to Q$ is local and surjective. \end{proof}

\begin{defn} \label{defn:benign} A morphism $h : P \to Q$ of monoids is called \emph{benign} iff $Q$ is isomorphic, as a monoid under $P$, to the quotient of $P$ by a subgroup $A \subseteq P^*$. \end{defn}

For example, the sharpening map $P \to \ov{P}$ is benign.  The next result is an easy exercise.

\begin{lem} \label{lem:benign}  Benign maps are surjective local maps and are stable under pushout and composition. \end{lem}

\begin{lem} \label{lem:pushout} Consider a commutative diagram of monoids as below.  $$\xym{ G \ar[r]^f \ar[d]_k & Q \ar[d]^h \\ S \ar[r]^i & P }$$ If this square is a pushout, then the natural map $\Cok f \to \Cok i$ is an isomorphism.  The converse holds when $G$ and $S$ are groups, $P$ is integral, and $i$ is monic.  In particular, $$\xym{ Q^* \ar[r] \ar[d]_{h^*} & Q \ar[d]^h \\ P^* \ar[r] & P }$$ is a pushout diagram whenever $h$ is strict and $P$ is integral. \end{lem}

\begin{proof} The first statement is a formality---direct limits commute amongst themselves.  We now prove the converse under the indicated assumptions.  In general the pushout of maps of monoids is difficult to describe, but the pushout of two maps out of a group is easy.  In particular, the pushout $S \oplus_G Q$ can be described as the quotient of $S \oplus Q$ by the equivalence relation $\sim$ where $(s,q) \sim (s',q')$ iff there is some $g \in G$ such that $s=s'+k(g)$ and $q'=q+f(g)$.  We first show that the natural map \bne{natty} S \oplus_G Q & \to & P \\ \nonumber [s,q] & \mapsto &  s+h(q) \ene is injective (we now drop notation for the monomorphism $i$).  If $[s,q]$ and $[s',q']$ have the same image under \eqref{natty}, then we have $s+h(q)=s'+h(q')$ in $P$, which implies $h(q)$ and $h(q')$ are equal in $\Cok i$, which, by the assumption that $\Cok f \to \Cok i$ is an isomorphism, implies that $q$ and $q'$ have the same image in $\Cok f$, which implies that $q'=q+f(g)$ for some $g \in G$.  We then have \be s+h(q) & = & s'+h(q) + h(f(g)) \\ & = & s' + h(q) + k(g) \ee in $P$, which implies that $s = s' + k(g)$ because $P$ is integral, hence $[s,q]=[s',q']$.  For surjectivity of \eqref{natty}: Given $p \in P$, the assumption that $\Cok f \to \Cok i$ is an isomorphism implies that we can find $q \in Q$ so that $h(q)$ and $p$ have the same image in $\Cok i$, which, since $S$ is a group, implies that we can write $p = h(q)+s$ for some $s \in S$, hence $p$ is the image of $[s,q]$ under \eqref{natty}. \end{proof}

\subsection{Refinements} \label{section:monoidrefinements} Let $h : Q \to P$ be a morphism of integral monoids.  Let $R \subseteq Q^{\rm gp}$ be the preimge of $P$ under $h^{\rm gp} : Q^{\rm gp} \to P^{\rm gp}$ so that we have a commutative diagram \bne{exactnessdiagram} \xym@C+10pt{ Q \ar@/_1pc/[rdd] \ar[rd]^i \ar@/^1pc/[rrd]^h \\ & R \ar[r]^-p \ar[d] & P \ar[d] \\ & Q^{\rm gp} \ar[r]^-{h^{\rm gp}} & P^{\rm gp} } \ene with \emph{cartesian} square.  The ``vertical" arrows in this diagram (including $i$ but not $h$) are monic (in particular, $R \subseteq Q^{\rm gp}$ is integral) and $R^{\rm gp} = Q^{\rm gp}$.  When we consider this diagram for different monoid homomorphisms $h$ at the same time, we will write $R_h$ instead of $R$ to avoid confusion.

\begin{defn} \label{defn:monoidrefinement} The map $h$ is called \emph{exact} iff $i : Q \to R$ is an isomorphism.  The map $h$ is called a \emph{refinement} (resp.\ \emph{good refinement}, \emph{strong refinement}) iff $\ov{p} : \ov{R} \to \ov{P}$ has a section $\ov{s}$ satisfying $\ov{i} = \ov{s} \ov{h}$ (resp.\ $p$ is an isomorphism, $\ov{p}$ is an isomorphism). \end{defn}

Obviously a good refinement is a strong refinement and a strong refinement is a refinement.

\begin{lem} \label{lem:injectivepullback} Let $h : Q \to P$ be a morphism of monoids with $h^{\rm gp}$ surjective.  Then \bne{hN} h^* : \Hom_{\Mon}(P,R) & \to & \Hom_{\Mon}(Q,R) \ene is injective for any integral monoid $R$.  In particular, a map $h$ of integral monoids with $h^{\rm gp}$ surjective is an epimorphism in the category of integral monoids.\end{lem}

\begin{proof} Suppose $f,g : P \to R$ are monoid homomorphisms with $fh = gh$.  Then $f^{\rm gp} h^{\rm gp} = g^{\rm gp} h^{\rm gp}$ so $f^{\rm gp} = g^{\rm gp} =:k $ since $h^{\rm gp}$ is surjective.  Since $R$ is integral $i : R \into R^{\rm gp}$ is injective, so we can show $f=g$ by showing $if = ig$.  Let $j : P \to P^{\rm gp}$ be the natural map.   Then $if=ig=kj$ by naturality of the groupification. \end{proof}

\begin{lem} \label{lem:refinements} Let $h : Q \to P$ be a map of integral monoids.  Then: \begin{enumerate} \item \label{hexact} The map $h$ is exact iff $\ov{h}$ is exact. \item \label{sharpening} The map $h$ is a refinement (resp. strong refinement) iff $\ov{h}$ is a refinement (resp.\ strong refinement).  \item \label{fine} If $Q$ and $P$ are fine, then the monoid $R$ in \eqref{exactnessdiagram} is also fine.   \item \label{sat} If $P$ is saturated, then $R$ is also saturated.  \item \label{groupsurjective} If $h$ is a refinement, then $\ov{h}^{\rm gp}$ is surjective.  \item \label{sectionuniqueness}  If $h$ is a refinement and $\ov{s}_i$ ($i=1,2$) are sections of $\ov{p}$ satisfying $\ov{i} =  \ov{s}_i \ov{h}$, then $\ov{s}_1=\ov{s}_2$.  That is, a section $\ov{s}$ as in Definition~\ref{defn:monoidrefinement} is necessarily unique. \item \label{groupiso} If $h^{\rm gp}$ is an isomorphism, then $h$ is a good refinement. \item \label{saturationrefinement}For any integral monoid $Q$, the saturation map $Q \to Q^{\rm sat}$ is a good refinement.  \item \label{fsrefinement} If $P$ is fs and $\ov{h}^{\rm gp}$ is surjective, then $h$ is a refinement. \end{enumerate}  \end{lem}

\begin{proof} Statement \eqref{hexact} is a straightforward diagram chase in the diagram \eqref{exactnessdiagram} and its analog for $\ov{h}$.  One sees easily that $R_h/Q^* \to R_{\ov{h}}$ is an isomorphism (in particular, $R_h \to R_{\ov{h}}$ is strict) and from this it is easy to see that $i : Q \to R_h$ is surjective iff $j : \ov{Q} \to R_{\ov{h}}$ is surjective (note that $i$ and $j$ are always injective).  The statement \eqref{sharpening} is obvious from the fact that $R_h \to R_{\ov{h}}$ is strict.  For \eqref{fine} note that $R$ is finitely generated because it is a general fact that a finite inverse limit of finitely generated monoids is finitely generated.  For \eqref{sat}, suppose $q \in R^{\rm gp} = Q^{\rm gp}$ is such that $nq \in R$ for some $n \geq 1$.  Then $h^{\rm gp}(nq) = n h^{\rm gp}(q) \in P$ by definition of $R$, hence $h^{\rm gp}(q) \in P$ since $P$ is saturated, hence $q \in R$ by definition of $R$.  For \eqref{groupsurjective}, first note that $Q^* \subseteq R^*$ so the Snake Lemma applied to $$ \xym{ 0 \ar[r] & Q^* \ar[d] \ar[r] & Q^{\rm gp} \ar@{=}[d]^{i^{\rm gp}} \ar[r] & \ov{Q}^{\rm gp} \ar[d]^{\ov{i}^{\rm gp}} \ar[r] & 0 \\ 0 \ar[r] & R^* \ar[r] & R^{\rm gp} \ar[r] & \ov{R}^{\rm gp} \ar[r] & 0 } $$ implies that $\ov{i}^{\rm gp} : \ov{Q}^{\rm gp} \to \ov{R}^{\rm gp}$ is surjective.  Since $\ov{p}$ has a section $\ov{s}$ (by definition of refinement), so does $\ov{p}^{\rm gp}$ (namely $\ov{s}^{\rm gp}$), so, in particular, $\ov{p}^{\rm gp}$ is surjective, hence the composition $\ov{h}^{\rm gp} = \ov{p}^{\rm gp} \ov{i}^{\rm gp}$ is surjective.  For \eqref{sectionuniqueness}, apply Lemma~\ref{lem:injectivepullback} to $\ov{h} : \ov{Q} \to \ov{P}$ using the previous result to conclude $\ov{s}_1=\ov{s}_2$ from the equality $\ov{s}_1 \ov{h} = \ov{s}_1 \ov{h}$.  For \eqref{groupiso}: If $h^{\rm gp}$ is an isomorphism, then so is its pullback $p : R \to P$.  Clearly \eqref{groupiso} implies \eqref{saturationrefinement}.  For \eqref{fsrefinement}, note that $P$ fs implies $\ov{P}$ fs and $\ov{P}^{\rm gp}$ free (Lemma~\ref{lem:saturated}), so we can find a section $s : \ov{P}^{\rm gp} \to \ov{Q}^{\rm gp}$ of the surjection $\ov{h}^{\rm gp}$.  It is clear from the definition of $R_{\ov{h}}$ that $s|\ov{P} : \ov{P} \to \ov{Q}^{\rm gp}$ actually takes values in $R_{\ov{h}} \subseteq \ov{Q}^{\rm gp}$ and that the sharpening of this map $\ov{P} \to \ov{R}_{\ov{h}} = \ov{R}$ is a section of $\ov{p}$ with the desired property.  \end{proof}

\subsection{Spec} \label{section:Spec} The set of prime ideals (\S\ref{section:idealsandfaces}) in a monoid $P$ is denoted $\Spec P$.  If $p_1,\dots,p_n$ generate $P$, then the map \be \{ {\rm \; faces \; of \; } P \, \} & \to & \{ {\rm \; subsets \; of \; } \{ p_1, \dots, p_n \} \, \} \\   F & \mapsto & \{ p_i : p_i \in F \} \ee is injective by Lemma~\ref{lem:faces}.  In particular:

\begin{lem} If $P$ is a finitely generated monoid then $\Spec P$ is finite. \end{lem}

The set $\Spec P$ is topologized by taking the sets \be U_p & := & \{ \p \in \Spec P : p \notin \p \} \ee (for each $p \in P$)  as basic opens.  If we instead view $\Spec P$ as the set of faces of $P$, then \be U_p & = & \{ F \in \Spec P : p \in F \}. \ee  Recall that every monoid has a unique maximal ideal $\m_P = P \setminus P^*$ and minimal ideal $\emptyset$.  The only open neighborhood of the maximal ideal is the entire space $\Spec P$, so the global section functor for sheaves on $\Spec P$ coincides with the stalk functor at the maximal ideal, hence it is exact.  

The space $\Spec P$ carries a sheaf of monoids $\M_P$ characterized by \be \M_P(U_p) & = & P_p, \ee and constructed in much the same way one would construct the structure sheaf of an affine scheme $\Spec A$ (\cite[Page~70]{H}).  We will construct this structure sheaf as a special case of more general results in \S\ref{section:Specrevisited}, but it is not hard to do directly.  The stalk of $\M_P$ at a prime $\p \in \Spec P$ with complementary face $F$ is given by \bne{stalkformula} \M_{P,\p} & = & F^{-1} P. \ene  A map of monoids $h : Q \to P$ induces a continuous map of topological spaces \be \Spec h : \Spec P & \to & \Spec Q \\ \nonumber \p & \mapsto & h^{-1}(\p). \ee  There is also an induced map $(\Spec h)^{-1} \M_Q \to \M_P$ of sheaves of monoids on $\Spec P$ whose stalk \be h_\p : \M_{Q,(\Spec h)(\p)} & \to & \M_{P,\p} \ee is identified with the natural map \be Q_{h^{-1} \p} & \to & P_{\p} \\ \nonumber [q,s] & \mapsto & [h(q),h(s)]. \ee  This map is clearly a \emph{local} map of monoids.

\begin{example} \label{example:Specgroup} If $A$ is a group, $\Spec A$ is the one point space with structure ``sheaf" $A$. \end{example}

In any finite space $X$ every point $x$ has a smallest open neighborhood $U_x$ and the stalk functor $F \mapsto F_x$ coincides with the section functor $F \mapsto F(U_x)$, so any map of sheaves inducing an isomorphism on stalks at $x$ actually induces an isomorphism on a neighborhood of $x$.  In fact, suppose $P$ is finitely generated and $\p \in \Spec P$ is a prime ideal in $P$ with complementary face $F$.  Then $F$ is also finitely generated (Lemma~\ref{lem:faces}), say by $f_1,\dots,f_n$, so that the set \be U & := & \cap_{f \in F} U_f \\ & = & \cap_{i=1}^n U_{f_i} \ee is open.  Since $U$ is the intersection of all basic opens containing $\p$, it is the smallest neighborhood of $\p$ in $\Spec P$ and, moreover, it is ``affine" in the sense that the natural map \be \Spec F^{-1} P & \to & \Spec P \ee is an isomorphism onto $U$.

\begin{lem} If $h :  Q \to P$ is surjective, then $\Spec h$ is an embedding (not necessarily closed!) of spaces. \end{lem}

\begin{proof} Since $h$ is surjective, $h^{-1}(\p) = h^{-1}(\q)$ clearly implies $\p=\q$, so $\Spec h$ is monic.  For $p \in P$, if we choose a lift $q \in Q$ with $h(q)=p$, then $(\Spec h)^{-1}(U_q) = U_p$, so every basic open subset of $\Spec P$ is obtained by intersecting an open subset of $\Spec Q$ with $\Spec P$, hence $\Spec h$ is an embedding. \end{proof}

\begin{example} \label{example:Specsurjectionnotclosed} If $h : \NN^2 \to \NN$ is the addition map $h(a,b) := a+b$, then $h$ is certainly surjective.  The image of $\Spec h$ consists of the closed point and the generic point of $\Spec( \NN^2 )$.  This image is certainly not closed since it contains the generic point but fails to contain two other points. \end{example}

\begin{lem} \label{lem:sharpeninghomeo} The sharpening map $f : P \to \ov{P}$ induces a homeomorphism $\Spec \ov{P} \to \Spec P$. \end{lem}

\begin{proof} Certainly $f$ is surjective, so $\Spec f$ is an embedding by the previous lemma, so it is enough to show that $\Spec f$ is surjective.  Given $\p \in \Spec P$, one checks easily that $f(\p)$ is a prime ideal of $\Spec \ov{P}$ with $f^{-1}(f(\p)) = \p$. \end{proof}

The quotient sheaf $\ov{\M}_P := \M_P / \M_P^*$ is a sheaf of \emph{sharp} monoids.  For many purposes, the sheaf $\ov{\M}_P$ on $\Spec P$ is more useful than the sheaf $\M_P$.  The homeomorphism $\Spec \ov{P} \to \Spec P$ induced by sharpening also induces an isomorphism of sheaves of monoids $\ov{\M}_P \to \ov{\M}_{\ov{P}}$.  The stalk of $\ov{\M}_P$ at a prime ideal $\p \in \Spec P$ with complementary face $F$ is given by \be \ov{\M}_{P,\p} & = & \ov{P_{\p}} \\ & = & \ov{F^{-1}P} \\ & = & P/F. \ee

We will return to our study of $\Spec$ in \S\ref{section:Specrevisited} after we properly set up the category in which $(\Spec P,\M_P)$ will live.

\begin{example} \label{example:SpecN}  The prototypical example to keep in mind is $\Spec \NN$.  The monoid $\NN$ has only the two obvious prime ideals: $\emptyset$ and $\m = \NN \setminus \{ 0 \}$.  The basic open sets $U_n$ are given by \be U_0 & = & \{ \emptyset, \m \} \\ U_{n \neq 0} & = & \{ \emptyset \} . \ee  Evidently then, the topological space $\Spec \NN$ is the two point ``Sierpinski space" where $\m$ is the closed point and $\emptyset$ is the generic point.  Since $\m \in \{ \emptyset \}^-$, we have a specialization map $F_{\m} \to F_{\emptyset}$ on stalks for any sheaf $F$ on $\Spec \NN$.  On a finite topological space, the specialization maps uniquely determine a sheaf, so the category of sheaves on $\Spec \NN$ is just the category of maps of sets.  The structure sheaf $\M_{\NN}$ is given by \be \M_{\m} & = & \NN_{\m} \\ & = & \{ 0 \}^{-1} \NN \\ & = & \NN \\ \M_{\emptyset} & = & \NN_{\emptyset} \\ & = & \ZZ, \ee with the obvious specialization map $\NN \to \ZZ$.  After sharpening, the structure sheaf becomes \be \ov{\M}_{\m} & = & \NN \\ \ov{\M}_{\emptyset} & = & \{ 0 \} , \ee with specialization map $\NN \to \{ 0 \}$. \end{example}

\subsection{Monoid algebra} \label{section:monoidalgebra} For a monoid $P$, let $\ZZ[P]$ be the free abelian group on $\{ [p] : p \in P \}$.  The unique $\ZZ$-bilinear map $\ZZ[P] \times \ZZ[P] \to \ZZ[P]$ satisfying $([p],[q]) \mapsto [p+q]$ serves as the multiplication map for a ring structure on $\ZZ[P]$ with multiplicative identity $1 = [0]$.  This defines a functor \bne{ZP} : \Mon & \to & \An \\ \nonumber P & \mapsto & \ZZ[P] \ene which is left adjoint to the forgetful functor $\An \to \Mon$ obtained by viewing a ring as a monoid under multiplication.  In other words, we have a bijection \bne{ZPAnadjunction} \Hom_{\An}(\ZZ[P],A) & = & \Hom_{\Mon}(P,A) \ene natural in $P \in \Mon$, $A \in \An$.

\subsection{Modules} \label{section:modules} For a monoid $P$, a \emph{module} over $P$ is a set $M$ equipped with an \emph{action} map \be P \times M & \to & M \\ (p,m) & \mapsto & p \cdot m \ee such that $0 \cdot m = m $ and $(p_1+p_2)\cdot m = p_1\cdot(p_2 \cdot m)$ for all $m \in M$, $p_1,p_2 \in P$.

\begin{example} An $\NN$ module is the same thing as a set $M$ equipped with an endomorphism $f : M \to M$.  The datum $(M,f)$ corresponds to the action $n \cdot m := f^n(m)$ on $M$ and one recovers an endomorphism $f : M \to M$ from an $\NN$ module $M$ by setting $f(m) := 1 \cdot m$. \end{example}

Modules over $P$ form a category $\Mod(P)$ where a morphism $M \to N$ is a map respecting the actions.  The category $\Mod(P)$ has all (small) limits.  The categorical product of $P$ modules $M_i$ is the set-theoretic product $\prod_i M_i$ with coordinate-wise $P$ action.  A general inverse limit is the set-theoretic inverse limit with the action it inherits from the embedding in the corresponding product.  The direct sum of $P$ modules $M_i$ is the set-theoretic disjoint union $\coprod_i M_i$.  Unlike the category of modules over a ring, finite products and coproducts in $\Mod(P)$ do \emph{not} coincide.  In particular, $\Mod(P)$ is not an abelian category.  The coequalizer of $f,g : M \rightrightarrows N$ is the quotient of $N$ by the smallest equivalence relation on $N \times N$ which is also a $P$ submodule.

A monoid $P$ becomes a module over itself by using addition in $P$ as the action.  If $P \to Q$ is a map of monoids, then $Q$ becomes a $P$ module similarly.  In particular, $P^{\rm gp}$ is naturally a $P$ module.  A module $M$ is \emph{finitely generated} iff there is a surjective map of $P$ modules $\coprod_n P \to M$ for some finite $n$.  

An ideal (in the sense of \S\ref{section:idealsandfaces}) of $P$ is a submodule of $P$.  A \emph{fractional ideal} is a $P$ submodule of $P^{\rm gp}$.  Just as in the case of modules over a ring, the forgetful functor $\Mod(P) \to \Sets$ admits a left adjoint $S \mapsto \coprod_S P$.  A module in the essential image of the latter functor is called \emph{free}.  Equivalently, $M \in \Mod(P)$ is free iff there is a subset $S \subseteq M$ (called a \emph{basis}) such that every element $m \in M$ can be uniquely written as $m = p \cdot s$ for $p \in P$, $s \in S$.  A $P$ module $M$ is called \emph{flat} iff it can be written as a filtered direct limit of free $P$-modules.  For more on flat and free modules, see \cite{logflatness}.

\begin{example} \label{example:freemodules} If $A \to B$ is an injective map of abelian groups, then $B$ is free as an $A$-module.  One can take as a basis any set $S \subseteq B$ containing exactly one representative of each element of $B/A$. \end{example}

For a $P$ module $M$, the free abelian group $\ZZ[M]$ on $M$ becomes a module (in the usual sense) over the monoid algebra $\ZZ[P]$ (\S\ref{section:monoidalgebra}).  This defines a functor \bne{ModPtoModZP} \ZZ[ \slot ] : \Mod(P) & \to & \Mod( \ZZ[P] ) \ene which clearly takes finitely generated modules to finitely generated modules and (proper) submodules to (proper) submodules.  The functor \eqref{ModPtoModZP} takes free (resp.\ flat) $P$ modules to free (resp.\ flat) $\ZZ[P]$-modules.  Since $\ZZ[P]$ is a noetherian ring when $P$ is a finitely generated monoid, it follows that any submodule of a finitely generated $P$ module is itself finitely generated (check the ACC).  The functor $\ZZ[ \slot ]$ admits a right adjoint ``forgetful functor" \bne{forget} \Mod( \ZZ[P] ) & \to & \Mod(P) \ene taking a $\ZZ[P]$ module $N$ to $N$ regarded as a $P$ module via the action $p \cdot n := [p]n$, where $[p] \in \ZZ[P]$ is the image of $p \in P$ in $\ZZ[P]$ and the juxtaposition is the action of $\ZZ[P]$ on $N$.  In particular, $\ZZ[ \slot ]$ preserves direct limits.

If $I$ is an ideal of a finitely generated monoid $P$, then we saw above that $I$ is finitely generated as a $P$ module, so we can find a finite subset $S \subseteq I$ inducing a surjection of $P$ modules $\coprod_S P \to I$.  That is, every $i \in I$ can be written in the form $i=s+p$ for some $s \in S$, $p \in P$.

\subsection{Saturation and density} \label{section:saturationanddensity}  Here we make some general remarks about the analogues (for monoids) of finite and integral morphisms of rings.  These results will be useful in \S\ref{section:propertiesofrealizations}.

\begin{defn} \label{defn:denseandfinite} A map of monoids $h : Q \to P$ is called \dots \begin{enumerate}[label=\dots] \item  \emph{saturated} iff, for any $p \in P$ with $np \in h(Q)$ for a positive integer $n$, we have $p \in h(Q)$.  \item \emph{dense} iff, for all $p \in P$, there is a positive integer $n$ (possibly depending on $p$) so that $np \in h(Q)$.  \item \emph{finite} iff $h$ makes $P$ a finitely generated $Q$-module. \end{enumerate} \end{defn}

For example, an integral monoid $P$ is saturated iff the map $P \to P^{\rm gp}$ is saturated in the above sense.

\begin{thm} \label{thm:dense}  {\bf (Gordan's Lemma)}.  Let $h : Q \to P$ be a map of monoids.  \begin{enumerate} \item \label{dense1} If $h$ is finite, then $P^{\rm gp} / Q^{\rm gp}$ is finite.  \item \label{dense2} If $P$ is finitely generated and $h$ is dense, then $h$ is finite. \item \label{dense3} If $Q$ is finitely generated, $P$ is integral, and $h$ is finite, then $P$ is fine and $h$ is dense. \item \label{dense4} If $Q$ is fine, $P$ is integral, $h$ is dense, and $P^{\rm gp}$ is finitely generated, then $h$ is finite. \item \label{dense5} If $Q$ is a fine monoid and $Q \into G$ is any injective monoid homomorphism from $Q$ to a finitely generated abelian group $G$, then the saturation $P$ of $Q$ in $G$ is a finitely generated $Q$-module. \end{enumerate} \end{thm}

\begin{example} \label{example:finitenotdense} Before giving the proof, let us give an example of an injective monoid homomorphism $\NN \to P$ which is finite but not dense.  (In this situation $P$ cannot be integral by part \eqref{dense3}.)  As a set, take $P := \NN \coprod \ZZ_{>0}s$.  The additional law on $P$ is defined so that $\NN \subseteq P$ is a submonoid (in fact a face) and so that $n + ms = (n+m)s$ and $ms+m's = (m+m')s$ for all $n \in \NN$, $m,m' \in \ZZ_{>0}$.  Then $P$ is generated as an $\NN$ module by $\{ 0, s \} \subseteq P$, but there is no $n \in \ZZ_{>0}$ for which $ns \in \NN$.  \end{example}

\begin{proof} (c.f.\ \cite[2.2.5]{Ogus}) For \eqref{dense1}, suppose $p_1,\dots,p_n$ generate $P$ as a $Q$-module.  This means every $p \in P$ can be written as $p = h(q)+p_i$ for some $q \in Q$ and some $i \in \{ 1, \dots, n \}$.  It is immediate that every element of $P^{\rm gp} / Q^{\rm gp}$ is equal to (the image of) $p_i-p_j$ for some $i,j \in \{ 1, \dots, n \}$.  

For \eqref{dense2}, suppose $p_1,\dots,p_n$ generate $P$ as a monoid.  Since $h$ is dense, we can write $n_ip_i = h(q_i)$ for some positive integers $n_i$ and some $q_i \in Q$.  Let $S \subseteq P$ be the set of all elements of $P$ which can be written as $\sum_{i=1}^n a_i p_i$ with $a_i \in \{ 0, \dots, n_i-1 \}$.  The set $S$ is clearly finite.  It remains to show that $S$ generates $P$ as a $Q$-module.  Any $p \in P$ can be written $\sum_{i=1}^n b_i p_i$ for some $b_i \in \NN$.  Write $b_i = m_in_i+a_i$ with $m_i \in \NN$, $a_i \in \{ 0,\dots, n_i-1 \}$.  Then $p = h(q)+s$ with $q = \sum_{i=1}^n m_i q_i$, $s = \sum_{i=1}^n a_i p_i \in S$. 

For \eqref{dense3}:  Obviously $P$ is finitely generated since $Q$ is finitely generated and $h$ is finite.  To see that $h$ is dense, consider an arbitrary $p \in P$ and let $Q_n$ be the $Q$-submodule of $P$ generated by $0,p,\dots,np$ so that we have an ascending chain of $Q$-submodules $$Q = Q_0 \subseteq Q_1 \subseteq Q_2 \subseteq \cdots $$ of $P$.  This ascending chain must be eventually constant (otherwise $\ZZ[ \slot ]$ of it would be an infinite strictly ascending chain of $\ZZ[Q]$-submodules of the finitely generated $\ZZ[Q]$-module $\ZZ[P]$, which can't happen because $\ZZ[Q]$ is noetherian) so for a large enough $n$ we have $Q_n = Q_{n-1}$, hence we can write $np = kp+h(q)$ for some $k < n$ and some $q \in Q$, hence by integrality of $P$ we find $(n-k)p = h(q) \in Q$.  

We reduce \eqref{dense4} to \eqref{dense5} as follows:  By replacing $Q$ with its image, we can assume $h$ is injective (note that the image $h(Q)$ is still fine).  Since a submodule of a finitely generated $Q$-module is finitely generated (by an ascending chain argument as in \eqref{dense3}), it suffices to show that the saturation of $Q$ in $P^{\rm gp}$ is finitely generated (since $P$ is a $Q$-submodule of the latter).  

We prove \eqref{dense5} in three steps:

\noindent {\bf Step 1:}  Reduction to the case where $Q$ is fine \emph{and sharp}.  For this step, we first observe that $P/Q^*$ is equal (as a submonoid of $G/Q^*$) to the saturation of $\ov{Q} = Q/Q^*$ in $G/Q^*$.  If this latter saturation is generated as a $\ov{Q}$-module by $\ov{g}_1,\dots,\ov{g}_n$, then we check easily that any lifts $g_1,\dots,g_n \in G$ of the $\ov{g}_i$ in fact lie in $P$ and generate $P$ as a $Q$-module.

\noindent {\bf Step 2:} Reduction from the case where $Q$ is sharp and fine to the case where $Q$ is sharp and fine \emph{and} $G$ \emph{is torsion-free}.  For this step, we write $G = \ZZ^n \oplus T$ where $T$ is finite.  Let $Q' := Q \cap \ZZ^n$, viewing $Q$ and $\ZZ^n$ as submonoids of $G$.  We appeal to the general fact that a finite inverse limit of finitely generated monoids is finitely generated to see that $Q$ is finitely generated; it is then obvious that $Q'$ is sharp and fine (since $Q' \subseteq Q$) and that $A := (Q')^{\rm gp}$ is torsion-free (since $Q' \subseteq \ZZ^n$, so $A \subseteq \ZZ^n$).  Let $P'$ denote the saturation of $Q'$ in $\ZZ^n$.  We see easily that $P = P' \oplus T$, so if we knew that $P'$ was generated as a $Q'$-module by $p_1',\dots,p_n'$, then we would find that $P$ is generated as a $Q$-module by the finitely-many elements $(p_i',t) \in P$ (with $t \in T$ arbitrary). 

\noindent {\bf Step 3:} The case where $Q$ is sharp and fine and $G \cong \ZZ^m$ is torsion-free.  Since $G$ is torsion-free, $G \into G_{\RR} := G \otimes_{\ZZ} \RR$ is injective.  Let \be C(Q) & := & \left \{ \sum_i q_i \otimes \lambda_i \in G_\RR : q_i \in Q, \lambda_i \in \RR_{\geq 0}  \right \} \ee be the cone over $Q$.  It is easy to see that $P = C(Q) \cap G$ is the monoid of ``lattice points" of this cone.  If we choose any finite set of generators $q_1,\dots,q_n$ for $Q$, then $$ G \cap \left \{ \sum_{i=1}^n q_i \otimes \lambda_i :  \lambda_i \in [0,1] \right \} $$ is a finite subset of $P$ and the classical Gordan's Lemma argument (much like the proof of \eqref{dense2}) shows that this subset of $P$ generates $P$ as a $Q$-module. \end{proof}

\begin{example} Even if $h : Q \to P$ is a finite map of finitely generated monoids, one cannot conclude that $h$ is dense.  Indeed, one cannot even conclude this when $Q = \{ 0 \}$ and $|P|=2$: Taking $P$ equal to the ``unique" monoid $P = \{ 0,1 \}$ with two elements not isomorphic to $\ZZ/2 \ZZ$ yields the desired counterexample because $n1 = 1$ for all positive integers $n$ in this monoid $P$. \end{example}

\subsection{Tensor product} \label{section:tensorproduct}  Let $P$ be a monoid and let $M,N,T$ be $P$ modules.  A function $f : M \times N \to T$ is called $P$-\emph{bilinear} iff \bne{bilinear} f(p \cdot m, n) & = & p \cdot f(m,n) \\ \nonumber f(m, p \cdot n) & = & p \cdot f(m,n) \ene for every $m \in M$, $n \in N$, $p \in P$.  Let $\Bilin_P(M \times N,T)$ denote the set of $P$-bilinear maps from $M \times N$ to $T$.  If $T$ is a $\ZZ[P]$ module, then it is clear that \bne{bilin} \Bilin_P(M \times N, T) & = & \Bilin_{\ZZ[P]}(\ZZ[M] \times \ZZ[N],T) , \ene where, on the left, $T$ is regarded as a $P$ module via the forgetful functor \eqref{forget}, and the right side is the set of bilinear maps of $\ZZ[P]$ modules in the usual sense.

\begin{prop} \label{prop:tensorproduct} For any $M,N \in \Mod(P)$, there is a $P$ module $M \otimes_P N$, unique up to unique isomorphism, with the following universal property:  There is a $P$-bilinear map $\tau : M \times N \to M \otimes_P N$ such that any $P$-bilinear map $f : M \times N \to T$ factors uniquely as $\ov{f} \tau$ for a $P$ module map $\ov{f} : M \otimes_P N \to T$. \end{prop}

\begin{proof}  The uniqueness argument is standard.  For existence, define $M \otimes_P N$ to be the quotient of $M \times N$ by the smallest equivalence relation $\sim$ enjoying the following two properties: \begin{enumerate} \item \label{sim1} $(p \cdot m,n) \sim (m, p \cdot n)$ for every $p \in P$, $m \in M$, $n \in N$. \item \label{sim2} If $(m_1,n_1) \sim (m_2,n_2)$ for some $m_i \in M$, $n_i \in N$, then $(p \cdot m_1,n_1) \sim (p \cdot m_2,n_2)$ for every $p \in P$. \end{enumerate}  For $(m,n) \in M \times N$, let $m \otimes n$ denote the image of $(m,n)$ in $M \otimes_P N$.  Regard $M \otimes_P N$ as a $P$ module using the action $p \cdot (m \otimes n) := (p \cdot m) \otimes n$.  This is well-defined since $\sim$ satisfies \eqref{sim2} and clearly satisfies the requisite property \be (p_1+p_2) \cdot (m \otimes n) & = & p_1 \cdot (p_2 \cdot (m \otimes n)) \ee for an action.  If we define $\tau : M \times N \to M \otimes_P N$ by $\tau(m,n) := m \otimes n$, then $\tau$ is $P$-bilinear because $\sim$ satisfies \eqref{sim1}.  

Suppose $f : M \times N \to T$ is $P$-bilinear.  Define an equivalence relation $\cong$ on $M \times N$ by declaring $(m_1,n_1) \cong (m_2,n_2)$ iff $f(m_1,n_1) = f(m_2,n_2)$.  It is clear from bilinearity that $\cong$ satisfies \eqref{sim1} and \eqref{sim2}, so, since $\sim$ is the smallest equivalence relation satisfying these properties, we have $$(m_1,n_1) \sim (m_2,n_2) \implies f(m_1,n_1) = f(m_2,n_2) $$ and we can therefore define a function $\ov{f} : M \otimes_P N \to T$ by $\ov{f}(m \otimes n) := f(m,n)$.  It is clear that this $\ov{f}$ is a $P$ module map and that $f = \ov{f} \tau$.  The uniqueness of $\ov{f}$ is automatic because $\tau$ is surjective (this is one place where the tensor product of $P$ modules is a little easier than the tensor product of modules over a ring). \end{proof}

The module $M \otimes_P N$ of Proposition~\ref{prop:tensorproduct} will be called the \emph{tensor product} of the $P$ modules $M$ and $N$.  Using the universal properties of tensor products (for modules over rings and monoids), formula \eqref{bilin}, and the adjointness of \eqref{ModPtoModZP} and \eqref{forget}, we obtain a natural isomorphism of $\ZZ[P]$ modules \bne{tensorformula} \ZZ[ M \otimes_P N ] & = & \ZZ[M] \otimes_{\ZZ[P]} \ZZ[N] \ene by showing that both sides have the same maps to any $\ZZ[P]$ module $T$.

\begin{example} \label{example:freemodule} Just as in the case of rings, it follows formally from the universal properties of free modules and tensor products that the free $P$ module $\coprod_S P$ on a set $S$ is obtained via base change from the free module on $S$ over the initial monoid $\{ 0 \}$: \be \coprod_S P & = & P \otimes_{ \{ 0 \} } ( \coprod_S \{ 0 \} ). \ee  Of course this does not say much since $\coprod_S \{ 0 \} = S$ and the above equality is clear from the construction (or characteristic property) of the tensor product. \end{example}

We will be most interested in the tensor product in the following situation.  Suppose $S$ is a submonoid of $P$ and $M$ is a $P$ module; then $P/S$ becomes a $P$ module via the projection $P \to P/S$ (c.f.\ \S\ref{section:modules}).  We will often use the notation \be M/SM & := & M \otimes_P (P/S). \ee  The tensor product $M/SM$ can be explicitly described as the quotient of $M$ by the equivalence relation $\sim$ where $m \sim m'$ iff $s \cdot m = s' \cdot m'$ for some $s,s' \in S$.  If we denote the image of $m \in M$ in $M/SM$ by $\ov{m}$ and the image of $p \in P$ in $P/S$ by $\ov{p}$ then $M/SM$ is a $P/S$ module via the action $\ov{p} \cdot \ov{m} := \ov{p \cdot m}.$  It is straightforward to check that this is well-defined and that \be \tau : M \times P/S & \to & M/SM \\ (m,\ov{p}) & \mapsto & \ov{p} \cdot \ov{m} \ee is a bilinear map with the requisite universal property.

Even more particularly, suppose $S = P^* \subseteq P$ and $M$ is a $P$ module.  Then the notation $$ \ov{M} := M/P^*M = M \otimes_P \ov{P} $$ is quite natural.

Another important special case of the tensor product is \emph{localization}.  Let $S$ be a submonoid of a monoid $P$, and let $P \to S^{-1} P$ be the localization of $P$ at $S$ (\S\ref{section:monoidbasics}).  For a $P$ module $M$, the tensor product $M \otimes_P S^{-1}P$ is usually denoted $S^{-1} M$ and has the ``expected" description.  Elements $[m,s]$ of $S^{-1} M$ are equivalence classes of pairs $(m,s)$ where $m \in M$ and $s \in S$.  Two such pairs $(m,s)$ and $(m',s')$ are equivalent iff \be (t + s') \cdot m & = & (t+s) \cdot m' \ee for some $t \in S$.  The $P$ module $S^{-1}M$ is also a module over the localized monoid $S^{-1} P$ (\S\ref{section:monoidbasics}) via the action $[p,s] \cdot [m,s'] := [p \cdot m, s+ s'].$

Starting with the tensor product, one can carry out the usual constructions familiar from the category of modules over a ring.  For example, the forgetful functor $P /\Mon \to \Mod(P)$ taking $P \to Q$ to $Q \in \Mod(P)$ admits a left adjoint \bne{Sym} \Sym_P^* : \Mod(P) & \to & P / \Mon \\ \nonumber M & \mapsto & \Sym^*_P M. \ene As in the case of rings, $\Sym_P^* M$ comes with an $\NN$ grading (c.f.\ \S\ref{section:Proj}) \bne{Symgrading} \Sym^*_P M & = & \coprod_{n \in \NN} \Sym^n_P M. \ene  The equality \be \Sym^*_P(M \coprod N) & = & ( \Sym_P^* M) \oplus_P (\Sym_P^* N) \ee follows formally as in the case of rings.

\begin{example} \label{example:freesymmetricmonoid} The symmetric monoid $\Sym_P^*( \coprod_S P)$ on the free $P$ module on a set $S$ can be described explicitly as follows: Let $S^{[n]} := S^n / \mathfrak{S}_n$ denote the set of $n$-element ``multisubsets" of $S$.  Then \be \Sym_P^* ( \coprod_S P) & = & \coprod_n ( P \times S^{[n]} ) \ee with addition law $(p,M)+(q,N) = (p+q,M+N)$.  The addition in the first coordinate here is the one for $P$, while the addition in the second coordinate is the addition law $S^{[m]} \times S^{[n]} \to S^{[m+n]}$ for multisubsets of $S$.  In particular, if we write $S$ instead of $\coprod_S  \{ 0 \} $ for the free $\{ 0 \}$ module on a set $S$, then we see that  \be \Sym_{ \{ 0 \} }^* S & = & \coprod_n S^{[n]}. \ee is the monoid of all multisubsets of $S$, graded by ``number of elements" (counting repeats) in the multiset.  We can alternatively describe this monoid of multisubsets of $S$ as the set of functions $S \to \NN$ with finite support, under the operation of coordinate-wise addition of functions.  The $\NN$ grading corresponding to the above coproduct decomposition is then $|f| = \sum_{s \in S} f(s)$.  If $S$ is finite, then this monoid is just $\NN^S$ with the grading given by the sum of the coordinates.  \end{example}

\section{Differentiable spaces} \label{section:differentiablespaces}  The purpose of this section is to recall some basic facts about the category $\DS$ of differentiable spaces.  This category provides a natural setting for differential geometry, incorporating singular spaces.  See \cite{GS} or \cite{diffspace} for further details.

\subsection{Basic notions}  \label{section:basicnotions} If $X$ is a locally ringed space and $I \subseteq \O_X$ is a sheaf of ideals, we define a new locally ringed space $\Z(I)$, called the \emph{zero locus} of $I$, as follows.  As a topological space, \be \Z(I) & := & \{ x \in X : I_x \subseteq \m_x \},  \ee with the topology inherited from $X$. (We use the usual notation for stalks of sheaves and for the unique maximal ideal $\m_x$ of the local ring $\O_{X,x}$.)  It is easy to see that $\Z(I)$ is closed in $X$.  By definition, the structure sheaf of $\Z(I)$ is given by \be \O_{\Z(I)} & := & i^{-1} (\O_X / I), \ee where $i : \Z(I) \into X$ is the inclusion.  Note that $i^{-1}( \O_X / I) = i^{-1} \O_X / i^{-1} I$ because $i^{-1}$ preserves finite limits.  This space $\Z(I)$ is a locally ringed space, and $i : \Z(I) \to X$ becomes a morphism of locally ringed spaces by defining $i^{\sharp}: i^{-1} \O_X \to \O_{\Z(I)}$ to be the natural quotient map.  The locally ringed space $\Z(I)$ represents the presheaf taking a locally ringed space $U$ to the set of $\LRS$ morphisms $f : U \to X$ such that $f^\sharp : f^{-1} \O_X \to \O_U$ kills $f^{-1} I$.

Let $X$ be a locally ringed space, $I \subseteq \O_X$ an ideal.  Let $\ov{I}_x$ denote the completion of the stalk $I_x$ in the topology it inherits from the $\m_x$-adic topology on $\O_{X,x}$ via the inclusion $I_x \subseteq \O_{X,x}$.  (This topology is not generally the same as the $\m_x$-adic topology on $I_x$.)  There is a natural inclusion $\ov{I}_x \subseteq \hat{\O}_{X,x}$, where $\hat{\O}_{X,x}$ denotes the $\m_x$-adic completion.  Let $t_x : \O_{X,x} \to \hat{\O}_{X,x}$ denote the natural map.  We often call it the ``Taylor series" map.  We say that $I$ is \emph{closed} iff the following holds: For any open subseteq $U \subseteq \O_X$ and any $f \in \O_X(U)$, if $t_x(f_x) \in \ov{I}_x$ for all $x \in U$, then $f \in I(U)$.

Recall \cite{GS}, \cite{diffspace} that a \emph{differentiable space} is a locally ringed space $X$ over $\RR$ locally isomorphic to the zero locus $\Z(I)$ of a closed ideal $I \subseteq \O_U$ for some open subset $U \subseteq \RR^n$ with its usual sheaf of smooth real-valued functions.  Differentiable spaces form a full subcategory $\DS$ of the category of locally ringed spaces over $\RR$.  Any smooth manifold $M$, with its usual sheaf of smooth, real-valued functions is a differentiable space.  The category $\DS$ has all finite inverse limits \cite[\Inverselimits]{diffspace}, in particular pullbacks.  The presheaf $X \mapsto \Gamma(X,\O_X)$ on $\DS$ is represented by the real line $\RR$, with its usual structure sheaf of smooth functions and its usual metric topology.  The forgetful functor $\DS \to \Top$ commutes with finite inverse limits \cite[\Inverselimits]{diffspace}.

If $X$ is a differentiable space, and $Z \subseteq X$ is a closed subspace of its underlying topological space, then both of the ideals \be I^{\rm big}(Z) & := & \{ f \in \O_X : t_x(f_x) = 0 {\rm \; for \; all \; } x \in Z \} \\ I^{\rm small}(Z) & := & \{ f \in \O_X : f(x) = 0 {\rm \; for \; all \; } x \in Z \} \\ & = & \{ f \in \O_X : f_x \in \m_x {\rm \; for \; all \; } x \in Z \} \ee are closed ideals of $\O_X$ and we clearly have $I^{\rm big}(Z) \subseteq I^{\rm small}(Z)$.  The zero locus of a closed ideal in a differentiable space is again a differentiable space \cite[\Zerolocus]{diffspace}, so we have closed embeddings of differentiable spaces $$ \Z(I^{\rm small}(Z)) \into \Z(I^{\rm big}(Z)) \into X.$$  One can check that $\Z(I^{\rm small}(Z))=\Z(I^{\rm big}(Z)) = Z$ as topological spaces by reducing to the local situation and using the following standard fact: For any closed subspace $Z$ of an open subspace $U \subseteq \RR^n$, we can find a smooth function $f : U \to \RR$ which is non-vanishing away from $Z$ whose Taylor series at any point of $Z$ is identically zero.  We call $\Z(I^{\rm small}(Z))$ (resp.\ $\Z(I^{\rm big}(Z))$) the \emph{small (resp.\ big) induced (differentiable space) structure} on the closed subspace $Z$.  The small induced structure is analogous to the reduced-induced closed subscheme structure in algebraic geometry.  The ``big ideal" $I^{\rm big}(Z)$ is rarely of any use in algebraic geometry because it is almost never quasi-coherent, so its zero locus is almost never a scheme (though it is a perfectly good locally ringed space).  The big induced structure is important to us because of the following result \cite[\Biginduced]{diffspace}.

\begin{lem}  \label{lem:big} The big induced structure $\Z(I^{\rm big}(Z))$ on a closed subspace $Z \subseteq X$ represents the presheaf taking $U \in \DS^{\rm op}$ to the set of $f \in \Hom_{\DS}(U,X)$ which factor through $Z$ on the level of topological spaces. \end{lem}

\begin{example} \label{example:Rplus}  For example, $X = \RR_+ = \{ x \in \RR : x \geq 0 \}$ is a closed subset of the real line $\RR$.  Its structure sheaf $\O_X$ in the big induced differentiable space structure is the quotient of the sheaf $\C^\infty$ of smooth functions on $\RR$ by the ideal sheaf $I$ consisting of functions with zero Taylor series at each point of $\RR_+$.  This is not the same thing as the restriction $ \C^\infty | X$ of the structure sheaf of $\RR$.  For example, the smooth function $f$ given by $\exp(-x^2)$ for $x < 0$ and zero for $x \geq 0$ yields a global section of $\C^\infty | X$ which is nonzero in the stalk $(\C^\infty | X)_0 = \C^\infty_0$, but this $f$ is clearly in $I$, hence $f$ maps to zero in $\O_X$ under $\C^\infty | X \to \O_X$. \end{example}

\begin{example} \label{example:bigideal} If $X = \RR$ is the real line and $Z = \{ 0 \}$ is the origin, the big ideal $I$ of $Z$ consists of all smooth functions on $X$ with zero Taylor series at the origin.  The big induced structure on $Z$ endows the point $Z$ with the ``sheaf" of rings given by the quotient $\C^\infty_0 / I_0$ of germs of smooth functions at the origin, modulo those with zero Taylor series; this ring is just the formal power series ring $\RR[[x]]$.  If one tried to do the analogous construction in algebraic geometry with, say, $X = \AA^1_{\CC}$, the big ideal $I$ is just given by \be I(U) & = & \left \{ \begin{array}{lll} \O_X(U), & \quad & 0 \notin U \\ 0, & & 0 \in U \end{array} \right . \ee because a rational function defined near the origin is determined by its power series at the origin.  This ideal sheaf isn't quasi-coherent.  The quotient locally ringed space is the one-point space with the ``sheaf" of rings $\O_{X,0} = \CC[x]_{(x)}$ (this isn't a scheme). \end{example}

The topological space underlying any differentiable space is $T_1$ \cite[\Pointsclosed]{diffspace} and locally compact in the sense of the following

\begin{defn} \label{defn:locallycompact} A topological space $X$ is locally compact iff each point $x \in X$ is in the interior of a compact Hausdorff subspace $Z \subseteq X$. \end{defn}

Note that $Z \subseteq X$ is not required to be a \emph{closed} subspace in the above definition.  This ensures that being locally compact is a local property.  Every open subspace of $\RR^n$ is locally compact and every closed subspace of a locally compact space is locally compact, hence every differentiable space is locally compact since the question is local and every differentiable space is locally a closed subspace of an open subspace of $\RR^n$.

\subsection{Positive functions} \label{section:positivefunctions}  Recall (\S\ref{section:basicnotions}) that the differentiable space $\RR$ represents the functor $X \mapsto \Gamma(X,\O_X)$.  In fact, the addition, multiplication, $0$, and $1$ maps for $\RR$ are maps of differentiable spaces, so that $\RR$ is a ring object in $\DS$, representing the functor \be \DS^{\rm op} & \to & \An \\ X & \mapsto & \Gamma(X,\O_X).\ee  We often regard $\RR$ as a monoid object in $\DS$ under multiplication, so that $\RR$ represents $X \mapsto \Gamma(X,\O_X)$, regarding the ring $\Gamma(X,\O_X)$ as a monoid under multiplication.  The subspace $\RR_+ \subseteq \RR$ of Example~\ref{example:Rplus} is in fact a (multiplicative) submonoid object of $\RR$.  

For a differentiable space $X$, let $\O_X^{\geq 0} \subseteq \O_X$ denote the subsheaf \be U & \mapsto & \{ f \in \O_X(U) : f(x) \geq 0 \; \forall \; x \in U \}. \ee  (The meaning of $f(x)$ is the image of $f$ in $\O_{X,x} / \m_x = \RR$---all points of a differentiable space are ``$\RR$-points"---but one can also think of $f(x) \in \RR$ as the image of $x$ under the map $f : U \to \RR$ corresponding to $f \in \O_X(U)$.)  Then $\O_X^{\geq 0}$ is a sheaf of (multiplicative) submonoids of $\O_X$, called the sheaf of \emph{non-negative functions}.  A section $f \in \O_X^{\geq 0}(U)$ is the same thing as a map of differentiable spaces $f : U \to \RR$ factoring set-theoretically through $\RR_+$, which, by Lemma~\ref{lem:big}, is the same thing as a map of differentiable spaces $f : U \to \RR_+$.  In other words, the monoid object $\RR_+$ in $\DS$ represents the functor $X \mapsto \Gamma(X,\O_X^{\geq 0})$.

Similarly, the open subspace $\RR^* = \RR \setminus \{ 0 \}$ of $\RR$ represents the sheaf of groups $X \mapsto \Gamma(X,\O_X^*)$ on $\DS$, and the open subspace $\RR_{>0} \subseteq \RR_+$ represents the sheaf of groups $X \mapsto \Gamma(X,\O_X^{>0})$ on $\DS$.  Here $\O_X^{>0}$ is the sheaf \be U & \mapsto & \{ f \in \O_X^*(U) : f(x) > 0 \; \forall \; x \in U \} \ee of \emph{positive functions}.  

The sign map $u \mapsto u/|u|$ defines a map of groups objects \bne{sign} \sign : \RR^* & \to & \{ \pm 1 \} \cong \ZZ / 2 \ZZ \ene in $\DS$, where $\ZZ / 2 \ZZ$ is the two element discrete group.  For any $X \in \DS$, we thus obtain a natural map of sheaves of abelian groups \bne{signmap} \sign: \O_X^* & \to & \underline{\ZZ/2\ZZ} \ene on $X$, where $\underline{\ZZ / 2 \ZZ}$ is the sheaf of locally constant maps to $\ZZ / 2 \ZZ$ (the sheaf of abelian groups represented by the discrete group object $\ZZ / 2 \ZZ$).  The map \eqref{signmap} admits an obvious section by taking a locally constant function $f : X \to \{ \pm 1 \} = \ZZ / 2 \ZZ$ to ``$f$" in $\O_X^*$, so we obtain a natural splitting of sheaves of abelian groups \bne{unitsplitting} \O_X^* & = & \O_X^{>0} \oplus \underline{\ZZ / 2 \ZZ}.\ene 

\subsection{Local properties of maps} \label{section:localproperties} By a \emph{property of maps} in a category $\C$, we mean a class of maps in $\C$ containing all identity maps and closed under composing with any isomorphism in $\C$.  For example, being an isomorphism is a ``property of maps."  If $\C$ is the category of locally ringed spaces, differentiable spaces, etc.,\footnote{that is, if $\C$ has a topology} and $\mathsf{P}$ is a property of maps in $\C$, then we say that $f : X \to Y$ has property $\mathsf{P}$ \emph{locally on the base} (or \emph{locally on} $Y$) iff there is an open cover $\{ U_i \}$ of $Y$ such that each map $f|f^{-1}(U_i) : f^{-1}(U_i) \to U_i$ has property $\mathsf{P}$.  We say that property $\mathsf{P}$ of $\C$ morphisms is \emph{local on the base} iff having property $\mathsf{P}$ and having property $\mathsf{P}$ locally on the base are equivalent.  Similarly, we say that $f$ has property $\mathsf{P}$ \emph{locally on the domain} iff there is an open cover $\{ U_i \}$ of $X$ such that each map $f|U_i :U_i \to Y$ has $\mathsf{P}$ and we say that $\mathsf{P}$ is \emph{local on the domain} if having $\mathsf{P}$ is equivalent to having $\mathsf{P}$ locally on the domain.

For $x \in X$, a \emph{neighborhood of} $x$ in $f$ is a pair $(U,V)$ consisting of a neighborhood $V$ of $f(x)$ in $Y$ and a neighborhood $U$ of $x$ in $f^{-1}(V) \subseteq X$.  We say that $f$ has property $\mathsf{P}$ \emph{at} a point $x \in X$ iff there is a neighborhood $(U,V)$ of $x$ in $f$ such that $f|U : U \to V$ has property $\mathsf{P}$.  We say that $f$ has property $\mathsf{P}$ \emph{locally on} $X$ \emph{and} $Y$ (or \emph{locally on} $f$) iff $f$ has property $\mathsf{P}$ at each $x \in X$.  We say that a property $\mathsf{P}$ of $\C$ morphisms is \emph{local on the domain and codomain} (or just ``\emph{is local}") iff, for all $\C$ morphisms $f : X \to Y$, having property $\mathsf{P}$ is equivalent to having property $\mathsf{P}$ locally on $f$.

\begin{example} \label{example:closedlocalonbase} In the category of topological spaces, the property of being an isomorphism, or a closed embedding is local on the base. \end{example}

\begin{example} \label{example:coveringspace} Call a map of topological spaces $f : X \to Y$ a \emph{trivial covering space} iff $X$ is isomorphic, as a space over $Y$, to a disjoint unions of copies of $Y$, each mapping to $Y$ by the identity map.  The property of being a trivial covering space is not local on the base; a covering space in the usual sense of topology is a map which is a trivial covering space locally on the base. \end{example}

The most important example is the following:

\begin{defn} \label{defn:localisomorphism} Call a map of topological spaces (or ringed spaces, schemes, differentiable spaces, \dots) $f :X \to Y$ a \emph{local isomorphism} iff ``$f$ is an isomorphism locally on $f$."  In other words, $f$ is a local isomorphism iff, for any $x \in X$, there is a commutative diagram $$ \xym{ U \ar[r] \ar[d]_g & X \ar[d]^f \\ V \ar[r] & Y } $$ where the horizontal arrows are open embeddings, $g$ is an isomorphism, and $x$ is in the image of $U \to X$.  Since the composition of $g$ and $V \to Y$ is also an open embedding, we can always take $U=V$, $g=\Id$ if desired, so one could equivalently say that $f : X \to Y$ is a local isomorphism iff, for each $x \in X$, there is an open embedding $i : U \to X$ with $x \in i(U)$ such that $fi$ is also an open embedding.  Notice that a map of topological spaces $f : X \to Y$ is a local isomorphism (usually one says ``local homeomorphism") iff it is open and locally one-to-one.  A local isomorphism which is surjective (on underlying spaces) is called a \emph{Zariski cover}. \end{defn}

\subsection{Smooth morphisms} \label{section:smoothmorphisms}  A morphism of differentiable spaces $f : X \to Y$ is called \emph{smooth at} $x \in X$ iff there is a neighborhood $(U,V)$ of $x$ in $f$ such that $f|U: U \to V$ can be factored as the inclusion of an open subspace $i : U \into V \times \RR^n$ (for some $n$), followed by the projection $\pi_1 : V \times \RR^n \to V$.  Such a morphism $f$ is called \emph{\'etale at} $x$ iff there is a neighborhood $(U,V)$ of $x$ in $f$ such that $f|U : U \to V$ is an isomorphism.  Clearly ``\'etale at $x$" implies ``smooth at $x$."  The morphism $f$ is called \emph{smooth} (resp.\ \emph{\'etale}) iff it is smooth (resp.\ \'etale) at $x$ for every $x \in X$.  

A differentiable space $X$ is called \emph{smooth} iff it is smooth over the terminal object $\Spec \RR$.  A smooth differentiable space is evidently the same thing as a \emph{smooth manifold} in the usual sense, but without the usual paracompactness and Hausdorff requirements on the underlying topological space, which we can always impose separately if necessary.

It is straightforward to check that a composition of smooth (resp.\ \'etale) morphisms is again smooth (resp.\ \'etale) and that the base change (pullback) of any smooth (resp.\ \'etale) morphism is again smooth (resp.\ \'etale).  The inclusion of an open subspace is \'etale.  The property of being smooth (resp.\ \'etale) is local in the sense of \S\ref{section:localproperties}.

\subsection{Proper and projective morphisms} \label{section:propermorphisms}  Here we recall the definition of a proper map of topological spaces, and some useful variants.

\begin{defn} \label{defn:proper} A continuous map of topological spaces $f : X \to Y$ is \emph{proper} iff $f$ is closed with compact fibers and for any distinct $x_1,x_2 \in X$ with $f(x_1)=f(x_2)$, there are disjoint open neighborhoods $U_1,U_2$ of $x_1,x_2$ in $X$. \end{defn}

\begin{lem} \label{lem:propermaps} Proper maps are closed under composition and base change. Being proper is local on the base. \end{lem}

\begin{proof} This is an exercise for the reader.  The only statement that is a little tricky is stability under base-change. \end{proof}

\begin{defn} \label{defn:properEuclidean} A map of differentiable spaces $f : X \to Y$ is called \emph{projective} iff there is a vector bundle $V$ on $Y$ such that $f$ factors as a closed embedding $X \into \PP(V)$ followed by the projection $\PP(V) \to Y$.  The $f$ is called \emph{Euclidean proper} iff, locally on $Y$, there is a compact differentiable space $E$ such that $f$ factors as a closed embedding $X \into Y \times E$ followed by the projection $Y \times E \to Y$.  \end{defn}

Clearly being proper Euclidean is local on the base and it is clear that every closed embedding of differentiable spaces is both projective and proper Euclidean.  Proper Euclidean maps are proper.  Every compact differentiable space embeds in $\RR^n$ for large enough $n$.  Every $\DS$ fiber bundle with compact fibers is proper Euclidean.  A map which is projective locally on the base is proper Euclidean.  Proper Euclidean maps are closed under composition and base change.

\subsection{Differentialization} \label{section:differentialization}  Let $\Sch_{\RR}$ denote the category of schemes of locally finite type over $\RR$.  To any $X \in \Sch_{\RR}$, there is an associated differentiable space $X^{\DS}$ called the \emph{differentialization} of $X$ and a map $X^{\DS} \to X$ of locally ringed spaces over $\RR$ which is terminal among $\LRS/\RR$ morphisms from a differentiable space to $X$.  Formation of $X^{\DS} \to X$ is functorial in $X$, and the functor \be \DS : \Sch_{\RR} & \to & \DS \ee is a kind of right adjoint in that it satisfies \bne{DSadjointness} \Hom_{\LRS / \RR} (Y, X) & = & \Hom_{ \DS }(Y,X^{\DS}) \\ \nonumber & = & \Hom_{\LRS/\RR}(Y,X^{\DS}) \ene for every $Y \in \DS$.  The set of points of $X^{\DS}$ is naturally bijective with the set $X(\RR)$ of $\RR$ points of $X$.  See \cite{diffspace} for further details.

Let $\Sch_{\CC}$ denote the category of schemes of locally finite type over $\CC$.  The above differentialization construction is analogous to the \emph{analytification} functor $X \mapsto X^{\AS}$ from $\Sch_{\CC}$ to the category $\AS$ of analytic spaces.  All results mentioned in the previous paragraph continue to hold when ``$\RR$," ``differentiable," ``differentialization," and ``$\DS$" are replaced by ``$\CC$," ``analytic," ``analytification," and ``$\AS$," respectively.

If $X = \Spec \RR[x_1,\dots,x_n]/(f_1,\dots,f_m)$, then the polynomials $f_i$ may be viewed as smooth functions on $\RR^n$.  A theorem of Malgrange implies that $(f_1,\dots,f_n)$ is a closed ideal in the ring of smooth functions on $\RR^n$.  The differentiable space $X^{\DS}$ is the zero locus of this ideal.

The differentialization functor $X \mapsto X^{\rm DS}$ takes closed embeddings to closed embeddings, open embeddings to open embeddings, $\AA^n$ to $\RR^n$, vector bundles to vector bundles, \'etale maps to \'etale maps, and so forth.  Keep in mind, however, that it does not take surjections to surjections because a surjective map of schemes need not be surjective on $\RR$ points.

Let $\Sch_{\CC}$ denote the category of schemes of locally finite type over $\CC$.  Recall that the base change functor \be \Sch_{\RR} & \to & \Sch_{\CC} \\  X & \mapsto & X \times_{\Spec \RR} \Spec \CC \ee admits a right adjoint \be W : \Sch_{\CC} & \to & \Sch_{\RR} \ee called the \emph{Weil restriction}.  In particular, $W(X)(\RR) = X(\CC)$ for $X \in \Sch_{\CC}$.  If $X = \Spec \CC[z_1,\dots,z_n]/(f_1,\dots,f_m)$, then \be W(X) & = & \Spec \RR[x_1,\dots,x_n,y_1,\dots,y_n]/( \Re f_1, \dots, \Re f_m, \Im f_1, \dots, \Im f_m ), \ee where one writes $z_j = x_j+iy_j$ and defines $\Re f_k$ and $\Im f_k$ by formally collecting the real and imaginary parts of $f_k$, treating the formal variables $x_j$ and $y_j$ as real numbers.  We define the \emph{differentialization} of $X \in \Sch_{\CC}$ to be the differentialization of $W(X)$ as defined above.

There is also a differentialization functor $X \mapsto X^{\rm \DS}$ from the category $\AS$ of analytic spaces to the category $\DS$ of differentiable spaces.  The differentiable space $X^{\rm \DS}$ is characterized by the existence of a bijection \be \Hom_{\LRS / \CC}(Y \otimes_{\RR} \CC, X) & = & \Hom_{\DS}(Y,X^{\rm \DS}) \ee natural in $Y \in \DS$.  Here $Y \otimes_{\RR} \CC$ is nothing but the topological space $Y$ equipped with the sheaf of \emph{complex valued} smooth functions $\O_Y \otimes_{\RR} \CC$.

We could also have defined the differentialization of $X \in \Sch_{\CC}$ in terms of the analytification and differentialization functors by the formula $X^{\DS} := (X^{\AS})^{\DS}$.  One checks easily using the ``adjunction" isomorphisms mentioned above $(X^{\AS})^{\DS} = W(X)^{\DS}$ by showing that both differentiable spaces represent the same functor.

\subsection{Monoids to differentiable spaces} \label{section:RP} Our main use for the differentialization construction of \S\ref{section:differentialization} is to associate differentiable spaces $\RR(P)$ and $\RR_+(P)$ to a finitely generated monoid $P$.  In fact, we could accomplish this through ``general nonsense" as in \S\ref{section:spaces}, but it is useful to have a concrete description of $\RR(P)$ and $\RR_+(P)$.

Since $\RR[P]$ is a finite type $\RR$ algebra, $\Spec \RR[P] \in \Sch_{\RR}$ is a finite type (affine) scheme over $\RR$.  For any $X \in \LRS/\RR$ we have \bne{RPadj} \Hom_{\LRS / \RR}(X,\Spec \RR[P]) & = & \Hom_{\RR-\Alg}(\RR[P] , \Gamma(X,\O_X)) \\ \nonumber & = & \Hom_{\Mon}(P, \Gamma(X,\O_X)). \ene  We set $\RR(P) := (\Spec \RR[P])^{\DS}$.  Note that the points of the topological space $\RR(P)$ are the $\RR$ points of $\Spec \RR[P]$, which are the monoid homomorphisms $P \to \RR$, regarding $\RR$ as a monoid under multiplication.  By the universal property of differentialization and \eqref{RPadj}, we have a natural bijection \bne{RRPadjunction} \Hom_{\DS}(X,\RR(P)) & = & \Hom_{\LRS/\RR}(X, \Spec \RR[P]) \\ \nonumber & = & \Hom_{\Mon}(P,\Gamma(X,\O_X)) \ene for $X \in \DS$.

Building on this construction, we next note that $\RR_+(P) := \Hom_{\Mon}(P,\RR_+)$ is a closed subspace of $\RR(P)$.  We regard $\RR_+(P)$ as a differentiable space by giving it the big induced differentiable space structure (\S\ref{section:basicnotions}) from $\RR(P)$.  A map $X \to \RR(P)$ corresponding to a monoid homomorphism $f : P \to \O_X(X)$ under \eqref{RRPadjunction} will factor set-theoretically through the closed subspace $\RR_+(P) \subseteq \RR(P)$ iff $f$ takes values in $\O_X^{\geq 0}(X)$, where $\O_X^{\geq 0} \subseteq \O_X$ is the submonoid of non-negative functions defined in \S\ref{section:PLDS}.  So \eqref{RRPadjunction} and the universal property of the big induced structure (Lemma~\ref{lem:big}) combine to yield a natural bijection \bne{RplusPadjunction} \Hom_{\DS}(X,\RR_+(P)) & = & \Hom_{\Mon}(P,\O_X^{\geq 0}(X)). \ene  In particular, if $P = \NN^n$, then $\RR_+^n := \RR_+(\NN^n)$ is the positive orthant in $\RR^n$ and represents the presheaf $X \mapsto \Gamma(X,\O_X^{\geq 0})^n.$  We already encountered the differentiable space $\RR_+$ in Example~\ref{example:Rplus}.

Formation of the differentiable spaces $\RR(P)$ and $\u{\RR}_+(P)$ is contravariantly functorial in $P$.  The adjunction formulas \eqref{RRPadjunction} and \eqref{RplusPadjunction} imply that the functors $\RR(\slot)$ and $\RR_+(\slot)$ take finite direct limits of monoids to finite inverse limits of differentiable spaces.

\section{Monoidal spaces I} \label{section:monoidalspacesI}  This section is essentially independent from the rest of the paper and should be readable by anyone with a general knowledge of topological spaces and sheaves, though a certain basic knowledge of algebraic geometry and the theory of monoids (at the level of \S\ref{section:monoids}) will probably be necessary to follow the finer points.  

The main object of study will be the category of \emph{locally monoidal spaces}, which is analogous to the category $\LRS$ of locally ringed spaces, but with monoids playing the role of rings.  Inside $\LRS$, we have the category $\Fans$ of \emph{fans} (\S\ref{section:fans}), which is analogous to the category $\Sch$ of schemes inside $\LRS$.  We develop some rudimentary ``algebraic geometry" of these spaces in the later sections, then formulate and prove a resolution of singularities statement in \S\ref{section:functorialresolution}.  Although essentially of a combinatorial nature, our proof of this result relies on the existence of a functorial resolution of singularities for varieties over $\CC$, as established by Bierstone and Milman in \cite{BM}.  The basic point is that the $\CC$-scheme realization of a monoid is a very faithful reflection of the monoid (\S\ref{section:CP}), so that we can propogate $\CC$-scheme results to general results about monoids, fans, and so forth.  For example, this ultimately allows us to extract general ``resolution of singularities" results from the corresponding results for $\CC$-schemes.  Our approach along these lines is not entirely new---we have discussed the history of these ideas in the Introduction.

\subsection{Definitions} \label{section:monoidalspacedefinitions} We warn the reader that our terminology here conflicts with Kato's terminology in \cite[\S9]{Kat2} and with Ogus's terminology \cite[Definition~1.3.4]{Ogus}, though it is consistent with \cite[\S3.7]{loc}.

\begin{defn} \label{defn:monoidalspace} A \emph{monoidal space} $X = (X,\M_X)$ is a topological space $X$ equipped with a sheaf of monoids $\M_X$ called its \emph{structure sheaf}.  We let $\M_{X,x}$ denote the stalk of $\M_X$ at $x \in X$ and we let $\m_x := \M_{X,x} \setminus \M_{X,x}^*$ denote the maximal ideal of $\M_{X,x}$.  Monoidal spaces form a category $\MS$ where a morphism $f = (f,f^\dagger) : (X,\M_X) \to (Y,\M_Y)$ is a pair $(f,f^\dagger)$ consisting of a map of topological spaces $f : X \to Y$ and a map of sheaves of monoids $f^\dagger : f^{-1} \M_Y \to \M_X$.  The category $\LMS$ of \emph{locally monoidal spaces} has the same objects as $\MS$, but an $\LMS$ morphism is an $\MS$ morphism $f$ as above such that \be f_x^\dagger : \M_{Y,f(x)} \to \M_{X,x} \ee is a local map of monoids (Definition~\ref{defn:local}) for each $x \in X$ (i.e.\ $f_x^\dagger(\m_{f(x)}) \subseteq \m_x$). \end{defn}

\begin{defn} A monoidal space $X$ is called \emph{sharp} iff $\M_X$ is a sheaf of sharp monoids (\S\ref{section:monoidbasics}).  By definition, sharp monoidal spaces form a full subcategory $\SMS$ of $\LMS$.  \end{defn}

One could of course view sharp monoidal spaces as a full subcategory of $\MS$, but we will have no use for maps of sharp monoidal spaces that are not local.  For any monoidal space $X$, the quotient $\ov{\M}_X := \M_X / \M_X^*$ is a sheaf of sharp monoids, and the quotient map $\M_X \to \ov{\M}_X$ is an $\LMS$ morphism.  The sharp monoidal space $\ov{X} := (X,\ov{\M}_X)$ is called the \emph{sharpening} of $X$.  Sharpening defines a functor $\LMS \to \SMS$ which is right adjoint to the inclusion $\SMS \into \LMS$.

\subsection{Prime systems} \label{section:primesystems} To construct relative Spec functors, blowups, and so forth in the category of locally monoidal spaces, we will follow the general strategy for doing this in the case of locally ringed spaces from \cite{loc}.  In fact, we initially used this machinery to construct inverse limits in $\LMS$ (following the analogous construction of inverse limits in $\LRS$ in \cite{loc}) before we discovered that this could be done more easily by using Theorem~\ref{thm:directlimitlocal} (we had trouble proving this result for some time), as in Theorem~\ref{thm:LMSinverselimits}.

\begin{defn} \label{defn:primesystem} Let $X$ be a monoidal space.  A \emph{prime system} on $X$ is the data of a subset $M_x \subseteq \Spec \M_{X,x}$ for each $x \in X$.  A \emph{primed monoidal space} $(X,M)$ is a monoidal space $X$ equipped with a prime system $M$.  Primed monoidal spaces form a category $\PMS$ where a morphism $(X,M) \to (Y,N)$ is a morphism of monoidal spaces $f : X \to Y$ such that $(f_x^\dagger)^{-1}(\p) \in N_{f(x)}$ for each $\p \in M_x$ for each $x \in X$.  \end{defn}

\begin{defn} \label{defn:primesystemexamples}  For a monoidal space $X$, the \emph{local prime system} $L_X$ on $X$ (or just $L$ if $X$ is clear from context) is the prime system with $L_x = \{ \m_x \}$ for each $x \in X$ and the \emph{terminal prime system} $T_X$ (or just $T$) is the prime system with $T_x = \Spec \M_{X,x}$ for every $x \in X$.  If $f : X \to Y$ is a map of monoidal spaces and $M$ is a prime system on $Y$, then the \emph{inverse image prime system} $f^* M$ is the prime system on $X$ with \be (f^*M)_x & := & \{ \p \in \Spec \M_{X,x} : (f_x^\dagger)^{-1} \p \in M_{f(x)} . \ee   \end{defn}

The local prime system can be viewed as a functor $L : \LMS \to \PMS$ making $\LMS$ a full subcategory of $\PMS$.  Similarly, the terminal prime system can be viewed as a functor $T : \MS \to \PMS$ which is clearly right adjoint to the forgetful functor $\PMS \to \MS$ forgetting the prime system.

All of the above constructions have analogues with ``monoidal" replaced by ``ringed" studied in \cite{loc}. 

\subsection{Localization} \label{section:localization} For $(X,M) \in \PMS$, we construct a new monoidal space $(X,M)^{\rm loc}$ (or just $X^{\rm loc}$ if $M$ is clear from context) called the \emph{localization} of $X$ at $M$ as follows.  As a set, we let $X^{\rm loc}$ be the set of pairs $(x,z)$ where $x \in X$ and $z \in M_x$. For an open subset $U$ of $X$ and a section $s \in \M_X(U)$, we set \be U_s & := & \{ (x,z) \in X^{\rm loc} : x \in U, \; s_x \notin z \}. \ee  These $U_s$ form the basic opens for a topology on $X^{\rm loc}$ in light of the formula \be U_{s} \cap V_t & = & (U \cap V)_{s+t} \ee (notation for restriction to $U \cap V$ dropped).  The map $\tau : X^{\rm loc} \to X$ given by $\tau(x,z) := x$ is continuous because $\tau^{-1}U = U_0$. 

We construct a sheaf of monoids $\M_{X^{\rm loc}}$ on $X^{\rm loc}$ as follows.  For each open $V \subseteq X^{\rm loc}$, we let $\M_{X^{\rm loc}}(V)$ denote the set of $$ s = (s(x,z))  \in \prod_{(x,z) \in V} (\M_{X,x})_z $$  satisfying the following \emph{local consistency condition}: For every $(x,z) \in V$, there is an open neighborhood $U$ of $x$ in $X$ and sections $p,q \in \M_X(U)$ with $q_x \notin z$ such that \be s(x',z') & = & p_{x'} - q_{x'} \ee for every $(x',z') \in U_q$.  This is a monoid under coordinatewise addition and is a sheaf in light of the local nature of the local consistency condition.  There is a natural isomorphism \be \M_{X^{\rm loc},(x,z)} & = & (\M_{X,x})_z \ee for each $(x,z) \in X^{\rm loc}$.  The map $\tau$ of topological spaces defined above lifts in an evident manner to an $\MS$ morphism \bne{tau} \tau : (X,M)^{\rm loc} & \to & X \ene such that \be \tau_{(x,z)}^\dagger : \M_{X,x} & \to & (\M_{X,x})_z \ee is the localization of $\M_{X,x}$ at the prime ideal $z$.

The above construction and detailed explanations of all above statements are given in the case of prime \emph{ringed} spaces in \cite[\S2.2]{loc}.  The following elementary result is proved in a manner identical to the proof of \cite[Theorem~2]{loc}: Just replace the word ``ring" (resp.\ ``ringed") with ``monoid" (resp.\ ``monoidal") everywhere in that proof, switch to multiplicative notation for monoids, and use the universal property of localization of monoids (\S\ref{section:monoidbasics}) in place of the one for rings.

\begin{thm} \label{thm:localization} Let $f : (X,M) \to (Y,N)$ be a morphism in $\PMS$.  Then there is a unique $\LMS$ morphism $\ov{f} : (X,M)^{\rm loc} \to (Y,N)^{\rm loc}$ making the diagram $$ \xym{ (X,M)^{\rm loc} \ar[d]_{\tau} \ar[r]^{\ov{f}} & (Y,N)^{\rm loc} \ar[d]_{\tau} \\ X \ar[r]^f & Y } $$ commute in $\MS$.  Localization defines a functor $\PMS \to \LMS$ retracting the local prime system functor $L : \LMS \to \PMS$ and right adjoint to it. \end{thm}

\begin{rem} \label{rem:localizationisloca} It is clear from the construction that formation of $(X,M)^{\rm loc}$ is local in nature: If $U$ is an open subspace of $X$, then $(U,M|U)^{\rm loc} = \tau^{-1}(U) \subseteq (X,M)^{\rm loc}$. \end{rem}

\subsection{Inverse limits} \label{section:inverselimits}  Here we summarize the basic properties of inverse limits in our various categories of monoidal spaces.

\begin{thm} \label{thm:LMSinverselimits}  The categories $\MS$, $\PMS$, $\LMS$, and $\SMS$ have all inverse limits.  The functors $$\MS \to \Top, \; \; \LMS \to \MS, \; \; L : \LMS \to \PMS, \; {\rm and} \; \LMS \to \SMS$$ preserve inverse limits.  \end{thm}

\begin{proof} ``The" inverse limit of a functor $i \mapsto X_i$ to $\MS$ can be constructed as follows.  Let $X$ be ``the" inverse limit of the underlying topological spaces of the $X_i$ and let $\pi_i : X \to X_i$ be the projections.  Let $\M_X$ be the direct limit of the sheaves of monoids $\pi_i^{-1} \M_{X_i}$ on $X$.  The structure maps $\pi_i^\sharp : \pi_i^{-1} \M_{X_i} \to \M_X$ to the direct limit allow us to view $\pi_i : X \to X_i$ as a $\MS$ morphism.  Since direct limits commute with stalks, we have \be \M_{X,x} & = & \dirlim \M_{X_i,\pi_i(x)} \ee for any $x \in X$, so that the map $\M_{X_i,\pi_i(x)} \to \M_{X,x}$ obtained from our structural $\MS$ morphism $\pi_i : X \to X_i$ is the usual structure map to the direct limit.  It is trivial to check that $X = (X,\M_X)$ is the inverse limit of the $X_i$ in $\MS$ and it is clear from this construction that $\MS \to \Top$ commutes with inverse limits.  

To construct the inverse limit of a functor $i \mapsto (X_i,M_i)$ in $\PMS$, first form the inverse limit $X$ of the $X_i$ in $\MS$ as discussed above, then endow it with the prime system $M$ where $\p \in M_x$ iff $(\pi_{i,x}^\dagger)^{-1}(\p) \in (M_i)_{\pi_i(x)}$ for every $i$.  

Now we have two ways to construct the inverse limit of a functor $i \mapsto X_i$ in $\LMS$.  

\noindent {\bf First approach:} (c.f.\ \cite[\S3.1]{loc})  Composing with the local prime system $L : \LMS \to \PMS$, we obtain an inverse limit system $i \mapsto (X_i,L_i)$ in $\PMS$, which,as we saw above, has inverse limits, so we can form the inverse limit $(X,M)$ of $i \mapsto (X_i,L_i)$.  Localization is a right adjoint (Theorem~\ref{thm:localization}), so it preserves inverse limits, hence $(X,M)^{\rm loc}$ is the inverse limit of $i \mapsto (X_i,L_i)^{\rm loc}$ in $\LMS$.  But $(X_i,L_i)^{\rm loc} = X_i$ because localization retracts $L$ (Theorem~\ref{thm:localization}), so $(X,M)^{\rm loc}$ is the desired inverse limit of $i \mapsto X_i$ in $\LMS$. 

\noindent {\bf Second approach:}  Let $X$ be the inverse limit of $i \mapsto X_i$ in $\MS$.  We claim that the structure maps $\pi_i : X \to X_i$ are in fact $\LMS$-morphisms.  By construction of $X$, for $x \in X$ we have $\M_{X,x} = \dirlim \M_{X_i,\pi_i(x)}$ and \be \pi_{i,x} : \M_{X_i,\pi_i(x)} & \to & \M_{X,x} \ee is the structure map to the direct limit.  This is indeed local by Theorem~\ref{thm:directlimitlocal}.  Next we claim that $X$ is in fact the inverse limit of $i \mapsto X_i$ in $\LMS$.  Consider a set of $\LMS$ maps $f_i : Y \to X_i$ so that \be (f_i) & \in & \invlim \Hom_{\LMS}(Y,X_i). \ee  Since $X$ is the inverse limit of $i \mapsto X_i$ in $\MS$, there is a unique $\MS$-morphism $f : Y \to X$ such that $\pi_i f = f_i$ for all $i$.  The issue is to show that $f$ is actually an $\LMS$ map.  Pick a point $y \in Y$.  We need to show that $f_y : \M_{X,f(y)} \to \M_{Y,y}$ is local.  We have a diagram $$ \bigoplus_i \M_{X_i,f_i(x)} \to \M_{X,f(y)} \to \M_{Y,y} $$ where the left map is surjective and local (c.f.\ the proof of Theorem~\ref{thm:directlimitlocal}) and the composition is local because the $f_i$ are $\LMS$ maps, so $f_y$ is local by Lemma~\ref{lem:local}.   

It is clear from this second approach that $\LMS \to \MS$ preserves inverse limits.  The fact that $L : \LMS \to \PMS$ preserves inverse limits is a similar application of Theorem~\ref{thm:directlimitlocal}.  The sharpening functor $\LMS \to \SMS$ commutes with inverse limits because it is a right adjoint (\S\ref{section:monoidalspacedefinitions}).  \end{proof}

Notice that the corresponding statements with ``monoidal" replaced by ``ringed" are certainly not true.  This is the main point where the theory of locally monoidal spaces differs from that of locally ringed spaces.

\begin{defn} \label{defn:benignMSmorphism} An $\MS$ morphism $f : X \to Y$ is \emph{benign} iff $f^\dagger_x : \M_{Y,f(x)} \to \M_{X,x}$ is a benign morphism of monoids (Definition~\ref{defn:benign}) for every $x \in X$. \end{defn}

For example, strict morphisms are benign and the sharpening morphism $\ov{X} \to X$ is benign.  Every benign $\MS$ morphism is in fact an $\LMS$ morphism because a benign map of monoids is local.

\begin{cor} \label{cor:benignstrictmorphisms} Strict (resp.\ benign, \dots) morphisms in $\LMS$ are closed under base change. \end{cor}

\begin{proof}  The point is that base change in $\LMS$ is the same as base change in $\MS$ so we just have to prove that isomorphisms (resp.\ benign maps, \dots) of monoids are closed under pushout, which is clear (resp.\ Lemma~\ref{lem:benign}). \end{proof}

\begin{cor} \label{cor:sharpening} An $\LMS$ map $ f : X \to Y$ which is surjective (resp.\ open, proper, a homeomorphism, \dots) on underlying topological spaces also has this property after any $\LMS$ base change.  \end{cor}

\begin{proof} $\LMS \to \Top$ preserves inverse limits by the theorem and the types of $\Top$ morphisms listed are stable under base change in $\Top$. \end{proof}

\subsection{Spec revisited} \label{section:Specrevisited} The $\Spec$ construction from \S\ref{section:Spec} defines a functor \bne{unsharpenedSpec} \Spec : \Mon^{\rm op} & \to & \LMS \\ \nonumber P & \mapsto & (\Spec P, \M_P). \ene  It is not hard to see directly from the definitions that we have a natural bijection \bne{Specadjunction} \Hom_{\LMS}(X,\Spec P) & = & \Hom_{\Mon}(P,\M_X(X)) \ene for $X \in \LMS$, $P \in \Mon$.  This bijection can be constructed in the same manner as the analogous natural bijection \bne{ringsSpecadjunction} \Hom_{\LRS}(X,\Spec A) & = & \Hom_{\An}(A,\O_X(X)) \ene for $X \in \LRS$, $A \in \An$ as in \cite[Err. I.1.8]{EGA}, though we will prove this ``by general nonsense" in just a moment.  Given $h : P \to \M_X(X)$, the corresponding map of topological spaces, abusively denoted $h : X \to \Spec P$, takes $x \in X$ to $h(x) := h^{-1}(\m_x) \in \Spec P$.  Here we write $h$ as abuse of notation for the composition of $h : P \to \M_X(X)$ and the map $\M_X(X) \to \M_{X,x}$ taking a global section to its stalk at $x$. 

\begin{example} \label{example:mapsfromSpecN} Recall the discussion of $\Spec \NN = \{ \emptyset, \m \}$ from Example~\ref{example:SpecN}.  To give an $\MS$-morphism $f : \Spec \NN \to X$ is to give points $f(\emptyset), f(\m) \in X$ with $f(\m) \in \{ f(\emptyset) \}^-$, and a commutative diagram of monoids as below.  $$ \xym{  \M_{X,f(\m)} \ar[d]_{f^\dagger_{\m}} \ar[r] & \M_{X,f(\emptyset)} \ar[d]^{f^\dagger_{\emptyset}} \\ \NN \ar[r] & \ZZ } $$ The horizontal arrows here are the generalization maps on stalks.  Such a map is an $\LMS$ morphism iff the vertical maps are local, which is equivalent to saying $(f^\dagger_{\m})^{-1}(0) = \{ 0 \}$ and $\M_{X,f(\emptyset)}$ is a group.  Compare the description \cite[II.4.4]{H} of maps out of a trait in algebraic geometry.  \end{example}

Given a monoid $P$, we can regard the one point space $*$ with ``sheaf" of monoids $P$ as a monoidal space $(*,P) \in \MS$.  The pullback of the ``sheaf" $P$ along the map from a space $X$ to $*$ is the constant sheaf $\underline{P}$ on $X$, so it is clear that \bne{homformula} \Hom_{\MS}(X,(*,P)) & = & \Hom_{\Mon(X)}(\underline{P},\M_X) \\ \nonumber & = & \Hom_{\Mon}(P,\M_X(X)). \ene We give $(*,P)$ the terminal prime system (Definition~\ref{defn:primesystemexamples}), so $(*,P,T) \in \PMS$, and we define \be \Spec P & := & (*,P,T)^{\rm loc}. \ee  As in \S\ref{section:Spec} we write $\M_P$ for the structure sheaf of $\Spec P$.  The structure map $\tau : (*,P,T)^{\rm loc} \to (*,P)$ is a $\MS$ morphism, hence we have a map \bne{SpecPmap} \tau^\dagger :  \underline{P} & \to & \M_P \ene of sheaves of monoids on $\Spec P$, where $\underline{P}$ is the constant sheaf associated to the monoid $P$.  As asserted in \S\ref{section:Spec}, we see from the general construction of localization that the stalk of $\tau^\dagger$ at $\p \in \Spec P$ is the localization $P \to P_{\p}$.  

Using the adjointness properties of $T$, localization (Theorem~\ref{thm:localization}), and \eqref{homformula} we see immediately that \be \Hom_{\LMS}(X,\Spec P) & = & \Hom_{\LMS}(X,(*,P,T)^{\rm loc}) \\ & = & \Hom_{\PMS}((X,L),(*,P,T)) \\ & = & \Hom_{\MS}(X,(*,P)) \\ & = & \Hom_{\Mon}(P,\M_X(X)), \ee as asserted in \eqref{Specadjunction} and in \S\ref{section:Spec}.  In particular, since $\Gamma(\Spec P,\M_P) = P$ (\S\ref{section:Spec}), we have \be \Hom_{\LMS}(\Spec P,\Spec Q) & = & \Hom_{\Mon}(Q,P) \ee for any monoids $P,Q$.

\begin{lem} \label{lem:Spec} Let $P$ be a monoid, $X := \Spec P$, $*$ the one point space.  Define a prime system $M$ on $(X,\underline{P}) \in \MS$ by $M_x := \{ x \} \subseteq \Spec \underline{P}_x = \Spec P$.  Then we have natural $\PMS$ morphisms $$ (X,\M_P,L) \to (X,\underline{P},M) \to (*,P,T) $$ inducing isomorphisms $$ \Spec P = (X,\M_P,L)^{\rm loc} = (X,\underline{P},M)^{\rm loc} = (*,P,T)^{\rm loc} $$ on localizations. \end{lem}

\begin{proof} The proof is the same as that of \cite[Lemma~3]{loc}. \end{proof}

\begin{lem} \label{lem:strictfanmaps} Suppose $h : Q \to P$ is a monoid homomorphism such that the induced $\LMS$ morphism $\Spec h : \Spec P \to \Spec Q$ is strict.  Then there is a submonoid $T \subseteq Q$ such that $P$ is isomorphic (as a monoid under $Q$) to $T^{-1}Q$.  In particular, if $Q$ is finitely generated, an $\LMS$ morphism $\Spec P \to \Spec Q$ is strict iff it is an open embedding. \end{lem}

\begin{proof} Suppose $f := \Spec h$ is strict.  Let $\eta = P \setminus P^*$ be the generic point of $\Spec P$, so that $\M_{P,\eta} = P$.  Then $f(\eta) = Q \setminus h^{-1}(P \setminus P^*) \in \Spec Q$ corresponds to the face $T := h^{-1}(P^*)$ of $Q$.  The stalk $\M_{Q,f(\eta)} \to \M_{P,\eta}$ of $f^\dagger : f^{-1} \M_Q \to \M_P$ at $\eta$ is the map $h' : T^{-1}Q \to P$ induced by $h$.  The map $h'$ is an isomorphism because $f$ is strict.  The map $f = \Spec h$ factors as $f = (\Spec l)(\Spec h')$ where $l : Q \to T^{-1}Q$ is the localization map because $h =h'l$.  When $Q$ is finitely generated, so is $T$ (Lemma~\ref{lem:faces}), so $\Spec l$ is an open embedding as discussed in \S\ref{section:Spec}---this is the nontrivial implication in the final statement of the lemma. \end{proof}

We conclude this section by pointing out some differences between the functor \bne{monoidSpecfunctor} \Spec : \Mon^{\rm op} & \to & \LMS \ene and the more familiar functor \bne{ringSpecfunctor} \Spec : \An^{\rm op} & \to & \LRS . \ene  First, the functor \eqref{ringSpecfunctor} takes surjections to closed embeddings (on the level of topological spaces, say), but \eqref{monoidSpecfunctor} does not (Example~\ref{example:Specsurjectionnotclosed}).  Second, the functor \eqref{ringSpecfunctor} preserves (finite) coproducts, while \eqref{monoidSpecfunctor} does not.  For one thing, it doesn't preserve initial objects: $0 = \{ 0 \}$ is the terminal object in $\Mon$, but $\Spec 0$ is the \emph{terminal} object in $\LMS$, not the initial object (which is the empty locally monoidal space).  For another thing, for \emph{any} monoids $P,Q$, $$\Spec P \coprod \Spec Q$$ is \emph{not} isomorphic to $\Spec R$ for any monoid $R$ (i.e.\ it isn't ``affine" so in particular it isn't isomorphic to $\Spec (P \times Q)$) because the former has precisely two points $x$ for which $\M_{X,x}$ is a group (two ``generic points"), while any $\Spec R$ has precisely one point $x$ for which $\M_{X,x}$ is a group (one ``generic point").  The fact that \eqref{ringSpecfunctor} happens to preserve finite coproducts might be viewed as something of a coincidence since it does not preserve infinite coproducts or general finite direct limits.  A closely-related fact is that the analog of Lemma~\ref{lem:strictfanmaps} does not hold for rings.  For example, $\Spec (A \times A) \to \Spec A$ is strict for any ring $A$, but is not an open embedding.  Similarly, \eqref{ringSpecfunctor} rarely takes localizations to open embeddings, but \eqref{monoidSpecfunctor} takes the localization of any finitely generated monoid at any finitely generated submonoid (e.g.\ any face) to an open embedding.  Finally, \eqref{monoidSpecfunctor} has the following ``good" property not enjoyed by \eqref{ringSpecfunctor}:

\begin{thm}  \label{thm:affineinverselimits} The functor $\Spec : \Mon^{\rm op} \to \MS$ commutes with arbitrary inverse limits.  In particular, so does the ``underlying space of $\Spec$" $\Mon^{\rm op} \to \Top$. \end{thm}

\begin{proof}  Let $\{ P_i \}$ be a direct system of monoids with direct limit $P$.  Let $\{ X_i := \Spec P_i \}$ be corresponding inverse limit system in $\MS$ with inverse limit $X$.  We want to show that $X = \Spec P$.  By Theorem~\ref{thm:LMSinverselimits}, $X$ is also the inverse limit of $\{ X_i \}$ in $\LMS$.  For any $Y \in \LMS$, we the have natural isomorphisms  \be \Hom_{\LRS}(Y,X) & = & \invlim \Hom_{\LRS}(Y,X_i) \\ & = & \invlim \Hom_{\Mon}(P_i,\M_Y(Y)) \\ & = & \Hom_{\Mon}(P,\M_Y(Y)) \\ & = & \Hom_{\LMS}(Y,\Spec P) \ee using the universal property \eqref{Specadjunction} of $\Spec$, so the result follows from Yoneda.  \end{proof}

\begin{rem} Although the above result is basically an elementary statement about monoids, it does not seem so easy to prove without some sort of geometric setup of monoidal spaces and so forth.  As an interesting exercise, the reader may wish to try to prove the following very particular special case from scratch:  \be \Spec (P \oplus Q) & = & (\Spec P) \times (\Spec Q). \ee  This is an equality in $\MS$, but even to establish it in $\Top$ or $\Sets$ is a good exercise! \end{rem}

\subsection{Fans} \label{section:fans} Here we define the full subcategory $\Fans \subseteq \LMS$ of \emph{fans} which is analogous to the full subcategory of schemes inside the category $\LRS$ of locally ringed spaces.  It is useful to keep in mind \emph{Kato's analogy}:

\begin{center} monoids : fans :: rings : schemes. \end{center} 

See \cite{fans} for further discussion of fans.

\begin{defn} \label{defn:fan} A locally monoidal space $F$ isomorphic to $\Spec P$ for a monoid (resp.\ finitely generated monoid, fine monoid, fs monoid, \dots) $P$ is called an \emph{affine fan} (resp.\ \emph{finite type affine fan}, \emph{fine affine fan}, \emph{fs affine fan}, \dots).  A locally monoidal space $F$ locally isomorphic to an affine fan (resp.\ finite type affine fan, fine affine fan, fs affine fan, \dots) is called a \emph{fan} (resp.\ \emph{locally finite type fan}, \emph{fine fan}, \emph{fs fan}, \dots).  A sharp locally monoidal space $F \in \SMS$ isomorphic to $(\Spec P, \ov{\M}_P)$ in $\SMS$ for a monoid (resp.\ finitely generated monoid, \dots) $P$ will be called a \emph{sharp affine fan} (resp.\ \emph{finite type sharp affine fan}, \dots).  A sharp locally monoidal space $F$ locally isomorphic to a sharp affine fan (resp.\ finite type sharp affine fan, \dots) will be called a \emph{sharp fan} (resp.\ locally finite type sharp fan, \dots).  Let $\Fans \subseteq \LMS$ denote the full subcategory of fans and let $\SFans \subseteq \SMS$ denote the full subcategory of sharp fans.  \end{defn}

Given a fan $X$ and an open subspace $U$ of its underlying topological space, we can regard $U = (U,\M_X|U)$ as a locally monoidal space by restricting the structure sheaf from $X$.  In fact this $U$ is a fan (just as for schemes, the question is local and reduces to the fact that the basic opens in $\Spec P$ are the $\Spec P_p$ for $p \in P$) and represents the functor on $\Fans$ (or on $\LMS$) taking $Y$ to the set of $\LMS$ maps $Y \to X$ factoring through $U$ on the level of topological spaces.  By definition, any monoidal space with an open cover by fans (resp.\ sharp fans) is again a fan (resp.\ sharp fan).   This allows us to construct new fans from old by gluing along open subfans.  The sharpening of a fan is clearly a sharp fan---that, is the sharpening functor $\LMS \to \SMS$ restricts to a sharpening functor $\Fans \to \SFans$.  The functor \eqref{monoidSpecfunctor} can be viewed as a functor \bne{SpectoFans} \Spec : \Mon^{\rm op} & \to & \Fans. \ene

\begin{example} \label{example:classicalfans}  For clarity, we will refer to a fan in the usual sense of toric varieties \cite[1.4]{F} as a \emph{classical fan}.  A classical fan $(\Sigma,N)$ gives rise to a fan in the sense of Definition~\ref{defn:fan} through the usual ``dualization process:"  We let $M := \Hom(N,\ZZ)$ and for each cone $\sigma \in \Sigma$, we let \be P_{\sigma} & := & \{ m \in M : m|\sigma \geq 0 \}. \ee  This $P_{\sigma}$ is an fs monoid.  Let $U_{\sigma} := \Spec P_{\sigma}$.  If $\sigma, \tau$ are two cones in $\Sigma$ with $\sigma$ a face of $\tau$, then $P_{\sigma}$ is the localization of $P_{\tau}$ at the face \be \{ m \in P_{\tau} : m|\sigma = 0 \}, \ee so we have an open embedding $U_{\sigma} \subseteq U_{\tau}$.  We can then glue the affine fans $U_{\sigma}$ to form an fs fan $X$ in such a way that the $U_{\sigma}$ form an affine cover of $X$, and the aforementioned open embeddings become inclusions of open subspaces of $X$.  This process defines a fully faithful functor from classical fans to $\Fans$. \end{example}

The category $\Fans$ has many advantages over the category of classical fans---it is a better category in which to ``do geometry."  The analogy to keep in mind is, roughly, 

\begin{center} classical fans : fans :: varieties : schemes. \end{center} 

For example, the category of classical fans probably does not have finite inverse limits, and even if it did, those limits wouldn't commute with any kind of geometric realization because, say, a fibered product of a diagram of varieties (taken in schemes) won't generally be a variety, even if the varieties are toric varieties and the maps are toric maps.  In contrast, we have:

\begin{prop} \label{prop:Fansinverselimits} The category $\Fans$ (resp.\ $\SFans$) has finite inverse limits and coproducts preserved by the inclusion $\Fans \subseteq \LMS$ (resp.\ $\SFans \subseteq \SMS$).  A finite inverse limit of locally finite type fans or sharp fans is again locally finite type. \end{prop}

\begin{proof}  The statement about coproducts is obvious.  The statement about inverse limits is proved by working locally and comparing the universal properties of the inverse limit and $\Spec$ as in \cite[Theorem~8]{loc}.  The finiteness statement corresponds to the fact that a finite direct limit of finitely generated monoids is again finitely generated. \end{proof}

\begin{lem} \label{lem:smallestngbd} Let $X$ be a locally finite type fan, $x \in X$.  There is a smallest open neighborhood $U_x$ of $x$ in $X$and we have isomorphisms \be U_x & = & \Spec \M_{X,x} \\ & = & \Spec \M_X(U_x) \ee natural in $(X,x)$.  If $f : X \to Y$ is a map of locally finite type fans and $x \in X$, we have a natural diagram \bne{fanmapdiagram} & \xym@C+20pt{ \Spec \M_{X,x} \ar[d] \ar[r]^-{\Spec f^\dagger_x} & \Spec \M_{Y,y} \ar[d] \\ X \ar[r]^-f & Y } \ene where the vertical arrows are open embeddings onto the smallest neighborhoods of $x,y$. \end{lem}

\begin{proof} The question is local and was noted for $\Spec P$, $P$ a finitely generated monoid in \S\ref{section:Spec}.  The second statement is an exercise---one has an analogous diagram for schemes, but the vertical arrows are almost never open embeddings.  \end{proof}

\begin{lem} \label{lem:strictfanmaps2} A map $f : X \to Y$ of locally finite type fans is strict iff it is locally an isomorphism (every point $x$ has a neighborhood $(U,V)$ in $f$ such that $f|U : U \to V$ is an isomorphism). \end{lem}

\begin{proof} The nontrivial implication is a geometric restatement of Lemma~\ref{lem:strictfanmaps}; it is also clear from the diagram \eqref{fanmapdiagram}. \end{proof}

\begin{example} \label{example:P1} Recall that we studied the affine fan $\Spec \NN$ in Example~\ref{example:SpecN}.  The simplest non-affine fan is perhaps the fan $\PP^1$ obtained by gluing two copies $U$, $V$ of $\Spec \NN$ along the common open subspace $\Spec \ZZ$, \emph{using the involution} of $\Spec \ZZ$ to make the gluing.  The resulting topological space $\PP^1$ has three points: two closed points $x$, $y$, and one generic point $\eta$.  The structure sheaf $\M_{\PP^1}$ of $\PP^1$ is characterized by the diagram of generalization maps $$ \xym@C+20pt{ \M_{\PP^1,x} \ar[r] \ar@{=}[d] & \M_{\PP^1,\eta} \ar@{=}[d] & \ar[l] \ar@{=}[d] \M_{\PP^1,y} \\ \NN \ar[r]^-{n \mapsto n} & \ZZ & \ar[l]_-{n \mapsto -n} \NN, } $$ which is also the diagram of restriction maps $$ \xym{ \M_{\PP^1}(U) \ar[r] & \M_{\PP^1}(U \cap V) & \ar[l] \M_{\PP^1}(V) } $$ for the cover $\{ U,V \}$ of $\PP^1$, thus we see that $\Gamma(\PP^1,\M_{\PP^1}) = \{ 0 \}$.  If we glue \emph{without using the involution} of $\Spec \ZZ$, we obtain a fan which might be called \emph{the affine line with the origin doubled}.  We can similarly construct a fan $\PP^n$ by appropriately gluing $n+1$ copies of $\Spec(\NN^n)$, though we will give a more formal construction in Example~\ref{example:Pn}.  See Example~\ref{example:realizationoffans} (\S\ref{section:realizationoffans}) for a continuation of this example.  \end{example}

\section{Spaces} \label{section:spaces}  We first give an axiomatic setup for the category $\Esp$ of ``spaces" we will work with.  The reader uninterested in our general abstraction can just look at the list of examples at the end of this section and move on to \S\ref{section:logstructures}.

\begin{defn} \label{defn:spaces} A \emph{category of spaces} $(\Esp,\AA^1)$ consists of a category $\Esp$, a functor $\Esp \to \Top$, called the \emph{underlying space functor}, denoted $X \mapsto |X|$, and a monoid object $\AA^1$ of $\Esp$ satisfying the following axioms, explained in more detail in the next section: \begin{enumerate}[label=(S\arabic*), ref=S\arabic*] \item \label{finitelimits} {\bf (Finite limits)} $\Esp$ has all finite inverse limits.  \item \label{coproducts} {\bf (Coproducts)} $\Esp$ has arbitrary (small) coproducts.  \item \label{coproductscommute} {\bf (Coproducts commute)} Coproducts in $\Esp$ commute with $X \mapsto |X|$.  \item \label{openembeddings} {\bf (Open embeddings)} Open subspaces are representable. \item \label{Zariskitopology} {\bf (Zariski topology)} The Zariski topology on $\Esp$ is subcanonical.  \item \label{coproductsandopenembeddings} {\bf (Coproducts and open embeddings)} The structure maps $X_i \to \coprod_i X_i$ to any coproduct in $\Esp$ are open embeddings.  \item \label{Zariskigluing} {\bf (Zariski Gluing)}  Objects of $\Esp$ can be glued Zariski locally. \item \label{openunits} {\bf (Open Units)} The group of units $\GG_m \subseteq \AA^1$ is representable by an open subspace.  \end{enumerate} Objects of $\Esp$ will be called \emph{spaces}. \end{defn}

\begin{rem} \label{rem:spaces}  The axioms \eqref{finitelimits} and \eqref{coproducts} concern only the category $\Esp$, the next five axioms concern the underlying space functor $\Esp \to \Top$, and \eqref{openunits} concerns everything. \end{rem}

\begin{defn} \label{defn:morphismofspaces} A $1$-\emph{morphism} of categories of spaces \be (F,\eta) : (\Esp,\AA^1) & \to & (\Esp', (\AA^1)') \ee is a functor $F : \Esp \to \Esp'$ together with a natural transformation $ \eta :  |F(\slot)|  \to  | \slot | $ of functors $\Esp \to \Top$ such that \begin{enumerate}[label=(M\arabic*), ref=M\arabic*]  \item \label{limitpres} $F$ preserves finite inverse limits and coproducts. \item \label{A1toA1} $F(\AA^1) = (\AA^1)'$ \emph{as monoid objects}---meaning that $$ \xym{ F(\AA^1 \times \AA^1) \ar[r]^{F(+)} \ar[d]_{(F\pi_1,F\pi_2)} & F(\AA^1) = (\AA^1)' \\ (\AA^1)' \times (\AA^1)' \ar[ru]_-{+'} } $$ commutes.  \item \label{openembeddingspreserved} The image $FU \to FX$ of an open embedding $U \to X$ in $\Esp$ is ``the" open embedding corresponding to $\eta_X^{-1}(|U|) \subseteq |FX|$.  \end{enumerate} A $2$-\emph{morphism} $\gamma : (F,\eta)  \to  (F',\eta') $ between $1$-morphisms \be (F,\eta),(F',\eta') : (\Esp,\AA^1) & \rightrightarrows & (\Esp', (\AA^1)') \ee is a natural transformation of functors $\gamma : F \to F'$ such that $\gamma(\AA^1) = \Id$ and such that the diagram $$ \xym{ |FX| \ar[r]^-{\eta_X} \ar[d]_{\gamma(X)}^\cong & |X| \\ |F'X| \ar[ru]_{\eta_X'} } $$ in $\Top$ commutes for each $X \in \Esp$.   The $2$-category of categories of spaces is called the \emph{universe} and is denoted $\Univ$.  \end{defn}

\subsection{On the axioms for spaces} \label{section:axiomsforspaces}  The meaning of \eqref{finitelimits}-\eqref{coproductscommute} in Definition~\ref{defn:spaces} is clear.  Note that the axiom \eqref{coproducts} in particular implies that $\Esp$ has an initial object, which we denote $\emptyset$.  Aside from the final axiom \eqref{openunits}, the axioms concern only the category $\Esp$ and the ``underlying space functor" $\Esp \to \Top$, which we always denote $X \mapsto |X|$.

The open embeddings axiom \eqref{openembeddings} for $\Esp \to \Top$ means that for any $X \in \Esp$ and any open subset $|U| \subseteq |X|$, the presheaf \be Y & \mapsto & \{ f : Y \to X : |f|(|Y|) \subseteq |U| \} \ee is representable by some $U \in \Esp$ so that $| \slot |$ of the natural $\Esp$-morphism $U \to X$ is an open embedding with image $|U| \subseteq |X|$.  (We don't really require \emph{equality}, as the notation suggests, but any open embedding with image $|U|$ is uniquely isomorphic, as a space over $|X|$, to $|U|$.)  We call a morphism $U \to X$ representing such a presheaf an \emph{open embedding}.  Note that if $i : U \to X$ and $i' : U' \to X$ are open embeddings with $|i|(|U|) = |i'|(|U'|)$, then $U$ and $U'$ are uniquely isomorphic as spaces over $X$---the category of open embeddings into $X$ is equivalent to the category of open subspaces of $|X|$ (with inclusions as the morphisms).  As part of this axiom, we will also require that the open embedding associated to the empty set $\emptyset \subseteq |X|$ ``is" the unique map $\emptyset \to X$ from the initial object $\emptyset$ of $\Esp$.  This is equivalent to demanding that $X \in \Esp$ is an initial object whenever $|X| = \emptyset$.  This should be viewed as a very mild ``faithfulness" assumption on the forgetful functor $X \mapsto |X|$.

\begin{rem} \label{rem:EspTopnotation}  To ease notation and terminology, we often, especially in later sections, blur the distinction between $X$ and $|X|$.  For example, by a ``point $x \in X$" we mean ``a point of $x \in |X|$" and by a ``neighborhood of $x$ in $X$" we mean an open embedding $U \into X$ such that $|U| \subseteq |X|$ is a neighborhood of $x$ in $|X|$. \end{rem}

\begin{lem} \label{lem:Espfiberedproducts}  Suppose $\Esp \to \Top$ satisfies \eqref{openembeddings}.  Then: \begin{enumerate} \item \label{Espfiberedproducts1} A composition of open embeddings is an open embedding.  \item \label{Espfiberedproducts1a} If $f : X \to Y$ and $g : Y \to Z$ are $\Esp$ morphisms such that $g$ and $gf$ are open embeddings, then $f$ is an open embedding.  \item \label{Espfiberedproducts2} The base change of any open embedding in $\Esp$ exists and is itself an open embedding.  \item \label{Espfiberedproducts3} Base change along an open embedding in $\Esp$ commutes with $\Esp \to \Top$. \end{enumerate}  \end{lem}

\begin{proof}  For \eqref{Espfiberedproducts1}, suppose $U \to V$ and $V \to X$ are open embeddings.   Then the composition $|U| \to |V| \to |X|$ is an open embedding in $\Top$ and the universal properties of $U \to V$ and $V \to X$ combine to show that the composition $U \to X$ ``is" the open embedding corresponding to $|U| \subseteq |X|$.  For \eqref{Espfiberedproducts1a}, first note that $|f|(|X|)$ is open in $|Y|$ because the statement \eqref{Espfiberedproducts1a} holds in $\Top$ (exercise!).  Next, use the fact that $g$ is an open embedding to see that $h \mapsto gh$ defines a bijection, natural in $T \in \Esp$, from the set of $h \in \Hom_{\Esp}(T,Y)$ for which $|h|(|T|) \subseteq |f|(|X|)$ to the set of $k \in \Hom_{\Esp}(T,Z)$ for which $|k|(|T|) \subseteq |gf|(|X|)$.  Since $gf$ is an open embedding, the latter set is bijective with $\Hom_{\Esp}(T,X)$ naturally in $T$.  Putting these two natural bijections together shows that $f$ is an open embedding.  For \eqref{Espfiberedproducts2} and \eqref{Espfiberedproducts3} suppose $f : X \to Y$ is an $\Esp$-morphism and $U \subseteq Y$ is an open embedding.  Then it is immediate from the universal property of open embeddings that the open embedding of spaces $f^{-1}(U) \to X$ corresponding to the open subspace $|f|^{-1}(|U|)$ of $|X|$ ``is" the fibered product $U \times_Y X$ (i.e.\ the usual diagram is cartesian); \eqref{Espfiberedproducts2} and \eqref{Espfiberedproducts3} follow. \end{proof}

\begin{defn} \label{defn:EspZariskitopology} We say that an $\Esp$-morphism $f : X \to Y$ is \emph{surjective} iff the map of topological spaces $|f|$ is surjective.  We define a \emph{local isomorphism} and a \emph{Zariski cover} in terms of our notion of open embedding in ``the usual way" (Definition~\ref{defn:localisomorphism}). \end{defn}  

We can jazz up Lemma~\ref{lem:Espfiberedproducts} as follows:

\begin{prop} \label{prop:Espfiberedproducts} Suppose $\Esp \to \Top$ satisfies \eqref{finitelimits} and \eqref{openembeddings}.  Then: \begin{enumerate} \item \label{Espfib1} The map of topological spaces underlying a local isomorphism (resp.\ Zariski cover) is a local homeomorphism (resp.\ Zariski cover).  \item \label{Espfib2} A composition of local isomorphisms (resp.\ surjections, Zariski covers) is a local isomorphism (resp.\ surjection, Zariski cover).  \item \label{Espfib3} If $f : X \to Y$ and $g : Y \to Z$ are $\Esp$ morphisms such that $g$ and $gf$ are local isomorphisms, then $f$ is a local isomorphism.  \item \label{Espfib4} Local isomorphisms and Zariski covers are stable under base change in $\Esp$ and base change along a local isomorphism commutes with $\Esp \to \Top$. \end{enumerate}  \end{prop}

\begin{proof}  All of these statements are readily deduced from the corresponding statements in Lemma~\ref{lem:Espfiberedproducts}.  Establishing \eqref{Espfib4} is the most difficult, so we will give a sketch of the argument.  We need to show that the diagram of topological spaces \bne{isitcartesian} & \xym{ |X'| \ar[r] \ar[d] & |X| \ar[d]^{| f |} \\ |Y'| \ar[r] & |Y| } \ene underlying a cartesian $\Esp$ diagram \bne{cartEspdiagram} & \xym{ X' \ar[r] \ar[d] & X \ar[d]^f \\ Y' \ar[r] & Y } \ene is cartesian when $f$ is a local isomorphism.  To see this, we first show that for any $x \in |X|$, there is a cartesian square of the form \bne{boxd} &  \xym{ U' \ar[r]^-j \ar@{=}[d] & U \ar@{=}[d] \\ U' \ar[r]^-j & U } \ene mapping to \eqref{cartEspdiagram} such that each map from a corner of \eqref{boxd} to the corresponding corner of \eqref{cartEspdiagram} is an open embedding and $x \in |U| \subseteq |X|$.  To do this, first use the fact that $f$ is a local isomorphism to find an open embedding $i : U \to X$ with $x \in |U| \subseteq |X|$ such that $fi : U \to Y$ is also an open embedding.  Now define $U' := X' \times_X U$, $U'' := Y' \times_Y U$ using \eqref{finitelimits}.  The projections $U' \to X'$, $U'' \to Y'$ are open embeddings by Lemma~\ref{lem:Espfiberedproducts} and by a simple ``category theory" exercise, we in fact have a natural isomorphism $U'=U''$, so we get our diagram.  Now one uses this diagram (for appropriate $x$) to check, using Lemma~\ref{lem:Espfiberedproducts}, that \begin{enumerate} \item $|X'| \to |Y'| \times_{|Y|} |X|$ is injective, \item $|X'| \to |Y'| \times_{|Y|} |X|$ is surjective, and \item the topology on $|X'| \subseteq |Y'| \times |X|$ is the ``subspace of the product" topology. \end{enumerate}  \end{proof}  

Proposition~\ref{prop:Espfiberedproducts} shows that when $\Esp \to \Top$ satisfies \eqref{finitelimits} and \eqref{openembeddings}, Zariski covers form the covers (or, really, ``generating covers") for a topology on $\Esp$, called the \emph{Zariski topology}.  The assumption in the Zariski topology axiom \eqref{Zariskitopology} is that this topology be subcanonical.  Explicitly, this means that whenever $f : X \to Y$ is a Zariski cover of spaces, the diagram of sets $$ \xym{ \Hom_{\Esp}(Y,T) \ar[r]^-{f^*} & \Hom_{\Esp}(X,T) \ar@<0.5ex>[r]^-{\pi_1^*} \ar@<-0.5ex>[r]_-{\pi_2^*} & \Hom_{\Esp}(X \times_Y X, T) } $$ is an equalizer diagram for any $T \in \Esp$.  If we assume \eqref{finitelimits}-\eqref{coproductsandopenembeddings}, then we can see that this Zariski topology is very much like the one on $\Top$:

\begin{lem} \label{lem:Zariskicovers} Suppose $\Esp \to \Top$ satisfies \eqref{finitelimits}-\eqref{coproductsandopenembeddings}.  Then: \begin{enumerate} \item \label{Zcovers1} If $f : X \to Y$ is a local isomorphism such that $|f|$ is injective (resp.\ bijective), then $f$ is an open embedding (resp.\ isomorphism). \item \label{Zcovers2} If $\{ f_i : U_i \to X \}$ is a set of open embeddings such that the $|U_i|$ cover $|X|$, then the coproduct of the $f_i$ yields a Zariski cover $f : U \to X$.  Furthermore, $U \times_X U = \coprod_{(i,j)} U_{ij}$, where $U_{ij} = U_i \times_X U_j$ is the open embedding to $X$ with image $|U_i| \cap |U_j|$.  \item \label{Zcovers3} For any space $X$ and any Zariski cover $V \to X$ there is a Zariski cover $U \to X$, constructed as in the previous part, and a map $U \to V$ of spaces over $X$. \end{enumerate}  \end{lem}

\begin{proof} We first prove the first part of \eqref{Zcovers2}:   The map $f : U \to X$ is a local isomorphism because the structure map $U_i \to U$ is an open embedding by \eqref{coproductsandopenembeddings} whose composition with $f$ is the open embedding $f_i$ and the $|U_i|$ cover $|U|$ by \eqref{coproductscommute}.  The map $|f|$ is surjective because the $|f_i|(|U_i|)$ cover $|X|$ by assumption and each $|f_i|$ factors through $|f|$.  This proves that $f$ is a Zariski cover.

\eqref{Zcovers1}:  The map of topological spaces $|f|$ underlying a local isomorphism $f : X \to Y$ is a local homeomorphism by Proposition~\ref{prop:Espfiberedproducts}\eqref{Espfib1}.  In particular $|f|$ is open, so it is a homeomorphism iff it is bijective, and its image $|f|(|X|)=:|U|$ is open in $|Y|$, so we can consider (using \eqref{openembeddings}) the corresponding open embedding $i : U \to Y$.  By the universal property of $i$, we can write $f=ig$ for some map $g : X \to U$.  Clearly $|g|$ is surjective and $g$ is a local isomorphism by Proposition~\ref{prop:Espfiberedproducts}\eqref{Espfib3}.  If $|f|$ is injective, then $|g|$ must also be injective.  We thus reduce to proving that if $f : X \to Y$ is a Zariski cover for which $|f|$ is injective (equivalently a homeomorphism), then $f$ is an isomorphism.  We first prove that $f$ is a monomorphism.  Suppose $g_1,g_2 : T \rightrightarrows X$ are maps of spaces with $fg_1=fg_2$.  Since $f$ is a local homeomorphism, we can choose, for each $x \in |X|$, an open embedding $i_x : U_x \to X$ such that $x \in |U_x|$ and $fi_x$ is also an open embedding.  Since $|f|$ is a homeomorphism, we see from Lemma~\ref{lem:Espfiberedproducts} that the diagram $$ \xym{ U_x \ar@{=}[r] \ar[d]_{i_x} & U_x \ar[d]^{fi_x} \\ X \ar[r] & Y } $$ is cartesian (here we just need $|f|$ injective).  If we now let $p_x : V_x \to T$ denote the base change of the open embedding $fi_x : U_x \to Y$ along the map $fg_1=fg_2$, then we see from the cartesian diagram above that $g_1p_x=g_2p_x$.  By \eqref{coproducts}, we can form the coproduct $V := \coprod_{x \in |X|} V_x$.  Let $p : V \to T$ be the coproduct of the $p_x$.  Since $|f|$ is surjective, the $|fi_x|(|U_x|)$ cover $|Y|$, so the $p_x$ are open embeddings such that the $|V_x|$ cover $T$ by Lemma~\ref{lem:Espfiberedproducts}, hence $p : V \to T$ is a Zariski cover by the first part of \eqref{Zcovers2} established above.  Since $g_1p_x=g_2p_x$ for all $x \in |X|$, we have $g_1p=g_2p$ by the universal property of the coproduct, so, since a Zariski cover is, in particular, an epimorphism, we conclude that $g_1=g_2$.  Now, since $f : X \to Y$ is monic, we have $X \times_Y X = X$ (with the projections $\pi_1,\pi_2: X \times_Y X \rightrightarrows X$ given by the the identity map of $X$), hence $\Id_X$ is in the equalizer of $$ \xym{ \Hom_{\Esp}(X,X) \ar@<0.5ex>[r]^-{\pi_1^*} \ar@<-0.5ex>[r]_-{\pi_2^*} & \Hom_{\Esp}(X \times_Y X,X), } $$ so, since $f$ is a Zariski cover and we assume in \eqref{Zariskitopology} that the Zariski topology is subcanonical, there is a unique $\Esp$-morphism $g : Y \to X$ such that $gf=\Id_X$.  The Zariski cover $f$ is also an epimorphism and we have $fgf = f$, so we conclude that $fg = \Id_Y$, hence $f$ and $g$ are inverse isomorphisms.    

Now we prove the second part of \eqref{Zcovers2}:   For $(i,j) \in I \times I$, we have open embeddings $s_{ij} : U_{ij} \to U_i$ and $t_{ij} : U_{ij} \to U_j$ with $f_is_{ij} = f_j t_{ij}$ equal to the open embedding $U_{ij} \to X$.  Composing $s_{ij}$ and $t_{ij}$ with the structure maps $U_i \to U$ and $U_j \to U$ (these are open embeddings by \eqref{coproductsandopenembeddings}) we obtain open embeddings $a_{ij} : U_{ij} \to U$, $b_{ij} : U_{ij} \to U$ with $$fa_{ij} = f_i s_{ij} = f_j t_{ij} = fb_{ij}$$ and hence a map $g_{ij} := (a_{ij},b_{ij}): U_{ij} \to U \times_X U$.  The map $g_{ij}$ is a local isomorphism by Proposition~\ref{prop:Espfiberedproducts}\eqref{Espfib3} because its composition with $\pi_1 : U \times_X U \to U$ is the open embedding (in particular, local isomorphism) $a_{ij}$ and $\pi_1$ is a local isomorphism by Proposition~\ref{prop:Espfiberedproducts}\eqref{Espfib4} since it is a base change of $f$.  The coproduct of the $g_{ij}$, over all $(i,j) \in I \times I$, gives a map $g : \coprod_{(i,j)} U_{ij} \to U \times_X U$, which is a local isomorphism by the same argument used to establish the first part of \eqref{Zcovers2} above.  The computation \be | \coprod_{(i,j)} U_{ij} | & = & \coprod_{(i,j)} |U_{ij}| \\ & = & ( \coprod_i |U_i| ) \times_{|X|} (\coprod_i |U_i|) \\ & = & |U| \times_{|X|} |U| \\ & = & |U \times_X U| \ee using \eqref{coproductscommute}, an elementary topology exercise, \eqref{coproductscommute}, and Proposition~\ref{prop:Espfiberedproducts}\eqref{Espfib4} shows that $|g|$ is a homeomorphism, hence $g$ is an isomorphism by \eqref{Zcovers1}.

\eqref{Zcovers3} is proved by the same elementary argument one would make in the case $\Esp=\Top$.   \end{proof}

Lemma~\ref{lem:Zariskicovers} implies that coproducts in $\Esp$ behave ``as one would expect in any geometric category."

\begin{lem} \label{lem:coproductsofspaces} Suppose $\Esp \to \Top$ satisfies \eqref{finitelimits}-\eqref{coproductsandopenembeddings}.  Let $\{ Y_i : i \in I \}$ be an indexed family of spaces with coproduct $Y$.  Then for distinct $i,j \in I$, $Y_i \times_Y Y_j = \emptyset$ is ``the" initial object of $\Esp$ and if $X \to Y$ is any $\Esp$-morphism and we set $X_i := X \times_Y Y_i$, then the natural maps $X_i \to X$ make $X$ ``the" coproduct of the $X_i$.  It follows that to give a map $f : X \to Y$ is to give a coproduct decomposition $X = \coprod_{i \in I} X_i$ of $X$ and a map $f_i : X_i \to Y$ for each $i \in I$.  \end{lem}

\begin{proof} By \eqref{coproductscommute} $|Y| = \coprod_i |Y_i|$ and by \eqref{coproductsandopenembeddings} $Y_i \to Y$ is an open embedding with underlying $\Top$-morphism given by the open embedding $|Y_i| \to |Y|$.  Since $\Esp \to \Top$ commutes with base change along an open embedding (Lemma~\ref{lem:Espfiberedproducts}), we find that $|Y_i \times_Y Y_j| = |Y_i| \times_{|Y|} |Y_j| = \emptyset$, hence $|Y_i \times_Y Y_j|=\emptyset$ by the ``mild faithfulness" of $\Esp \to \Top$ assumed in \eqref{openembeddings}.  For the next statement, note that the maps $X_i \to X$ are open embeddings with $|X| = \coprod_i |X_i|$ by Lemma~\ref{lem:Espfiberedproducts}.  Using \eqref{coproductscommute} and \eqref{coproductsandopenembeddings}, we then see that the natural map $\coprod_i X_i \to X$ is a local isomorphism and a homeomorphism on spaces, hence it is an isomorphism by Lemma~\ref{lem:Zariskicovers}\eqref{Zcovers1}.  \end{proof}

Lemma~\ref{lem:Zariskicovers} also implies that, for any $X,Y \in \Esp$, the presheaf \bne{Esppresheaf} |U| & \mapsto & \Hom_{\Esp}(U,Y) \ene on the topological space $|X|$ is a sheaf.  Here $|U|$ is an open subset of $|X|$ and $U \to X$ denotes ``the" corresponding open embedding.  As in Remark~\ref{rem:EspTopnotation} we will often blur the distinction between $X$ and $|X|$ and write ``$U$" instead of ``$|U|$" and say that this is a sheaf ``on $X$."

To motivate the Zariski Gluing axiom \eqref{Zariskigluing}, let us first make a simple observation:

\begin{prop} \label{prop:Zariskigluing} Suppose $\Esp \to \Top$ satisfies \eqref{finitelimits}-\eqref{coproductsandopenembeddings}.  Let $X$ be a space.  Then: \begin{enumerate} \item \label{pushoutsalongopenembeddings} Suppose $|U_1|$ and $|U_2|$ are open subsets of $|X|$ covering $|X|$.  Let $U_i \to X$ ($i=1,2$) be the open embedding corresponding to $|U_i|$ and let $U_{12} \to U_i$ ($i=1,2$) be the open embedding corresponding to $|U_1| \cap |U_2| \subseteq |U_i|$.  Then $$ \xym{ U_{12} \ar[r] \ar[d] & U_1 \ar[d] \\ U_2 \ar[r] & X } $$ is a pushout diagram in $\Esp$. \item \label{unionofopenembeddings} Let $i \mapsto |U_i|$ be a map of partially ordered sets from a well-ordered set $I$ to the set of open subsets of $|X|$, ordered by inclusion.  Assume that $|X| = \cup_i |U_i|$.  For $i \leq j$ in $I$, let $U_i \to U_j$ be the open embedding of spaces corresponding to the open embedding $|U_i| \subseteq |U_j|$ in $\Top$.  Then $X$, with the open embeddings $U_i \to X$ as the structure maps, is the direct limit, in $\Esp$, of $i \mapsto U_i$. \end{enumerate} \end{prop}

\begin{proof} Using the properties of the Zariski topology from Lemma~\ref{lem:Zariskicovers}, one can prove this in the same way it is proved in the case $\Esp=\Top$. \end{proof}

\begin{defn} \label{defn:Zariskigluingdatum} A \emph{Zariski gluing datum} is a direct limit system in $\Esp$ of one of the following two types: \begin{enumerate} \item Two open embeddings $j_i : X_0 \to X_i$ ($i=1,2$), viewed as a functor to $\Esp$ from the category $I$ with three objects $0$, $1$, $2$ and exactly two non-identity morphisms $0 \to i$ ($i=1,2$).  \item A functor $i \mapsto X_i$ from a well-ordered set $I$ to $\Esp$ such that, for all $i,j \in I$ with $i \leq j$, the ``transition function" $X_i \to X_j$ is an open embedding. \end{enumerate} \end{defn}

The Zariski Gluing axiom \eqref{Zariskigluing} asserts that every Zariski gluing datum has a direct limit $X$, that the structure maps $X_i \to X$ are open embeddings, and that this direct limit is preserved by the forgetful functor $\Esp \to \Top$.  Informally, this axiom asserts that one can form the pushout of two open embeddings and the union of a (possibly transfinite) sequence of open embeddings, and that these limits have ``the expected properties."  Notice that, in the case of a pushout along open embeddings, the assumptions that the structure maps $X_i \to X$ are open embeddings and that the pushout square is preserved by $\Esp \to \Top$ ensure that the pushout square is also a \emph{pullback} square (Lemma~\ref{lem:Espfiberedproducts}).

\begin{prop} \label{prop:Zariskigluingspreserved} Any $1$-morphism of spaces $(F,\eta)$ (Definition~\ref{defn:morphismofspaces}) preserves direct limits of Zariski gluing data.\end{prop}

\begin{proof}  Let $(F,\eta)$ be a $1$-morphism from $(\Esp,\AA^1)$ to $(\Esp',(\AA^1)')$.\footnote{The monoid objects are irrelevant here and the assumptions \eqref{A1toA1} and \eqref{limitpres} on $(F,\eta)$ are not needed.}  Consider a pushout square \bne{pushoutsq} & \xym{ X_0 \ar[r]^-{j_1} \ar[d]_{j_2} & X_1 \ar[d]^{k_1} \\ X_2 \ar[r]^-{k_2} & X } \ene in $\Esp$ where $j_1$ and $j_2$ are open embeddings and $|X_0| = |X_1| \cap |X_2|$.  By assumption \eqref{Zariskigluing}, $k_1$ and $k_2$ are also open embeddings and the underlying square in $\Top$ is also a pushout, so $|X| = |X_1| \cup |X_2|$.  By the property \eqref{openembeddingspreserved} of $(F,\eta)$ in Definition~\ref{defn:morphismofspaces}, the square \bne{pushoutsq2} & \xym{ FX_0 \ar[r]^-{Fj_1} \ar[d]_{Fj_2} & FX_1 \ar[d]^{Fk_1} \\ FX_2 \ar[r]^-{Fk_2} & FX } \ene is the diagram of open embeddings whose underlying $\Top$ diagram is \bne{pushoutsq3} & \xym{ \eta_X^{-1}|X_0| \ar[r] \ar[d] & \eta_X^{-1}|X_1| \ar[d] \\ \eta_X^{-1}|X_2| \ar[r] & \eta_X^{-1}|X| = |FX| } \ene where $\eta_X : |FX| \to |X|$ is the $\Top$ morphism given by evaluating the natural transformation $\eta : |F(\slot)| \to |\slot|$ on $X$.  Since we have $$ (\eta_X^{-1}|X_1|) \cap (\eta_X^{-1}|X_2|) = \eta_X^{-1}(|X_1| \cap |X_2|) = \eta_X^{-1}|X_0| $$ and $$(\eta_X^{-1}|X_1|) \cup (\eta_X^{-1}|X_2|) = \eta_X^{-1}(|X_1| \cup |X_2|) = \eta_X^{-1}(|X|) = |FX|,$$ the diagram \eqref{pushoutsq2} is a pushout by Proposition~\ref{prop:Zariskigluing}\eqref{pushoutsalongopenembeddings}.  The case of a ``union of open embeddings" is proved in the same way using Proposition~\ref{prop:Zariskigluing}\eqref{unionofopenembeddings}.  \end{proof}

To explain the axiom \eqref{openunits} for $(\Esp,\AA^1)$, recall that $\AA^1$ is a monoid object, so for any $U \in \Esp$, the set $\AA^1(U) = \Hom_{\Esp}(U,\AA^1)$ has a monoid structure, natural in $U \in \Esp$.  The axiom \eqref{openunits} says that the presheaf \be \GG_m : \Esp^{\rm op} & \to & \Sets \\ U & \mapsto & \AA^1(U)^* = \Hom_{\Esp}(U,\AA^1)^* \ee given by the group of units in this monoid is representable by some space (this actually follows from \eqref{finitelimits}, as we will see in \S\ref{section:monoidstospaces}), also denoted $\GG_m$, and that the induced map $\GG_m \to \AA^1$ should be an open embedding.  Note that $\GG_m$ is actually a presheaf of groups, so that the representing object $\GG_m$ will be a group object in $\Esp$.

For example, when $\Esp=\Top$, the monoid object $\AA^1$ is just a monoid $\AA^1$ equipped with a topology such that the addition map $+ : \AA^1 \times \AA^1 \to \AA^1$ is continuous.  The ``group of units presheaf" $\GG_m$ defined above is represented by the group of units $\GG_m$ in the monoid $\AA^1$, with the subspace topology from the inclusion $\GG_m \into \AA^1$.  The latter map is an open embedding iff $\GG_m$ is an \emph{open} subset of $\AA^1$.  So, for example, while any monoid $P$ can be regarded as a monoid object in $\Top$ by giving $P$ the anti-discrete topology, the subspace $P^* \subseteq P$ will not be open unless $P=P^*$ is a group.  On the other hand, any monoid $P$ can be regarded as a monoid object in $\Top$ by giving $P$ the \emph{discrete} topology and in that case $(\Top,P)$ will be a category of spaces (with the identity as the underlying space functor).  

We close this section with a few remarks about $1$- and $2$-morphisms in the $2$-category $\Univ$.

\begin{lem} \label{lem:1morphismofspaces} Suppose $(F,\eta)$ is a $1$-morphism of spaces as in Definition~\ref{defn:morphismofspaces}.  Then $F$ takes local isomorphisms (resp.\ Zariski covers) to local isomorphisms (resp. Zariski covers). \end{lem}

\begin{proof} This follows from the property \eqref{openembeddingspreserved}. \end{proof}

\begin{lem} \label{lem:2morphismofspaces} Suppose that $(F,\eta)$, $(F',\eta')$ are $1$-morphisms of spaces from $(\Esp,\AA^1)$ to $(\Esp',(\AA^1)')$.  Assume that for every $X \in \Esp$ we have $x=y$ whenever $x,y \in |X|$ are specializations of each other.  Then any isomorphism of functors $\gamma$ from $F$ to $F'$ with $\gamma(\AA^1)=\Id$ is in fact an invertible $2$-morphism in $\Univ$. \end{lem}

\begin{proof}  The issue is to show that the diagram \bne{2mordia} & \xym{ |FX| \ar[r]^-{\eta_X} \ar[d]_{\gamma(X)}^\cong & |X| \\ |F'X| \ar[ru]_{\eta_X'} } \ene in $\Top$ commutes for each $X \in \Esp$.  First observe that it is ``essentially commutative" in the sense that the diagram of partially-ordered sets $$ \xym{ \Open(|FX|) & \ar[l]_-{\eta_X^{-1}}  \ar[ld]^{(\eta'_X)^{-1}} \Open(|X|) \\ \Open(|F'X|)  \ar[u]^{\gamma(X)^{-1}}_\cong } $$ commutes---here $\Open(Y)$ denotes the set of open subsets of a topological space $Y$.  To see this, take an open subset $|U| \subseteq |X|$ and let $U \to X$ be the corresponding open embedding in $\Esp$.  Then \bne{Espprimediagram} & \xym{ F U \ar[r]^-{\gamma(U)}_-\cong \ar[d] & F'U \ar[d] \\ FX \ar[r]^-{\gamma(X)}_-\cong & F'X } \ene commutes in $\Esp'$ since $\gamma$ is natural and the $\Top$-diagram underlying \eqref{Espprimediagram} is $$ \xym{ \eta^{-1}_X(|U|) \ar[d] \ar[r]^-{|\gamma(U)|}_-\cong & (\eta_X')^{-1}(|U|) \ar[d] \\ |FX| \ar[r]^-{|\gamma(X)|}_-\cong & |F'X| }$$  by the property \eqref{openembeddingspreserved} of the $1$-morphisms $(F,\eta)$, $(F',\eta')$.  Now one uses the assumption on $|X|$ to show that this ``essential commutativity" implies commutativity. \end{proof}

\subsection{Spaces to locally monoidal spaces} \label{section:spacestolocallymonoidalspaces}  The monoid object $\AA^1$ for a category of spaces $(\Esp,\AA^1)$ gives rise to a sheaf $\str_X$ of monoids on $X$ (really on $|X|$) for each $X \in \Esp$ defined by \bne{strX} \str_X(U) & := & \AA^1(U) := \Hom_{\Esp}(U,\AA^1). \ene  This is a special case of \eqref{Esppresheaf}.  We call $\str_X$ the \emph{structure sheaf} of $X$.  The sheaf of monoids $\str_X$ is functorial in $X \in \Esp$ in the sense that an $\Esp$-morphism $f : X \to Y$ induces a map of sheaves of monoids $f^\sharp : f^{-1} \str_Y \to \str_X$ on $X$ (really $X$ should be $|X|$ and $f^{-1}$ should be $|f|^{-1}$, but we want to start indulging in these sort of notational abuses).  We see in Lemma~\ref{lem:localstalks} below that the axiom \eqref{openunits} ensures that the map $f^\sharp$ is local, so that \be (f,f^\sharp) : (|X|,\str_X) & \to & (|Y|,\str_Y) \ee is a morphism of locally monoidal spaces (Definition~\ref{defn:monoidalspace}).  We thus obtain, for any category of spaces $(\Esp,\AA^1)$, a functor \bne{EsptoLMS} \Esp & \to & \LMS \\ \nonumber X & \mapsto & |X| = (|X|,\str_X). \ene

\begin{lem} \label{lem:localstalks} Let $f : X \to Y$ be a map of spaces.  For any $x \in X$, the stalk $f^\sharp_x : \str_{Y,f(x)} \to \str_{X,x}$ of $f^\sharp$ at $x$ is a local map of monoids (\S\ref{section:monoidbasics}). \end{lem}

\begin{proof} We need to show that $m \in \str_{Y,f(x)}^*$ assuming $f^\sharp_x(m) \in \str_{X,x}^*$.  This $m$ is the germ of a map from $Y$ to $\AA^1$ near $f(x)$, represented by a map $\ov{m} : U \to \AA^1$ defined on some neighborhood $U$ of $f(x)$ in $Y$.  Then $f^\sharp_x(m)$ is the germ of $\ov{m}f : f^{-1}(U) \to \AA^1$ at $x$.  To say that the latter germ is in $f^\sharp_x(m) \in \str_{X,x}^*$ is to say that $\ov{m}f$ factors through $\GG_m \into \AA^1$ after possibly shrinking $f^{-1}(U)$ to a smaller neighborhood of $x$.  But $\GG_m \into \AA^1$ is an open embedding, so this is the same thing as saying that $\ov{m}f(x) \in \GG_m$, which is similarly the same thing as saying that $\ov{m}$ factors through $\GG_m \subseteq \AA^1$ after replacing $U$ with a smaller neighborhood of $f(x)$, which is the same thing as saying $m \in \str_{Y,f(x)}^*$. \end{proof}

\begin{rem} \label{rem:EsptoLMS}  The functor \eqref{EsptoLMS} is not generally a $1$-morphism of spaces (Definition~\ref{defn:morphismofspaces}) because it does not generally preserve finite inverse limits.  It does, however, take open embeddings to open embeddings, as is clear from its definition.  It follows, as in the proof of Proposition~\ref{prop:Zariskigluingspreserved}, that \eqref{EsptoLMS} preserves direct limits of Zariski gluing data.   \end{rem}

\subsection{Examples} \label{section:spacesexamples}  Now we discuss the examples of ``categories of spaces" that we have in mind.  In most cases of serious interest to us, $\Esp$ will be a subcategory of the category $\LRS$ of locally ringed spaces, and the monoid object $\AA^1$ will be the space respresenting the presheaf \be \Esp^{\rm op} & \to & \Mon \\ X & \mapsto & \O_X(X) \ee (here $\O_X(X)$ is regarded as a monoid under multiplication), so that the sheaf of monoids $\str_X$ on $X$ defined in \eqref{strX} is nothing but the structure sheaf $\O_X$.  The major exception is that, when $\Esp = \DS$, the ``more useful" choice of monoid object is $\AA^1 = \RR_+$, so that the sheaf $\str_X$ is the sheaf of ``non-negative functions" on $X$, discussed in \S\ref{section:positivefunctions}.

We now list some categories of spaces (the functor $\Esp \to \Top$ is ``the obvious underlying space functor" unless otherwise mentioned) together a description of the structure sheaf $\str_X$ and the group of units $\GG_m$ in each case: \begin{enumerate}[label=(E\arabic*), ref=E\arabic*] \item \label{discretespaces} $(\Sets,P)$, $P$ any monoid.  Here $\str_X(U) = \Hom_{\Sets}(U,P) = P^U$, $\GG_m = P^*$.  The underlying space functor is the ``discrete topology" functor $\Sets \to \Top$.  \item $(\Top,\RR_+)$.  Spaces are topological spaces, $\str_X$ is the sheaf of continuous maps to $\RR_+$, viewing the latter as a monoid object in its usual metric topology, and $\GG_m = \RR_{>0}$ with its usual topological group structure.  More generally, one can use any topological monoid as $\AA^1$, provided that the group of units $\GG_m \subseteq \AA^1$ is open.  \item $(\LRS,\Spec \ZZ[x])$.  Spaces are locally ringed spaces, $\str_X = \O_X$, and $\GG_m = \Spec \ZZ[x,x^{-1}]$. \item $(\Sch,\Spec \ZZ[x])$.  Spaces are schemes, $\str_X = \O_X$, $\GG_m = \Spec \ZZ[x,x^{-1}]$.  \item There are variants of the above where we work over some fixed base and impose certain finiteness conditions.  In particular, we will often consider the case where $\Esp = \Sch_\CC$ (or $\Esp = \Sch_\RR$) is the category of locally finite type schemes over $\CC$ (resp.\ $\RR$), with $\AA^1 = \AA^1_{\CC} = \Spec \CC[x]$ (resp.\ $\Spec \RR[x]$).  \item $(\LMS,\Spec \NN)$.  Spaces are locally monoidal spaces, $\str_X = \M_X$, $\GG_m = \Spec \ZZ$ (the fan, not the scheme).  \item $(\Fans,\Spec \NN)$.  Spaces are fans, $\str_X = \M_X$, $\GG_m = \Spec \ZZ$. Again, we could impose certain finiteness conditions, such as taking $\Fans$ to be the category of locally finite type fans. \item $(\AS,(\CC,\cdot))$.  Spaces are analytic spaces, $\str_X = \O_X$, $\GG_m  = \CC^* = \GL_1(\CC)$ with its usual complex Lie group structure. \item $(\DS,\RR)$.  Spaces are differentiable spaces, $\str_X = \O_X$, $\GG_m = \RR^* = \GL_1(\RR)$ with its usual Lie group structure. \item $(\DS,\CC)$.  Spaces are differentiable spaces, $\str_X = \O_X \otimes_{\RR} \CC$.  Here $\CC \in \DS$ is isomorphic to $\RR^2$ as a differentiable space, but it has a different monoid structure.  \item $(\DS,\RR_+)$.  Spaces are differentiable spaces.  The monoid object $\AA^1$ is the non-negative reals under multiplication (\S\ref{section:positivefunctions}), so $\str_X = \O_X^{\geq 0}$.  \end{enumerate}

\begin{rem} \label{rem:spacesexamples}  Let us make a few comments concerning the validity of the axioms of Definition~\ref{defn:spaces} in the above examples.  The only conceivable difficulty is probably \eqref{Zariskigluing}, which is a statement about direct limits, and these are always a little touchy in categories like $\Sch$.  Here is a good way to proceed:  First, one has a very concrete description of direct limits in the category $\RS$ of ringed spaces and in the category $\MS$ of monoidal spaces (to form the direct limit of $i \mapsto X_i$ in, say, $\RS$, just form the corresponding direct limit $X$ in spaces and endow it with the structure sheaf $\O_X$ given by the inverse limits of the pushforwards of the $\O_{X_i}$ to $X$).  Using this, one checks \eqref{Zariskigluing} for $\RS$ and $\MS$ (in fact, one must check it in $\Top$ first of all!).  Now one uses the fact that \eqref{Zariskigluing} holds in $\RS$ to deduce it in, say, schemes, by noting that the property of being a scheme is local and an open embedding of ringed spaces is a map of schemes when the ringed spaces happen to be schemes, so one thus sees that the sort of direct limits needed for \eqref{Zariskigluing} in, say, schemes, are obtained simply by computing the corresponding direct limit in $\RS$ and noting that it is a scheme and the structure maps to it are maps of schemes.  The same argument allows one to check \eqref{Zariskigluing} in $\LRS$, $\DS$, $\AS$, \dots.  One similarly deduces \eqref{Zariskigluing} for $\LMS$ and for $\Fans$ from the $\MS$ case. \end{rem}

We have already encountered several $1$-morphisms in $\Univ$.  For example, the various differentialization and analytification functors of \S\ref{section:differentialization} yield $1$-morphisms \be (\Sch_{\RR},\Spec \RR[x]) & \to & (\DS,\RR) \\ (\Sch_{\CC},\Spec \CC[x]) & \to & (\AS,\CC) \\ (\AS,\CC) & \to & (\DS,\CC). \ee  Similarly, we have base change morphisms \be (\Sch_{\ZZ},\Spec \ZZ[x]) & \to & (\Sch_{\RR},\Spec \RR[x]) \\ (\Sch_{\RR}, \Spec \RR[x]) & \to & (\Sch_{\CC}, \Spec \CC[x]). \ee  In general the underlying space functor $X \mapsto |X|$ doesn't yield a $1$-morphism to $(\Top,|\AA^1|)$ in $\Univ$ because it doesn't preserve inverse limits.  But for differentiable and anlytic spaces, we do have such $1$-morphisms.

\subsection{Monoids to spaces} \label{section:monoidstospaces}  Fix a category of spaces $(\Esp,\AA^1)$.  We now explain how to associate a space $\AA(P) \in \Esp$ to every (finitely generated) monoid $P$ in a manner contravariantly functorial in $P$.  The space $\AA(P)$ is often called the \emph{realization} of $P$ in $\Esp$.  This realization construction is the cornerstone of toric geometry and log geometry.

\begin{lem} \label{lem:AP} For every finitely generated monoid $P$, the presheaf \be \AA(P) : \Esp^{\rm op} & \to & \Sets \\ X & \mapsto & \Hom_{\Mon}(P,\str_X(X)) = \Hom_{\Mon}(P,\Hom_{\Esp}(X,\AA^1)) \ee is representable by a space, also denoted $\AA(P)$.  The functor \be \AA : \Mon^{\rm op} & \to & \Esp \\ P & \mapsto & \AA(P) \ee preserves finite inverse limits (i.e.\ it takes a finite direct limit of monoids to a finite inverse limits of spaces). \end{lem}

\begin{proof} Notice that \be \Hom_{\Mon}(\NN,\Hom_{\Esp}(X,\AA^1)) & = & \Hom_{\Esp}(X,\AA^1) \ee (naturally in $X$), so $\AA(\NN)$ is representable by $\AA^1$.  Similarly, one sees that $\AA(\NN^n)$ is representable by $(\AA^1)^n =: \AA^n$, which exists since we assume that $\Esp$ satisfies \eqref{finitelimits}.  In general, any finitely generated monoid $P$ has a finite presentation (\S\ref{section:monoidbasics})---i.e.\ $P$ sits in a coequalizer diagram \bne{presentationofP} \NN^n \rightrightarrows \NN^m \to P. \ene It follows from the universal property of inverse limits that ``the" inverse limit of \bne{inverselimitofpresentation} \AA^m \rightrightarrows \AA^n \ene (which is representable by \eqref{finitelimits}) represents $\AA(P)$.  The limit preservation statement follows formally from the ``adjointness" property of $P \mapsto \AA(P)$. \end{proof}

\begin{rem} Lemma~\ref{lem:AP} holds for any category $\Esp$ with finite inverse limits and any monoid object $\AA^1$ of $\Esp$. \end{rem}

\begin{example} \label{example:GmisAZ} When $P = \ZZ$, the presheaf \be \AA(Z) : \Esp^{\rm op} & \to & \Sets \\ X & \mapsto & \Hom_{\Mon}(\ZZ,\str_X(X)) \\ & = & \str_X^*(X) \ee is the one represented by the ``group of units" $\GG_m$, so we have a natural isomorphism $\AA(\ZZ) = \GG_m$.  The map of spaces $\AA(\ZZ) \to \AA(\NN)$ obtained by applying the functor $\AA$ of Lemma~\ref{lem:AP} to the map of monoids $\NN \to \ZZ$ is similarly identified with the natural inclusion $\GG_m \to \AA^1$ coming from the definition of $\GG_m$. \end{example}

\subsection{Fans to spaces} \label{section:realizationoffans} Let $(\Esp,\AA^1)$ be a category of spaces.  In Lemma~\ref{lem:AP} we associated a space $\AA(P) \in \Esp$ to each finitely generated monoid $P$, characterized up to unique isomorphism by the existence of a bijection \bne{affinebij} \Hom_{\Esp}(X,\AA(P)) & = & \Hom_{\Mon}(P,\str_X(X)) \ene natural in $P$ and $X \in \Esp$.  Here $\str_X(X)$ is the monoid of global section of the structure sheaf $\str_X$ of the locally monoidal space $|X| = (|X|,\str_X)$ given by the image of $X \in \Esp$ under the functor \eqref{EsptoLMS} of \S\ref{section:spacestolocallymonoidalspaces}.  Combining \eqref{affinebij} with the  natural bijection \be \Hom_{\Mon}(P,\str_X(X)) & = & \Hom_{\LMS}(|X|,\Spec P) \ee (\eqref{Specadjunction} in \S\ref{section:Specrevisited}) we obtain a bijection \bne{affinebij2} \Hom_{\Esp}(X,\AA(P)) & = & \Hom_{\LMS}(|X|,\Spec P) \ene natural in $X$ and $P$.  By taking $X = \AA(P)$ and considering the image of the identity map $\Id \in \Hom_{\Esp}(\AA(P),\AA(P))$ under this bijection, we obtain an $\LMS$-morphism $\tau_P : |\AA(P)| \to \Spec P$ (``the unit of the adjunction"), so that the bijection \eqref{affinebij2} is given explicitly by $f \mapsto \tau_P |f|$.

Our goal now is to generalize this by replacing the finite type affine fan $\Spec P$ with an arbitrary (locally finite type) fan $Y$.  That is, we are going to associate, to each such $Y$, a space $\AA(Y)$, and an $\LMS$-morphism $\tau_Y : |\AA(P)| \to Y$ such that $f \mapsto \tau_Y |f|$ yields a bijection \bne{EspFansadjunction} \Hom_{\Esp}(X,\AA(Y)) & = & \Hom_{\LMS}(|X|,Y) \ene for each $X \in \Esp$.  This will yield a functor \bne{FanstoEsp} \AA : \Fans & \to & \Esp \ene from the category $\Fans$ of \emph{locally finite type fans} to spaces.  The space $\AA(Y)$ will be called the \emph{realization} of $Y$ in $\Esp$ and the functor \eqref{FanstoEsp} will be called the \emph{realization functor} to $\Esp$.

\begin{rem} \label{rem:AYrepresentable} For a fan $Y$ and a space $X$, if we set \be \AA(Y)(X) & := & \Hom_{\LMS}(|X|,Y), \ee then we can view $\AA( \slot )$ as a functor from fans to the category of presheaves on $\Esp$.  To construct our functor \eqref{FanstoEsp} we just need to show that each presheaf $\AA(Y)$ is actually \emph{representable}, at least when $Y$ is locally finite type.  When $\AA(Y)$ is representable, we also abusively denote the representing space $\AA(Y)$. \end{rem}

\begin{rem} \label{rem:EspFansadjunction} For a fan $Y$, having a space $\AA(Y)$ and an $\LMS$-morphism $\tau_Y : |\AA(P)| \to Y$ such that $f \mapsto \tau_Y |f|$ yields a bijection \eqref{EspFansadjunction} is equivalent to having a space $\AA(Y)$ and \emph{some} bijection \eqref{EspFansadjunction} natural in $X \in \Esp$ because we can always construct $\tau_Y$ as the ``unit of the adjunction" as we did above.  The point is that it will be convenient in what follows to have an ``explicit formula" for the bijection \eqref{EspFansadjunction} so that we can more easily construct ``new such bijections from existing ones." \end{rem}

\begin{rem} \label{rem:FanstoEsp}  The functor \eqref{FanstoEsp} is characterized up to unique isomorphism of functors by the existence of the bijections \eqref{EspFansadjunction} natural in $X \in \Esp$, $Y \in \Fans$, but of course there are still lots of ``choices" that go into the construction.  For example, we can construct \eqref{FanstoEsp} in such a way that $\AA(\Spec \NN)$ is any space \emph{isomorphic} to $\AA^1$.  Since we want \eqref{FanstoEsp} to be a $1$-morphism of categories of spaces (Definition~\ref{defn:morphismofspaces}) we can and do arrange that $\AA(\Spec \NN)$ is \emph{equal} to $\AA^1$.  \end{rem}

Lemma~\ref{lem:AP} and the discussion above, prove that $\AA(Y)$ is representable for any finite type \emph{affine} fan $Y$.  Our goal now is to bootstrap up from this to see that $\AA(Y)$ is representable for any locally finite type fan $Y$.

\begin{rem} \label{rem:finitetypehypothesis}  In some categories of spaces, we can see directly that $\AA(\Spec P)$ is representable for \emph{all} monoids $P$, not just those of finite type.  For example, in $(\Sch,\AA^1 = \Spec \ZZ[x])$, we can take $\AA(\Spec P) := \Spec \ZZ[P]$ (which we could \emph{not} generally do if we insisted on working with locally finite type schemes).  We obtain the desired ``adjunction formula" in this case as a composition \be \Hom_{\Sch}(X,\AA(P)) & = & \Hom_{\An}(\ZZ[P], \O_X(X) ) \\ & = & \Hom_{\Mon}(P,\O_X(X)) \\ & = & \Hom_{\Mon}(P,\str_X(X)) \\ & = & \Hom_{\LMS}(|X|,\Spec P) \ee of the natural bijections \eqref{ringsSpecadjunction}, \eqref{ZPAnadjunction}, and \eqref{Specadjunction} (the third equality is just the meaning of the ``structure sheaf" $\str_X$ in this category of spaces---this $\AA^1$ represents $X \mapsto \O_X(X)$).  For such categories of spaces we will be able to construct the functor \eqref{FanstoEsp} with $\Fans$ equal to the category of \emph{all} fans---that is, we will prove that the presheaf $\AA(Y)$ of Remark~\ref{rem:AYrepresentable} is representable for \emph{any} fan $Y$. \end{rem}

\begin{lem} \label{lem:fanpresentation}  Let $Y$ be a fan.  Then there is a well-ordered set $I$ and a map of partially ordered sets $i \to Y_i$ from $I$ to the set of open subspaces of $Y$ (ordered by inclusion) such that: \begin{enumerate} \item For each successor $i = j+1$ in $I$ we have $Y_i = Y_j \cup U_i$, where $U_i$ is an affine fan. \item For each limit $i$ in $I$, we have $Y_i = \cup_{j<i} Y_j$. \item $Y = \cup_{i \in I} Y_i$. \end{enumerate}  If $Y$ is locally finite type, then we can arrange that the $U_i$ are finite type and if $Y$ is finite as a topological space (which holds, for example, if $Y$ is quasi-compact and locally finite type), we can take $I$ finite. \end{lem}

\begin{proof} Choose, for each $y \in Y$, an open neighborhood $U_y$ of $y$ in $Y$ such that $(U_y,\O_Y|U_y)$ is an affine fan.  (If $Y$ is locally finite type, we can further arrange that $U_y$ is finite type.)  Choose a well-ordering of (the set underlying) $Y$.  Take $I := Y$, $Y_i := \cup_{j \leq i} U_j$.  \end{proof}

\begin{lem} \label{lem:Acoproducts} Let $\{ Y_i \}$ be a set of fans such that, for each $i$, the presheaf $\AA(Y_i)$ of Remark~\ref{rem:AYrepresentable} is representable.  Then $\coprod_i \AA(Y_i)$ represents the presheaf $\AA(\coprod_i Y_i)$. \end{lem}

\begin{proof} This is a straightforward exercise with Lemma~\ref{lem:coproductsofspaces}.  \end{proof}

\begin{lem} \label{lem:Aopenembeddings} Let $Y$ be a fan for which the presheaf $\AA(Y)$ of Remark~\ref{rem:AYrepresentable} is representable and let $U \subseteq Y$ be an open subfan of $Y$.  Let $\AA(U) \to \AA(Y)$ be the open embedding of spaces corresponding to the open subset $\tau_Y^{-1}(U)$ of $|\AA(Y)|$.  Let $\tau_U$ be the restriction of $\tau_Y$ to $|\AA(U)|$.  Then $f \mapsto \tau_U |f|$ yields a bijection \be \Hom_{\Esp}(X,\AA(U)) & = & \Hom_{\LMS}(|X|,U) \ee for every $X \in \Esp$, hence the presheaf $\AA(U)$ of Remark~\ref{rem:AYrepresentable} is representable by the open subspace of $\AA(Y)$ corresponding to the open subset $\tau_Y^{-1}(U)$ of $|\AA(Y)|$. \end{lem}

\begin{proof} This is a straightforward exercise with the universal property of the open embedding $\AA(U) \to \AA(Y)$. \end{proof}

\begin{rem} \label{rem:openunits} Notice that Lemma~\ref{lem:Aopenembeddings} only makes use of axiom \eqref{openunits} for $\Esp$ through its consequence (Lemma~\ref{lem:localstalks}) that the functor \eqref{EsptoLMS} actually ``takes values" (on morphisms) in $\LMS \subseteq \MS$.  This proves that the conclusion of Lemma~\ref{lem:localstalks} is actually \emph{equivalent} to the axiom \eqref{openunits} (assuming all the other axioms).  In fact, it is a good exercise for the reader to check directly (assuming all axioms for spaces other than \eqref{openunits}) that the $\MS$ map $|i| : |\GG_m| \to |\AA^1|$ attached to the natural $\Esp$-morphism $i : \GG_m \to \AA^1$ is an $\LMS$ map iff $i$ is an open embedding. \end{rem}

\begin{rem} \label{rem:localizationstoopenembeddings} Suppose $P$ is a finitely generated monoid and $S$ is a finitely generated submonoid of $P$.  Then the localization map $P \to S^{-1}P$ induces an open embedding of (finite type affine) fans $\Spec S^{-1}P \to \Spec P$.\footnote{In fact, every open embedding of finite type affine fans is of this form.}  Lemma~\ref{lem:Aopenembeddings} then shows that the map of spaces $\AA(S^{-1}P) \to \AA(P)$ constructed in Lemma~\ref{lem:AP} is an open embedding.  We could also have seen this directly in \S\ref{section:monoidstospaces} as follows:  Let $p$ be the sum of a finite set of generators for $S$.  Then $S^{-1}P = P_p$ and we have a pushout diagram \bne{pushP} \xym{ \NN \ar[r]^-p \ar[d] & P \ar[d] \\ \ZZ \ar[r]^-p & P_p } \ene of monoids.  The functor $\AA$ of Lemma~\ref{lem:AP} preserves finite inverse limits, hence $\AA(S^{-1}P) \to \AA(P)$ is a base change of the map $\AA(\ZZ) \to \AA(\NN)$.  The latter map is just the natural inclusion $\GG_m \to \AA^1$  (Example~\ref{example:GmisAZ}), which is an open embedding by axiom \eqref{openunits}, hence its base change $\AA(S^{-1}P) \to \AA(P)$ is also an open embedding by Lemma~\ref{lem:Espfiberedproducts}. \end{rem}

\begin{lem} \label{lem:AZariskigluing} Let $i \mapsto Y_i$ be a Zariski gluing datum (Definition~\ref{defn:Zariskigluingdatum}) in the category of fans with direct limit $Y$.  Suppose that the presheaf $\AA(Y_i)$ of Remark~\ref{rem:AYrepresentable} is representable for all $i \in I$.  Then $i \mapsto \AA(Y_i)$ is a Zariski gluing datum in $\Esp$ and its direct limit represents the presheaf $\AA(Y)$.  \end{lem}

\begin{proof}  The direct limit system $i \mapsto \AA(Y_i)$ is a Zariski gluing datum in $\Esp$ by Lemma~\ref{lem:Aopenembeddings}.  Write $\tau_i : |\AA(Y_i)| \to Y_i$ instead of $\tau_{Y_i}$ to ease notation.  Lemma~\ref{lem:Aopenembeddings} also shows that $\tau_i = \tau_j | |\AA(Y_j)|$ for each map $i \to j$ in the indexing category $I$.  Since $\Esp$ satisfies \eqref{Zariskigluing} we can form the direct limit $\AA(Y)$ of the $\AA(Y_i)$.  The functor \eqref{EsptoLMS} preserves such direct limits (Remark~\ref{rem:EsptoLMS}), so $|\AA(Y)|$ is also the direct limit of $i \mapsto \AA(Y_i)$ and we can define the $\LMS$-morphism $\tau_Y : |\AA(Y)| \to Y$ as the direct limit of the $\tau_i$.  It remains to prove that $f \mapsto \tau_Y|f|$ yields a bijection \be \Hom_{\Esp}(X,\AA(Y)) & = & \Hom_{\LMS}(|X|,Y) \ee for each $X \in \Esp$.

To establish surjectivity, consider an $\LMS$-morphism $g : |X| \to Y$.  Define $X_i \to X$ to be the open embedding in $\Esp$ corresponding to $\tau_Y^{-1}(Y_i) \subseteq |X|$.  Note that the image $|X_i| \to |X|$ of $X_i \to X$ under \eqref{EsptoLMS} is also the open embedding in $\LMS$ corresponding to $\tau_Y^{-1}(Y_i) \subseteq |X|$.  Using the fact that $Y$ is the direct limit of the Zariski gluing datum $i \mapsto Y_i$ in $\Fans$, we see that $X$ is the direct limit of $i \mapsto X_i$ in $\Esp$ and $|X|$ is the direct limit of $i \mapsto |X_i|$ in $\LMS$ by Proposition~\ref{prop:Zariskigluing}.  Let $g_i : |X_i| \to Y$ be the restriction of $g$ to $|X_i|$.  By the universal property of $\AA(Y_i)$, we have unique $\Esp$-maps $f_i : X_i \to \AA(Y_i)$ such that $g_i = \tau_i |f_i|$ for each $i \in I$.  Using the uniqueness we check that $f_i = f_j|X_i$ for each map $i \to j$ in $I$, so we can define $f : X \to \AA(Y)$ to be the direct limit of the $f_i$.  We see that $\tau_Y |f| = g$ by observing that $(\tau_Y |f| ) | X_i = g_i = g|X_i$ and using the fact that $|X|$ is the direct limit of the $|X_i|$.

For injectivity, suppose $f_1,f_2 : X \to \AA(Y)$ are two $\Esp$ morphisms such that $\tau_Y|f_1| = \tau_Y |f_2| =: g$.  Then we define $X_i \to X$ as above and use the universal property of the $\AA(Y_i)$ to check that $f_1|X_i = f_2|X_2$ for all $i \in I$, hence $f_1=f_2$ because $X$ is the direct limit of $i \mapsto X_i$.    \end{proof}

\begin{thm} \label{thm:AY} For any category of spaces $(\Esp,\AA^1)$ and any locally finite type fan $Y$, the presheaf $\AA(Y)$ of Remark~\ref{rem:AYrepresentable} is representable.  If $\AA(\Spec P)$ is representable for every monoid $P$, then $\AA(Y)$ is representable for every fan $Y$.  The resulting functor \eqref{FanstoEsp} (taking $\Fans$ to be the category of locally finite type fans in general, or \emph{all} fans if $\AA(\Spec P)$ is representable for every monoid $P$) \begin{enumerate} \item \label{copres} preserves coproducts, \item \label{ilpres} preserves inverse limits, and \item \label{oe} takes an open embedding of fans $U \to Y$ to an open embedding of spaces $\AA(U) \to \AA(Y)$ with image $\tau_Y^{-1}(U)$. \end{enumerate} \end{thm}

\begin{proof} Let $i \mapsto Y_i$ be as in Lemma~\ref{lem:fanpresentation}.  By transfinite induction, it suffices to prove that $\AA(Y_i)$ is representable under the assumption that $\AA(Y_j)$ is representable for every $j<i$.  If $i=j+1$ is a limit, then we have $Y_i = Y_j \cup U_i$ with $U_i$ affine (and finite type if $Y$ is locally finite type).  Both $\AA(Y_j)$ and $\AA(U_i)$ are representable.  Furthermore, $\AA(Y_j \cap U_i)$ is representable and the maps $\AA(Y_j \cap U_i) \to \AA(Y_j)$ and $\AA(Y_j \cap U_i) \to \AA(U_i)$ are open embeddings by Lemma~\ref{lem:Aopenembeddings}, so Lemma~\ref{lem:AZariskigluing} shows that $\AA(Y_i)$ is representable by the pushout of the two aforementioned open embeddings.  If $i$ is a limit, then we have $Y_i = \cup_{j<i} Y_j$.  Since each $\AA(Y_j)$ is representable and each map $\AA(Y_j) \to \AA(Y_k)$ ($j \leq k < i$) is an open embedding by Lemma~\ref{lem:Aopenembeddings}, Lemma~\ref{lem:AZariskigluing} shows that $\AA(Y_i)$ is representable by $\dirlim \{ \AA(Y_j) : j < i \}$. 

Once the representability is known, \eqref{copres} follows from Lemma~\ref{lem:Acoproducts} and \eqref{oe} follows from \ref{lem:Aopenembeddings}.   The fact that \eqref{FanstoEsp} preserves inverse limits follows formally from the natural bijections \eqref{EspFansadjunction} by the usual argument showing that a right adjoint preserves inverse limits. \end{proof}

\begin{thm} \label{thm:fans2initial}  Let $\Fans$ be the category of locally finite type fans.  Then the category of spaces $(\Fans,\Spec \NN)$ is ``the" initial object in the $2$-category $\Univ$ of categories of spaces (Definition~\ref{defn:morphismofspaces}). \end{thm}

\begin{proof}  Let $(\Esp,\AA^1)$ be a category of spaces.  As in Remark~\ref{rem:FanstoEsp}, we can assume that the functor $\AA$ in \eqref{FanstoEsp} obtained from Theorem~\ref{thm:AY} takes $\Spec \NN$ to $\AA^1$.  The formula \eqref{EspFansadjunction} then makes it clear that the equality $\AA(\Spec \NN) = \AA^1$ is an equality \emph{of monoid objects}.  For $Y \in \Fans$, let $\eta_Y : |\AA(Y)| \to |Y|$ be the map of topological spaces underlying the $\LMS$ morphism $\tau_Y : |\AA(Y)| \to Y$.  Then we see that \be (\AA,\eta) : (\Fans,\Spec\NN) & \to & (\Esp,\AA^1) \ee is a $1$-morphism of categories of spaces (Definition~\ref{defn:morphismofspaces}) by using the properties of \eqref{FanstoEsp} in Theorem~\ref{thm:AY}.

It remains to show that if we have two $1$-morphism of spaces \be (\AA,\eta),(\AA',\eta') : (\Fans,\Spec\NN) & \rightrightarrows & (\Esp,\AA^1) \ee then there is a unique invertible $2$-morphism in $\Univ$ from $(\AA,\eta)$ to $(\AA',\eta')$.  Since the topological space underlying any fan is ``sober," Lemma~\ref{lem:2morphismofspaces} says that such an invertible $2$-morphism is the same thing as an isomorphism of functors $\gamma$ from $\AA$ to $\AA'$ with $\gamma(\AA^1)=\Id$.  We will prove the existence and uniqueness of $\gamma$ simultaneously.  Recall from Definition~\ref{defn:morphismofspaces} that $\AA$ and $\AA'$ preserve finite inverse limits and coproducts and take Zariski covers to Zariski covers (Lemma~\ref{lem:1morphismofspaces}).

Set $\AA^n := \Spec \NN^n$.  Let \be \pi_i = \Spec(e_i : \NN \to \NN^n) : \AA^n & \to & \AA^1 \ee ($i=1,\dots,n$) be the projections.  Since $\AA$ and $\AA'$ preserve finite inverse limits (in particular finite products), we can define $\gamma(\AA^n) : \AA(\AA^n) \to \AA'(\AA^n)$ to be the unique isomorphism making the diagrams \bne{Acommute} & \xym{ \AA(\AA^n) \ar[r]^-{\AA(\pi_i)} \ar[d]_{\gamma(\AA^n)}^\cong & \AA(\AA^1) = \AA^1 \ar@{=}[d] \\ \AA'(\AA^n) \ar[r]^-{\AA'(\pi_n)} & \AA'(\AA^1) = \AA^1 } \ene commute for each $i$.  Notice that this ensures that $\gamma(\AA^1) = \Id$.  (For uniqueness of $\gamma$, note that any other isomorphism of functors $\gamma'$ from $\AA$ to $\AA'$ with $\gamma'(\AA^1)=\Id$ must also make the diagrams \eqref{Acommute}, so we must have $\gamma(\AA^n) = \gamma'(\AA^n)$ by the universal property of the product.)  Next we claim that \bne{Af} & \xym{ \AA(\AA^m) \ar[r]^-{\AA(f)} \ar[d]_{\gamma(\AA^m)}^\cong & \AA(\AA^1) = \AA^1 \ar@{=}[d]^{\gamma(\AA^1) = \Id} \\ \AA'(\AA^m) \ar[r]^-{\AA'(f)} & \AA'(\AA^1) = \AA^1 } \ene commutes for any map of fans $f : \AA^m \to \AA^1$.  For this we need to make use of a fact about $(\Fans,\Spec \NN)$ that does not hold for other categories of spaces, namely that our completely arbitrary map $f \in \Hom_{\Fans}(\AA^m,\AA^1)$ is in the submonoid of $\Hom_{\Fans}(\AA^m,\AA^1)$ generated by the elements $\pi_1,\dots,\pi_m \in \Hom_{\Fans}(\AA^m,\AA^1)$, so we can write $f = \sum_i a_i \pi_i$ for some $a_1,\dots,a_m \in \NN$ (this is just a geometric restatement of the fact that any $f \in \NN^n$ can be written $f = \sum_i a_i e_i$).  Using the commutativity of \eqref{Acommute} and the fact that the equalities $\AA(\AA^1)=\AA^1$, $\AA'(\AA^1)=\AA^1$ are equalities of \emph{monoid objects} we establish the commutativity of \eqref{Af} by the computation \be \AA(f) & = & \AA(\sum_i a_i \pi_i) \\ & = & \sum_i a_i \AA(\pi_i) \\ & = & \sum_i a_i( \AA'(\pi_i) \gamma(\AA^m) ) \\ & = & \left ( \sum_i a_i \AA'(\pi_i) \right ) \gamma(\AA^m) \\ & = & \AA'(f) \gamma(\AA^m). \ee Next we claim that for any map of fans $f : \AA^m \to \AA^n$, the diagram \bne{Af2} & \xym{ \AA(\AA^m) \ar[r]^-{\AA(f)} \ar[d]_{\gamma(\AA^m)}^\cong & \AA(\AA^n) \ar[d]^{\gamma(\AA^n)}_\cong \\ \AA'(\AA^m) \ar[r]^-{\AA'(f)} & \AA'(\AA^n)  } \ene commutes.  To see this, we again use the universal property of the product $\AA'(\AA^n)$ to see that it suffices to check that \be \AA'(\pi_i) \gamma(\AA^n) \AA(f) & = & \AA'(\pi_i) \AA'(f) \gamma(\AA^m), \ee which follows from the commutativity results established above.

Now let $Y$ be a finite type affine fan.  Using the fact that any finitely generated monoid has a presentation (\S\ref{section:monoidbasics}), we see that we can write $Y$ as an equalizer \bne{presentY} & Y \to \AA^m \rightrightarrows \AA^n \ene in $\Fans$.  By the commutativity result established above, the solid diagram \bne{AYA} & \xym{ \AA(Y) \ar@{.>}[d]_{\gamma(Y)}^\cong \ar[r] & \AA(\AA^m) \ar[d]_{\gamma(\AA^m)}^\cong \ar@<.5ex>[r] \ar@<-.5ex>[r] & \AA(\AA^n) \ar[d]_{\gamma(\AA^n)}^\cong \\ \AA'(Y) \ar[r] & \AA'(\AA^m) \ar@<.5ex>[r] \ar@<-.5ex>[r] & \AA'(\AA^n) } \ene commutes.  Since $\AA$ and $\AA'$ preserve equalizers, we obtain the isomorphism $\gamma(Y)$ making the resulting diagram commute.  Using the fact that any two presentations of a finitely generated monoid are dominated by a third (\S\ref{section:monoidbasics}), we check that the isomorphism $\gamma(Y)$ does not depend on the choice of the equalizer diagram \eqref{presentY}.  Using this and the fact that we can sit any map $f : Y \to Y'$ of finite type affine fans in a commutative diagram \bne{presentamap} & \xym{ Y \ar[d]_{f} \ar[r] & \AA^m \ar[d] \ar@<.5ex>[r] \ar@<-.5ex>[r] & \AA^n \ar[d] \\ Y' \ar[r] & \AA^k \ar@<.5ex>[r] \ar@<-.5ex>[r] & \AA^l } \ene where the rows are equalizers, we find that \bne{maincommutativity} & \xym{  \AA(Y) \ar[r]^-{\AA(f)} \ar[d]_{\gamma(Y)}^\cong & \AA(Y') \ar[d]^{\gamma(Y')}_\cong \\ \AA'(Y) \ar[r]^-{\AA'(f)} & \AA'(Y') } \ene commutes for any such $f$.  For uniqueness: If $\gamma'$ is another invertible $2$-morphism from $\AA$ to $\AA'$, we already argued above that $\gamma(\AA^n) = \gamma'(\AA^n)$ for every $n$; then naturality of $\gamma'$ ensures that $\gamma'(Y)$ would also complete the diagram \eqref{AYA}, hence we must have $\gamma(Y)=\gamma(Y')$.

Now we can finish the proof that there is at most one invertible natural transformation $\gamma: \AA \to \AA'$ with $\gamma(\AA^1)=\Id$.  Suppose $\gamma$ and $\gamma'$ are two such.  We need to show that $\gamma(Y) = \gamma'(Y)$ for an arbitrary locally finite fan $Y$.  We've argued above that this is true when $Y$ is finite type affine.  For a general such $Y$, \emph{choose} open embeddings $f_i : U_i \to Y$ with $U_i$ finite type affine such that the $U_i$ cover $Y$; let $f : U \to Y$ be the coproduct of the $f_i$, so that $f$ is a Zariski cover of $Y$ in $\Fans$.  Since both $\gamma(Y)$ and $\gamma'(Y)$ will complete the diagram \bne{AAA} & \xym{  \AA(U_i) \ar[r]^-{\AA(f_i)} \ar[d]_{\gamma(U) = \gamma'(U)}^\cong & \AA(Y) \ar@{.>}[d] \\ \AA'(U) \ar[r]^-{\AA'(f_i)} & \AA'(Y) } \ene and since $\AA$ ``preserves the coproduct $U$," we see that $\gamma \AA(f) = \gamma' \AA(f)$, hence $\gamma=\gamma'$ because $\AA(f)$ is a Zariski cover and $\Esp$ satisfies \eqref{Zariskitopology} (hence $\AA(f)$ is an epimorphism in $\Esp$). 

To finish our construction of $\gamma$ we now consider a fixed locally finite type fan $Y$.  \emph{Choose} $\{ f_i : i \in I \}$ and define $f$ as in the previous paragraph.  Since $\AA$ and $\AA'$ preserve coproducts, there is a unique isomorphism $\gamma(U) : \AA(U) \to \AA'(U)$ making the diagram \bne{gU} & \xym{ \AA(U_i) \ar[r] \ar[d]_{\gamma(U_i)}^\cong & \AA(U) \ar[r]^-{\AA(f)} \ar[d]_{\gamma(U)}^\cong & \AA(Y) \ar@{.>}[d]^\cong_{\gamma(Y)} \\ \AA'(U_i) \ar[r] & \AA(U) \ar[r]^-{\AA(f)} & \AA(Y) } \ene commute for every $i \in I$.  Here $\gamma(U_i)$ is the isomorphism constructed two paragraphs above for the finite type affine fan $U_i$.  We will now argue that there is a unique isomorphism $\gamma(Y)$ making the diagram commute as indicated.

For each $(i,j) \in I \times I$, set $U_{ij} = U_i \cap U_j = U_i \times_Y U_j$ and choose a set of open embeddings $\{ V_k \to U_{ij} : k \in K(i,j) \}$ covering $U_{ij}$ such that each $V_k$ is finite type affine.  Let $K$ be the coproduct (in sets) of the $K(i,j)$ and let $V$ be the coproduct of all the $V_k$, $k \in K$.  For $k \in K(I,J)$, the composition of $V_k \to U_{ij}$ and $U_{ij} \to U \times_Y U$ is an open embedding $V_k \to U \times_Y U$.  The coproduct $g$ of these open embeddings, over all $k \in K$, is a Zariski cover $g : V \to U \times_Y U$.  For $(i,j) \in I \times I$, $k \in K(i,j)$, the situation is summed up in the commutative diagram of $\Fans$ below.  \bne{bigD} & \xym{ V_k \ar[rrrd] \ar[rd] \ar[rddd] \\ & V \ar[rr] \ar[dd] \ar[rd]^g & & U_j \ar[d] \\ & & U \times_Y U \ar[r]^-{\pi_2} \ar[d]_{\pi_1} & U \ar[d]^f \\ & U_i \ar[r] & U \ar[r]^-f & Y} \ene  Now we will construct an isomorphism from $\AA$ applied to this diagram to $\AA'$ applied to this diagram to obtain a gigantic commutative diagram that we don't intend to draw.  For the finite type affine fans $V_k$, $U_i$, and $U_j$ we have the isomorphisms $\gamma(V_k)$, $\gamma(U_i)$, and $\gamma(U_j)$ constructed earlier.  These make two obvious squares in the ``gigantic diagram" commute (the instances of \eqref{maincommutativity} where $Y \to Y'$ is $V_k \to U_i$ or $V_k \to U_j$).  Since $\AA$ and $\AA'$ preserve finite inverse limits, the isomorphism $\gamma(U)$ appearing in \eqref{gU} gives rise to an isomorphism $\gamma(U \times_Y U) : \AA(U \times_Y U) \to \AA'(U \times_Y U)$ making two more squares in the ``gigantic diagram" commute.  Next we use the fact that $\AA$ and $\AA'$ preserve coproducts to get an isomorphism $\gamma(V) : \AA(V) \to \AA'(V)$ (``the coproduct of the $\gamma(V_k)$") making everything in the gigantic diagram constructed up to this point commute.  Now, in the part of the gigantic diagram we have constructed, we can look at the composition $$\AA(V_k) \to \AA(V) \to \AA(U \times_Y U),$$ followed by either $\AA(\pi_1)$ or $\AA(\pi_2)$, followed by $\gamma(U)$.  By chasing around our gigantic diagram we see that the two maps $\AA(V_k) \to \AA'(U)$ thus defined are equal.  Since this is true for all $k$, and $\AA(V)$ is ``the" coproduct of the $\AA(V_k)$, and $\AA(k)$ is a Zariski cover (hence epi), we find that \be \gamma(U) \AA(\pi_1) & = & \gamma(U) \AA(\pi_2). \ee  Now, since $\AA(f)$ is a Zariski cover and the lower right square in \eqref{bigD} is cartesian (as is $\AA$ or $\AA'$ of it) we obtain a unique morphism $\gamma(Y) : \AA(Y) \to \AA(Y')$ making the ``gigantic diagram" commute---of course $\gamma(Y)$ must be an isomorphism because we can exchange the roles of $\AA$, $\AA'$ (and exhange all the $\gamma$'s with their inverses) to get the inverse of $\gamma(Y)$.  In particular this $\gamma(Y)$ completes \eqref{gU} to a commutative diagram.

Next one argues that the isomorphism $\gamma(Y)$ constructed above does not depend on the \emph{choice} of cover $\{ f_i \}$ of $Y$ by finite type affines.  This is done by using the fact that any two open covers have a common refinement which is a cover by finite type affines (exercise!), much as we argued independence of the presentation in constructing $\gamma(Y)$ for affine $Y$.  (The detailed diagram chase is very much like the one above and is left to the reader.)  Finally, one proves that \eqref{maincommutativity} commutes for any map $f : Y \to Y'$ of locally finite type fans by taking a cover $\{ U_i \}$ of $Y'$ by finite type affines and a refinement $\{ V_j \}$ of $\{ f^{-1}U_i \}$ by finite type affines, much as we checked commutativity of \eqref{maincommutativity} in the affine case by choosing a ``presentation of $f$." \end{proof}

\section{Log spaces} \label{section:logspaces} The main purpose of this section is to explain the basic notions of \emph{logarithmic geometry} (\S\ref{section:logstructures}), in the sense of Fontaine, Illusie, and Kato.  We have made some effort to set things up in a fairly general way, so that we can speak of various kinds of ``log spaces" in a unified manner.  However, we should mention at the outset that our approach is not sufficiently general to incorporate the usual setup for log geometry in the algebraic setting, in that we have no way to systematically incorporate the \'etale site of a scheme into our framework.  

It probably would have been possible to arrange this by axiomatizing the notions of smooth and \'etale maps in our category of ``spaces," but we decided that that setup would be too abstract. (It would also have required us to work in some more abstract setting of ``monoidal topoi" rather than the monoidal spaces of \S\ref{section:monoidalspacesI} because we would be led naturally to consider log structures on the \'etale site of a given space.)  We can, however, speak about Zariski log schemes.  The main point is that our setup will be sufficiently general to allow us to treat, say, (positive) log differentiable spaces and log analytic spaces on the same footing.

\subsection{Log structures} \label{section:logstructures} We begin by recalling the basic notions of log geometry (c.f.\ \cite{Kat1}).  Let $X = (X,\str_X)$ be a monoidal space (\S\ref{section:monoidalspacesI})---that is, a topological space $X$ equipped with a sheaf of monoids $\str_X$.  One can keep in mind the case where $X$ is a ringed space and $\str_X = \O_X$, regarded as a sheaf of monoids under multiplication.  A \emph{prelog structure} on $X$ is a map $\alpha_X : \M_X \to \str_X$ of sheaves of monoids on $X$.  We often write $\M_X$ to denote a prelog structure, leaving $\alpha_X$ implicit.  Prelog structures on $X$ form a category $\PreLog(X)$, which is nothing but the category of sheaves of monoids over $\str_X$.  A prelog structure is called a \emph{log structure} iff \bne{logstr} \alpha_X^* : \M_X^* \to \str_X^* \ene is an isomorphism (of sheaves of abelian groups on $X$).  

For a point $x \in X$ and a prelog structure $\alpha_X : \M_X \to \str_X$, the submonoid $F_x := \alpha_{X,x}^{-1} \str_{X,x}^*$ of $\M_{X,x}$ is a face (\S\ref{section:idealsandfaces}) called the \emph{face of} $\M_X$ at $x$.  If $\M_X$ is a log structure, we often implicitly identify $\str_X^*$ and $\M_X^*$ via the isomorphism \eqref{logstr}, thereby regarding $\str_X^*$ as a submonoid of $\M_X$.   A morphism of log structures is, by definition, a morphism of prelog structures, so log structures form a full subcategory $\Log(X)$ of $\PreLog(X)$.  For example, the inclusion $\str_X^* \into \str_X$ defines a log structure called the \emph{trivial log structure}.  

For any prelog structure $\M_X$, there is an \emph{associated log structure} $\M_X^a$ defined by \be \M_X^a & := & \M_X \oplus_{ \alpha_X^{-1} \str_X^*  } \str_X^* , \ee where the pushout here is taken in the category of sheaves of monoids\footnote{The category $\Mon$ of monoids has all direct and inverse limits, but it is fair to say that pushouts are a bit tricky at times.} on $X$, and the map $\M_X^a \to \str_X$ is induced by the map $\alpha_X$ and the inclusion $\str_X^* \subseteq \str_X$.  There is a natural map $\M_X \to \M_X^a$ which is initial among maps from $\M_X$ to a log structure.  The functor $\M_X \mapsto \M_X^a$ is left adjoint to and retracts the inclusion $\Log(X) \into \PreLog(X)$.  The categories $\Log(X)$ and $\PreLog(X)$ have all direct and inverse limits.  For our purposes, it is enough to know that the direct limit of log structures is the same as the one in sheaves of monoids \emph{under} $\str_X^*$, and that the functor $\M_X \to \M_X^a$ preserves direct limits, because it is a left adjoint.\footnote{The associated log structure functor does \emph{not} preserve inverse limits.}

For a prelog structure $\M_X$, the \emph{characteristic monoid} of $\M_X$ is $\ov{\M}_X := \M_X / \alpha_X^{-1} \str_X^*$.  The stalk $\ov{\M}_{X,x}$ is the quotient of $\M_{X,x}$ by the face of $\M_X$ at $x$.  If $\M_X$ is a log structure, then $\ov{\M}_X = \M_X / \M_X^*$ is just the sharpening of $\M_X$.  Formation of the characteristic monoid commutes with taking associated log structures.  The characteristic monoid $\ov{\M}_X$ is a sheaf of sharp monoids (\S\ref{section:monoidbasics}).  A morphism $\M_X \to \N_{X}$ of prelog structures induces a morphism on characteristics \emph{with trivial kernel} (a local morphism in the sense of \S\ref{section:monoidbasics}).  For a point $x \in X$, the \emph{characteristic extension} is the extension \bne{characteristicextension} (0 \to \str_{X,x}^* \to \M_{X,x}^{\rm gp} \to \ov{\M}_{X,x}^{\rm gp} \to 0) & \in & \Ext^1( \ov{\M}_{X,x}^{\rm gp}, \str_{X,x}^*).\ene

If $f : X \to Y$ is a map of monoidal spaces (i.e.\ a map of topological spaces $f$ and a map $f^\sharp : f^{-1} \str_Y \to \str_X$ of sheaves of monoids on $X$) and $\M_Y$ is a prelog structure on $Y$, then $f^\sharp f^{-1}\alpha_Y : f^{-1} \M_Y \to \str_X$ is a prelog structure on $X$.  We let $f^* \M_Y$ denote the associated log structure and we call it the \emph{inverse image log structure}.  We have $f^*(\M_Y^a) = f^* \M_Y$.  If the morphism $f^\sharp$ is local (i.e.\ $f$ is a morphism of \emph{locally} monoidal spaces), then we have $\ov{f^* \M_Y} = f^{-1} \ov{\M}_Y$.  

\begin{example} \label{example:LMSmaps} A morphism of locally ringed spaces can be viewed as a morphism of locally monoidal spaces.  We will ultimately be interested only in maps of monoidal spaces arising from maps of ``spaces" as in \S\ref{section:spaces}.  Lemma~\ref{lem:localstalks} ensures that all such maps are maps of \emph{locally} monoidal spaces. \end{example}

\begin{defn} \label{defn:logmonoidalspace} A \emph{log monoidal space} $X$ is a monoidal space $(X,\str_X)$ equipped with a log structure $\M_X$.  A morphism of log monoidal spaces $f : X \to Y$ is a map of monoidal spaces $f : X \to Y$ together with a map of log structures $f^\dagger : f^* \M_Y \to \M_X$ on $(X,\str_X)$. \end{defn}

\subsection{Charts and coherence}  \label{section:chartsforlogstructures} One generally does not work with completely arbitrary log structures.  Certain ``coherence" conditions are imposed to obtain a well-behaved theory.

\begin{defn} \label{defn:chart} A \emph{chart} for a log structure $\M_X$ on a monoidal space $X$ is a monoid homomorphism $h : P \to \M_X(X)$ such that the corresponding map $\underline{P} \to \M_X$ induces an isomorphism on associated log structures: $h^a : \underline{P}^a \cong \M_X$.  A chart is called \emph{finitely generated} (resp.\ \emph{fine}, \dots) iff the monoid $P$ is finitely generated (resp.\ fine, \dots).  A log structure is called \emph{quasi-coherent} (resp.\ \emph{coherent}, \emph{fine}, \emph{fs}, \dots) iff it locally has a chart (resp.\ finitely generated chart, fine chart, fs chart, \dots). \end{defn}

If $h : P \to \M_X(X)$ is a chart, then it is easy to see that the inclusion $i : h(P) \into \M_X(X)$ of the image of $h$ is also a chart.  If $h : P \to \M_Y(Y)$ is a chart and $f : X \to Y$ is a morphism of ringed spaces, then the composition of $h$ and $\M_Y(Y) \to (f^*\M_X)(X)$ is a chart for $f^* \M_X$ because formation of associated log structures commutes with formation of inverse image log structures.  In particular, the inverse image of a quasi-coherent (resp.\  coherent, \dots) log structure is quasi-coherent (resp.\ coherent, \dots).  For the same reason, if $h : P \to \M_X(X)$ is a chart and $U$ is an open subset of $X$, then the composition of $h$ and the restriction map $\M_X(X) \to \M_X(U)$ is a chart for $\M_X|U$ which we will also abusively denote $h : P \to \M_X(U)$ and refer to as the \emph{restriction} of $h$ (to $U$).

\begin{defn} \label{defn:characteristicchart} A chart $h:P \to \M_X(X)$ is called a \emph{characteristic chart} at $x \in X$ iff the composition $P \to \M_X(X) \to \ov{\M}_{X,x}$ is an isomorphism. \end{defn}

Under fairly mild hypotheses, one can construct characteristic charts, but we should note that there are fine log structures on $\Spec \RR$ without characteristic charts (Example~\ref{example:nocharacteristicchart}).  A prelog structure is called \emph{integral} iff $\M_X$ is a sheaf of integral monoids (\S\ref{section:monoidbasics}).  It is easy to see that the log structure associated to an integral prelog structure is integral.

\begin{defn} \label{defn:chartformorphism} Let $f : X \to Y$ be a morphism of log monoidal spaces.  A \emph{chart (resp.\ fine chart, \dots) for} $f$ is a commutative diagram of monoids $$ \xym{ P \ar[r]^-a & \M_X(X) \\ Q \ar[u]^h \ar[r]^-b & \M_Y(Y) \ar[u]_{f^\dagger} }$$ where $a$ and $b$ are charts (resp.\ fine charts, \dots).  For $x \in X$, a \emph{chart for} $f$ \emph{near} $x$ is a chart for $f|U : U \to V$ for a neighborhood $(U,V)$ of $x$ in $f$. \end{defn}

We will see in Lemma~\ref{lem:chartformorphism1} that every map of coherent log monoidal spaces has finitely generated local charts---we will then give a slightly more refined statement when ``coherent" is replaced by ``fine" in Lemma~\ref{lem:secondchart}.  Given a chart for $f$ as in the above definition, an open subspace $V$ of $Y$ and an open subspace $U$ of $f^{-1}(V)$, one can compose $a$ (resp.\ $b$) with the restriction $\M_X(X) \to \M_X(U)$ (resp.\ $\M_Y(Y) \to \M_Y(V)$) to obtain a chart for $f|U :U \to V$ which we will systematically denote $$ \xym{ P \ar[r]^-a & \M_X(U) \\ Q \ar[u]^h \ar[r]^-b & \M_Y(V) \ar[u]_{f^\dagger} }$$ and refer to as the \emph{restriction} of the given chart to $(U,V)$.

The rest of this section is devoted to establishing some basic results about charts.  Most of these results can be found in \cite{Kat1}---we have made an effort here to give precise statements and proofs here as these results will be useful to us later in our study of log smooth maps (\S\ref{section:logsmoothness}).

Suppose we have a chart for $f$ as above.  For $x \in X$, let $F := a_x^{-1}\str_{X,x}^*$, $G := b_x^{-1} \str_{Y,f(x)}^*$ be the faces of $a$ and $b$ at $x$ and $f(x)$.  Assuming $F$ and $G$ are finitely generated (by Lemma~\ref{lem:faces}, this assumption holds when $P$ and $Q$ are finitely generated), we can pass to a neighborhood $(U,V)$ of $x$ in $f$ so that $a(F) \subseteq \O_X^*(U)$ and $b(G) \subseteq \O_Y^*(V)$.  We thus obtain a chart $$ \xym{ F^{-1}P \ar[r]^-a & \M_X(U) \\ G^{-1}Q \ar[u]^h \ar[r]^-b & \M_Y(V) \ar[u]_{f^\dagger} }$$ for $f|U : U \to V$ such that \be \ov{\M}_{X,x} & = & (F^{-1} P) / a_x^{-1}(\O_{X,x}) \\ & = & (F^{-1} P)/F^{\rm gp} \\ & = & P/F \\ & = & \ov{F^{-1}P} \ee and $\ov{\M}_{Y,f(x)} = \ov{G^{-1}Q}$ similarly.  Here we have used Lemma~\ref{lem:sharpening} and the fact that formation of characteristic monoids commutes with formation of associated log structures.  We will have occassion to use this argument at various points and will refer to it as the \emph{Shrinking Argument}.

\begin{lem} \label{lem:chart1} Let $f : \M_X \to \N_{X}$ be a map of integral prelog structures on a monoidal space $(X,\str_X)$.  Then $f$ induces an isomorphism on associated log structures iff it induces an isomorphism $\ov{f} : \ov{\M}_X \to \ov{\N}_{X}$ on characteristic monoids. \end{lem} 

\begin{proof} This can be checked on stalks.  Using the fact that the monoids in question are integral, one can carry out the usual ``Five Lemma" diagram chase in $$ \xym{ 0 \ar[r] & \str_{X,x}^* \ar@{=}[d] \ar[r] & \M_{X,x}^a \ar[d]^{f_x^a} \ar[r] & \ov{\M}_{X,x} \ar[d]^{\ov{f}_x} \ar[r] & 0 \\ 0 \ar[r] & \str_{X,x}^*  \ar[r] & \N_{X,x}^a  \ar[r] & \ov{\N}_{X,x}  \ar[r] & 0} $$ to show that $f_x^a$ is an isomorphism iff $\ov{f}_x$ is an isomorphism. \end{proof}

\begin{lem} \label{lem:characteristic} Let $\M_X$ be a coherent (resp.\ fine, fs) log structure on a monoidal space $(X,\str_X)$.  For any point $x \in X$, the monoid $\ov{\M}_{X,x}$ is finitely generated (resp.\ fine, fs). \end{lem}

\begin{proof} If $h : P \to \M_X(U)$ is a chart near $x$, then the characteristic $\ov{\M}_{X,x}$ coincides with the characteristic of the prelog structure $\alpha_X h : \underline{P} \to \O_X$, which is the quotient of $P$ by the face (\S\ref{section:idealsandfaces}) $ \alpha_{X,x} h^{-1}(\str_{X,x}^*)$ of $\alpha_X h$ at $x$, so any property of $P$ inherited by quotients by faces is inherited by $\ov{\M}_{X,x}$.  The properties of finite generation, integrality, and saturation are in fact inherited by all quotients (Lemma~\ref{lem:integral}, Lemma~\ref{lem:saturated}).\end{proof}

\begin{lem} \label{lem:chartformorphism1} Let $f : X \to Y$ be a morphism of coherent (resp.\ fine) log monoidal spaces.  For any $x \in X$ there exists a coherent (resp.\ fine) chart for $f$ near $x$. \end{lem}

\begin{proof} Since $X$ and $Y$ are coherent (resp.\ fine) we can find a neighborhood $(U,V)$ of $x$ in $f$ and finitely generated (resp.\ fine) charts $Q \to \M_Y(V)$, $P \to \M_X(U)$ for $\M_Y|V$ and $\M_X|U$.  We can assume these charts are monic by passing to images if necessary.  Choose generators $q_1,\dots,q_m$ for $Q$.  The element $f^\dagger(q_i) \in \M_X(U)$ may not be in $P \subseteq \M_X(U)$, but, at least on a neighborhood $U_i$ of $x$ in $U$, we can write $f^\dagger(q_i)=u_ip_i$ for some $p_i \in P$, $u_i \in \str_X^*(U_i)$.  By replacing $U$ with the intersection of the $U_i$, we can assume $u_1,\dots,u_n \in \str_X^*(U)$.  Let $G \subseteq \str_X^*(U)$ be the subgroup generated by the $u_i$ and let $P' \subseteq \M_X(U)$ be the image of the chart $P \oplus G \to \M_X(U)$.  Then $P' \into \M_X(U)$ is a coherent (resp. fine) chart and $Q \to \M_X(U)$ factors through $P'$. \end{proof}

\begin{lem} \label{lem:localiso} Let $X$ be a monoidal space, $f^\dagger : \M_X \to \N_X$ a map of fine log structures on $X$ inducing an isomorphism $\ov{f}_x^\dagger : \ov{\M}_{X,x} \to \ov{\N}_{X,x}$ on stalks of characteristics at a point $x \in X$.  Then there is a neighborhood $U$ of $x$ so that $f|U$ is an isomorphism of log structures on $U$. \end{lem}

\begin{proof} By the previous lemma and the Shrinking Argument, we can find a fine chart $$ \xym{ F^{-1}P \ar[r] & \N_X(U) \\ G^{-1}Q \ar[u] \ar[r] & \M_X(U) \ar[u]_{f^\dagger} }$$ for $f^\dagger$ on a neighborhood $U$ of $x$ such that $\ov{\N}_{X,x} = P/F$ and $\ov{\M}_{X,x} = Q/G$.  The assumption on $\ov{f}_x^\dagger$ hence implies that the indicated arrow in the diagram $$ \xym{ F^{\rm gp} \ar[r] & F^{-1} P \ar[r] & P/F \\ G^{\rm gp} \ar[u] \ar[r] & G^{-1} Q \ar[u] \ar[r] & G/Q \ar[u]_{\cong} } $$ is an isomorphism, hence the left square of this diagram is a pushout (using Lemma~\ref{lem:sharpening} and Lemma~\ref{lem:pushout}) and we compute \be \N_X|U & = & (F^{-1}P)^a \\ & = & (G^{-1}Q)^a \oplus_{(G^{\rm gp})^a} (F^{\rm gp})^a \\ & = & (G^{-1} Q)^a \oplus_{\str_X^*} \str_X^* \\ & = & (G^{-1}Q)^a \\ & = & \M_X|U \ee using the fact that formation of associated log structures commutes with direct limits.   \end{proof}

\begin{lem} \label{lem:fschart} Let $X$ be a monoidal space, $\M_X$ a fine log structure on $X$, $x \in X$.  Suppose the characteristic extension \eqref{characteristicextension} splits.\footnote{This hypothesis is satisfied, for example, if $X$ is fs, for then $\ov{\M}_{x,x}$ is fs by Lemma~\ref{lem:characteristic}, hence $\ov{\M}_{X,x}^{\rm gp}$ is free by Lemma~\ref{lem:saturated}.}  Then the natural map $\M_{X,x} \to \ov{\M}_{X,x} =: P$ admits a section $s$ lifting to a characteristic chart $h : P \to \M_X(U)$ on a neighborhood $U$ of $x$.  \end{lem}  

\begin{proof}  It is a straightforward exercise using integrality of $P$ (Lemma~\ref{lem:characteristic}) to show that a splitting $\ov{s} : \ov{\M}_{X,x}^{\rm gp} \to \M_{X,x}^{\rm gp}$ gives rise to a section $s = \ov{s}|P$ of  $\M_{X,x} \to P$.  Using finite presentation of $P$, one can lift $s$ to $h : P \to \M_X(V)$.  Conclude by applying Lemma~\ref{lem:localiso} to $h^a : \underline{P}^a \to \M_X|V$. \end{proof}

\begin{lem} \label{lem:firstchart} Let $X$ be a monoidal space, $\M_X$ a fine log structure on $X$, $x \in X$.  Suppose $G$ is a finitely generated abelian group and $f : G \to \M_{X,x}^{\rm gp}$ is a group homomorphism whose composition with $\M_{X,x}^{\rm gp} \to \ov{\M}_{X,x}^{\rm gp}$ is surjective.\footnote{Such a map $f$ can always be chosen because $\ov{\M}_{X,x}$ is fine by Lemma~\ref{lem:characteristic}.}  Let $Q := f^{-1}( \M_{X,x})$.  Then $f|Q : Q \to \M_{X,x}$ lifts to a fine chart $Q \to \M_X(U)$ on a neighborhood of $x$. \end{lem}

\begin{proof} Apply Lemma~\ref{lem:chartproduction} with $P=\M_{X,x}$ to see that $Q^* = f^{-1}(\O_{X,x}^*)$, $\ov{Q} \to \ov{\M}_{X,x}$ is an isomorphism, and $Q$ is fine (note $\ov{\M}_{X,x}$ is fine by Lemma~\ref{lem:characteristic}).  Since $Q$ is finitely generated (hence finitely presented---\S\ref{section:monoidbasics}), $Q \to \M_{X,x}$ lifts to $Q \to \M_X(U)$ for a neighborhood $U$ of $x$.  The corresponding map $\underline{Q} \to \M_X|U$ of prelog structures induces an isomorphism on characteristics at $x$ by what we just proved, so, after possibly shrinking $U$, $Q \to \M_X(U)$ is a chart by Lemma~\ref{lem:localiso}. \end{proof}

\begin{lem} \label{lem:secondchart} Let $f : X \to Y$ be a map of fine log monoidal spaces, $x \in X$.  Suppose $a : P \to \M_X(X)$ and $b : Q \to \M_Y(Y)$ are fine charts.  Then, after possibly shrinking $X$ to a neighborhood of $x$, we can find a commutative diagram of monoids $$ \xym{ P \ar[r]^-g & P' \ar[r]^-{a'} & \M_X(X) \\ & Q \ar[u] \ar[r]^-b & \M_Y(Y) \ar[u]_{f^\dagger} } $$ where $a=a'g$ and the square is a fine chart for $f$. \end{lem}

\begin{proof}  Construct $P'$ by applying Lemma~\ref{lem:firstchart} to the groupification of $Q \oplus P \to \M_X(X)$ (or rather, its composition with $\M_X(X)^{\rm gp} \to \M_{X,x}^{\rm gp}$). \end{proof}

In particular, any two fine charts map to a third:

\begin{lem} \label{lem:thirdchart} Suppose $a_i : Q_i \to \M_X(X)$ ($i=1,2$) are fine charts for a log monoidal space $X$ and $x \in X$.  Then, after possibly replacing $X$ with a neighborhood of $x$, we can find a fine chart $a : Q \to \M_X(X)$ and monoid homomorphisms $g_i : Q_i \to Q$ such that $a_i = ag_i$. \end{lem}

\begin{proof} Apply the previous lemma with $f=\Id_X$. \end{proof}

\begin{example} \label{example:nocharacteristicchart}  The characteristic extension \eqref{characteristicextension} need not split, even for a fine log structure on $\Spec \RR$.  Let $P$ be the submonoid of $\ZZ \oplus \ZZ/ 4 \ZZ$ generated by $(1,1)$, $(1,0)$, and $(0,2)$.  This $P$ is manifestly fine.  We have $P^* = \ZZ/2\ZZ$ generated by $(0,2)$ and $P^{\rm gp} = \ZZ \oplus \ZZ / 4 \ZZ$.  The sharpening $\ov{P} = P/P^*$ is the submonoid of $\ZZ \oplus \ZZ / 2 \ZZ$ generated by $(1,1)$ and $(1,0)$ and we have $\ov{P}^{\rm gp} = \ZZ \oplus \ZZ / 2 \ZZ$.  Define a monoid homomorphism $h : P \to \RR$ by setting $h(p) = 0$ if $p \in P \setminus P^*$, $h(0,2) = -1$, and $h(0,0)=1$.  We have a commutative diagram $$ \xym{ 0 \ar[r] & P^* \ar@{=}[d] \ar[r] & P \ar[d] \ar[r] & \ov{P} \ar[d] \ar[r] & 0 \\ 0 \ar[r] & \ZZ/2\ZZ \ar[r] \ar[d]_i & \ZZ \oplus \ZZ/4\ZZ \ar[d] \ar[r] & \ZZ \oplus \ZZ / 2 \ZZ \ar[r] \ar@{=}[d] & 0 \\ 0 \ar[r] & \RR^* \ar[r] & M^{\rm gp} \ar[r] & \ZZ \oplus \ZZ /  2 \ZZ \ar[r] & 0 } $$ where the vertical arrows on top are groupifications and the bottom two rows are exact sequences of abelian groups.  The map $i$ includes $\ZZ/ 2\ZZ$ as $\{ \pm 1 \} \subseteq \RR^*$.  The bottom row is the characteristic sequence of the (fine!) log structure $M \to \RR$ associated to the prelog structure $h$ (note $M = P \oplus_{P^*} \RR^*$ by definition of the associated log structure, hence $M^{\rm gp} = P^{\rm gp} \oplus_{P^*} \RR^*$ because groupification commutes with direct limits).  The middle extension is non-trivial, hence the bottom extension must also be non-trivial because the inclusion $i$ admits a retract $r : \RR^* \to \{ \pm 1 \}$ given by $r(u) := u/|u|$. \end{example}

\subsection{Definition of log spaces} \label{section:definitionoflogspace} Fix some category of spaces $(\Esp,\u{\AA}^1)$ as in \S\ref{section:spaces}.  A \emph{log space} is a pair $X = (\u{X},\M_X)$ consisting of a space $\u{X}$ equipped with a log structure $\alpha_X : \M_X \to \str_X$, where $\str_X$ is the structure sheaf of monoids on $|\u{X}|$ ``represented by $\u{\AA}^1$," as in \S\ref{section:spaces}.  A log space is called \emph{coherent} (resp.\ \emph{fine}, \dots) iff the log structure $\M_X$ is coherent (resp.\ fine, \dots).  Log spaces form a category $\LogEsp$ where a morphism $f : X \to Y$ is an $\Esp$-morphism $\u{f} : \u{X} \to \u{Y}$ (which induces a map of locally monoidal spaces $(|\u{X}|,\str_X) \to (|\u{Y}|,\str_Y)$ as in \S\ref{section:spaces}) together with a morphism of log structures $f^\dagger : f^* \M_Y \to \M_X$ on the locally monoidal space $\u{X}$.  

There is an obvious forgetful functor \bne{LogEsptoEsp} \LogEsp & \to & \Esp \\ \nonumber X & \mapsto & \u{X} \\ \nonumber (f : X \to Y) & \mapsto & (\u{f} : \u{X} \to \u{Y}). \ene  Consequently, $\LogEsp$ inherits an ``underlying topological space" functor \bne{LogEsptoTop} \LogEsp & \to & \Top \\ \nonumber X & \mapsto & |\u{X}|, \ene which we also denote by $X \mapsto |X|$ by abuse of notation.

We typically denote a log space by $X$, reserving $\u{X}$ for its underlying space, not to be confused with its underlying \emph{topological} space $|X|=|\u{X}|$.  We will say ``point of $X$," ``sheaf on $X$," etc.\ as abuse of notation for ``point of $|\underline{X}|$," ``sheaf on $|\underline{X}|$," etc.

The forgetful functor \eqref{LogEsptoEsp} has a right adjoint \bne{EsptoLogEsp} \Esp & \to & \LogEsp \ene obtained by mapping a space $\underline{X}$ to the log space abusively denoted $\underline{X}$ obtained by equipping $\underline{X}$ with the trivial log structure (\S\ref{section:logstructures}).  If $f : \underline{X} \to \underline{Y}$ is a map of spaces, then a $\LogEsp$ morphism $g$ with $\underline{g} = f$ will be called a \emph{lift} of $f$.

\begin{prop} \label{prop:LogEsplimits} The category $\LogEsp$ has finite inverse limits and arbitrary coproducts, both commuting with the forgetful functor $X \mapsto \u{X}$ of \eqref{LogEsptoEsp}. \end{prop}

\begin{proof} If $\{ X_i \}$ is a finite inverse limits system of log spaces, then we can construct an inverse limit $X$ of $\{ X_i \}$ by endowing the inverse limit space $\u{X}$ of $\{ \u{X}_i \}$ with the log structure given by the direct limit (in the category of log structures) of the log structures $\pi_i^* \M_{X_i}$, where $\pi_i : \u{X} \to \u{X}_i$ is the projection map for the inverse limit $\u{X}$.  To form the coproduct of log spaces $\{ X_i \}$, we first form the coproduct $\u{X}$ of $\{ \u{X}_i \}$ in spaces.  Since the underlying space functor for spaces commutes with coproducts, $|\u{X}| = \coprod |\u{X}_i|$ and since representable presheaves are sheaves we have $\str_X = \coprod \str_{X_i}$ and we can endow $\u{X}$ with the log structure $\M_X := \coprod_i \M_{X_i}$.  This $X$ will serve as the coproduct of $\{ X_i \}$.   \end{proof}

\begin{lem} \label{lem:inverselimits}  A finite inverse limit $X$ of coherent log spaces $X_i$ is coherent. \end{lem}

\begin{proof} The question and the construction of inverse limits are local in nature, so we can assume, using Lemma~\ref{lem:chartformorphism1} (or its proof) that there are finitely generated charts $a_i : Q_i \to \M_{X_i}(X_i)$ for all the $X_i$ lifting to charts $Q_i \to Q_j$ for all the transition functions $X_i \to X_j$.  Since formation of associated log structures commutes with direct limits and pullbacks, the natural map $\dirlim Q_i \to  \M_X(X)$ is a finitely generated chart. \end{proof}

Unfortunately, a finite inverse limit of \emph{fine} log spaces is not necessarily fine (i.e.\ is not necessarily integral), though we will see in \S\ref{section:integration} that the category $\FineLogEsp$ of fine log spaces \emph{does} have finite inverse limits, though they will \emph{not} commute with \eqref{LogEsptoEsp}.

There is a tautological monoid homomorphism \bne{tautmap} \NN & \to & \M_{\u{\AA}^1}(\u{\AA}^1) = \Hom_{\Esp}(\u{\AA}^1,\u{\AA}^1) \ene given by mapping $1 \in \NN$ to $\Id \in \Hom_{\Esp}(\u{\AA}^1,\u{\AA}^1)$.  We let $\AA^1$ denote the log space whose underlying space is $\u{\AA}^1$ and whose log structure $\M_{\AA^1}$ is the one associated to \eqref{tautmap}.   Using the fact that $\u{\AA}^1$ (tautologically) represents the functor \be \Esp & \to & \Mon \\ \u{X} & \mapsto & \str_{ \u{X} }( \u{X} ) \ee we see that $\AA^1$ represents the functor \be \LogEsp & \to & \Mon \\ X & \mapsto & \M_X(X). \ee  (We will prove a more general statement along these lines in Proposition~\ref{prop:APdescription}.)  This endows $\AA^1$ with the structure of a monoid object in $\LogEsp$.

\begin{prop} \label{prop:logspacesarespaces}  For any category of spaces $(\Esp,\u{\AA}^1)$, the pair $(\LogEsp,\AA^1)$ consisting of the category of log spaces and the monoid object $\AA^1$ defined above is itself a category of spaces and the functor \eqref{LogEsptoEsp} is a $1$-morphism of categories of spaces (Definition~\ref{defn:morphismofspaces}).  The association of ``log spaces" to ``spaces" defines an endomorphism \be \Log : \Univ & \to & \Univ \ee of the $2$-category $\Univ$.  \end{prop}

\begin{proof}  Proposition~\ref{prop:LogEsplimits} shows that $\LogEsp$ satisfies \eqref{finitelimits} and \eqref{coproducts} and the \eqref{LogEsptoEsp} preserves finite inverse limits and coproducts.  By construction of $\AA^1$, we see that \eqref{LogEsptoEsp} takes $\AA^1$ to $\u{\AA}^1$.  Because of the way we defined the underlying space functor for $\LogEsp$, that proposition also shows that the axiom \eqref{coproductscommute} for $\Esp$ implies \eqref{coproductscommute} for $\LogEsp$.  For the open embeddings axiom \eqref{openembeddings}, suppose $X$ is a log space and $|\u{U}|$ is an open subset of $|X| = |\u{X}|$.  Let $\u{U} \to \u{X}$ be the open embedding in $\Esp$ corresponding to $|\u{U}|$.  Then if we set $U := (\u{U},\M_X|U)$, then it is clear that $U \to X$ ``is the" open embedding in $\LogEsp$ corresponding to $|\u{U}|$.  This discussion also shows that \eqref{LogEsptoEsp} satisfies the condition \eqref{openembeddingspreserved} in Definition~\ref{defn:morphismofspaces}.  The axioms \eqref{Zariskitopology}-\eqref{Zariskigluing} for $\LogEsp$ are established easily from the fact that they hold in $\Esp$ and in the category of monoidal spaces.  The final axiom \eqref{openunits} is also clearly inherited from $\Esp$---in fact, the group object $\GG_m \in \LogEsp$ is nothing but $\u{\GG}_m$ with the trivial log structure. \end{proof}

\begin{rem} \label{rem:spacesandlogspaces} The above proposition in some sense puts $(\LogEsp,\AA^1)$ back on the same footing as $(\Esp,\u{\AA}^1)$ because ``log spaces" are just the ``spaces" for another category of spaces, so it now becomes superfluous to handle ``log spaces" differently from ``spaces."  Nevertheless, we will often maintain a somewhat artificial distinction between the two, especially if we want to discuss them at the same time. \end{rem}

\begin{rem} One can talk about ``log log spaces" by iterating the functor $\Log$, but in practice there is no use for the category of ``log log spaces." \end{rem}

\subsection{Strict maps of log spaces} \label{section:strictmapsoflogspaces}  

\begin{defn} \label{defn:strict} A morphism $f: X \to Y$ of log spaces is called \emph{strict} iff $f^\dagger : f^* \M_Y \to \M_X$ is an isomorphism. \end{defn}

\begin{lem} \label{lem:strictmorphisms} Strict morphisms are closed under base change.  For a $\LogEsp$ morphism $f : X \to Y$ and a strict $\LogEsp$ morphism $g : Y \to Z$, $gf$ is strict iff $f$ is strict.  A morphism $f : X \to Y$ of integral log spaces is strict iff $\ov{f}^\dagger : f^{-1} \ov{\M}_Y \to \ov{\M}_X$ is an isomorphism.  \end{lem}

\begin{proof} The first two statements are straightforward exercises with the definitions.  The third statement follows from Lemma~\ref{lem:chart1} and the fact that $\ov{f^* \M_Y} = f^{-1} \ov{\M}_Y$ because $f$ is local (\S\ref{section:logstructures}). \end{proof}

\begin{rem} \label{rem:logspacesstack} The forgetful functor $\LogEsp \to \Esp$ from log spaces to spaces is a \emph{fibered category} in the sense of \cite[3.1]{Vis}, \cite[2.1]{minimality}.  The \emph{cartesian arrows} in $\LogEsp$, in the stack-theoretic sense, are exactly the strict morphisms.  See \cite{minimality} for further discussion. \end{rem}

\subsection{Log spaces to locally monoidal spaces} \label{section:logspacestomonoidalspaces} We saw in Proposition~\ref{prop:logspacesarespaces} that $(\LogEsp,\AA^1)$ is a category of spaces, so we have a functor \bne{LogEsptoLMS} \LogEsp & \to & \LMS \\ \nonumber X & \mapsto & |X| \ene as a special case of the functor \eqref{EsptoLMS} of \S\ref{section:spacestolocallymonoidalspaces}.  Recall from \S\ref{section:definitionoflogspace} that the monoid object $\AA^1 \in \LogEsp$ represents the functor $X \mapsto \M_X(X)$ so that the ``structure sheaf" of the locally monoidal space $|X|$ \emph{is} the sheaf of monoids $\M_X$ given by the domain of the log structure $\alpha_X : \M_X \to \str_X$ on $X$.  One can view the maps $\alpha_X$ as a natural transformation of functors from \eqref{LogEsptoLMS} to the composition of the forgetful functor \eqref{LogEsptoEsp} and the functor \eqref{EsptoLMS}.

Our interest in the category $\LMS$ is mainly due to the fact that it is the target of the functor \eqref{LogEsptoLMS}.  We can compose \eqref{LogEsptoLMS} with the sharpening functor $\LMS \to \SMS$ (\S\ref{section:monoidalspacedefinitions}) to obtain a functor \bne{LogEsptoSMS} \LogEsp & \to & \SMS. \ene  We will usually denote this functor by $X \mapsto \ov{X}$ on objects and $f \mapsto \ov{f}$ on morphisms, to avoid cumbersome notation.

\begin{lem} \label{lem:chart2} Let $X$ be a log space with integral log structure $\M_X$.  Suppose $h : P \to \M_X(X)$ is a map of monoids and $$f :\ov{X} \to (\Spec P,\ov{\M}_P)$$ is the induced map of sharp monoidal spaces.  Then $h$ is a chart iff $\ov{f}^\dagger : f^{-1} \ov{\M}_P \to \ov{\M}_X$ is an isomorphism (i.e.\ $f$ is strict). \end{lem}

\begin{proof}  In the proof we will write $P_X$ for the sheaf of locally constant functions to $P$ on $X$.  By Lemma~\ref{lem:chart1}, $h : P_X \to \M_X$ is a chart iff $\ov{h} : \ov{P}_X \to \ov{\M}_{X}$ is an isomorphism, where $\ov{P}_X = P_X / (\alpha_X h)^{-1}(\str_{X,x}^*)$ is the characteristic monoid of the prelog structure $\alpha_X h : P_X \to \str_X$.  On $Y = \Spec P$, recall that there is a natural map $\tau^\dagger : \u{P} \to \M_P$ of sheaves of monoids on $Y$ (here $\u{P}$ is the constant sheaf of monoids on $Y$ associated to $P$), corresponding to the tautological map $P \to \M_P(Y)$ under the adjunction between ``constant sheaf" and ``global sections" (\S\ref{section:Specrevisited}).  Composing $\tau^\dagger$ with the sharpening, we obtain a natural map $\ov{\tau}^\dagger : \u{P} \to \ov{\M}_P$.  The original $h$ factors as the composition of $f^{-1}n : P_X \to f^{-1} \M_P$ and the map $f^{-1} \M_P \to \M_X$ corresponding to the map from $X$ to the locally monoidal space $(\Spec P,  \M_P)$.  In particular, $\ov{h}$ factors as the composition of $f^{-1} \ov{n} : \ov{P}_X \to f^{-1} \ov{\M}_P$ and $\ov{f}^\dagger$, so it is enough to prove that $f^{-1} \ov{\tau}^\dagger$ is an isomorphism, which is clear from the way $f$ is constructed from $h$ (check on stalks, say). \end{proof}

\begin{prop} If $X$ is a fine log space then the sharp monoidal space $\sms{X}$ is fine in the sense of Definition~\ref{defn:coherentmonoidalspace}. \end{prop}

\begin{proof} This is clear from the definitions using Lemma~\ref{lem:chart2}. \end{proof}

\begin{rem} \label{rem:strictmorphisms} If $f : X \to Y$ is a strict morphism of log spaces in the sense of Definition~\ref{defn:strict}, then $\lms{f}$ need not be a strict $\LMS$ morphism in the sense of Definition~\ref{defn:monoidalspacechart} because $f^{-1} \M_Y$ and $f^* \M_Y$ may not coincide.  However, if $f$ is strict, then the $\SMS$ morphism $\ov{f}$ will be strict, and the converse holds as long as $X$ and $Y$ are integral---see Lemma~\ref{lem:strictmorphisms}.  \end{rem}

\subsection{Fans to log spaces} \label{section:fanstologspaces} Let $(\Esp,\u{\AA}^1)$ be a category of spaces, $(\AA^1,\LogEsp)$ the associated category of log spaces (\S\ref{section:definitionoflogspace}).  Since $(\AA^1,\LogEsp)$ is itself a category of spaces (Proposition~\ref{prop:logspacesarespaces}), the functor \eqref{FanstoEsp} of \S\ref{section:realizationoffans} may be viewed as a functor \bne{FanstoLogEsp} \AA : \Fans & \to & \LogEsp  \ene (in fact a $1$-morphism of spaces).  To avoid confusion, we will usually denote the analogous functor (in fact: $1$-morphism of spaces) for the category of spaces $(\Esp,\u{\AA}^1)$ by \bne{FanstoEsp2} \u{\AA} : \Fans & \to & \Esp. \ene As the notation suggests, the functor \eqref{FanstoEsp2} ``is" (isomorphic by a unique $2$-isomorphism to) the composition of \eqref{FanstoLogEsp} and the forgetful functor \eqref{LogEsptoEsp}.  Indeed, the latter composition is a $1$-morphism of spaces (Proposition~\ref{prop:logspacesarespaces}) so the two $1$-morphisms of spaces in question are uniquely $2$-isomorphic in $\Univ$ because $\Fans$ is ``the" $2$-initial object in $\Univ$ by Theorem~\ref{thm:fans2initial}.

Let us now make the relationship between \eqref{FanstoLogEsp} and \eqref{FanstoEsp} a little more concrete.  We can compose $\Spec : \Mon^{\rm op} \to \Fans$ with \eqref{FanstoEsp2} or \eqref{FanstoLogEsp} to obtain functors \bne{MontoEsp} \u{\AA} : \Mon^{\rm op} & \to & \Esp \\ \nonumber P & \mapsto & \u{\AA}(P) \\ \label{MontoLogEsp} \AA : \Mon^{\rm op} & \to & \LogEsp \\ \nonumber P & \mapsto & \AA(P). \ene  To ease notation, we let $\str_P$ denote the structure sheaf of $\u{\AA}(P)$ (the sheaf of monoids on $|\u{\AA}(P)|$ represented by $\u{\AA}^1$).  According to Lemma~\ref{lem:AP}, $\u{\AA}(P)$ represents the presheaf \bne{uAP} \Esp & \to & \Sets \\ \nonumber \u{X} & \mapsto & \Hom_{\Mon}(P,\str_{\u{X}}(\u{X}). \ene From this, we have a tautological map of monoids \bne{tautchart2} P & \to & \str_P(\u{\AA}(P)) \ene  by considering the image of the identity map under the natural bijection \be \Hom_{\Esp}(\u{\AA}(P),\u{\AA}(P)) & = & \Hom_{\Mon}(P,\str_P(\u{\AA}(P))). \ee

\begin{prop} \label{prop:APdescription} The log space $\AA(P)$ represents the presheaf \be \LogEsp & \to & \Sets \\ X & \mapsto & \Hom_{\Mon}(P,\M_X(X)) \ee and is naturally isomorphic to the space $\u{\AA}(P)$ equipped with the log structure $\M_P$ associated to \eqref{tautchart2}.  \end{prop}

\begin{proof} The first statement holds by construction of $\AA(P)$ (Lemma~\ref{lem:AP}, \S\ref{section:realizationoffans}) because $\M_X$ is the structure sheaf of the log space $X$, as discussed above.  For the next statement, Yoneda's Lemma reduces us to showing that $(\u{\AA}(P), \M_P)$ represents the same presheaf.

Suppose $X$ is a log space, $Y$ is a \emph{pre}log space, and $\u{f} :\u{X} \to \u{Y}$ is a map of spaces.  By the adjointness property of associated log structures, formation of associated log structures commutes with inverse images and giving a map of log structures $f^\dagger : f^* (\M_Y^a) \to \M_X$ is the same thing as giving a map of prelog structures $f^\dagger_{\rm pre} : f^{-1} \M_Y \to \M_X$.  In particular, if \bne{globalchartT} P & \to & \M_Y(Y) \ene is a global chart for a log structure $\M_Y$, then giving a map of log structures $f^\dagger : f^* \M_Y \to \M_X$ is nothing but the data of a monoid homomorphism $f^\dagger : P \to \M_X(X)$ making the following diagram of monoids commute: $$ \xym{ P \ar[d] \ar[rr]^-{f^\dagger} & & \M_X(X) \ar[d] \\ \M_Y(Y) \ar[r]^-{\alpha_Y} & \str_Y(Y) \ar[r] & \str_X(X).  }$$  

By combining the above discussion (in the case where $\u{Y} = \u{\AA}(P)$ and \eqref{globalchartT} is obtained from \eqref{tautchart2}) and the fact that $\u{\AA}(P)$ represents \eqref{uAP}, we see that $(\u{\AA}(P), \M_P)$ represents the same presheaf as $\AA(P)$. \end{proof}

\subsection{Examples} \label{section:logspacesexamples} Now that we have the basic theory of spaces and log spaces, it is time for some examples.  The categories of log spaces associated to some of the categories of spaces mentioned in \S\ref{section:spaces} are tabulated below.  In the left column, we give the category of spaces, in the middle column we give the name of the corresponding ``log spaces," and in the right column we introduce alternative notation for the functor \eqref{MontoLogEsp} of the previous section.

\begin{tabular}{lll} $(\Top,\RR_+)$ & positive log topological spaces & $\RR_+( \slot )$ \\ $(\LRS,\Spec \ZZ[x])$ & log locally ringed spaces & $\Spec \ZZ[\slot]$ \\ $(\Sch,\Spec \ZZ[x])$ & (Zariski) log schemes & $\Spec \ZZ[ \slot ]$ \\ $(\DS,\RR)$ & log differentiable spaces & $\RR(\slot)$ \\ $(\DS,\RR_+)$ & positive log differentiable spaces & $\RR_+(\slot)$ \\ $(\DS,\CC)$ & complex log differentiable spaces & $\CC(\slot)$ \\ $(\AS,\CC)$ & log analytic spaces & $\CC(\slot)$ \end{tabular}

We will make a specialized study of (positive) log differentiable spaces in \S\ref{section:lds}.

In the next example, we describe the images of some fans under \eqref{FanstoLogEsp} for various categories of log spaces.

\begin{example} \label{example:realizationoffans} Recall the fan $\PP^1$ from Example~\ref{example:P1}.  The realizations $\AA(\PP^1)$ of this fan in various categories of log spaces are as follows: \begin{enumerate} \item In $(\LRS,\AA^1)$ or $(\Sch,\AA^1)$, $\AA(\PP^1)$ is the usual scheme $\PP^1$ (over $\ZZ$) with ``log structure at $0$ and $\infty$." \item In complex analytic spaces, $\AA(\PP^1)$ is of course the analytification of (the complexification of) the previous example---i.e.\ the Riemann Sphere. \item In $(\LDS,\RR)$, $\AA(\PP^1) = \RR \PP^1 \in \LDS$ is a circle, with log structure at ``opposite poles." \item In $(\LDS,\RR_+)$, $\AA(\PP^1) \in \PLDS$ is the closed interval $[0,\infty]$ with log structure ``at the endpoints." \end{enumerate} The realizations of $\PP^n$ in these categories of spaces should also be clear to the reader, though the realization in $\PLDS$ may be less obvious.

Let $\AA^n$ denote the affine fan $\Spec ( \NN^n)$ and let $U_n$ denote the ``quasi-affine" fan obtained from $\AA^n$ by removing the unique closed point.  Then one can construct (as in algebraic geometry or elsewhere) a map of fans $f : U_{n+1} \to \PP^n$ which is a locally trivial $\GG_m := \Spec \ZZ$ bundle.  In fact we claim that $f$ is a $\GG_m$-torsor under the action of the group fan $\GG_m$ on $U_n$ inherited from the coordinatewise scaling action of $\GG_m$ on $\AA^{n+1}$.  Recall (\S\ref{section:inverselimits}) that $\Spec \ZZ \to \Spec \{ 0 \}$ is a universal homeomorphism in $\LMS$, so we are claiming that $f$ is a (universal) homeomorphism in $\LMS$ and that $\M_{U_n}$ is locally (on opens pulled back from $\PP^n$ and in fact on the preimage of each standard open in the usual cover of $\PP^n$) isomorphic, as a sheaf of monoids under $f^{-1} \M_{\PP^n}$, to $f^{-1} \M_{\PP^n} \oplus \underline{\ZZ}$.  Incidentally, the spaces $U_{n+1}$ and $\PP^n$ both have \be 2^{n+1}-1 & = & 1 + 2 + 2^2 + \cdots + 2^n \ee points.  If we realize $f$ in $\PLDS$ we find that \be \RR_+(f) : \RR_+(U_n) = \RR_+^{n+1} \setminus \{ 0 \} & \to & \RR_+(\PP^n) \ee is an $\RR_+(\GG_m) = \RR_{>0}$ torsor under the coordinate-wise scaling action, so that we have \be \RR_+(\PP^n) & = & (\RR_+^{n+1} \setminus \{ 0 \}) / \RR_{>0}. \ee  Each orbit for this action has a unique representative $t \in \RR_+^{n+1}$ with $\sum_i t_i = 1$, thus we see that \be \RR_+(\PP^n) & = & \{ t \in \RR_+^{n+1} : \sum_i t_i = 1 \} \\ & = & \{ t \in [0,1]^{n+1} : \sum_i t_i =1 \} \ee is the standard $n$-simplex, appropriately viewed as a manifold with corners.  

The reader may also want to think about the maps $\tau$ and $\u{\tau}$ in these examples.  \end{example}

Recall that the differentiable spaces $\u{\RR}(P)$ and $\u{\RR}_+(P)$ were described ``explicitly" in \S\ref{section:RP}.

\subsection{Properties of realizations} \label{section:propertiesofrealizations}  Most results in log geometry rely on a careful understanding of the log spaces $\AA(P)$ (and the underlying spaces $\u{\AA}(P)$) and the maps between these spaces induced by monoid homomorphisms.  The rest of this section is devoted to some elementary results along these lines.

\begin{lem} \label{lem:saturationhomeo} Let $h : Q \to P$ be a map of finitely generated monoids so that $h$ makes $P$ a finitely generated $Q$-module (\S\ref{section:modules}).  Then: \begin{enumerate} \item \label{U1} $\ZZ[Q] \to \ZZ[P]$ is a finite map of rings. \item \label{U2} The induced map of analytic spaces $\CC(P) \to \CC(Q)$ is finite (proper with finite fibers). \item \label{U3} The induced map of differentiable spaces $\RR(P) \to \RR(Q)$ is finite. \item \label{U3b} The induced map of differentiable spaces $\RR_+(P) \to \RR_+(Q)$ is finite. \item \label{U4} If we assume furthermore that $h$ is injective and $P$ and $Q$ are fine, then the induced map of differentiable spaces $\RR_+(P) \to \RR_+(Q)$ is a homeomorphism. \end{enumerate} \end{lem}

\begin{proof} For \eqref{U1} just note that if $p_1,\dots,p_n \in P$ generate $P$ as a $Q$-module, then the corresponding elements $[p_1],\dots,[p_n] \in \ZZ[P]$ generate $\ZZ[P]$ as a $\ZZ[Q]$-module.  This implies that $\Spec \CC[P] \to \Spec \CC[Q]$ is a finite (i.e.\ proper and quasi-finite) map of finite type (affine) $\CC$-schemes, so \eqref{U2} follows from standard GAGA results since the only issue is to show that the map of analytic spaces corresponding to a proper map of finite type $\CC$-schemes is proper (quasi-finiteness obviously passes to associated analytic spaces).  For \eqref{U3}, notice that we have a commuative diagram of topological spaces $$ \xym{ \RR(P) \ar[r] \ar[d] & \CC(P) \ar[d] \\ \RR(Q) \ar[r] & \CC(Q) } $$ where the horizontal arrows are closed embeddings (they are the fixed loci of the automorphisms induced by complex conjugation).  Although the diagram may not be cartesian, one can still conclude finiteness (or properness) of the left vertical arrow from finiteness (or properness) of the right vertical arrow, which we know by \eqref{U2}.  We see that \eqref{U3} implies \eqref{U3b} by a similar argument using the closed embeddings $\RR_+(P) \into \RR(P)$.  For \eqref{U4}: First recall that Theorem~\ref{thm:dense} says that $h$ is dense, so for every $p \in P$ there is an $n > 0$ such that $np \in Q$.  Since every element of $\RR_+$ has a unique $n^{\rm th}$ root, we see easily that the continuous map $\RR_+(P) \to \RR_+(Q)$ is bijective.  Since it is finite (hence a closed map) by \eqref{U3b}, it is a homeomorphism. \end{proof}

For our later purposes we will need to know some circumstances under which a monoid homomorphism $h$ induces a surjection $\underline{\RR}(h)$ on ``$\RR$ points" and/or a surjection $\underline{\RR}_+(h)$ on ``$\RR_+$ points."

\begin{lem} \label{lem:surjectivity} Let $h : Q \to P$, $x : Q \to \RR_{\geq 0}$ be monoid homomorphisms and let $x_{\CC} : Q \to \CC$ be the monoid homomorphism obtained by composing $x$ with the inclusion $\RR_{\geq 0} \into \CC$ of (multiplicative) monoids.  If there is a monoid homomorphism $y : P \to \CC$ such that $yh=x_{\CC}$, then there is a monoid homomorphism $z : P \to \RR_{\geq 0}$ such that $zh=x$.  In particular, surjectivity of \be h^* : \Hom_{\Mon}(P,\CC) & \to & \Hom_{\Mon}(Q,\CC) \ee implies surjectivity of \be h^* : \Hom_{\Mon}(P,\RR_{\geq 0}) & \to & \Hom_{\Mon}(Q,\RR_{\geq 0}). \ee The same statements hold with ``$\RR_{\geq 0}$" replaced everywhere by ``$\RR$" when $Q$ is fs and $h^{\rm gp}$ is an isomorphism. \end{lem}

\begin{proof} The first statement is trivial: The absolute value map $| \slot | : \CC \to \RR_{\geq 0}$ is a monoid homomorphism retracting the inclusion $\RR_{\geq 0} \into \CC$, so one can take $z = |y|$.  For the last statement, let $G := x^{-1}(\RR^*) = x_{\CC}^{-1}(\CC^*)$ and $F := y^{-1}(\CC^*)$ be the faces of $Q$ and $P$ determined by $x$ and $y$, respectively, so we have a diagram of abelian groups  \bne{abgpdiagram} & \xym{ G^{\rm gp} \ar[d]_{\sign x} \ar[r]^-{f} & F^{\rm gp} \ar@{.>}[ld]^t \\ \{ \pm 1 \} = \ZZ / 2 \ZZ } \ene where we set $f := h^{\rm gp}|G^{\rm gp}$.  If \eqref{abgpdiagram} has a completion $t$ as indicated, then since $F$ is a face, \be z : P & \to & \RR \\ p & \mapsto & \left \{ \begin{array}{lll} 0, & \quad & p \notin F \\ t(p)|y(p)|, & & p \in F \end{array} \right . \ee is a well-defined monoid homomorphism with $zh=x$.  Since $h^{\rm gp}$ is an isomorphism, $f$ is injective, and the obstruction to completing \eqref{abgpdiagram} lies in $\Ext^1(\Cok f,\ZZ/2\ZZ)$.  Applying the Snake Lemma to the exact diagram $$ \xym{ 0 \ar[r] & G^{\rm gp} \ar[r] \ar[d]^f & Q^{\rm gp} \ar[r] \ar[d]^{h^{\rm gp}} & (Q/G)^{\rm gp} \ar[r] \ar[d]^g & 0 \\ 0 \ar[r] & F^{\rm gp} \ar[r] & P^{\rm gp} \ar[r] & (P/F)^{\rm gp} \ar[r] & 0 }$$ defining $g$ we find that $\Cok f  =  \Ker g .$ Since $Q$ is fs and $G$ is a face of $Q$, $Q/G$ is sharp (Lemma~\ref{lem:sharpening}) and fs (Lemma~\ref{lem:saturated}), hence $(Q/G)^{\rm gp}$ is free (Lemma~\ref{lem:saturated}), hence so is its subgroup $\Cok f = \Ker g$, hence the aforementioned $\Ext$ group vanishes. \end{proof}

\begin{example} \label{example:nosurjectivity} One cannot replace ``fs" with ``fine" in Lemma~\ref{lem:surjectivity}, even when $P=Q^{\rm sat}$.  Let $Q$ be the (fine!) submonoid of $\ZZ \oplus \ZZ/4\ZZ$ generated by $(1,0)$, $(0,2)$, and $(1,3)$.  Then $Q^{\rm gp} = \ZZ \oplus \ZZ/4\ZZ$ and $Q^{\rm sat} = \NN \oplus \ZZ/4\ZZ$.  The monoid homomorphism \be x : Q & \to & \RR \\ (a,b) & \mapsto & \left \{ \begin{array}{lll} 0, & \quad & a > 0 \\ -1, & & (a,b)=(0,2) \\ 1, & & (a,b)=(0,0)  \end{array} \right . \ee does not extend to $z : Q^{\rm sat} \to \RR$ because $z(0,1)$ would have to be a square root of $-1$. \end{example}

\subsection{Groups to group spaces} \label{section:groupstogroupspaces} Since the functor $\AA$ of \eqref{MontoLogEsp} preserves inverse limits, it takes group objects to group objects and actions to actions.  Any \emph{group} (all groups are assumed finitely generated abelian) $G$ is a group object in $\Mon^{\rm op}$ with comultiplication $G \to G \oplus G$ given by $g \mapsto (g,g)$.  If $P$ is any monoid, the group $P^{\rm gp}$ acts on the monoid $P$ via the coaction map $P \to P \oplus P^{\rm gp}$ given by $p \mapsto (p,p)$.  This group action is ``universal" in the sense that any coaction map $P \to P \oplus G$ determines an obvious monoid map $P \to G$ (which is equivalent to a group map $P^{\rm gp} \to G$) so that the action of $G$ on $P$ specified by the coaction map coincides with the $G$ action on $P$ obtained from the universal action and the map $P^{\rm gp} \to G$. The notation $\GG(P) := \AA(P^{\rm gp})$ is often convenient.  The log structure on $\GG(P)$ is the trivial one because the chart $P^{\rm gp} \to \str_{P^{\rm gp}}(\GG(P))$ defining this log structure factors through the units.  Hence there is little need to distinguish between $\GG(P)$ and $\u{\GG}(P)$.

It will be important to understand various geometric realizations of such group objects and group actions.  Let us begin with a few examples concerning the structure of the group objects $\GG(G)$ in the differentiable setting.

\begin{example} \label{example:RG}  If $G \cong \ZZ^r \oplus \ZZ / a_1 \ZZ \oplus \cdots \oplus \ZZ/a_k \ZZ$ is a finitely generated abelian group, then $\underline{\RR}(G)$ is the disjoint union of $\epsilon(a_1) \cdots \epsilon(a_k)$ copies of $(\RR^*)^k$, where $\epsilon(a_i)$ is the number of $a_i^{\rm th}$ roots of unity in $\RR$ (one if $a_i$ is odd, two if $a_i$ is even).  \end{example}

\begin{example} \label{example:RGplus} If $G$ is a finitely generated abelian group of rank $r$, then $\underline{\RR}_+(G) \cong \RR_{>0}^r$ and we will usually write $\RR_{>0}(G)$ instead of $\underline{\RR}_+(G)$.  The torsion of $G$ is irrelevant because $1$ is the only non-negative root of unity in $\RR$.  The differentiable space $\RR_{>0}(G)$ represents the presheaf \bne{RGplusmodular} X & \mapsto & \Hom_{\Ab}(G,\O_X^{>0}(X)), \ene where $\O_X^{>0}$ is the sheaf of positive functions (\S\ref{section:positivefunctions}).  The natural splitting of $\O_X^*$ in \eqref{unitsplitting} gives a coproduct decomposition \bne{coproductdecomp} \underline{\RR}(G) & = & \coprod_g \RR_{>0}(G), \ene where $g$ runs over $\Hom_{\Ab}(G,\{ \pm 1 \})$.  Explicitly, the component of $\underline{\RR}(G)$ indexed by $g$ represents the subpresheaf of $X \mapsto \Hom_{\Ab}(G,\O_X^*(X))$ consisting of group homomorphisms $G \to \O_X^*(X)$ which can be written as $fg$ (this juxtaposition is a product, not a composition) for some group homomorphism $f : G \to \O_X^{>0}(X)$.  Since $\RR_{>0}(G)$ is connected, \eqref{coproductdecomp} is nothing but the decomposition of $\underline{\RR}(G)$ into its connected components.\end{example}

It is clear from Examples~\ref{example:RG} and \ref{example:RGplus} that $\underline{\RR}(G)$ and $\RR_{>0}(G)$ are smooth differentiable spaces for any finitely generated abelian group $G$.

Suppose we have an exact sequence of groups \bne{groupsequence} 0 \to A \to B \to G \to 0. \ene  Pushing out along $A \to B$ we obtain a map of exact sequences \bne{groupsequence2} \xym{ 0 \ar[r] & A \ar[d] \ar[r] & B \ar[d] \ar[r] & G \ar@{=}[d] \ar[r] & 0 \\ 0 \ar[r] & B \ar[r] & B \oplus_A B \ar[r] & G \ar[r] & 0.} \ene  The sequence on the bottom of \eqref{groupsequence2} splits canonically:  The surjection has a section given by $g \mapsto [-b,b]$, where $b \in B$ is any lift of $g \in G$ (the choice is irrelevant).  Another way of saying this is to note that we have a map of exact sequences \bne{groupsequence3} \xym{ 0 \ar[r] & A \ar[d] \ar[r] & B \ar[d] \ar[r] & G \ar@{=}[d] \ar[r] & 0 \\ 0 \ar[r] & B \ar[r] & B \oplus G \ar[r] & G \ar[r] & 0} \ene where the left square is a pushout because giving group homomorphisms $f,g : B \to D$ with $f|A=g|A$ is the same thing as given a group homomorphism $B \oplus G \to D$ (the bijection is $(f,g) \mapsto (f,f-g)$).  A geometric interpretation of this discussion goes as follows: the map $B \to G$ in the sequence \eqref{groupsequence} together with the tautological action of $B$ on itself induce an action of $G$ on $B$ in $\Mon^{\rm op}$, which is a fiberwise action for the map $B \to A$ in $\Mon^{\rm op}$ corresponding to the map $A \to B$.  This fiberwise action of $G$ on $B$ makes $B$ a $G$-torsor over $A$ in $\Mon^{\rm op}$ which becomes trivial when pulled back to itself.   

\begin{lem} \label{lem:groupsmooth} Let $A \to B$ be an injective map of finitely generated abelian groups with cokernel $G$.  Let $N$ denote the order of the torsion subgroup of $G$.  Then: \begin{enumerate} \item \label{groupsmooth1} The map $\Spec \ZZ[B] \to \Spec \ZZ[A]$ is faithfully flat and is a $\Spec \ZZ[G]$-torsor, locally trivial in the flat topology.  \item \label{groupsmooth2} $\Spec \ZZ[1/N][B] \to \Spec \ZZ[1/N][A]$ is a $\Spec \ZZ[1/N][G]$-torsor, locally trivial in the \'etale topology.  The scheme $\Spec \ZZ[1/N][G]$ is smooth over $\Spec \ZZ[1/N]$, hence $\Spec \ZZ[1/N][B] \to \Spec \ZZ[1/N][A]$ is a smooth map of schemes.  \item \label{groupsmooth3} The map $\RR^*(B) \to \RR^*(A)$ is a smooth $\DS$-morphism.  It is \'etale when $G$ is finite.  \item  \label{groupsmooth4} The map $\RR_{>0}(B) \to \RR_{>0}(A)$ is a (globally) trivial $\RR_{>0}(G)$-bundle.  In particular it is smooth.  It is a diffeomorphism when $G$ is finite.  The results of this part hold even when the hypothesis that $h$ is injective is replaced with the hypothesis that $h$ has torsion kernel. \end{enumerate} \end{lem}

\begin{proof}  For \eqref{groupsmooth1}, note that $A \to B$ makes $B$ a free $A$-module (of rank equal to $|G|$) (Example~\ref{example:freemodule}).  Since $\ZZ[ \slot ]$ takes free modules to free modules (\S\ref{section:modules}), $\ZZ[B]$ is a free $\ZZ[A]$-module of positive rank, hence $\ZZ[A] \to \ZZ[B]$ is faithfully flat.  Let us take as the meaning of ``torsor" in this algebraic context the statement that the group \be \Hom_{\Sch}(\Spec k, \Spec \ZZ[G]) & = & \Hom_{\Ab}(G,k^*) \ee acts simply transitively on the set of $k$-points of $\Spec \ZZ[B]$ lying over a given $k$-point of $\Spec \ZZ[A]$ for each algebraically closed field $k$.  This is just the statement that $$0 \to A \to B \to G \to 0 $$ remains exact after applying $\Hom_{\Ab}(\slot,k^*)$, which holds since $k^*$ is an injective abelian group (it is divible because $X^n = u$ can be solved for each $u \in k^*$ since $k$ is algebraically closed).  Since $\ZZ[ \slot ]$ preserves direct limits, our discussion above shows that this torsor can be trivialized by pulling it back to itself; since it is a flat map, this torsor is trivial in the flat topology.

For \eqref{groupsmooth2}, let $B'$ be the subgroup of $B$ consisting of those $b \in B$ such that $nb \in A$ for some positive integer $n$.  The smallest such $n$ always divides $N$, hence $\ZZ[1/N][A] \to \ZZ[1/N][B']$ is an \'etale cover because it can be presented by adjoining various $n^{\rm th}$ roots of units $[a] \in \ZZ[1/N][A]^*$ with $n$ invertible in $\ZZ[1/N]$.  The torsor can be trivialized after pulling back to this \'etale cover, since the sequence $$0 \to A \to B \to G \to 0 $$ splits after pushing out along $A \to B'$.  The map $\ZZ[1/N] \to \ZZ[1/N][G]$ is smooth for similar reasons. 

For \eqref{groupsmooth4}, first note that the torsoriality (simple transitivity on fibers) results immediately from the fact that $\RR_{>0}$ is a divisible abelian group, as in the proof of \eqref{groupsmooth1}.  For the other statements, note that the functor $\RR_{>0}$ is insensitive to torsion, so we can assume $A$ and $B$ are free.  We can choose bases of $A \cong \ZZ^s$ and $B \cong \ZZ^{s+r}$ so that the matrix representation of $A \to B$ is in Smith Normal Form---in particular, it is a diagonal matrix with diagonal entries $e_1,\dots,e_s,0,\dots,0$, where the $e_i$ are positive integers and there are $r$ zeros ($r=0$ when $G$ is finite).  Again, since $\RR_{>0}$ is insensitive to torsion, $\RR_{>0}(G) = \RR_{>0}^r$.  The map in question is then given by \be \RR_{>0}^{s+r} & \to & \RR_{>0}^{s} \\ (x_1,\dots,x_{r+s}) & \mapsto & (x_1^{e_1},\dots,x_s^{e_s})  \ee with $\RR_{>0}(G)$ acting by rescaling the last $r$ coordinates.  The map \be f : \RR_{>0}^s \times \RR_{>0}^r & \to & \RR_+^{s+r} \\ (y_1,\dots,y_s,z_1,\dots,z_r) & \mapsto & (y_1^{1/e_1},\dots,y_s^{1/e_s},z_1,\dots,z_r) \ee is then an isomorphism $\RR_{>0}(A) \times \RR_{>0}^r \to \RR_{>0}(B)$ commuting with the maps to $\RR_{>0}(A)$ and respecting the $\RR_{>0}(G)$-action. 

Statement \eqref{groupsmooth3} follows from \eqref{groupsmooth4} using the component decomposition discussed in Example~\ref{example:RGplus}. \end{proof}

\begin{lem} \label{lem:monoidsmooth}  Suppose $h : Q \to P$ is a map of fine monoids such that $\ov{h} : \ov{Q} \to \ov{P}$ is an isomorphism and $\Ker h^{\rm gp}$ (resp.\ $\Ker h^{\rm gp}$ and $\Cok h^{\rm gp}$) is (resp.\ are) torsion.  Then the $\DS$ morphisms ${\RR}_+(h) : \underline{\RR}_+(P) \to \underline{\RR}_+(Q)$ and $\underline{\RR}(h) : \underline{\RR}(P) \to \underline{\RR}(Q)$ are smooth (resp.\ \'etale). \end{lem}

\begin{proof}  Since $\ov{h}$ is an isomorphism, $h :Q  \to P$ is a pushout of $h^* : Q^* \to P^*$ by Lemma~\ref{lem:pushout}, so the maps ${\RR}_+(h)$ and ${\RR}(h)$ are pullbacks of ${\RR}_+(h^*)$ and ${\RR}(h^*)$, hence it suffices to prove that the latter maps are smooth (resp.\ \'etale).  Since $\ov{h}$ is an isomorphism, so is $\ov{h}^{\rm gp}$ and the Snake Lemma applied to $$ \xym{ 0 \ar[r] & Q^* \ar[d]_{h^*} \ar[r] & Q^{\rm gp} \ar[d]^{h^{\rm gp}} \ar[r] & \ov{Q}^{\rm gp} \ar[d]^{\ov{h}^{\rm gp}} \ar[r] & 0 \\ 0 \ar[r] & P^* \ar[r] & P^{\rm gp} \ar[r] & \ov{P}^{\rm gp} \ar[r] & 0 } $$ yields isomorphisms \be \Ker h^* & = & \Ker h^{\rm gp} \\ \Cok h^* & = & \Cok h^{\rm gp}, \ee so the result follows from Lemma~\ref{lem:groupsmooth}. \end{proof}

\subsection{Integration} \label{section:integration} Let $X$ be a log space (\S\ref{section:logspaces}).  We would like to produce a map of log spaces $X^{\rm int} \to X$ which is terminal among all maps from an integral log space to $X$.  That is, we would like to produce a right adjoint to the inclusion \be \IntLogEsp & \to & \LogEsp \ee of the full subcategory of integral log spaces.  If $X$ is coherent with a global chart $P \to \M_X(X)$ (using a finitely generated monoid $P$), then it is easy to see from the modular interpretation of $\AA(P^{\rm int}) \to \AA(P)$ that \be X^{\rm int} & := & X \times_{\AA(P)} \AA(P^{\rm int}) \ee will do the job.  More generally, if $X$ is coherent, then we can produce the desired $X^{\rm int}$ by choosing local charts $P_i \to \M_X(X)$ (using finitely generated $P_i$) then gluing the $U_i^{\rm int}$ together using our gluing axiom \eqref{Zariskigluing} for space (\S\ref{section:axiomsforspaces}) and the fact that the ``preimage" of $U_{ij}$ in $U_i^{\rm int}$ is canonically identified with that in $U_j^{\rm int}$ since both satisfy the universal property of $U_{ij}^{\rm int}$.  This produces a right adjoint to the inclusion \be \FineLogEsp & \to & \CohLogEsp \ee in any category of log spaces.  Since it is a right adjoint, the functor $X \mapsto X^{\rm int}$ preserves inverse limits.  We see immediately that $\FineLogEsp$ has all finite inverse limits:  We first form the inverse limit in $\CohLogEsp$, then we apply $X \mapsto X^{\rm int}$ to obtain the inverse limit in $\FineLogEsp$.  We will revisit this ``integration" construction in \S\ref{section:integrationrevisited}.

There is an analogous \emph{saturation} construction, but we find it best to delay this construction until \S\ref{section:saturation}, at which point it can be obtained from general nonsense.

\section{Log differentiable spaces}  \label{section:lds} In this section we give a specialized study of log differentiable spaces.  Recall (\S\ref{section:definitionoflogspace}) that we made the following:

\begin{defn} \label{defn:LDS} The category $\LDS$ (resp.\ $\PLDS$) of \emph{log differentiable spaces} (resp.\ \emph{positive log differentiable spaces}) is the category of log spaces associated to the category of spaces $(\DS,\RR)$ (resp.\ $(\DS,\RR_+)$) as in \S\ref{section:logspaces}. \end{defn}  

To spell it out, a log differentiable space (resp.\ positive log differentiable space) $X = (\underline{X},\M_X)$ is a differentiable space $\underline{X}$ equipped with a \emph{log structure} $\M_X \to \O_X$ (resp.\ $\M_X \to \O_X^{\geq 0}$).  Recall that a log structure is a map of sheaves of monoids inducing an isomorphism on groups of units.  A morphism $f : X \to Y$ in $\LDS$ (or $\PLDS$) is a map $\u{f} : \u{X} \to \u{Y}$ of differentiable spaces, together with a map $f^\dagger : f^* \M_Y \to \M_X$ of log structures on $\underline{X}$.  We will adhere to the notational conventions introduced in \S\ref{section:definitionoflogspace}.  We also refer the reader to \S\ref{section:logspaces} for basic notions such as inverse limits, strict maps, and integration.  Recall (\S\ref{section:fanstologspaces}) that we have inverse-limit-preserving functors \be \RR : \Mon^{\rm op} & \to & \LDS \\ \RR_+ : \Mon^{\rm op} & \to & \PLDS. \ee  Here $\Mon$ is the category of finitely generated monoids.  The object $\RR(P)$ (resp.\ $\RR_+(P)$) is characterised up to unique isomorphism by the existence of a natural bijection \be \Hom_{\LDS}(X,\RR(P)) & = & \Hom_{\Mon}(P,\M_X(X)) \\  \Hom_{\PLDS}(X,\RR_+(P)) & = & \Hom_{\Mon}(P,\M_X(X)) \ee for each $X \in \LDS$ (resp.\ $X \in \PLDS$).  For more on $\RR(P)$ and $\RR_+(P)$, see \S\ref{section:examples}.

After developing a theory of \emph{log smoothness} in \S\ref{section:logsmoothness} we will see how manifolds with corners fit naturally into the picture (\S\ref{section:manifoldswithcorners}).  We will digress in \S\ref{section:integrationrevisited} and \S\ref{section:boundary} to describe some general ``log space" constructions that we preferred to defer until now for reasons of context.  We also give a treatment of the \emph{Kato-Nakayama space} construction, which, from our point of view, is an important functor between log analytic spaces and our category $\PLDS$.

\subsection{Positive log differentiable spaces}  \label{section:PLDS} Here we make some brief remarks about the relationship between $\PLDS$ and $\LDS$.  For a differentiable space $X$, we will often refer to a (pre)log structure $\M_X \to \O_X^{\geq 0}$ as a \emph{positive (pre)log structure}, to distinguish it from a (pre)log structure $\M_X \to \O_X$ \emph{in the usual sense}.  Given a positive log structure, we can view it as a prelog structure in the usual sense by composing with the inclusion $\O_X^{\geq 0} \into \O_X$.  We can then take the associated log structure---concretely this is given by \be \M_X^a & = & \M_X \oplus_{\O_X^{>0}} \O_X^*. \ee  In fact we have an isomorphism \be \M_X \oplus_{\O_X^{>0}} \O_X^* & \to & \M_X \oplus \u{\ZZ/2 \ZZ} \\ (m,u) & \mapsto & (|u|m, u/|u|), \ee with inverse $(m,\pm 1) \mapsto (m,\pm 1)$.  Here $\u{\ZZ/2 \ZZ}$ denotes the sheaf of locally constant functions to $\ZZ / 2 \ZZ = \{ \pm 1 \}$.  Thus we obtain a functor $\PLDS \to \LDS$, which is easily seen to be faithful in light of the aforementioned isomorphism, though it is not full.

\begin{example}  The automorphism group of the positive log structure \be \NN \oplus \RR_{>0} & \to & \RR_{>0} \\ (n,u) & \mapsto & \left \{ \begin{array}{lll} u, & \quad & n=0 \\ 0, & & n>0 \end{array} \right . \ee is $\RR_{>0}$ via the map sending $r \in \RR_{>0}$ to the isomorphism $(n,u) \mapsto (n,r^n u)$.  The associated log structure is \be \NN \oplus \RR^* & \to & \RR^* \\ (n,u) & \mapsto & \left \{ \begin{array}{lll} u, & & n=0 \\ 0, & & n>0,\end{array} \right . \ee which has automorphism group $\RR^*$. \end{example}

For a differentiable space $X$ and a point $x \in X$, the group $\O_{X,x}^{>0}$ is divisible (injective).\footnote{The local ring of any differentiable space is always a quotient of the ring $A_n$ of germs of smooth functions at the origin of $\RR^n$ (for some $n$), so it suffices to note that the groups $A_n^*$ are divisible.}  Consequently, the characteristic extension \eqref{characteristicextension} of a positive log structure always splits and one easily proves the following variant of Lemma~\ref{lem:fschart}:

\begin{lem} \label{lem:positivechart} Let $\M_X$ be a fine positive log structure on a differentiable space $X$.  Then for any $x \in X$, there exists a neighborhood $U$ of $x$ in $X$ and a characteristic chart $\ov{\M}_{X,x} \to \ov{\M}_X(U)$ for $\M_X|U$. \end{lem}

\subsection{Examples} \label{section:examples}  This section contains some basic examples of log differentiable spaces.  

\begin{example} \label{example:RPlog}  We should emphasize that $\RR(\NN^n)$ (resp.\ $\RR_+(\NN^n)$) is just $X := \RR^n$ (resp.\ the positive orthant $X := \RR_+^n$) equipped with its usual differentiable structure.  The log structure on $X$ is the one associated to the chart $\NN^n \to \O_X(X)$ taking the standard basis element $e_i$ to the $i^{\rm th}$ coordinate function $x_i$.  Whenever we view $\RR^n$ (resp.\ $\RR_+^n$) as a log differentiable space (resp.\ positive log differentiable space) it is understood to have this \emph{usual log structure}, unless stated otherwise.  \end{example}

\begin{example} We will typically be interested in $\RR(P)$ and $\RR_+(P)$ when $P$ is fs, but let us note the following:  If $P$ is an arbitrary fine monoid and $P^{\sat}$ is its saturation, then $P \to P^{\rm sat}$ is an injective finite map of fine monoids (Theorem~\ref{thm:dense}), hence the induced map of differentiable spaces $\RR_+(P^{\sat}) \to \RR_+(P)$ is a homeomorphism on underlying topological spaces (Lemma~\ref{lem:saturationhomeo}).  It is not generally an isomorphism of differentiable spaces, however.  \end{example}

\begin{example} If $P$ is the submonoid of $\NN$ generated by $2$ and $3$, then $P^{\sat}= \NN$, and we have an isomorphism of topological spaces $\RR_+ = \RR_+(P)$.  The structure sheaf of $\RR_+(P)$ is the subsheaf of the structure sheaf of $\RR_+$ consisting of those $f$ with $f'(0)=0$. \end{example}

\begin{example}  \label{example:teardrop}  Let $x,y$ be the usual coordinates on $\RR_{+}^2 = \{ (x,y) \in \RR^2 : x,y \geq 0 \}$ and let $u,v$ be the usual coordinates on $\RR \times \RR_+ = \{ (u,v) \in \RR^2 : v \geq 0 \}$.   Consider the \emph{continuous} maps \be f : \RR_+^2 & \to & \RR \times \RR_+ \\ (x,y) & \mapsto & \left ( \frac{x^2-y^2}{\sqrt{x^2+y^2}} , \frac{ 2xy }{\sqrt{x^2+y^2}} \right ) \ee (this extends continuously to $(0,0)$ by setting $f(0,0)=(0,0)$) and \be g : \RR \times \RR_+ & \to &  \RR_+^2 \\ (u,v) & \mapsto & \left ( \sqrt{ \frac{u^2+v^2+u\sqrt{u^2+v^2} }{2} } , \sqrt{ \frac{u^2+v^2-u\sqrt{u^2+v^2} }{2} } \right ). \ee  One checks easily that $fg= \Id$ and $gf = \Id$, so $f$ and $g$ provide a \emph{homeomorphism} $\RR_{+}^2 \cong \RR \times \RR_+$, which is not really the point of this example.  The point is that $f$ and $g$ are smooth away from $(0,0)$, so they provide a \emph{diffeomorphism} \bne{gluing} \RR_+^2 \setminus \{ (0,0) \} & \cong & (\RR \times \RR_+) \setminus \{ (0,0) \}. \ene  

In fact, we claim that \eqref{gluing} is an isomorphism of log differentiable spaces when $\RR_+^2$ is given the usual log structure as in Example~\ref{example:RPlog}) and $\RR \times \RR_+$ is given $\pi_2^*$ of the usual log structure on $\RR_+$.  That is, $\RR_+^2$ is given the log structure associated to the prelog structure \be \NN^2 & \to & \C^\infty(\RR_+^2) \\ (1,0) & \mapsto & x \\ (0,1) & \mapsto & y, \ee and $\RR \times \RR_+$ is given the log structure associated to the map $\NN \to \C^\infty(\RR \times \RR_+)$ given by $1 \mapsto v$.  To say that $f$ is a morphism of log differentiable spaces away from $(0,0)$ then amounts to the claim that $f^*v = 2xy(x^2+y^2)^{-1/2}$ is in the subsheaf of monoids of $\C^\infty(\RR_+^2 \setminus (0,0))$ generated by $x,y$ and the invertible functions.  Indeed, this is clear from the formula for $f^* v$ because $2(x^2+y^2)^{-1/2}$ is an invertible smooth function away from $(0,0)$.  Similarly, to say that $g$ is a morphism of log differentiable spaces away from $(0,0)$ is equivalent to saying that $g^*x$ and $g^*y$ are in the submonoid of $\C^\infty(\RR \times \RR_+ \setminus (0,0))$ generated by $v$ and the units (away from $(0,0)$).  On the open subspace $\{ v > 0 \}$, this is obvious because $g^*x$ and $g^* y$ are invertible smooth functions.  On the open subspace $\{ u > 0 \}$, for example, the function $g^* x$ is again an invertible smooth function.  The key point is to prove that, for example, $v^{-1} g^*y$ extends as a smooth function to $\{ v=0 \}$ on the region $\{ u >0 \}$, which is a basic exercise with L'Hospital's Rule.

The log differentiable space \be X & := & \RR_+^2 \coprod_{ \RR_+^2 \setminus ( 0,0) \cong \RR \times \RR_+ \setminus (0,0) } \RR \times \RR_+ \ee obtained by gluing $\RR_+^2$ to $\RR \times \RR_+$ along the complement of the origin via \eqref{gluing} is called the \emph{teardrop}.  Let $x_0 \in X$ denote the origin in the chart $\RR_+^2 \subseteq X$.  The idempotent automorphism $\phi(x,y) := (y,x)$ of $\RR_+^2$ extends to an idempotent automorphism $\phi : X \to X$ given by $(u,v) \mapsto (-u,v)$ on the other chart $\RR \times \RR_+$.  This $\phi$ is an automorphism of log differentiable spaces, so its mapping torus $T$ has a natural log differentiable space structure.  The log differentiable space $T$ is called the \emph{teared Klein body}.  The mapping torus $T$ comes with a fiber bundle structure $\pi : T \to S^1$ with fiber $X$ and monodromy $\phi$.  Since $\phi$ fixes $x_0$, $\pi$ comes with a section $s : S^1 \to T$ given by $x_0 \in X$ in each fiber.  As a topological space, $T$ is a 3-manifold with boundary given by the Klein bottle.  The boundary $\boundary T$ of $T$ (in the ``manifolds with corners" sense of \S\ref{section:boundary}), however, will be an interval bundle $\boundary T \to S^1$ with monodromy given by ``flipping the interval" ($t \mapsto 1-t$).  The map $\boundary : \boundary T \to T$ is 2-to-1 over $s$.  The double boundary $\boundary : \boundary^2 T \to \boundary T$ is the connected double cover of the circle.      \end{example}

\subsection{Integration revisited} \label{section:integrationrevisited}  Recall that in \S\ref{section:integration} we constructed a right adjoint $X \mapsto X^{\rm int}$ to the inclusion of fine log spaces into integral log spaces, for \emph{any} category of spaces.

For certain categories of spaces we can do even better.  Suppose $(X,\O_X)$ is a locally ringed space and $\alpha_X : \M_X \to \O_X$ is a prelog structure on $X$.  Let $I(\M_X) \subseteq \O_X$ denote the ideal generated by local sections of the form $\alpha_X(m)-\alpha_X(n)$ where $m,n$ are local sections of $\M_X$ with the same image in $\M_X^{\rm int}$.  We would like to construct $X^{\rm int}$ as follows:  Let $i : X^{\rm int}(\M_X) \into X$ be the zero locus of $I(\M_X)$, so that $X^{\rm int} := X^{\rm int}(\M_X)$ and we have a natural map $i^{-1} \M_X^{\rm int} \to \O_{X^{\rm int}}$ and, in fact, a natural map \bne{intmap} X^{\rm int} = (X^{\rm int}(\M_X) ,\M_X^{\rm int}) & \to & (X,\M_X) \ene of \emph{pre}log locally ringed spaces.  By taking associated log structures everywhere we may view \eqref{intmap} as a map of log locally ringed spaces.  In particular, this construction defines a functor $X \mapsto X^{\rm int}$ from log locally ringed spaces to integral log locally ringed spaces, which is easily seen to be right adjoint to the inclusion the other way.  By explicitly examining the pushout construction of the associated log structure $\M_X^a$, one sees that $I(\M_X^a) = I(\M_X)$, hence $X^{\rm int}(\M_X) = X^{\rm int}(\M_X^a)$ as locally ringed spaces.  Similarly, it is clear from the pushout construction of $\M_X^a$ that $(\M_X^{\rm int})^a = (\M_X^a)^{\rm int}$ as log structures on this common space.  In particular, if $a : P \to \M_X(X)$ is a finitely generated chart for a log structure $\M_X$ on a locally ringed space $X$, then $I(\M_X)  =  I(P)$ is the (Taylor closure of the) ideal generated by the  global sections $\alpha_X(ap) - \alpha_X(aq)$ where $p,q \in P$ have the same image in $P^{\rm int}$ and $P^{\rm int} \to \M_X^{\rm int}(X^{\rm int})$ is a chart.  The ideal $I(P)$ here coincides with the inverse image of the ideal of the closed embedding $\underline{\RR}(P^{\rm int}) \into \underline{\RR}(P)$, hence we see that \bne{Xint} X^{\rm int} & = & X \times_{\AA(P)} \AA(P^{\rm int}) \ene as log locally ringed spaces.  We also see, in a new way, that the functor $X \mapsto X^{\rm int}$ takes coherent log locally ringed spaces to fine log locally ringed spaces.

One would like to make the above construction in other categories of spaces, but the trouble is that the zero locus $X^{\rm int}(\M_X)$ might not be a ``space."  For example, if $X = (X,\O_X)$ is a scheme, the ideal $I(\M_X) \subseteq \O_X$ defined above will not be quasi-coherent for an arbitrary log structure.  There are similar issues with analytic spaces.  But with differentiable spaces, there is no problem at all:  We need only make sure that we take the \emph{closure} of the ideal $I(\M_X)$, to ensure that its zero locus is a differentiable space (\S\ref{section:differentiablespaces}).  We thus see that both functors \be \IntLDS & \into & \LDS \\ \IntPLDS & \into & \PLDS \ee admit right adjoints.

\subsection{Boundary} \label{section:boundary}  A manifold with corners $X$ has a kind of \emph{boundary} which is ``more refined" than its usual topological boundary (c.f.\ \cite{Joyce}).  In the model case $X = \RR_+^n$, the boundary of $X$, denoted $\boundary X$, is the disjoint union of the $n$ coordinate hyperplanes of $X$, regarded as a manifold with corners as usual.  There is a natural map of differentiable spaces $\u{\boundary} X \to \u{X}$, but this is \emph{not} a morphism of manifolds with corners.  At least, it is not a morphism of manifolds with corners in \emph{our} setup.  Kottke and Melrose \cite{KM} do consider this a ``$b$-map," but not an ``interior $b$-map."

In this section, we give a general construction of ``boundary" for various kinds of log spaces.  This is \emph{not} a completely general construction defined for all log spaces (it won't make sense for $\Fans$, for example), though it will make sense in all other categories of log spaces mentioned in \S\ref{section:spaces}.  

Let us give the algebraic construction first.  Suppose $P$ is a monoid, $F \subseteq P$ is a face.  Then $I := P \setminus F$ is a prime ideal of $P$ and $\ZZ[I]$ is an ideal of $\ZZ[P]$.  It is not generally a prime ideal, even when $\ZZ$ is replaced by a field $k$ (though it \emph{will} be prime in that case if $P$ is integral and $P^{\rm gp}$ is torsion free).  For more details on these sort of ideals, see \cite[7.7]{logflatness} (we do not need any of those results for the present constructions).  The quotient ring $\ZZ[P]/\ZZ[I]$ is identified with $\ZZ[F]$ by the obvious map, so that the quotient projection $\ZZ[P] \to \ZZ[F]$ retracts the natural map $\ZZ[F] \to \ZZ[P]$.  This quotient projection is a sort of ``\emph{wrong-way map}."  We define the \emph{boundary} of the log scheme $X := \Spec \ZZ[P]$ to be \be \boundary X & := & \coprod_{F} \Spec \ZZ[F], \ee where the disjoint union runs over maximal proper faces of $P$.  The boundary $\boundary X$ is also a log scheme and there is a natural map of schemes $\u{\boundary} X \to \u{X}$ given by the disjoint union of the closed embeddings corresponding to the aforementioned quotient projections.  The map $\u{\boundary} X \to \u{X}$ does not (generally) lift to a map of \emph{log} schemes because, for one thing, the inclusion $F \subseteq P$ of monoids does not generally have a retract.  We can of course repeat this discussion replacing $\ZZ$ with another base ring, such as $\RR$ or $\CC$.

We can describe $\Delta X$ in a more intrinsic way, on the level of \emph{pre}log locally ringed spaces.  Suppose $(X,\M_X \to \O_X)$ is a prelog locally ringed space.  As a set, $\Delta X$ is the set of pairs $(x,F)$ consisting of a point $x$ of $X$ and a maximal proper face $F \subseteq \ov{\M}_{X,x}$.  (The sharpening map $\M_{X,x} \to \ov{\M}_{X,x}$ induces an order-preserving bijection on faces by Lemma~\ref{lem:sharpeninghomeo}, so this choice of $F$ is the same thing as the choice of a maximal proper face of $\M_{X,x}$.)  For an open subset $U \subseteq X$ and $m \in \ov{\M}_X(U)$, let \be U_m & := & \{ (x,F) \in \boundary X : x \in U, \, m_x \in F \}. \ee  The formula \be U_m \cap V_n & = & (U \cap V)_{m+n} \ee implies that the sets $U_m$ form a basis for a topology on $\boundary X$.  The map \bne{boundary} \boundary : \boundary X & \to & X \\ \nonumber (x,F) & \mapsto & x \ene is continuous because $\boundary^{-1}(U) = U_0$.  The sheaf of monoids $\boundary^{-1} \ov{\M}_X$ on $\boundary X$ comes with a tautological subsheaf of monoids $\ov{\M}_{\boundary X}$ given by \be \ov{\M}_{\boundary X}(V) & = & \{ m \in (\boundary^{-1} \ov{\M}_X)(V) : m_x \in F {\rm \; for \; all \; } (x,F) \in V \} . \ee  Let $\M_{\boundary X} \subseteq \boundary^{-1} \M_X$ be the subsheaf of monoids consisting of sections of $\boundary^{-1} \M_X$ whose image in $\Delta^{-1} \ov{\M}_X$ lies in $\ov{\M}_{\Delta X}$.  Let $I \subseteq \boundary^{-1}\O_X$ be the ideal generated by the image of \be \boundary^{-1} \alpha_X : \boundary^{-1} \M_X \setminus \M_{\boundary X} & \to & \boundary^{-1} \O_X \ee and let $\O_{\boundary X} := \boundary^{-1} \O_X / I$ so that the quotient map $\boundary^{-1} \O_X \to \O_{\Delta X}$ lifts the map \eqref{boundary} to a map of locally ringed spaces $\boundary : \boundary X \to X$.  The composition $$ \M_{\boundary X} \into \boundary^{-1} \M_X \to \boundary^{-1} \O_X \to \O_{\Delta X} $$ determines a prelog structure on $\boundary X$.  We view $\boundary X$ as a log locally ringed space by giving it the log structure associated to $\M_{\boundary X}$.

We leave it to the reader to check that: \begin{enumerate}[label=(B\arabic*), ref=B\arabic*] \item The log locally ringed space $\boundary X$, and the map of locally ringed spaces $\boundary : \u{\boundary} X \to \u{X}$ depend only on the log structure $\M_X^a$ associated to the prelog structure $\M_X$. \item If we specialize the above general construction to $X = \Spec \ZZ[P]$ we recover the earlier notion of boundary. \item If the log structure on $X$ is coherent (or fine, fs, or any other property of monoids inherited by faces), then so is the log structure on $\boundary X$. \end{enumerate}  Assuming the log structure on $X$ is coherent, we also have: \begin{enumerate}[label=(B\arabic*), ref=B\arabic*] \setcounter{enumi}{3} \item \label{Bface} For each $x \in \boundary X$, $\ov{\M}_{\boundary X,x}$ is a proper face of $\ov{\M}_{X,\boundary(x)}$.  \item \label{Bfinite} Locally on $X$, the boundary map $\boundary : \u{\boundary} X \to \u{X}$ is a finite disjoint union of closed embeddings.  In particular $\boundary$ is \emph{finite} (proper with finite fibers).  \item \label{Bbasechange} Formation of the boundary commutes with \emph{strict} base change of log locally ringed spaces.  That is, if $X$ is a coherent log locally ringed space and $f : X' \to X$ is a strict map of log locally ringed spaces, then we have a cartesian diagram \be & \xym{ \boundary X' \ar[r] \ar[d] & \boundary X \ar[d] \\ \u{X}' \ar[r] & \u{X} } \ee of log locally ringed spaces.  \end{enumerate} For \eqref{Bbasechange} it helps to use \eqref{Bfinite} to know that $\boundary : \u{\boundary} X \to \u{X}$ is close enough to being a closed embedding that any base change of it, calculated in $\LRS$, is the same as the one calculated in $\RS$.

\begin{rem} \label{rem:boundary} Boundary is \emph{not} functorial for general maps of log locally ringed spaces. \end{rem}

Having defined the boundary of an arbitrary log locally ringed space, we can specialize our construction to define the boundary of other kinds of log spaces.  For log differentiable spaces or positive log differentiable spaces, we need only make sure that we take the closure of the ideal $I$ defined in the general construction.  The usual issues arise for schemes: for an \emph{arbitrary} log structure on a scheme $X$, we do not know that the ideal $I$ from the general construction is quasi-coherent, so we do not know in general that the log locally ringed space $\boundary X$ is actually a log scheme.  However, as long as the log structure is coherent, we can use the formal properties above to see that $\boundary X$ is a coherent log scheme.

The following set-theoretic description of the boundary map might be helpful.  Let $F$ be a face of a (finitely generated) monoid $P$.  Let $\AA(P)$ denote one of the usual log spaces we attach to $P$ ($\Spec \ZZ[P]$, $\Spec \RR[P]$, $\RR(P)$, $\RR_{\geq 0}(P)$, $\CC(P)$, etc.)  The points of $\AA(P)$ are described in terms of monoid homomorphisms $x : P \to k$, where $k$ is a field (in the algebraic cases we run over \emph{all} fields and impose an equivalence relation, for $\AA(P) = \CC(P)$, we use $k = \CC$, for $\RR_+(P)$ we use $k = \RR_+$, which is not really a field, but the discussion will make sense, \dots).  Now, if $x : F \to k$ is a point of $\AA(F)$, then the fact that $F$ is a face implies that \be \ov{x} : P & \to & k \\ p & \mapsto & \left \{ \begin{array}{lll} x(p), & \quad & p \in F \\ 0, & & p \notin F \end{array} \right . \ee is a monoid homomorphism extending $x$.  On the level of sets, the boundary map $\u{\boundary} \AA(P) \to \u{\AA}(P)$ is the disjoint union, over maximal proper faces $F \subseteq P$, of the maps $x \mapsto \ov{x}$ described above.

\begin{example} \label{example:boundaryofRn}  When $P = \NN^n$, the maximal proper faces of $P$ are the $n$ ``coordinate hyperplanes" $\NN^{n-1} \into \NN^n$.  The boundary map $\boundary : \u{\boundary} \AA(P) \to \u{\AA}(P)$ is the disjoint union  \be \Delta : \coprod_{i=1}^n \underline{\AA}^{n-1} & \to & \underline{\AA}^n \ee of the inclusions of these hyperplanes. \end{example}

\begin{example} If $X$ is a manifold with corners (free log smooth positive log differentiable space), then it follows from Example~\ref{example:boundaryofRn} that $\underline{ \boundary} X \to \underline{X}$ is a finite-to-one proper map whose image is the topological boundary of $\underline{X}$. \end{example}

\begin{lem} \label{lem:boundary}  Suppose $X$ is a coherent log locally ringed space.  Then, at least locally on $X$, there is a positive integer $n$ such that the iterated boundary $\Delta^n X$ is empty. \end{lem}

\begin{proof} Since the quesiton is local we can assume $X$ has a global chart using a finitely generated monoid $P$.  Then the characteristic monoid $\ov{\M}_{X,x}$ at any point of $X$ is some quotient of $P$, so there is a global bound on the minimal number of generators of $\ov{\M}_{X,x}$.  By Lemma~\ref{lem:faces}, the minimal number of generators must decrease on passing to a proper face, so the result follows from property \eqref{Bface} of the boundary. \end{proof}

The following ``combinatorial condition" on a log locally ringed space often arises in the literature in some form:

\begin{defn} \label{defn:tame} A log locally ringed space $X$ is called \emph{tame} iff $\underline{\boundary} : \underline{ \boundary X} \to \underline{X}$ is \emph{globally} a finite disjoint union of closed embeddings. \end{defn}

\subsection{Log smoothness} \label{section:logsmoothness}  In this section we introduce \emph{log smooth} maps and establish some of their basic properties.  Our definition is modelled on Kato's \emph{Chart Criterion for Log Smoothness} in log algebraic geometry \cite{Kat1}---see \S\ref{section:scholium} for further discussion.  While it would be possible to develop a general theory of log smoothness for log spaces (\S\ref{section:logspaces}) by starting with an appropriate class of maps of spaces deemed ``smooth," we will restrict ourselves in this section to the cases of log differentiable spaces and positive log differentiable spaces.

\begin{defn} \label{defn:smoothchart} Let $f : X \to Y$ be a map of fine $\LDS$ (resp.\ $\PLDS$).  A fine chart $$\xym{ P \ar[r]^-a & \M_X(X) \\ Q \ar[u]^-h \ar[r]^-b & \M_Y(Y) \ar[u]_{f^\dagger} }$$ for $f$ is called \emph{smooth} iff $h$ is monic and the $\DS$ morphism $\underline{X} \to \underline{Y} \times_{\underline{\RR}(Q)} \underline{\RR}(P)$ (resp.\ $\underline{X} \to \underline{Y} \times_{\underline{\RR}_+(Q)} \underline{\RR}_+(P)$) is smooth. \end{defn}

If $x \in X$ and $(U,V)$ is a neighborhood of $x$ in $f$, then the restriction (\S\ref{section:chartsandcoherence}) of a smooth chart for $f$ to $(U,V)$ is again a smooth chart.  The Shrinking Argument of \S\ref{section:chartsandcoherence} works equally well for smooth charts, as follows.  Suppose one has a smooth chart for a map $f$ as in Definition~\ref{defn:smoothchart}.  Let $F := a_x^{-1} \O_{X,x}^*$ and $G := b_y^{-1} \O_{Y,y}^*$ be the faces of $a$ and $b$ at $x \in X$ and $y := f(x)$.  Then after replacing $(X,Y)$ with a neighborhood of $x$ in $f$, we can assume our smooth chart factors as $$\xym{ P \ar[r] & F^{-1}P \ar[r]^-a & \M_X(X) \\ Q \ar[r] \ar[u]^-h & G^{-1}Q \ar[u]^-h \ar[r]^-b & \M_Y(Y) \ar[u]_{f^\dagger} }$$ and it is easy to see that the right square above is also a smooth chart (use the fact that $\underline{\RR}(F^{-1}P) \to \underline{\RR}(P)$ is an open embedding and similarly with $\underline{\RR}$ replaced by $\underline{\RR}_+$ and/or $P,F$ replaced by $Q,G$).

\begin{thm} \label{thm:logsmooth} For a map $f : X \to Y$ of fine $\LDS$ or $\PLDS$ and a point $x \in X$, the following are equivalent: \begin{enumerate} \item \label{anychart} For any neighborhood $W$ of $f(x)$ in $Y$ and any fine chart $b : Q \to \M_Y(W)$ for $\M_Y|W$ there is a neighborhood $(U,V)$ of $x$ in $f|f^{-1}(W) : f^{-1}(W) \to W$ and a smooth chart for $f : U \to V$ extending the restriction of $b$ to $V$. \item \label{somechart} There is a neighborhood $(U,V)$ of $x$ in $f$ and a smooth chart for $f : U \to V$. \end{enumerate} \end{thm}

The proof of Theorem~\ref{thm:logsmooth} is rather technical and lengthy---it will be relegated to \S\ref{section:proof}.

\begin{defn} \label{defn:logsmooth} A map $f : X \to Y$ of fine $\LDS$ or $\PLDS$ is called \emph{log smooth at} $x \in X$ iff it satisfies the equivalent conditions of Theorem~\ref{thm:logsmooth}.  The map $f$ is called \emph{log smooth} iff it is log smooth at $x$ for every $x \in X$.  A fine $\LDS$ or $\PLDS$ $X$ is called \emph{log smooth} iff it is log smooth over $\Spec \RR$ (a point with trivial log structure).  \end{defn}
 
\begin{example} \label{example:RPlogsmooth} If $h : P \into Q$ is an injective map of fine monoids, then the $\LDS$ morphism $\RR(h) : \RR(Q) \to \RR(P)$ (c.f.\ Example~\ref{example:RPlog}) is log smooth because we can use $h$ as a chart for $\RR(h)$ to check \eqref{somechart} in Theorem~\ref{thm:logsmooth}.  In particular, $\RR(P)$ is log smooth.  Similarly, $\RR_+(Q) \to \RR_+(P)$ is a log smooth map of $\PLDS$ and $\RR_+(P)$ is a log smooth $\PLDS$. \end{example}
 
We will spend the rest of this section working out the basic properties of log smooth maps.

\begin{prop} \label{prop:logsmooth} Let $f : X \to Y$ be a map of fine $\LDS$ (resp.\ $\PLDS$).  Suppose $Y$ has the trivial log structure.  Then $X$ is log smooth at $x \in X$ iff there is a neighborhood $(U,V)$ of $x$ in $f$ and a fine chart $P \to \M_X(U)$ for $\M_X|U$ such that the corresponding $\DS$ morphism $\underline{U}  \to  \underline{V} \times \underline{\RR}(P)$ (resp.\ $\underline{U}  \to  \underline{V} \times \underline{\RR}_+(P)$) is smooth.  \end{prop}

\begin{proof}  Unravel the statement of Theorem~\ref{thm:logsmooth} with $W=Y$, $Q=0$. \end{proof}

\begin{thm} \label{thm:strictlogsmooth} A strict morphism $f : X \to Y$ of fine $\LDS$ or $\PLDS$ is log smooth iff $\underline{f} : \underline{X} \to \underline{Y}$ is a smooth $\DS$ morphism. \end{thm}

\begin{proof} We will give the proof for fine $\LDS$.  The argument for $\PLDS$ is identical.  Suppose $f : X \to Y$ is strict and $\underline{f}$ is smooth.  We want to show that $f$ is log smooth.  Fix $x \in X$ and suppose $V$ is a neighborhood of $f(x)$ in $Y$ and $b : Q \to \M_Y(V)$ is a fine chart for $\M_Y|V$.  Set $U := f^{-1}(V)$.  Since $f$ is strict, the composition of $b$ and $f^\dagger : \M_Y(V) \to \M_X(U)$ is a fine chart for $\M_X|U$.  The natural map $\underline{X} \to \underline{Y} \times_{ \underline{\RR}(Q) } \underline{\RR}(Q)$ is just the map $\underline{f}$, thus we see that $f$ is log smooth.  

Conversely, suppose $f$ is strict and log smooth.  Fix any $x \in X$.  We want to show that $\underline{f}$ is smooth at $x$.  This is local, so we can freely replace $(X,Y)$ with any neighborhood of $x$ in $f$.  Hence, by definition of log smooth and the Shrinking Argument, we can assume we have a smooth chart $$ \xym{ P \ar[r]^-a & \M_X(X) \\ Q \ar[u]^h \ar[r]^-b & \M_Y(X) \ar[u]_{f^\dagger} } $$ where $P = F^{-1}P$ and $Q = G^{-1}Q$ (so $F$ and $G$ are the groups of units in $P$, $Q$ respectively), where $F = a_x^{-1} \O_{X,x}^*$, $G = b_y^{-1} \O_{Y,y}^*$, $y = f(x)$ as usual.  Since $f$ is strict, $\ov{f}^\dagger_x : \ov{\M}_{Y,y} \to \ov{\M}_{X,x}$ is an isomorphism, hence $P/F \to Q/G$ is also an isomorphism because $P/F \to \ov{\M}_{X,x}$ and $Q/G \to \ov{\M}_{Y,y}$ are isomorphisms since $a$ and $b$ are charts, hence $Q \to P$ is the pushout of $G=Q^* \to F=P^*$, hence $\underline{\RR}(P) \to \underline{\RR}(Q)$ is smooth since it is a pullback of $\underline{\RR}(F) \to \underline{\RR}(G)$, which is smooth by Lemma~\ref{lem:groupsmooth}.  But then $\underline{f}$ is smooth because it is a composition of $\underline{X} \to \underline{Y} \times_{\underline{\RR}(Q)} \underline{\RR}(P)$, which is smooth since the chart above is smooth, and the projection $\underline{Y} \times_{\underline{\RR}(Q)} \underline{\RR}(P) \to \underline{Y}$, which is smooth because it is a base change of $\underline{\RR}(P) \to \underline{\RR}(Q)$. \end{proof}

\begin{thm} \label{thm:logsmoothness}  Log smoothness is a local property (\S\ref{section:localproperties}) of $\LDS$ and $\PLDS$ morphisms.  Log smooth morphisms are stable under composition and base change in the category of fine $\LDS$ / $\PLDS$ (\S\ref{section:integration}). \end{thm}

\begin{proof} It is clear from criterion \eqref{somechart} in Theorem~\ref{thm:logsmooth} that log smoothness is local.  The proofs of the other statements are the same in $\LDS$ and $\PLDS$ so we will just give the $\LDS$ proofs.

For stability under composition, suppose $f : X \to Y$ and $g : Y \to Z$ are log smooth and $x \in X$.  Since $g$ is log smooth at $f(x)$, we can find a neighborhood $(V,W)$ of $f(x)$ in $g$ and a smooth chart \bne{firstsmoothchart} & \xym{ P \ar[r]^-a & \M_Y(V) \\ Q \ar[u]^h \ar[r] & \M_Z(W) \ar[u]_{g^\dagger} } \ene for $g : V \to W$.  Since $f$ is log smooth at $x$ we can find a neighborhood $(U,V)$ of $x$ in $f$ and a smooth chart \bne{secondsmoothchart} & \xym{ R \ar[r] & \M_X(U) \\ P \ar[u]^k \ar[r]^-a & \M_Y(V) \ar[u]_{f^\dagger} } \ene for $f : U \to V$ (we can assume this $V$ is the same as the old $V$ after shrinking the old $V$ if necessary).  We then have a cartesian $\DS$ diagram \bne{bigcartesiandiagram} & \xym{ \underline{U} \ar[r] & U'' \ar[d] \ar[r] & U' \ar[r] \ar[d] & \underline{\RR}(R) \ar[d] \\ & \underline{V} \ar[r] & V' \ar[r] \ar[d] & \underline{\RR}(P) \ar[d] \\ & & \underline{W} \ar[r] & \underline{\RR}(Q) } \ene where the cartesian squares define $U''$, $U'$, and $V'$.  The map $\underline{V} \to V'$ is smooth since \eqref{firstsmoothchart} is smooth, hence $U'' \to U'$ is smooth by stability of smooth $\DS$ morphisms under base change.  The map $\underline{U} \to U''$ is smooth since \eqref{secondsmoothchart} is a smooth, hence the composition $\underline{U} \to U'$ is also smooth, hence $$ \xym{ R \ar[r] & \M_X(U) \\ Q \ar[r] \ar[u]^{kh} & \M_Z(W) \ar[u]_{(gf)^\dagger} } $$ is a smooth chart for $gf : U \to W$ so $gf$ is smooth at $x$. 

For stability under base change, suppose $$ \xym{ X \ar[r]^{p_2} \ar[d]_{p_1} & X_2 \ar[d]^{f_2} \\ X_1 \ar[r]^{f_1} & Y } $$ is a cartesian diagram in the category of fine $\LDS$ and $f_2$ is log smooth.  Fix a point $x \in X$ and set $x_i := p_i(x)$, $y := f_1(x_1) = f_2(x_2)$.  We want to prove $p_1$ is log smooth at $x$.  Since $X_1$ and $Y$ are fine and $f_2$ is log smooth at $x_2$, we can find (using Lemma~\ref{lem:secondchart}) a neighborhood $(U_1,V)$ of $x_1$ in $f_1$, a neighborhood $(U_2,V)$ of $x_2$ in $f_2$, and fine charts $$ \xym{ P_1 \ar[r] & \M_{X_1}(U_1) \\ Q \ar[r]^-b \ar[u]^{h_1} & \M_Y(V) \ar[u]_{f_1^\dagger} } \quad \quad {\rm and} \quad \quad \xym{ P_2 \ar[r] & \M_{X_2}(U_2) \\ Q \ar[r]^-b \ar[u]^{h_2} & \M_Y(V) \ar[u]_{f_2^\dagger} } $$ for $f_1 : U_1 \to V$ and $f_2 : U_2 \to V$ so that the chart on the right is smooth.  Set $P' := P_1 \oplus_Q P_2$.  If we let $U'$ denote the fiber product $U_1 \times_V U_2$ taken in $\LDS$ (not in $\fLDS$) then $$ \xym{ P' \ar[r] & \M_{U'}(U') \\ P_1 \ar[u] \ar[r] & \M_{X_1}(U_1) \ar[u]_{\pi_1^\dagger} } $$ is a finitely generated (but not necessarily fine) chart for $\pi_1 : U' \to U_1$ and the $\DS$ morphism $$g : \underline{U}' \to \underline{U}_1 \times_{\underline{\RR}(P_1)} \underline{\RR}(P') = \underline{U}_1 \times_{\underline{\RR}(Q)} \underline{\RR}(P_2) $$ is smooth because it is a base change of $\underline{U}_2 \to \underline{V} \times_{\underline{\RR}(Q)} \underline{\RR}(P_2). $  Let $P := (P')^{\rm int}$.  Then $U := (U')^{\rm int} \cong U' \times_{\RR(P')} \RR(P)$  (c.f.\ \eqref{Xint} in \S\ref{section:integration}) is a neighborhood of $x$ in $X$ and $$ \xym{ P \ar[r] & \M_X(U) \\ P_1 \ar[u] \ar[r] & \M_{X_1}(U_1) \ar[u]_{p_1^\dagger} } $$  is a smooth chart for $p_2 : U \to U_1$ because $\underline{U} \to \underline{U}_1 \times_{\underline{\RR}(P_1)} \underline{\RR}(P)$ is nothing but the base change $g \times_{\underline{\RR}(P')} \underline{\RR}(P)$ of $g$.  \end{proof}

\begin{lem} \label{lem:logsmoothcharchart} Let $X$ be a log smooth differentiable space, $x$ a point of $X$ such that $P = \ov{\M}_{X,x}^{\rm gp}$ is torsion free (every point of $X$ has this property if $X$ is fs).  Then there is a neighborhood $U$ of $x$ and a characteristic chart $P \to \M_X(U)$ at $x$ (Definition~\ref{defn:chart}) such that the corresponding $\DS$ morphism $\underline{U} \to \underline{\RR}(P)$ is smooth. \end{lem}

\begin{proof}  By Proposition~\ref{prop:logsmooth} we can find a neighborhood $V$ of $x$ and a fine chart $h : Q \to \M_X(V)$ so that the corresponding $\DS$ morphism $\underline{h} : \underline{V} \to \underline{\RR}(Q)$ is smooth.  Set $G := h_x^{-1} \O_{X,x}$.  Since $h$ is a chart, the map $Q/G \to P = \ov{\M}_{X,x}$ is an isomorphism, hence $Q^{\rm gp} / G^{\rm gp} \to P^{\rm gp}$ is an isomorphism.  By Lemma~\ref{lem:splitting} and the hypothesis on $P^{\rm gp}$ we in fact have a splitting of monoids $G^{-1} Q = G^{\rm gp} \oplus P$.  Since $Q \to G^{-1} Q$ is a localization, $\RR(G^{-1} Q) \subseteq \RR(Q)$ is an open subspace, whose preimage $U$ under $h : V \to \RR(Q)$ still contains $x$.  The restriction $\underline{h}|\underline{U} : \underline{U} \to \underline{\RR}(G^{-1} Q)$ is still a smooth $\DS$ morphism because it is a base change of $\underline{h}$.  As discussed in Example~\ref{example:RG}, the differentiable space $\underline{\RR}(G^{\rm gp})$ is a smooth manifold in the usual sense (it is not merely \emph{log} smooth), so the projection $\pi : \RR(G^{-1} Q ) = \RR(G^{\rm gp}) \times \RR(P) \to \RR(P)$ is not only log smooth, but $\underline{\pi}$ is a smooth $\DS$ morphism.  The composition $k : U \to \RR(P)$ of $h|U$ and $\pi$ also has $\underline{k}$ smooth.  The map $k$ corresponds to a monoid homomorphism $k : P \to \M_X(U)$ whose composition with $\M_X(U) \to \ov{\M}_{X,x}=P$ is the identity.  Consequently, we can assume (Lemma~\ref{lem:fschart}), after possibly shrinking $U$ to a smaller neighborhood of $x$ (which won't destroy smoothness), that $k$ is a chart as desired. \end{proof} 

The following ``positive" variant of Lemma~\ref{lem:logsmoothcharchart} is proved in the same manner (replace Example~\ref{example:RG} with Example~\ref{example:RGplus}).

\begin{lem} \label{lem:positivelogsmoothcharchart} Let $X$ be a positive log smooth differentiable space, $x$ a point of $X$ such that $P = \ov{\M}_{X,x}^{\rm gp}$ is torsion free (every point of $X$ has this property if $X$ is fs).  Then there is a neighborhood $U$ of $x$ and a characteristic chart $P \to \M_X(U)$ at $x$ (Definition~\ref{defn:chart}) such that the corresponding $\DS$ morphism $\underline{U} \to \underline{\RR}_+(P)$ is smooth. \end{lem}

\subsection{Proof of Theorem~\ref{thm:logsmooth}} \label{section:proof}  Clearly \eqref{anychart}$\implies$\eqref{somechart}.  The difficulty is to prove \eqref{somechart}$\implies$\eqref{anychart}.  Suppose \eqref{somechart} holds so we have a smooth chart $$ \xym{ S \ar[r]^-c & \M_X(U) \\ T \ar[u] \ar[r]^-d & \M_Y(V) \ar[u]_{f^\dagger} } $$ for $f : U \to V$ for a neighborhood $(U,V)$ of $x$ in $f$.  We can (and will) shrink $(U,V)$ to any smaller neighborhood of $x$ in $f$ and the restricted chart continues to be smooth on the smaller neighborhood.   We want to show that \eqref{anychart} holds, so consider a neighborhood $W$ of $f(x)$ in $Y$ and a fine chart $b : Q \to \M_Y(W)$.  The conclusion we want to make allows us to shrink $(f^{-1}(W),W)$ to smaller neighborhoods of $x$ in $f$, so we can shrink $(f^{-1}(W),W)$ and $(U,V)$ if necessary to assume $(f^{-1}(W),W) = (U,V)$.  Shrinking further if necessary, Lemma~\ref{lem:thirdchart} implies that we can assume there is a fine chart $b' : Q' \to \M_Y(V)$ and monoid homomorphisms $g : Q \to Q'$ and $t : T \to Q'$ with $d=b't$ and $b = b'g$.  By the Shrinking Argument, we can assume, after possibly shrinking again, that the monoids $T$, $Q'$, and $Q$ are all equal to their localizations at the preimage of $\O_{Y,f(x)}^*$ and that $S$ is equal to its localization at the preimage of $\O_{X,x}^*$ so that $\ov{S} \to \ov{\M}_{X,x}$ is an isomorphism.  In particular, we can assume the sharpening $\ov{t} : \ov{T} \to \ov{Q}'$ of $t$ is an isomorphism.  Set $P' := S \oplus_T Q'$.  The monoid $P'$ is fine: it is finitely generated because $S$ and $Q'$ are finitely generated; it is integral because $t$ is the pushout of $T^* \to (Q')^*$ by Lemma~\ref{lem:pushout} so $P' = S \oplus_{T^*} (Q')^*$ and $S$ is integral.  We have a commutative diagram \bne{proofdiagram} \xym{ S \ar[r]^-s & P' \ar[r]^-{a'} & \M_X(U) \\ T \ar[u] \ar[r]^t & Q' \ar[u]^h \ar[r]^-{b'} & \M_Y(V) \ar[u]_{f^\dagger} } \ene with $a's=c$.  Since $s$ is a pushout of $t$, the map $\ov{s} : \ov{S} \to \ov{P}'$ is an isomorphism, hence $\ov{P}' \to \ov{\M}_{X,x} \cong \ov{S}$ is an isomorphism, hence we can assume $a'$ is a chart after possibly shrinking again (Lemma~\ref{lem:localiso}).  Using the fact that the big square in \eqref{proofdiagram} is a smooth chart and the left square is pushout, we see easily that the right square in \eqref{proofdiagram} is a smooth chart using stability of smooth $\DS$ morphisms under composition and base change.  Now we forget about the left square in \eqref{proofdiagram} and obtain the desired conclusion via the following:

\begin{lem} \label{lem:logsmooth} Suppose $f : X \to Y$ is a map of fine $\LDS$ or fine $\PLDS$, $x \in X$, the square in the solid diagram $$\xym{ P \ar@{.>}[r]^t & P' \ar[r]^-{a'} & \M_X(X) \\ Q \ar[r]^g \ar@{.>}[u]^k & Q' \ar[u]^-h \ar[r]^-{b'} & \M_Y(Y) \ar[u]_{f^\dagger} }$$ is a smooth chart for $f$, and the composition $b := b'g$ is a fine chart for $\M_Y$.  Then, after possibly replacing $X$ with a neighborhood of $x$, the diagram can be completed as indicated using some fine monoid $P$ so that: \begin{enumerate} \item \label{smoothness} The maps $\underline{\RR}(P') \to \underline{\RR}(P \oplus_Q Q')$ and $\underline{\RR}_+(P') \to \underline{\RR}_+(P \oplus_Q Q')$ are \'etale $\DS$ morphisms. \item \label{newchart} The ``big" square is also a smooth chart for $f$. \end{enumerate} \end{lem}

\begin{proof}  First of all, \eqref{smoothness} (plus the knowledge that $a := a't$ is a fine chart for $\M_X$ and $k$ is monic) implies \eqref{newchart} because the only issue is then to show that the $\DS$ morphism \bne{maptobesmooth} \underline{X} & \to & \underline{Y} \times_{\underline{\RR}(Q)} \underline{\RR}(P) \ene is smooth.  But \eqref{maptobesmooth} is a composition of \be  \underline{X} & \to & \underline{Y} \times_{\underline{\RR}(Q')} \underline{\RR}(P') , \ee which is smooth because the original chart witnesses log smoothness, and a base change \be \underline{Y} \times_{\underline{\RR}(Q')} \underline{\RR}(P') & \to & \underline{Y} \times_{\underline{\RR}(Q')} \underline{\RR}(P \oplus_Q Q') = \underline{Y} \times_{\underline{\RR}(Q)} \underline{\RR}(P) \ee of the map in \eqref{smoothness}, hence \eqref{maptobesmooth} is smooth.  The same argument of course applies in the ``positive" context using the positive version of \eqref{smoothness}.

We now prove \eqref{smoothness} by carefully constructing $P$, $k$, $t$---we use roughly the same argument Kato uses to produce the charts in his Chart Criterion for Log Smoothness \cite[3.13]{Kat1}. Let $F' := (a'_x)^{-1} \O_{X,x}^*$ be the face of the chart $a'$ at $x$, so we have an isomorphism $P'/F' \cong \ov{\M}_{X,x}$.  By the Shrinking Argument, we can assume, after possibly replacing $X$ with a neighborhood of $x$, that $P' = (F')^{-1} P'$, so that $F' = (P')^*$ (Lemma~\ref{lem:sharpening}).  Choose $p =(p_1,\dots,p_r) : \NN^r \to P'$ so that the images of the $p_i$ in the $\QQ$ vector space $((P')^{\rm gp}/ (Q')^{\rm gp}) \otimes \QQ$ form a basis (this can be done because the image of $P'$ certainly spans this vector space).  Consider the monoid homomorphism $s:=(hg,p) : Q \oplus \NN^r \to P'$ and its groupification $s^{\rm gp} : Q^{\rm gp} \oplus \ZZ^r \to (P')^{\rm gp}$.  Set $y := f(x)$.  Since $b$ and $b'$ are charts, $b_y^{\rm gp} : Q^{\rm gp} \to \ov{\M}_{Y,y}^{\rm gp}$ is surjective and similarly for $b'$, $Q'$, so $Q^{\rm gp}$ and $(Q')^{\rm gp}$ have the same image in $\ov{\M}_{X,x}^{\rm gp} = (P')^{\rm gp} / F'$ (namely the image of $(\ov{f}^\dagger_x)^{\rm gp}$).  It follows from this and our choice of $p$ that the composition of $s^{\rm gp}$ and the projection $(P')^{\rm gp} \to (P')^{\rm gp} /F'$ has (finitely generated) torsion cokernel $K$.  Suppose $p' \in (P')^{\rm gp}$ is a lift of one of the generators of $K$, so that we can write $mp' = s^{\rm gp}(q,v)$ (modulo $F'$) for some $(q,v) \in Q^{\rm gp} \oplus \ZZ^r$ and some integer $m > 1$.  Then we can sit $Q^{\rm gp} \oplus \ZZ^r$ inside a ``larger" abelian group $A$ defined by the pushout square $$ \xym@C+10pt{ \ZZ \ar[r]^-{(q,v)} \ar[d]_{\cdot m} & Q^{\rm gp} \oplus \ZZ^r \ar[d]^i \\ \ZZ \ar[r] & A } $$ and extend $s^{\rm gp}$ to $t := (s^{\rm gp},p') : A \to (P')^{\rm gp}$ so that $p'$ is in the image of $t$ modulo $F'$.  Note that $\cdot m$ becomes an isomorphism after applying $\otimes \QQ$, hence so does its pushout $i$.  Repeating this process for each generator of $K$, we eventually produce an inclusion $i : Q^{\rm gp} \oplus \ZZ^r \into A$ of finitely generated abelian groups with $i \otimes \QQ$ an isomorphism, and a group homomorphism $t : A \to (P')^{\rm gp}$ extending $s^{\rm gp}$ such that the composition of $t$ and $(P')^{\rm gp} \to (P')^{\rm gp}/F'$ is surjective.  Set $P := t^{-1}(P')$, so we have a monoid homomorphism $t : P \to P'$, and an inclusion $k : Q \into P$ of monoids (and even an inclusion $Q \oplus \NN^r \into P$).  Set $a := a'k$, and let $F := t^{-1}(F') = a_x^{-1} \O_{X,x}^*$.  By Lemma~\ref{lem:chartproduction}, $P$ is a fine monoid, $F=P^*$, and $\ov{P} \to \ov{P}' \cong \ov{\M}_{X,x}$ is an isomorphism, hence Lemma~\ref{lem:localiso} implies that, after possibly shrinking $X$ to a neighborhood of $x$, $a$ is a chart for $\M_X$.  It is clear that $P^{\rm gp}$ contains $Q^{\rm gp} \oplus \ZZ^r$ and is contained in $A$, so the inclusions $Q^{\rm gp} \oplus \ZZ^r \into P^{\rm gp} \into A$ all become $\QQ$ vector space isomorphisms upon application of $\otimes \QQ$.

We now show that the maps in \eqref{smoothness} are \'etale.  Let $G' := (b'_y)^{-1}\O_{Y,y}^*$, $G := b_y^{-1} \O_{Y,y}^*$ be the faces of $b'$ and $b$ at $y$.  We have a commutative diagram of monoids \bne{diagramA} & \xym{  P \ar[r] & T \ar[r]^e & P' \\ G^{-1}Q \ar[r] \ar[u]^k & (G')^{-1}Q' \ar[u]_h \\ Q \ar[u] \ar[r] & Q' \ar[u] } \ene where the top square is a pushout defining the monoid $T$.  We have $Q/G \cong Q'/G' \cong \ov{\M}_{Y,y}$ because $b$ and $b'$ are charts, so the map of monoids $G^{-1} Q \to (G')^{-1}Q'$ induces an isomorphism on sharpenings, hence so does its pushout $P \to T$, hence $e : T \to P'$ also induces an isomorphism $\ov{e} : \ov{T} \to \ov{P}'$ because $\ov{P} \to \ov{P}'$ is an isomorphism.    The map $\underline{\RR}(G^{-1}Q) \into \underline{\RR}(Q)$ is an open embedding (and similarly with $\underline{\RR}$ replaced by $\underline{\RR}_+$ and/or $G$, $Q$ replaced by $G'$, $Q'$), so $\underline{\RR}(T) \into \underline{\RR}(Q' \oplus_Q P)$ is also an open subspace through which the first map in \eqref{smoothness} factors (and similarly with $\underline{\RR}$ replaced by $\underline{\RR}_+$), so we reduce to showing the the maps $\underline{\RR}(P') \to \underline{\RR}(T)$ and $\underline{\RR}_+(P') \to \underline{\RR}_+(T)$ are \'etale.  Since $\ov{e}$ is an isomorphism, Lemma~\ref{lem:monoidsmooth} reduces us to proving that $e^{\rm gp}$ has torsion cokernel and cokernel.  This is the same thing as proving $e^{\rm gp} \otimes \QQ$ is an isomorphism of $\QQ$ vector spaces.  Groupification and tensoring with $\QQ$ preserve direct limits, so \eqref{diagramA} yields a pushout diagram \bne{diagramB} & \xym@C+15pt{  P^{\rm gp} \otimes \QQ \ar[r] & T^{\rm gp} \otimes \QQ \ar[r]^-{e^{\rm gp} \otimes \QQ} & (P')^{\rm gp} \otimes \QQ \\ Q^{\rm gp} \otimes \QQ \ar[u]^{k^{\rm gp} \otimes \QQ} \ar[r] & (Q')^{\rm gp} \otimes \QQ \ar[u] } \ene of $\QQ$ vector spaces.  Pick a basis $v_1,\dots,v_m$ for $(Q')^{\rm gp} \otimes \QQ$.  Since $P^{\rm gp} \otimes \QQ = (Q^{\rm gp} \oplus \ZZ^r) \otimes \QQ$, the standard basis vectors $e_i \in \ZZ^r$ map to a basis for the cokernel of $k^{\rm gp} \otimes \QQ$, hence, since the square is a pushout, their images $e_i$ in $T^{\rm gp} \otimes \QQ$, together with the $v_i$, form a basis for $T^{\rm gp} \otimes \QQ$.  But by construction, this basis for $T^{\rm gp} \otimes \QQ$ maps to the basis $v_1,\dots,v_m,p_1,\dots,p_r$ for $(P')^{\rm gp} \otimes \QQ$ under $e^{\rm gp} \otimes \QQ$.

\end{proof}

\subsection{Manifolds with corners} \label{section:manifoldswithcorners} Let us now briefly explain how manifolds with corners arise naturally in the setting of (positive) log differentiable spaces.

\begin{defn} \label{defn:mfdwcorners1} {(\bf Provisional)}  A \emph{manifold with corners} is a differentiable space locally isomorphic to an open subspace of $\underline{\RR}_+^n$ (\S\ref{section:RP}) for various $n$. \end{defn}

Notice that any differentiable space which is smooth over a manifold with corners is itself a manifold with corners.

\begin{defn} \label{defn:free} A fine log structure $\M_X$ on a space $X$ is called \emph{free} iff $\ov{\M}_{X,x}$ is a free monoid (\S\ref{section:monoidbasics}) for every $x \in X$. \end{defn}

Lemma~\ref{lem:fschart} implies that a free log structure admits a characteristic chart $\NN^r \to \M_X(U)$ near any given point, hence every free log structure is fs.  Observe that the question of whether a fine log structure is free depends only on the characteristic $\ov{\M}_{X}$.  In other words, freeness is really a property of the (fine) sharp monoidal space $\sms{X}$ underlying a fine log space $X$ (the image of $X$ under the functor \eqref{LogEsptoSMS} of \S\ref{section:logspacestomonoidalspaces}).

\begin{prop} \label{prop:mfdcorners} Let $X$ be a log smooth differentiable space (resp.\ positive log smooth differentiable space) with free log structure.  Then the underlying differentiable space $\underline{X}$ is a smooth manifold\footnote{We remind the reader that our sense of ``smooth manifold" is purely local: We do not demand that the underlying topological space be paracompact, nor even Hausdorff.} (resp.\ a manifold with corners in the sense of Definition~\ref{defn:mfdwcorners1}). \end{prop}

\begin{proof}  Since $(\NN^n)^{\rm gp} = \ZZ^n$ is free, this follows easily from Lemma~\ref{lem:logsmoothcharchart} (or Lemma~\ref{lem:positivelogsmoothcharchart} in the positive case).   \end{proof}

Motivated by the above result, we make the following

\begin{defn} \label{defn:mfdwcorners} A \emph{manifold with corners} is a positive log smooth differentiable space $X$ with free log structure (Definition~\ref{defn:free}). \end{defn}

For example, $\RR_+(\NN^n)$ (Example~\ref{example:RPlog}) is a manifold with corners by this definition.  This definition is also local in nature, so anything locally isomorphic in $\PLDS$ to an open subspace of $\RR_+(\NN^n)$ is again a manifold with corners.  Our next task is to reconcile Definitions~\ref{defn:mfdwcorners1} and \ref{defn:mfdwcorners}.  We first need an ``invariance of domain" result.

\begin{lem} \label{lem:idealofcoordinatehyperplanes} Let $Z \subseteq \RR^n$ be the closed subspace given by the union of the first $k$ coordinate hyperplanes.  The small ideal of $Z$ (i.e.\ the ideal of smooth functions vanishing on $Z$ as in \S\ref{section:differentiablespaces}) is generated by $x_1 \cdots x_k$. \end{lem}

\begin{proof} The proof is by induction on $k$.  The case $k=0$ is trivial, so we now assume $k>0$.  The result can be checked on stalks and follows from the induction hypothesis away from the origin, so we can concentrate on the stalk of this ideal at the origin.  Consider a smooth function $f$ defined on a neighborhood $U$ of the origin and vanishing on $Z \cap U$.  We must show that $f = x_1 \cdots x_k v$ for a smooth function $v$ after possibly restricting to a smaller neighborhood of the origin.  Set \be g(t,\ov{x}) & := & f(x_1,\dots,x_{k-1}, tx_k, x_{k+1},\dots,x_n). \ee Using the fact that $f$ vanishes on the $k^{\rm th}$ coordinate hyperplane, together with the Fundamental Theorem of Calculus and the Chain Rule, we compute \be f(\ov{x}) & = & g(1,\ov{x}) - g(0,\ov{x}) \\ & = & \int_0^1 \frac{\partial g}{\partial t} (t,\ov{x}) dt \\ & = & \int_0^1 \frac{\partial(tx_k)}{\partial t} \frac{\partial f}{\partial x_k}(x_1,\dots,tx_k,\dots,x_n) dt \\ & = & x_k \int_0^1 \frac{\partial f}{\partial x_k}(x_1,\dots,tx_k,\dots,x_n) dt, \ee which shows that we can write $f = x_k h$ for a smooth function $h$ defined on $U$.  Since $f$ vanishes on the first $k-1$ coordinate hyperplanes, $h$ must vanish on the complement $W$ of the $k^{\rm th}$ coordinate hyperplane in the union $Z'$ of the first $k-1$ coordinate hyperplanes.  But $W$ is dense in $Z'$, so by continuity $h$ vanishes on $Z'$ and hence by induction we can write $h = x_1 \cdots x_{k-1} v$ for some smooth function $v$ on some neighborhood of the origin.  Then we have $f = x_1 \cdots x_k v$ as desired.  \end{proof}

\begin{lem} The differentiable spaces $X = \RR^k_+ \times \RR^{n-k}$ (with coordinates $x_1,\dots,x_n$) and $Y = \RR^l_+ \times \RR^{n-l}$ (with coordinates $y_1,\dots,y_n$) are diffeomorphic iff $k=l$.  For any diffeomorphism $f : X \to Y$, there is a permutation $\sigma$ of $\{ 1, \dots, k \}$ such that $f^* y_i = u_i x_{\sigma(i)}$ for positive units $u_i \in \Gamma(X,\O_X^{>0})$ for $i=1,\dots,l$. \end{lem}

\begin{proof} First of all, if $f : X \to Y$ is any isomorphism of differentiable spaces, and $Z \subseteq X$ is any closed subset of the space underlying $X$, then it is clear that $f$ restricts to an isomorphism of differentiable spaces $Z \to f(Z)$ when $Z$ (resp.\ $f(Z)$) is given the small induced structure from $X$ (resp.\ $Y$) (\S\ref{section:differentiablespaces}).  Equivalently, the ideal $I^{\rm small}(Z)$ of $\O_X$ coincides with the ideal $f^{-1} I^{\rm small}(f(Z))$.  

Suppose $f : X \to Y$ is a diffeomorphism as in the lemma.  Since $f$ is, in particular, a homeomorphism, it must take the topological boundary $\partial X$ of $X$ homeomorphically onto the topological boundary $\partial Y$ of $Y$.  By Lemma~\ref{lem:idealofcoordinatehyperplanes}, the small ideal of $\partial X$ is generated by $x_1 \cdots x_k$ and the small ideal of $\partial Y$ is generated by $y_1 \cdots y_l$, so $f$ restricts to a $\DS$ isomorphism $\partial f : \Z(x_1 \cdots x_k) \to \Z(y_1 \cdots y_l)$.  (Although Lemma~\ref{lem:idealofcoordinatehyperplanes} concerns $\RR^n$, it implies the analogous results for $X$ and $Y$ because the structure sheaves of the latter are quotients of the structure sheaf of $\RR^n$.)  It is understood in the remainder of the proof that $\partial X = \Z(x_1 \cdots x_k)$ and $\partial Y = \Z(y_1 \cdots y_l)$ are given the small induced differentiable space structures from $X$ and $Y$.  The module of differentials $\Omega_{\partial X}$ (relative to $\RR$ of course) is generated by $dx_1,\dots,dx_n$ subject to the relation \be 0 & = & d(x_1 \cdots x_n) \\ & = & x_2 \cdots x_k dx_1 + x_1x_3 \cdots x_k dx_2+ \cdots + x_1 \cdots x_{k-1} dx_k.\ee  The points $x$ of $\partial X$ contained in only one of the coordinate hyperplanes $\Z(x_1), \dots , \Z(x_k)$ have $\Omega_{\partial X,x} \cong \RR^{n-1}$, while the points in two or more of these coordinate hyperplanes have $\Omega_{X,x} \cong \RR^n$.  Consequently, $\partial f$ must take the open subspace $W \subseteq \partial X$ of points $x$ of the former type homeomorphically onto the analogous open subspace $W' \subseteq \partial Y$.  But $W$ is the disjoint union, over $i=1,\dots,k$, of the subspaces \be \Z(x_i)' & := & \Z(x_i) \setminus ( \Z(x_1) \cup \cdots \Z(x_{i-1}) \cup \Z(x_{i+1}) \cup \cdots \cup \Z(x_k)) \ee of $X$ and $W'$ is the disjoint union, over $i=1,\dots,l$ of the analogous subspaces $\Z(y_i)'$ of $Y$, hence we must have $k=l$ and there is a permutation $\sigma$ of $\{ 1, \dots, k \}$ such that $f^{-1}(\Z(y_i)') = \Z(x_{\sigma(i)})'$ for $i=1,\dots,k$.  The homeomorphism $f$ also preserves closures, and the closure of $\Z(x_i)'$ in $X$ is $\Z(x_i)$, so we also have $f^{-1}(\Z(y_i)) = \Z(x_{\sigma(i)})$ for $i=1,\dots,k$.  But $y_i$ generates the small ideal of the closed subpace $\Z(y_i)$ and $x_{\sigma(i)}$ generates the small ideal of the closed subspace $\Z(x_{\sigma(i)})$, so the ideal of $\O_X$ generated by $f^* y_i$ coincides with the one generated by $x_{\sigma(i)}$, which implies the desired result.  The units $u_i$ are necessarily positive since $f^*y_i$ and $x_{\sigma(i)}$ are non-negative. \end{proof}

\begin{lem} \label{lem:mfdcornerslifting} Let $X$ and $Y$ be open subspaces of the positive log differentiable space $\RR_+(\NN^n)$.  Any $\DS$ isomorphism $\underline{f} : \underline{X} \to \underline{Y}$ lifts uniquely to an $\LDS$ isomorphism $f : X \to Y$. \end{lem}

\begin{proof}  Such a lift $f$, if it exists, is certainly unique since the structure map $\alpha_Y$ for the log structure on $Y$ is monic (it is just the inclusion of the subsheaf of monoids generated by the units and the coordinate functions).  It therefore suffices to construct a lift locally (the unique lift of $\underline{f}^{-1}$ then serves as an inverse for the lift $f$).  Locally, we are in the situation of the previous lemma and it is clear from that lemma and the description of the log structure on $\RR_+(\NN^n)$ that we have the desired lift.  \end{proof}

\begin{thm} \label{thm:mfdcorners} If $\underline{X} \in \DS$ is a manifold with corners in the sense of Definition~\ref{defn:mfdwcorners1}, then there is a unique positive log structure $\M_X$ on $\underline{X}$ which restricts to the usual one (Example~\ref{example:RPlog}) on any chart of $\underline{X}$.  The positive log differentiable space $X := (\underline{X},\M_X)$ is a manifold with corners in the sense of Definition~\ref{defn:mfdwcorners}. \end{thm}

\begin{proof}  Choose an atlas of open subspaces $\underline{U}_i \subseteq \RR^+_n$ (the $n$ can vary with $i$) and diffeomorphisms $\underline{f}_i : \underline{U}_i \to \underline{X}$ onto open subspaces of $\underline{X}$ whose images cover $\underline{X}$.  View $\underline{U}_i$ as a positive log differentiable space $U_i$ by pulling back the usual log structure on $\RR_+^n = \underline{\RR}(\NN^n)$.  By Lemma~\ref{lem:mfdcornerslifting}, the gluing automorphisms (transition functions) $\underline{f}_{ij} \in \Aut_{\DS}(\underline{U}_{ij})$ lift (uniquely) to $\PLDS$ isomorphisms $f_{ij}$ which inherit the cocycle condition from the $\underline{f}_{ij}$ in light of the uniqueness.  The $U_i$ then form an atlas for a positive log structure on $X$ which is clearly as desired.  The uniqueness follows easily from the uniqueness statement in Lemma~\ref{lem:mfdcornerslifting}.   \end{proof}

Proposition~\ref{prop:mfdcorners} and Theorem~\ref{thm:mfdcorners} show that manifolds with corners in the sense of Definitions~\ref{defn:mfdwcorners1} and \ref{defn:mfdwcorners} are ``the same thing."  However, we view Definition~\ref{defn:mfdwcorners1} as being ``inferior" in that it does not give the right notion of \emph{morphisms} of manifolds with corners.  It should be emphasized here that the log structure in Theorem~\ref{thm:mfdcorners} is not functorial in $\underline{X} \in \DS$---thought it is functorial under \emph{isomorphisms}.

\begin{prop} \label{prop:boundarylogsmooth} Let $X$ be a log smooth log differentiable space (resp.\ positive log differentiable space).\footnote{The conclusion and proof will make sense in many other contexts.}  Then the boundary $\boundary X$ (\S\ref{section:boundary}) is also log smooth.  The boundary of a manifold with corners is a manifolds with corners. \end{prop}

\begin{proof} The question is local, so we can assume $X$ has a smooth global chart $P \to \M_X(X)$ with $P$ fine.  By general properties of the boundary construction (\eqref{Bbasechange} of \S\ref{section:boundary}) we have a cartesian diagram $$ \xym{ \boundary X \ar[d] \ar[r] & \boundary \AA(P) \ar[d] \\ \u{X} \ar[r] & \u{\AA}(P) }$$ in $\LDS$ (resp.\ $\PLDS$) where $\AA(P)$ denotes $\RR(P)$ (resp.\ $\RR_+(P)$).  Since the horizontal arrows are strict, the underlying diagram in $\DS$ is also cartesian, so, since smooth $\DS$ morphisms are stable under base change and the bottom horizontal arrow is smooth by definition of ``smooth chart", the top horizontal arrow is smooth on the underlying differentiable spaces---but we also have \be \boundary \AA(P) & = & \coprod_F \AA(F) , \ee where $F$ runs over maximal proper faces of $P$, so, near any point of $\boundary X$ we obtain a smooth chart using one of these $F$.  In the manifolds with corners case we can take $P = \NN^n$, hence each $F$ is $\NN^{n-1}$. \end{proof} 

\subsection{Kato-Nakayama space}  \label{section:KNspace}  In \cite{KN}, Kato and Nakayama explained how to associate, to a fine log analytic space $X$, a topological space $X^{\KN}$, which we call the \emph{Kato-Nakayama space} (or just the \emph{KN space}) of $X$.  This space comes with a natural proper, surjective map of topological spaces $\tau : X^{\KN} \to X$.  In fact, Kato and Nakayama explain how to endow $X^{\KN}$ with a natural sheaf of $\CC$-algebras so that $\tau$ becomes a map of ringed spaces over $\CC$ (the ringed space $X^{\KN}$ is not \emph{locally} ringed in general).  The purpose of this section is to revisit this construction from the point of view of differential geometry.  We will interpret the KN space as a functor \bne{KNspace} \LAS & \to & \PLDS \\ \nonumber X & \mapsto & X^{\KN} \ene from fine log analytic spaces to fine positive log differentiable spaces.

The topological space $X^{\KN}$ can be constructed for any log analytic space (fine or not) as follows:  Points of $X^{\KN}$ are pairs $(x,f)$ consisting of a point $x$ of $X$ and a group homomorphism $f : \M_{X,x}^{\rm gp} \to S^1$ satisfying $f(u) = u(x) /|u(x)|$ for each $u \in \O_{X,x}^* \subseteq \M_{X,x}$.  Define a function $\tau : X^{\KN} \to X$ by $\tau(x,f) := x$.  Given an open subspace $U$ of $X$ and a section $m \in \M_X(U)$, we obtain a tautological function $\phi_m$ from $\tau^{-1}(U)$ to $S^1$ by setting $\phi_m(x,f) := f(m_x)$.  We endow the set $X^{\KN}$ with the smallest topology so that $\tau$ and the maps $\phi_m$ are continuous (for the usual metric topology on $S^1$).

It is worth noting that the above construction can be made for a \emph{pre}log analytic space $X$ and that the resulting topological space $X^{\KN}$ depends only on the associated log analytic space of $X$.

If one runs through the above construction in the case where $X = \CC(P)$ is the fine log analytic space associated to a fine monoid $P$, then one finds the following equality on the level of topological spaces: \bne{KNformula} \CC(P)^{\KN} & = & \RR_+(P) \times S^1(P), \ene where $$S^1(P) := \Hom_{\Mon}(P,S^1) = \Hom_{\Ab}(P^{\rm gp},S^1) $$ has its usual smooth manifold structure---it is a finite disjoint union of tori and represents the functor $X \mapsto \Hom_{\Ab}(P^{\rm gp},S^1(X))$, where $S^1(X)$ is the group of $\DS$-morphisms $X \to S^1$ for the usual group object structure of $S^1$.  The point is that a pair $(x,y) \in \RR_+(P) \to S^1(P)$ determines a point $\tau(x,y) = xy \in \AA(P)$ (i.e.\ a monoid homomorphism $P \to \CC$) by composing $(x,y) : P \to \RR_+ \times S^1$ with the multiplication map $\RR_+ \times S^1 \to \CC$ (thinking of $\RR_+ = \RR_{\geq 0}$ and $S^1$ as multiplicative submonoids of $\CC$ in the usual way).   (The $y$ in the pair $(x,y)$ can serve as the $f$ in the general Kato-Nakayama construction, so that $(x,y) \mapsto (xy,y)$ yields the homeomorphism from the RHS of \eqref{KNformula} to the usual construction of the KN space of $\AA(P)$.)  The point is to promote the equality \eqref{KNformula} on the level of topological spaces to the level of positive log differentiable spaces, endowing $\RR_+(P)$ with its usual positive log structure.

\begin{thm} \label{thm:KNspace}  There exists a functor $F : \LAS \to \PLDS$ from fine log analytic spaces to fine positive log differentiable spaces satisfying the following properties: \begin{enumerate} \item \label{KNspace1} $F$ preserves finite inverse limits.  \item \label{KNspace2} If $\u{X} \in \LAS$ has trivial log structure, then $F(\u{X}) = \u{X}^{\DS}$ is the differentialization of $\u{X}$ (\S\ref{section:differentialization}) with the trivial positive log structure. \item If $P$ is a fine monoid and $\CC(P)$ is the associated fine log analytic space, then $F(\CC(P)) = \RR_+(P) \times S^1(P)$ in $\PLDS$. \end{enumerate} Any other functor satisfying these properties is isomorphic, via a unique isomorphism, to $F$. \end{thm}

\begin{proof} Many details in the proof are routine tedium; we will highlight the main points.  For a fine monoid $P$, we set $FP := \RR_+(P) \times S^1(P)$ to ease notation.  Say $X$ is a fine log analytic space.  Then, at least locally, $X$ has a fine chart $P \to \M_X(X)$.  This chart gives a cartesian diagram $$ \xym{ X \ar[r] \ar[d] & \CC(P) \ar[d] \\ \u{X} \ar[r] & \u{\CC}(P) } $$ in $\LAS$.  If the three axioms are to hold, $FX$ is determined by the cartesian diagram $$ \xym{ FX \ar[r] \ar[d] & FP \ar[d] \\ \u{X}^{\DS} \ar[r] & \u{\CC}(P)^{\DS} } $$ in $\PLDS$.  The main claim is that $FX$, thus \emph{defined}, does not depend (up to canonical isomorphism) on the chosen chart.  If we assume this claim, then we can construct such a functor $F$ as follows:  \emph{Choose} an open cover $\{ U_i \}$ of each $X \in \LAS$ on which $X$ has a fine chart.  \emph{Choose} such a chart $P_i \to \M_X(U_i)$.  \emph{Choose} a cartesian product $FU_i$ as above.  Then, assuming the claim, we can glue the $F U_i$ (defined by our chosen cartesian diagrams) along the canonical isomorphisms to define $FX$.  Assuming we made the obvious \emph{choices} when $X = \u{X}$ and when $X = \CC(P)$, the resulting functor will satisfy the three axioms.  The essential uniqueness of the functor $F$ is established similarly from the claim.

Now we prove the claim.  Suppose we have two global charts $P,Q \rightrightarrows \M_X(X)$.  We want to show that the (fine) positive log differentiable spaces $F_P(X)$ and $F_Q(X)$ defined by the above recipe are canonically isomorphic.  Since we are going to show that the isomorphism is canonical anyway, we can fix some point $x \in X$ and work locally near $x$.  By Lemma~\ref{lem:thirdchart} we can therefore assume that $P$ and $Q$ both map to a third chart $R$ (we have introduced here the \emph{choice} of such a third chart $R$ receiving maps from $P$ and $Q$).

This gives us a commutative diagram in $\PLDS$ \bne{PLDSdia} & \xym{ F_R(X) \ar[r] \ar[d] & FR \ar[r] \ar[d] & FP \ar[d] \\ \u{X}^{\DS} \ar[r] & \u{\CC}(R)^{\DS} \ar[r] & \u{\CC}(P)^{\DS} } \ene where the left square is cartesian by definition of $F_R(X)$.  (We have a similar diagram with $P$ replaced by $Q$.)  Now, if we knew the right square was cartesian, then the ``big square" would be cartesian, and we would have a canonical isomorphism $F_R(X) \cong F_P(X)$ resulting from the fact that both spaces make the big square cartesian.  We first check that this right square is cartesian when $P \to R$ is a localization.  In this case, the horizontal arrows in the right square are open embeddings and the cartesianness can hence be checked on the level of underlying sets, which is easily done.  Now by the Shrinking Argument and the case just handed, we reduce to treating the case where $h : P \to R$ induces an isomorphism on sharpenings $\ov{P} \to \ov{R}$.  By Lemma~\ref{lem:sharpening}, this means we have a pushout diagram \bne{RPpush} & \xym{ P^* \ar[d] \ar[r]^-{h^*} & R^* \ar[d] \\ P \ar[r]^-h & R } \ene of fine monoids.  Thinking in terms of the ``modular" interpretations of the $\PLDS$ involved, the cartesianness of the right square in \eqref{PLDSdia} is equivalent to the following:  Suppose $Y \in \PLDS$ and we have monoid homomorphisms \be (a,b) : P & \to & \M_Y(Y) \times S^1(Y) \\ g : R & \to & \CC(Y) = \O_Y(Y) \times_{\RR} \CC \ee so that \bne{sothat} \ov{a} \cdot b = gh : P & \to & \CC(Y) = \O_Y(Y) \otimes_{\RR} \CC . \ene  Here $S^1(Y)$ is the group of $\DS$-morphisms $Y \to S^1$, $\CC(Y)$ is the monoid of $\DS$-morphisms $Y \to \CC$, $\ov{a} : P \to \RR_+(Y)$ is the map induced by $a$, and the $\cdot$ in $\ov{a} \cdot b$ makes use of the map \bne{multy} \RR_+ \times S^1 & \to & \CC \ene of monoid objects in $\DS$ given by $(\lambda,u) \mapsto \lambda u$.  We need to show that there is a unique monoid homomorphism \be (c,d) : R & \to & \M_Y(Y) \times S^1(Y) \ee so that $(a,b) = (c,d) h$ and $g = \ov{c} \cdot d$.  Since we have the pushout diagram \eqref{RPpush}, such a $(c,d)$ is the same thing as a monoid homomorphism \be (c^*,d^*) : R^* & \to & \M_Y^*(Y) \times S^1(Y) = \RR_{>0}(Y) \times S^1(Y) \ee satisfying $(a^*,b^*) = (c^*,d^*) h^*$ and $g^* = c^* \cdot d^*$.  The map \eqref{multy} induces an isomorphism \bne{multystar} \RR_{>0} \times S^1 & \to & \CC^* \ene of group objects in $\DS$, and hence an isomorphism of groups \be \RR_{>0}(Y) \times S^1(Y) & \to & \CC^*(Y), \ee so that $(c^*,d^*)$ is uniquely determined by $g^*$ and the requirement that $g^* = c^* \cdot d^*$; the other necessary equality $(a^*,b^*) = (c^*,d^*)h^*$ then results from \eqref{sothat}.
   
The result above yields an isomorphism $F_P(X) \cong F_R(X)$ and a similar isomorphism $F_Q(X) \cong F_R(X)$, hence an isomorphism $F_P(X) \cong F_Q(X)$.  To know that this isomorphism is truly canonical, we have to argue that it doesn't depend on the chosen chart $R$ receiving maps from $Q$ and $P$.  But this argument is routine:  If we chose a different such $R$, say $R'$, then we find some $R''$ to which both $R$ and $R'$ map, then we argue that the isomorphism $F_P(X) \cong F_Q(X)$ constructed using $R$ (or $R'$) coincides with the one constructed using $R''$ (all of these isomorphisms ultimately result from the essential uniqueness of cartesian products). \end{proof}

\subsection{Scholium}  \label{section:scholium}  The purpose of this section is to clear up some issues with our approach that might be in the minds of readers familiar with log geometry in the algebraic setting.  This discussion is not logically necessary for the comprehension of the rest of the paper and can safely be skipped by the uninterested reader.

If $X$ is a scheme and $U \subseteq X$ is an open subscheme, one often considers the log structure \be \N_{U/X} & := & \{ f \in \O_X : f|U \in \O_U^* \} \ee on $X$ given by functions ``invertible on $U$".  In general this will be very ill-behaved (or uninteresting by Hartog's Theorem) unless the complementary closed subset $D$ is a divisor.  If the pair $(U,X)$ is locally isomorphic to $(T,X(\Delta))$ for some toric variety $X(\Delta)$ with torus $T$, then a foundational result of Kato ensures that $\N_{U/X}$ is a fine log structure.  However, even if $X$ is a toric variety and $D \subseteq X$ is a torus-invariant Cartier divisor, the log structure $\N_{U/X}$ need not be quasi-coherent (e.g.\ take $X = \Spec \CC[x,y,u,v]/(xy-uv)$, $D = \Z(x)$).  Thus Ogus was motivated to study ``mildly incoherent" (or ``psuedo-coherent"... there is no accepted terminology) log structures.  The log structures $\N_{U/X}$ have a certain appeal in light of their ``topological nature" (for example, it is easy to understand their pullbacks, and hence it is often easy to produce maps between them), but one must be a bit careful with their usage.

In the setting of differential geometry, the log structure $\N_{U/X}$ as defined above is very poorly behaved and probably should never be considered.  For example, even in the ``simplest possible case" $X=\RR$ (with coordinate $x$), $U= \RR^*$, $D = \Z(x)$, the log structure $\N_{U/X}$ contains the fine log structure \be \M_X & := & \{ x^n u : n \in \NN, \, u \in \O_X^* \} \ee that we would normally use, but it is ``much bigger" because $\N_{U/X}$ will contain functions with zero Taylor series at the origin, and such functions cannot be in $\M_X$ because the $n^{\rm th}$ derivative of $x^n u$ at the origin is $n! u(0) \neq 0$.  In fact, one can show that the stalk at the origin $\N_{U/X,0}$ defines a log structure on the ring $\O_{X,0}$ of germs of smooth functions on $\RR$ at the origin which is not finitely generated (not of the form $P^a$ for any monoid homomorphism $P \to \O_{X,0}$ with $P$ finitely generated).  For reasons of this nature, we will never make use of the log structures $\N_{U/X}$ in this paper.

Our definition of ``log smooth" (Definition~\ref{defn:logsmooth}) is a rather subtle variation on Kato's chart criterion for log smoothness.  In log algebraic geometry one begins by defining \emph{formally log smooth} maps to be the class of maps with the (local) right lifting property with respect to strict square zero closed embeddings of fine log schemes.  One then defines \emph{log smooth} maps to be formally log smooth maps of fine log schemes where the underlying map of schemes is of locally finite presentation.   This definition of ``log smooth" has the advantage that it is clearly closed under composition and base change.   Kato then proves that this notion of log smooth map is equivalent to another criterion in terms of charts, much like our definition.  Of course, since the two notions are equivalent, one could define log smooth maps in terms of the chart criterion, but then it is rather challenging to prove that such maps are closed under composition and base change.  For differentiable spaces, one cannot define the expected notion of smooth maps (\S\ref{section:smoothmorphisms}) in terms of differential or infinitesimal lifting criteria, so one cannot define log smoothness in this manner either, thus we have resorted to working directly with the chart criterion, which is not particularly easy.  For example, one thing that concerned us early on in this work is the following subtlety in the proof of the chart criterion: in log algebraic geometry, one always works \'etale locally, and, \'etale locally, one can always extract $n^{\rm th}$ roots of units in the structure sheaf as long as $n$ is invertible.  There is no analogous statement for differentiable spaces (there is no hope of extracting a square root of a negative unit).  In particular, there are various points in the proof of the chart criterion where Kato performs this \'etale local unit extraction, so we were concerned that it might be necessary to place various additional hypotheses on, say, the charts we would be willing to consider in the definition of log smooth.  In the end, we got everything to work without any such restrictions, but this somewhat explains the delicacy of the proof of Theorem~\ref{thm:logsmooth}.

\section{Monoidal spaces II} \label{section:monoidalspacesII}

\subsection{Proj} \label{section:Proj} Let $P = \coprod_{n} P_n$ be an $\NN$ \emph{graded monoid}.  The coproduct here is the direct sum of $P_0$ modules, which is the disjoint union on the level of sets (\S\ref{section:modules}).  We require $p+q \in P_{m+n}$ for $p \in P_m$, $q \in P_n$.  If $p \in P_n$, then we say $p$ is of \emph{degree} $n$.  Observe that an $\NN$ grading on a monoid $P$ is the same thing as a monoid homomorphism $P \to \NN$.  By convention, we agree that all $\NN$ graded monoids are generated as monoids under $P_0$ by $P_1$.  This will be clear in the applications (blowup) that we have in mind.  

From $P$, we can form a locally monoidal space $\Proj P$, equipped with an $\LMS$ map $\Proj P \to \Spec P_0$, in the same way one defines the locally ringed space $\Proj A$ associated to a graded ring (\cite[Page~76]{H}).  Points of $\Proj P$ are prime ideals\footnote{There is no notion of ``homogeneous" for ideals in a graded monoid.  All elements of $P$ and ideals of $P$ are of course ``homogeneous" in the naive sense of being contained in one of the $P_n$.} $\p \subseteq P$ not containing the \emph{irrelevant prime} $P_{>0} := P_1 \coprod P_2 \coprod \cdots$.  Basic open sets in $\Proj P$ are given by \bne{Up} U_p := \{ \p \in \Proj P : p \notin \p \}, \ene where $p \in P$.  The structure sheaf $\M_X$ of $X = \Proj P$ is defined so that, for an open subset $V \subseteq \Proj P$, a section of $\M_X(V)$ is an element $$ s = (s(\p)) \in \prod_{\p \in V} P_{(\p)} $$ satisfying an evident local consistency condition.  Here $P_{(\p)}$ denotes the monoid of degree zero elements in the $\ZZ$ graded localized monoid $P_{\p}$.  Formation of $\Proj P$ is contravariantly functorial in the graded monoid $P$.  

For a monoid $P$, let $P \oplus \NN$ be the graded monoid with the ``obvious" grading \be P \oplus \NN & = & \coprod_n P \times \{ n \} \ee (the grading corresponding to projection on the second factor).  Suppose $\p \subseteq P \oplus \NN$ is a prime not containing the irrelevant prime.  Then it is easy to check that $\p = \q \times \NN$ where $\q = \{ p \in P : (p,0) \in \p \}$.  As in the case of graded rings, we have a natural isomorphism \be \Proj (P \oplus \NN) & = & \Spec P \ee of locally monoidal spaces.  For $i \in P_1$, there is a natural map of graded monoids $P \to P_{(i)} \oplus \NN$ given by $p \mapsto ([p,ni],n)$ for $p \in P_n$.  Here $P_{(i)}$ is the monoid of degree zero elements in the $\ZZ$-graded localized monoid $P_i = (\NN i)^{-1}P$.  Explicitly, elements of $P_{(i)}$ are equivalence classes $[p,ni]$ of pairs with $p \in P_n$, where $(p,ni) \sim (q,mi)$ iff \be p+(m+a)i & = & q+(n+a)i \ee for some $a \in \NN$.  If $P$ is integral, then $P_{(i)}$ can be viewed as the submonoid of $P^{\rm gp}$ consisting of elements of the form $p-ni$ with $p \in P_n$.  If $P$ is fine and $p_1,\dots,p_k$ are generators for $P$, then the elements $p_j - n_j i$, where $n_j$ is the degree of $p_j$, generate $P_{(i)}$, thus we see that $P_{(i)}$ is fine.  As in the case of graded rings, the induced $\LMS$ morphisms \bne{coverofProj}  \Spec P_{(i)} = \Proj (P_{(i)} \oplus \NN) & \to & \Proj P \ene are the inclusions of the open subspaces $U_i$; these cover $\Proj P$ as $i$ runs through a set of generators for $P_1$ as a $P_0$ module.  (Because of our convention, any prime ideal of $P$ not containing $P_{>0}$ fails to contain $i$ for some $i \in P_1$.)  In particular, we see that $\Proj P$ is a fan in the sense of Definition~\ref{defn:fan}, fine if $P$ is fine.

If $h : Q \to P$ is a map of graded monoids and $q \in Q$, then it is clear from definition \eqref{Up} that $(\Proj h)^{-1}(U_q) = U_{h(q)}$.  That is, the diagram $$ \xym{ \Proj P \ar[d] \ar[r] & \Proj Q \ar[d] \\ \Spec P_{(f(q))} \ar[r] & \Spec Q_{(q)} } $$ is cartesian, so the map $h$ is ``affine".  In particular, if $h$ is surjective, then so is $Q_q \to P_{f(q)}$ and so is the degree zero part of this map, so, locally on the base, $\Proj h$ is given by $\Spec$ of a surjective map of monoids.

\begin{example} \label{example:Pn} In Example~\ref{example:freesymmetricmonoid}, we encountered the $\NN$-graded monoid $\NN^S$ of functions $f : S \to \NN$ associated to a finite set $S$ with grading $|f| := \sum_{s \in S} f(s)$.  Taking $S = \{ 0, \dots, n \}$, we can define the fan $\PP^n$ using the Proj construction via the formula \be \PP^n & := & \Proj (\NN^S). \ee  The reader may which to review the $\Proj$ construction by showing that $\PP^n$ can be built by gluing $n+1$ copies of $\AA^n = \Spec \NN^n$ in ``the usual way." \end{example} 

\subsection{Blowup} \label{section:blowup} For any monoid $P$, we define the \emph{blowup} of an ideal $I \subseteq P$ to be the locally monoidal space $\Bl_I P$ over $\Spec P$ given by $\Proj$ of the graded \emph{Rees monoid} \bne{Reesmonoid} R & := & P \coprod I \coprod I^2 \coprod \cdots. \ene Here $I^n$ is the ideal of $P$ consisting of elements of $P$ of the form $i_1+\cdots+i_n$ where $i_1,\dots,i_n \in I$.  When $P$ is integral and $i \in I = R_1$,  $R_{(i)}$ can be viewed as the submonoid of $P^{\rm gp}$ consisting of elements of the form $t-ni$, where $t \in I^n$ and $n \in \NN$, and the structure map $P \to (R_i)_0$ is just the natural inclusion $p \mapsto p-0i$.  Since $I$ is finitely generated when $P$ is fine (\S\ref{section:idealsandfaces}), it is clear that $R$ is fine when $P$ is fine, hence $\Bl_I P$ is a fine fan covered by finitely many opens $\Spec R_{(i)}$.  The structure map $P \to R_{(i)}$ is an injective map of monoids inducing an isomorphism on associated groups.  We have proved:

\begin{prop} \label{prop:blowup} For a fine monoid $P$ and an ideal $I \subseteq P$, the blowup $\Bl_I P$ is a fine fan over $\Spec P$ which can be covered by finitely many opens $\Spec Q$ such that $P \to Q$ is an injective map of fine monoids inducing an isomorphism on associated groups. \end{prop}

\begin{example}  Consider the case where $P = \NN^m$ and $I = \NN^m \setminus \{ 0 \}$ is the unique maximal ideal of $P$.  Then \be I^n & = & \{ (a_1,\dots,a_m) \in P : \sum_{i=1}^m a_i \geq n  \} . \ee  Let $e_1,\dots,e_n \in P$ be the standard basis vectors.  For $i = 1, \dots, m$, the monoid $(R_{e_i})_0$ is the free submonoid of $P^{\rm gp} = \ZZ^m$ generated by $$ e_1-e_i, \dots, e_{i-1}-e_i, e_i, e_{i+1}-e_i, \dots, e_m-e_i, $$ and the monoid homomorphism $P \to (R_{e_i})_0$ is the natural inclusion.  This will ring a bell for anyone who has blown up the origin in $\AA^n$.  \end{example}

\subsection{Charts and coherence} \label{section:chartsandcoherence} Here we introduce various coherence conditions on locally monoidal spaces and sharp monoidal spaces.  

\begin{defn} \label{defn:monoidalspacechart} An $\LMS$ morphism $f : X \to Y$ is called \emph{strict} iff $f^\dagger : f^{-1} \M_Y \to \M_X$ is an isomorphism.  A strict morphism from $X \in \LMS$ to $(\Spec P,\M_P)$ for some monoid $P$, is called a \emph{chart} for $X$.  For $X \in \SMS$, a strict morphism $f : X \to (\Spec P,\ov{\M}_P)$ will also be called a \emph{chart}.  \end{defn}

\begin{defn} \label{defn:coherentmonoidalspace} A locally monoidal space or sharp monoidal space is called \emph{quasi-coherent} (resp.\ locally finite type, fine, fs, free, \dots) iff it locally admits a chart (resp.\ where the monoids $P$ can be taken finitely generated, fine, fs, free, \dots).  We denote the full subcategory of $\LMS$ (resp.\ $\SMS$) consisting of fine locally monoidal spaces (resp.\ fine sharp fans) by $\fLMS$ (resp.\ $\fSMS$). \end{defn}

If $(X,\M_X)$ is a quasi-coherent monoidal space (resp.\ locally finite type, \dots) then the sharpening $(X,\ov{\M}_X)$ is a quasi-coherent (resp.\ coherent, \dots) sharp monoidal space.  This is clear from the definitions and the fact (Lemmas~\ref{lem:integral}, \ref{lem:saturated}) that the sharpening of a finitely generated (resp.\ fine, fs, \dots) monoid is finitely generated (resp.\ fine, fs, \dots).

\begin{example} \label{example:LMSgroup} Suppose $X$ is a topological space equipped with a sheaf of (abelian) groups $\M_X$.  Since a group is a monoid, we can view $X=(X,\M_X)$ as a locally monoidal space.  First observe that \be \Hom_{\MS}(Y,X) & = & \Hom_{\LMS}(Y,X) \ee for any monoidal space $Y$ because all monoid homomorphisms out of groups are local (Definition~\ref{defn:local}).  Recall (Example~\ref{example:Specgroup}) that $\Spec A$ is a point when $A$ is a group.  It follows that $X$ is quasi-coherent iff the sheaf $\M_X$ is locally constant.  If this is the case, then $X$ is coherent iff it is fs iff the stalks of $\M_X$ are finitely generated.  Similarly, $X$ is a fan iff the space underlying $X$ is discrete. \end{example}

It might help to keep in mind the following extensions of Kato's analogy (\S\ref{section:fans}):

\begin{tabular}{rcl} monoids : rings & :: & monoidal spaces : ringed spaces \\ & :: & locally monoidal spaces : locally ringed spaces \\ & :: & fans : schemes \\ & :: & locally finite type fans : locally finite type schemes \\ & :: & quasi-coherent l.m.s. : ``quasi-coherent l.r.s." \end{tabular}

Note that ``quasi-coherent locally ringed space" has an obvious definition, but, to our knowledge, there has never been a study of such locally ringed spaces.  It might therefore seem strange to introduce the notion of, say, quasi-coherent locally monoidal spaces, but variants of this, particularly, say, ``fine sharp monoidal spaces," arise quite naturally in logarithmic geometry (c.f.\ \S\ref{section:logspaces}) and it is certainly worth studying them systematically, as we will soon see.

\subsection{Coherent sheaves} \label{section:coherentsheaves} For $X \in \MS$, an $\M_X$ \emph{module} is a sheaf $M$ on $X$ equipped with the data of an $\M_X(U)$ module structure on $M(U)$ for each open $U \subseteq X$ making the restriction maps $M(U) \to M(V)$ linear for $\M_X(U) \to \M_X(V)$ in the obvious sense (c.f.\ \cite[Page~109]{H}).  These modules form a category denoted $\Mod(\M_X)$, or just $\Mod(X)$.  The \emph{tensor product} of $M,N \in \Mod(X)$, denoted $M \otimes_{\M_X} N$, is defined to be the sheaf associated to the presheaf \be U & \mapsto & M(U) \otimes_{\M_X(U)} N(U). \ee  (See \S\ref{section:tensorproduct} for the tensor product of modules over a monoid.)  It carries a natural $\M_X$ module structure.

Given a monoid $P$ and a $P$ module $M$, one can form an $\M_P$ module $M^{\sim}$ on $(\Spec P,\M_P)$ in much the same way one would form the quasi-coherent sheaf $M^{\sim}$ on the affine scheme $\Spec A$ from a module $M$ over a ring $A$ (\cite[Page~110]{H}).  The fastest way to define $M^{\sim}$ is to set \be M^\sim & := & \underline{M} \otimes_{\underline{P}} \M_P, \ee where the underlines denote constant sheaves and we make implicit use of the map \eqref{tau}.  The stalk of $M^{\sim}$ at a prime $\p \in \Spec P$ with complementary face $F$ is given by the localization of $M$ at $F$ (\S\ref{section:tensorproduct}): \bne{stalksofMsim} M^\sim_{\p} & = & M \otimes_P \M_{P,\p} \\ \nonumber & = & M \otimes_P F^{-1} P \\ \nonumber & = & F^{-1} M. \ene

\begin{defn} \label{defn:coherentsheaf} For a quasi-coherent locally monoidal space $X$, $M \in \Mod(X)$ is called \emph{quasi-coherent} (or \emph{a quasi-coherent sheaf} on $X$) iff, locally on $X$, there is a chart $f : X \to \Spec P$ for $X$ and a $P$ module $N$ so that $M \cong f^{-1}(N^\sim)$ as an $f^{-1} \M_P \cong \M_X$ module.  We reserve the word \emph{coherent} for the case where $X$ is fine, $P$ can be taken fine, and $N$ can be taken finitely generated.  We then have an evident notion of \emph{coherent ideals} for a fine locally monoidal space. \end{defn}

\begin{example} \label{example:coherentsheaves} For any monoid $P$, the quasi-coherent sheaf $(P^{\rm gp})^{\sim}$ on $\Spec P$ is the constant sheaf $\underline{P^{\rm gp}}$.  Indeed, the natural map \be \underline{P^{\rm gp}} & \to & \underline{P^{\rm gp}} \otimes_{\underline{P}} \M_P \ee is an isomorphism because its stalk at a face $F$ of $\Spec P$ is the natural map \be P^{\rm gp} & \to & P^{\rm gp} \otimes_P F^{-1} P, \ee which is an isomorphism because $$ \xym{ P \ar[r] \ar[d] & F^{-1} P \ar[d] \\ P^{\rm gp} \ar@{=}[r] & P^{\rm gp} } $$ is a pushout diagram of monoids (equivalently, $P$ modules) in light of the universal property of localization. \end{example}

Given a locally monoidal space $X$ and $\M_X$ module $M$, we obtain an $\ov{\M}_X$ module $\ov{M}$ on the monoidal space $(X,\ov{\M}_X)$ obtained by sharpening $X$ by setting \bne{sharpeningmodule} \ov{M} & := & M \otimes_{\M_X} \ov{\M}_X. \ene  In particular, for a monoid $P$ and a $P$ module $M$, we obtain a module $\ov{M}^\sim$ on $\Spec P$.

\begin{defn} For a quasi-coherent sharp monoidal space $X$, $M \in \Mod(X)$ is called \emph{quasi-coherent} iff, locally on $X$, there is a chart $f : X \to (\Spec P,\ov{\M}_P)$ for $X$ and a $P$ module $N$ so that $M \cong f^{-1}(\ov{N}^\sim)$ as an $f^{-1} \ov{\M}_P \cong \ov{\M}_X$ module.  If $P$ can be taken fine and $N$ can be taken finitely generated, then $M$ will be called \emph{coherent}.  \end{defn}

It is clear from these definitions that $M \mapsto \ov{M}$ takes quasi-coherent sheaves on $X$ to quasi-coherent sheaves on the sharpening $\ov{X}$.

\begin{lem} \label{lem:coherentsheaves2} Let $X$ be a fine fan.  An $\M_X$ module $M$ is coherent (Definition~\ref{defn:coherentsheaf}) iff locally on $X$, there is a fine monoid $P$, a finitely generated $P$ module $N$, an \emph{isomorphism} $f :X \to \Spec P$, and an isomorphism $M \cong f^{-1}(N^\sim)$ of $f^{-1} \M_P \cong \M_X$ modules. \end{lem}

\begin{proof} The condition in question is the same as in the definition of ``coherent," except that here ``strict map" is replaced by ``isomorphism."  In any case, the condition is local, but a strict map of fine fans is a local isomorphism anyway (Lemma~\ref{lem:strictfanmaps2}) so there is no difference. \end{proof}

\begin{lem} \label{lem:coherentsheaves} Let $P$ be a monoid, $M$ a $P$ module, $X := \Spec P$.  Then $\Gamma(X,M^\sim) = M$.  For any quasi-coherent sheaf $M$ on $X$, we have $M = \Gamma(X,M)^\sim$. \end{lem}

\begin{proof} This can be proved in the same manner as the analogous statement for rings, though the first statement is even easier here:  The global section functor on $\Spec P$ is just the stalk functor at the maximal ideal $P \setminus P^*$ with complementary face $P^*$ (see \S\ref{section:Spec}), and it is clear from the description of the stalks of $M^\sim$ in \eqref{stalksofMsim} that this stalk is nothing but $M = (P^*)^{-1} M$. \end{proof}

\subsection{Relative Spec} \label{section:relativeSpec}  Let $X$ be a monoidal space, $f : \M_X \to Q$ a map of sheaves of monoids on $X$.  Let $f^* L_X$ be the inverse image of the local prime system of $X$ under the corresponding map $f : (X,Q) \to (X,\M_X)$ of monoidal spaces.  Explicitly, \be (f^* L_X)_x & := & \{ \p \in \Spec Q_x : f_x^{-1}(\p) = \m_x \}. \ee Set \be \Spec_X Q & := & (X,Q,f^*L_X)^{\rm loc} . \ee  It is clear the the construction of $(X,Q,f^*L_X) \in \PMS$ is contravariantly functorial in $Q \in \M_X / \Mon(X)$, so we obtain a functor \bne{SpecX} \Spec_X : (\M_X / \Mon(X))^{\rm op} & \to & \LMS/X \ene called the \emph{relative Spec functor}.  Note that $\Spec_X \M_X = X$ because localization retracts $L$.  The usual $\Spec$ functor is compatible with the relative one:

\begin{lem} \label{lem:relativeSpec} Let $h : P \to Q$ be a morphism of monoids and let $X := \Spec P$.  Then we have a natural isomorphism of locally monoidal spaces over $X$ \be \Spec_X (Q^\sim) & = & \Spec Q. \ee  \end{lem}

\begin{proof} It is clear from the construction of inverse limits in $\PMS$ (\S\ref{section:inverselimits}) and the definition of $Q^\sim$ (\S\ref{section:coherentsheaves}) that $$ \xym{ (X,Q^\sim, (h^\sim)^* L_X ) \ar[r] \ar[d] & (X,\M_P,L_X) \ar[d] \\ (*,Q,T) \ar[r] & (*,P,T) } $$ is cartesian in $\PMS$.  Since localization commutes with inverse limits (it is a right adjoint), the diagram stays cartesian after localizing, in which case the right vertical arrow becomes an isomorphism (Lemma~\ref{lem:Spec}), hence the localization of the left vertical arrow furnishes the desired isomorphism. \end{proof}

\begin{prop} \label{prop:relativeSpecadjointness} The relative Spec functor \eqref{SpecX} is right adjoint to the functor \be \LMS/X & \to & (\M_X / \Mon(X))^{\rm op} \\ (f : Y \to X) & \mapsto & (f^\dagger : \M_X \to f_* \M_Y). \ee \end{prop}

\begin{proof} We want to establish a natural bijection \be \Hom_{\LMS/X}(f:Y \to X, \Spec_X P) & = & \Hom_{\M_X/\Mon(X)}(k : \M_X \to P,f_* \M_Y) . \ee  By the adjointness $(f^{-1},f_*)$, an element $g \in \Hom_{\M_X/\Mon(X)}(k : \M_X \to P,f_* \M_Y)$ is the same thing as a completion of $$ \xym{ & \ar@{.>}[ld]_g (Y,\M_Y) \ar[d]^f \\ (X,P) \ar[r]^k & (X,\M_X) } $$ in $\MS$, which is the same thing as a completion $g$ of the solid diagram $$ \xym{ \Spec_X P = (X,P,k^*L_X)^{\rm loc} \ar[d]_{\tau} & \ar@{.>}[l]_-h \ar@{.>}[ld]_g (Y,\M_Y,L_Y) \ar[d]^{L(f)} \\ (X,P,k^*L_X) & \ar[l]^k (X,\M_X,L_X) } $$ in $\PMS$.  By the universal property of localization (Theorem~\ref{thm:localization}) and the fact that $L(f)^{\rm loc} = f$ (same theorem), the map $g \mapsto g^{\rm loc}$ yields the desired bijection, with inverse $h \mapsto \tau h$. \end{proof}

\begin{rem} The analog of Proposition~\ref{prop:relativeSpecadjointness} also holds for locally \emph{ringed} spaces, by the same proof, though this statement is curiously absent in \cite{loc}. \end{rem}

\begin{cor} The functor \eqref{SpecX} preserves inverse limits, so that, in particular, \be \Spec_X (P \otimes_{\M_X} Q) & = & (\Spec_X P) \times_X (\Spec_X Q). \ee \end{cor}

\begin{cor} \label{cor:pullbackSpec} Let $Y$ be a locally monoidal space, $\M_Y \to P$ a map of sheaves of monoids on $Y$, $f : X \to Y$ an $\LMS$ morphism, $f^* P = f^{-1} P \otimes_{f^{-1} \M_Y} \M_X$ the pullback of $P$.  Then the $\LMS$ diagram $$ \xym{ \Spec_X (f^*P) \ar[d]_{\tau} \ar[r] & \Spec_Y P \ar[d]^\tau \\ X \ar[r]^f & Y } $$ is cartesian. \end{cor}

\begin{proof}  We want to show that for each $\LMS$ morphism $g : Z \to X$, the set of $\LMS$ morphisms $h : Z \to \Spec_Y P$ with $fg=\tau h$ is in bijective correspondence with the set of $\LMS /X$ morphisms $k : Z \to \Spec_X (f^*P)$.  By the Proposition, the set of such $h$ is bijective with the set of completions $$ \xym{ f_* \M_X \ar[r] & f_*g_* \M_Z \\ \M_Y \ar[u] \ar[r] & P \ar@{.>}[u]^h } $$ in $\Mon(Y)$, which, by the adjointness $(f^{-1},f_*)$ is bijective with the set of completions  $$ \xym{  \M_X \ar[r] & g_* \M_Z \\ f^{-1} \M_Y \ar[u] \ar[r] & f^{-1} P \ar@{.>}[u]^h } $$ in $\Mon(X)$, which, by the universal property of tensor product, is bijective with the set of maps $h : f^* P \to g_* \M_Z$, which, by the Proposition, is bijective with the set of such $k$. \end{proof}

The relative Spec construction is compatible with sharpening in the following sense:

\begin{cor} \label{cor:sharpeningSpec} Let $X$ be a locally monoidal space, $\M_X \to P$ a map of sheaves of monoids on $X$.  Then the $\LMS$ diagram $$ \xym{ \Spec_{\ov{X}} \ov{P} \ar[d] \ar[r] & \Spec_X P \ar[d] \\ \ov{X} \ar[r] & X } $$ is cartesian and there is a natural $\SMS$ isomorphism \be \ov{ \Spec_{\ov{X}} \ov{P} } & = & \ov{ \Spec_X P }. \ee \end{cor}

\begin{proof}  Let $f : \ov{X} \to X$ be the sharpening map.  The $\ov{X}$ module $\ov{P}$ is nothing but $f^* \M_X$---c.f.\ \eqref{sharpeningmodule}, so the cartesian diagram is just the one from the previous corollary.  The natural isomorphism is obtained from this cartesian diagram by noting that the sharpening functor $\SMS \to \LMS$ is a right adjoint (to the inclusion the other way), so it preserves inverse limits. \end{proof}

\begin{thm} \label{thm:relativeSpec} Let $X$ be a fan (resp.\ quasi-coherent locally monoidal space), $Q$ a quasi-coherent monoid under $\M_X$.  Then $\Spec_X Q$ is a fan (resp.\ quasi-coherent locally monoidal space).  If $X$ is fine and $Q$ is integral and locally finite type, then $\Spec_X Q$ is fine.  \end{thm}

\begin{proof}  The question is local, so we can assume there is an isomorphism (resp.\ a strict map) $h : X \to \Spec P =: Y$ for some monoid $P$, and a monoid homomorphism $P \to R$ such that $Q = h^{-1}(R^\sim)$.  We have $\Spec_Y (R^{\sim}) = \Spec R$ by Lemma~\ref{lem:relativeSpec} and we have a cartesian diagram $$ \xym{ \Spec_X Q \ar[r]^f \ar[d] & \Spec R \ar[d] \\ X \ar[r]^h & Y } $$ by Corollary~\ref{cor:pullbackSpec}.  When $X$ is a fan ($h$ is an isomorphism), so is $f$ so $\Spec_X Q \cong \Spec R$ is a fan, and when $h$ is strict, so is its base change $f$ (Corollary~\ref{cor:benignstrictmorphisms}), so $\Spec_X Q$ is quasi-coherent.  If $X$ is fine and $Q$ is integral and locally finite type, then we can take $P$ fine and $R$ finitely type (under $P$, hence also absolutely) and integral.     \end{proof}

\subsection{Integration and saturation} \label{section:integrationandsaturation} Let $X$ be a locally monoidal space.  Let $\M_X^{\rm int}$ be the sheaf associated to the presheaf $U \mapsto \M_X(U)^{\rm int}$.  The functor $P \mapsto P^{\rm int}$ is a left adjoint (to the inclusion of integral monoids in monoids) so it commutes with direct limits, hence with stalks, hence we have \be (\M_{X}^{\rm int})_x & = & (\M_{X,x})^{\rm int} \ee so we can simply write $\M_{X,x}^{\rm int}$ without ambiguity.  There is a natural map $\M_X \to \M_X^{\rm int}$ of sheaves of monoids on $X$.  Set \be X^{\rm int} & := & \Spec_X \M_{X}^{\rm int} . \ee  We obtain a functor $X \mapsto X^{\rm int}$ which is right adjoint to the inclusion of \emph{integral} locally monoidal spaces (stalks are integral monoids) into $\LMS$.  In particular, $X \mapsto X^{\rm int}$ commutes with inverse limits in $\LMS$.

If $X$ is coherent (resp.\ a locally finite type fan), then locally on $X$ there is a strict map (resp.\ an isomorphism) $f : X \to \Spec P$ for a finitely generated monoid $P$.  It is straightforward to see that $\M_P^{\rm int} = (P^{\rm int})^\sim$, so from Lemma~\ref{lem:relativeSpec} and Corollary~\ref{cor:pullbackSpec} we obtain a cartesian diagram $$ \xym{ X^{\rm int} \ar[d] \ar[r] & \Spec P^{\rm int} \ar[d] \\ X \ar[r]^f & \Spec P } $$ (locally on $X$) where the horizontal arrows are strict (resp.\ isomorphisms), thus we see that $X^{\rm int}$ is a fine monoidal space (resp. a fine fan).

The same constructions can be repeated with integration replaced by saturation to obtain a functor $X \mapsto X^{\rm sat}$ which takes fine locally monoidal spaces (resp.\ fine fans) to fs locally monoidal spaces (resp.\ fs fans).

\subsection{Relative Proj} In the context of locally monoidal spaces, we will make use of the relative version of the $\Proj$ construction of \S\ref{section:Proj}, which we now summarize.  To a locally monoidal space $X$, an $\NN$ graded sheaf of monoids on $X$ and a map of sheaves of graded graded monoids $f : \M_X \to P$, we can associate a monoidal space $\Proj_X P$ over $X$.  Points of $\Proj_X P$ are pairs $(x,z)$ consisting of a point $x \in X$ and a prime ideal $z \subseteq P_x$  not containing the irrelevant prime $(P_x)_{>0}$ and such that $f_x^{-1}(z) = \m_x \subseteq \M_{X,x}$.  For an open subset $U \subseteq X$ and a section $p \in P(U)$, the set \be U_p & := & \{ (x,z) \in \Proj_X P : x \in U, p \notin z \} \ee is open in $\Proj_X P$ and such opens form a basis for its topology.  The structure sheaf of $\Proj_X P$ is defined similarly to the structure sheaf for usual $\Proj$ and the structure sheaf for relative $\Spec$.  We assume for convenience that $P$ is locally generated, as a sheaf of monoids under $P_0$, by global sections of $P_1$.  For a section $i \in P_1(U)$, we have an open embedding \bne{Projcovering} \Spec_U (P_{(i)}) = \Proj_U  (P_{(i)} \oplus \underline{\NN}) & \subseteq & \Proj_X P. \ene  As $U$ runs over opens in $X$ and $i$ runs over $P_1(U)$ these opens cover $\Proj_X P$.  

The following basic properties of the relative Proj construction are obtained from the analogous properties of the relative Spec construction by using the natural coverings \eqref{Projcovering}.

\begin{lem} \label{lem:relativeProj} Let $P=\coprod_n P_n$ be an $\NN$ graded monoid, $X := \Spec P_0$.  There is a natural isomorphism \be \Proj P & = & \Proj_X (P^{\sim}) \ee of locally monoidal spaces over $X$. \end{lem}

\begin{lem} \label{lem:pullbackProj} Let $Y$ be a locally monoidal space, $P$ a graded sheaf of monoids under $\M_Y$, $f : X \to Y$ an $\LMS$ morphism, $f^* P = f^{-1} P \otimes_{f^{-1} \M_Y} \M_X$ the pullback of $P$.  Then the $\LMS$ diagram $$ \xym{ \Proj_X (f^*P) \ar[d]_{\tau} \ar[r] & \Proj_Y P \ar[d]^\tau \\ X \ar[r]^f & Y } $$ is cartesian. \end{lem}

\begin{lem} \label{lem:sharpeningProj} Let $X$ be a locally monoidal space, $P$ a graded sheaf of monoids under $\M_X$.  Then the $\LMS$ diagram $$ \xym{ \Proj_{\ov{X}} \ov{P} \ar[d] \ar[r] & \Proj_X P \ar[d] \\ \ov{X} \ar[r] & X } $$ is cartesian and hence the top horizontal arrow induces an $\SMS$ isomorphism upon sharpening. \end{lem}

\section{Algebraic considerations} \label{section:algebraicconsiderations}  Now that we have the basic theory of monoidal spaces in place, our goal in the next section (\S\ref{section:monoidalspacesIII}) will be to prove a kind of ``resolution of singularities" theorem for (reasonable) locally monoidal spaces with variants for fans and for sharp locally monoidal spaces.  To do this, we will bootstrap up from a similar resolution result for (reasonable) monoids.  In order to make this work, we will need to know that the complex monoid algebra functor $P \mapsto \CC[P]$ is ``faithful" in various senses, so that we can use resolution of singularities for $\Spec \CC[P]$ to get it for $P$.  The purpose of this section is mainly to establish the necessary ``faithfulness" results.

\subsection{Scheme realization of fans} \label{section:schemetheoreticrealizationoffans}  Let $\Sch$ denote the category of schemes, viewed as a category of spaces (\S\ref{section:spaces}) with the monoid object given by $\AA^1 = \Spec \ZZ[x]$, under multiplication.  The realization functors of \S\ref{section:fanstologspaces} will be denoted \bne{FanstoSch} \u{\AA} : \Fans & \to & \Sch \\ \label{FanstoLogSch} \AA : \Fans & \to & \LogSch. \ene The scheme $\u{\AA}(F)$ (resp.\ the log scheme $\AA(F)$) will be called the \emph{scheme realization} (resp.\ \emph{log scheme realization}) of the fan $F$.   While it is not necessary to assume that $F$ is locally finite type, that assumption does ensure that $\u{\AA}_{\ZZ}(F)$ is a locally finite type scheme (over $\ZZ$).

These functors can be constructed by alternative ``general nonsense."  In fact we can construct functors \bne{LMStoLRS} \u{\AA}_{\ZZ} : \LMS & \to & \LRS \\ \label{LMStoLogLRS} \AA_{\ZZ} : \LMS & \to & \LogLRS \ene satisfying the adjointess relations \bne{adj} \Hom_{\LRS}(Y,\u{\AA}_{\ZZ}(X)) & = & \Hom_{\LMS}((Y,\O_Y),X) \\ \label{adj2} \Hom_{\LogLRS}(Y,\AA_{\ZZ}(X)) & = & \Hom_{\LMS}((Y,\M_Y),X). \ene  Since $\Sch \subseteq \LRS$ is a full subcategory, these adjointness relations imply that \eqref{LMStoLRS} and \eqref{LMStoLogLRS}, when restricted to $\Fans \subseteq \LMS$, must agree with the functors \eqref{FanstoSch} and \eqref{FanstoLogSch}.  To construct these general functors, consider an arbitrary $X=(X,\M_X) \in \LMS$.  Letting $\ZZ[\M_X]$ denote the sheaf associated to the presheaf \be U & \mapsto & \ZZ[\M_X(U)], \ee we obtain a ringed space $(X,\ZZ[\M_X]) \in \RS$.  This ringed space in fact comes with a natural prime system $S$ defined by letting $S_x$ be the set of prime ideals of $(\ZZ[\M_X])_x = \ZZ[\M_{X,x}]$ whose intersection with $\M_{X,x} \subseteq \ZZ[\M_{X,x}]$ is the unique maximal ideal of the monoid $\M_{X,x}$.  The ringed space $(X,\ZZ[\M_X])$ also comes with an obvious log structure, namely the one associated to the natural map $\M_X \to \ZZ[\M_X]$.  Let $\u{\AA}_{\ZZ}(X)$ be the locally ringed space obtained by ``localizing $(X,\ZZ[\M_X])$ at $S$" (c.f.\ \cite{loc} or the analogous construction in \S\ref{section:localization}) and let $\AA_{\ZZ}(X)$ be the log locally ringed space obtained by pulling back the log structure on $(X,\ZZ[\M_X])$ along the structure map $\u{\AA}_{\ZZ}(X) \to (X,\ZZ[\M_X])$.  Now suppose that $Y = (Y,\O_Y) \in \LRS$ is a locally ringed space.  If we regard $Y$ as a primed ring space $(Y,L) \in \PRS$ by giving it the ``local prime system" $L$ (with $L_y = \{ \m_y \}$ for each $y \in Y$), then it is clear from our definition of $S$ that \bne{PRSLMS} \Hom_{\PRS}((Y,L),(X,\ZZ[\M_X],S)) & = & \Hom_{\LMS}((Y,\O_Y),X), \ene where $(Y,\O_Y)$ is the locally monoidal space whose structure sheaf of monoids is the multiplicative monoid $\O_Y$.  Combining \eqref{PRSLMS} with the universal property of localization \be \Hom_{\LRS}(Y,\AA_{\ZZ}(X)) & = & \Hom_{\PRS}((Y,L),(X,\ZZ[\M_X],S)) \ee we obtain the adjunction formula \eqref{adj}.  The formula \eqref{adj2} is obtained similarly; one must work with the category of ``log primed ring spaces."

For the log scheme realization $\AA(F)$ of a fan $F$, we will often refer to the $\LMS$ map \bne{orbitmap} \tau : \AA(F)^\dagger & \to & F \ene as the \emph{orbit map}, for reasons that will be made clear in Lemma~\ref{lem:orbitmap2}.  In the case of an affine fan $F = \Spec P$, the orbit map \bne{affineorbitmap} \tau : \Spec \ZZ[P] & \to & \Spec P \ene is given as follows.  We view $\Spec P$ as the set of faces of the monoid $P$.  We view the points of $\Spec \ZZ[P]$ as ring homomorphisms $x : \ZZ[P] \to k$ ($k$ a field), up to the usual notion of equivalence.  Such an $x$ yields a monoid homomorphism $x|P : P \to k$ by restriction to $P \subseteq \ZZ[P]$.  Then $\tau(x) = x_P^{-1}(k^*)$.

\begin{lem} \label{lem:orbitmap1} For any fan $F$, the orbit map \eqref{orbitmap} is surjective and quasi-compact\footnote{This means that $\tau^{-1}(U)$ is quasi-compact for every quasi-compact open subspace $U$ of $F$.} (as a map of topological spaces). \end{lem}

\begin{proof}  It is quasi-compact because it is ``affine:"  The preimage of an open affine $U = \Spec P$ of $F$ under $\tau$ is $\Spec \ZZ[P]$, by the construction of $\tau$ (\S\ref{section:realizationoffans}), so the result follows from the fact that $\Spec$ of any monoid or ring is quasi-compact (same proof for monoids as for rings).  Surjectivity is local on the base, so we can assume $F = \Spec P$ is affine.  We need to show that for any face $F \subseteq P$ there is a field $k$ and a monoid homomorphism $x : P \to k$ with $x^{-1}(k^*)=F$.  Since $F$ is a face, we can just define $x(p)$ to be $1$ if $p \in F$ and zero otherwise, using any field we want.  \end{proof}

\begin{lem} \label{lem:orbitmap2} Let $P$ be a monoid, $k$ an algebraically closed field, $X := \Spec \ZZ[P]$, $G := \Spec \ZZ[P^{\rm gp}]$, so that $G$ acts on $X$ as in \S\ref{section:groupstogroupspaces}.  Two $k$-points $x,y \in X(k)$ have the same image under $$ \xym{ X(k) \ar[r] & X \ar[r]^-{\tau} & \Spec P } $$ iff they lie in the same orbit of the $G(k)$-action.  If we assume $P$ is fs, this result still holds without the assumption that $k$ is algebraically closed.  \end{lem}

\begin{proof} An element $g \in G(k)$ is the same thing as a group homomorphism $g : P^{\rm gp} \to k^*$; similarly $x,y$ correspond to monoid homomorphisms $x,y : P \to k$, and $g \cdot x$ corresponds to the monoid homomorphism $p \mapsto x(p)g(p)$, so it is clear that $y = g \cdot x$ implies $x^{-1} k^* = y^{-1} k^*$ as faces of $P$.  For the converse: having the same image under that map means $x$ and $y$ determine the same face $x^{-1} k^* = y^{-1} k^* =: F$ of $P$.  Then $\ov{g} := y/x : F^{\rm gp} \to k^*$ is a group homomorphism.  If we could extend $\ov{g}$ to $g : P^{\rm gp} \to k^*$ that we'd have $g \cdot x = y$.  Such an extension can always be found if $k$ is algebraically closed, for then $k^*$ is injective (i.e.\ divisible) so $\Hom(\slot,k^*)$ is exact.  In the fs case, we can also find such an extension $g$ because $P^{\rm gp} / F^{\rm gp}$ is free (Lemma~\ref{lem:fssplitting}). \end{proof}

The next result says that ``blowup commutes with scheme realization."

\begin{lem} \label{lem:blowuprealization} Let $P$ be a monoid, $I \subseteq P$ an ideal.  There is a natural isomorphism of schemes \be \Bl_{\ZZ[I]} \Spec \ZZ[P] & = & \u{\AA}_{\ZZ}( \Bl_I \Spec P).\ee \end{lem}

\begin{proof}  Let $R = P \coprod I \coprod I^2 \coprod \cdots$ be the Rees monoid.  First note that $\ZZ[I] \subseteq \ZZ[P]$ is the ideal generated by $\{ [i] \in \ZZ[P] : i \in I \}$, so $\ZZ[I]^n$ is the ideal of $\ZZ[P]$ generated by the elements $$[i_1]\cdots[i_n] = [i_1+\cdots+i_n] ,$$ where $i_1,\dots,i_n \in I$, which makes it clear that $\ZZ[I]^n = \ZZ[I^n]$, and hence that \be \ZZ[R] & = & \ZZ[P] \oplus \ZZ[I] \oplus \ZZ[I]^2 \oplus \cdots \ee is the Rees algebra of the ideal $\ZZ[I] \subseteq \ZZ[P]$.  The fan $\Bl_I \Spec P$ is covered by affine opens $\Spec (R_i)_0$ as $i$ runs over $I$, hence its scheme realization $\u{\AA}_{\ZZ}( \Bl_I \Spec P)$ is covered by affine opens $\Spec \ZZ[(R_i)_0]$.  Similarly, $\Bl_{\ZZ[I]} \Spec \ZZ[P]$ is covered by affine opens $\Spec (\ZZ[R]_{[i]})_0$.  As a special case of the compatibility \eqref{tensorformula} between $\ZZ[ \slot ]$ and tensor products, we see that $(\ZZ[R]_{[i]})_0 = \ZZ[(R_i)_0]$.  Both schemes in question are covered by the same open affines and it is clear that the gluing data is the same. \end{proof}

The results of this section hold equally well, by the same proofs, with $\ZZ$ replaced everywhere by $\CC$.  

\subsection{Complex monoid algebras}  \label{section:CP}  Let $P$ be a fine monoid.  Later in the paper we will use results about the $\CC$ scheme $\Spec \CC[P]$ to establish results (particularly resolution of singularities) about the monoid $P$.  To do this, we need to know a few basic relationships between properties of $P$ and properties of $\CC[P]$.  For the purposes of this paper, the following definition of ``complex variety" is convenient.

\begin{defn} \label{defn:complexvariety} A \emph{complex variety} is a separated scheme of finite type over $\CC$ each of whose connected components is integral (reduced and irreducible). \end{defn}

Any (abelian) group $A$ is an abelian group object in $\Mon^{\rm op}$ with comultiplication given by the diagonal map $\Delta : A \to A \oplus A$.  The monoid algebra functor \be \CC[ \slot ] : \Mon & \to & (\CC / \An) \ee preserves direct limits, so it takes (abelian) group objects in $\Mon^{\rm op}$ to (abelian) group objects in $(\CC / \An)^{\rm op}$ (affine, abelian group schemes over $\CC$).  For any monoid $P$, the (co)group $P^{\rm gp}$ acts on $P$ via the ``diagonal" map \be P & \to & P \oplus P^{\rm gp}  \\ p & \mapsto & (p,p), \ee and hence $\CC[P^{\rm gp}]$ acts on $\CC[P]$.  On the level of $\CC$ points, we obtain, for each $\CC$-algebra map $h : \CC[P^{\rm gp}] \to \CC$, a $\CC$-algebra automorphism abusively denoted $h : \CC[P] \to \CC[P]$.  The $\CC$-algebra map $h$ corresponds to a monoid homomorphism $h : P^{\rm gp} \to \CC$ (with $\CC$ regarded as a monoid under multiplication), which is the same thing as a monoid homomorphism $h : P^{\rm gp} \to \CC^*$.  The corresponding $\CC$-algebra automorphism $h : \CC[P] \to \CC[P]$ is given by \bne{aut} h \left ( \sum_{p} a_p [p] \right ) & \mapsto & \sum_p a_p h(p) [p]. \ene

\begin{lem} \label{lem:enoughmapstoCstar} Let $k$ be a field, $A$ a finitely generated abelian group such that the order of the torsion subgroup of $A$ is prime to the characteristic of $k$.  Then for all $a,b \in A$, there exists a finite field extension $k \subseteq K = K(a,b)$ and a group homomorphism $h : A \to K^*$ with $h(a) \neq h(b).$  Variant: If $k$ is infinite (or $A$ is torsion), then we can find one finite extension $k \subseteq K$ that will work for all $a,b \in A$.  In particular, if $k$ is algebraically closed of characteristic zero, then for any finitely generated abelian group $A$ and any distinct $a,b \in A$, there is a group homomorphism $h : A \to k^*$ with $h(a) \neq h(b).$\end{lem}

\begin{proof} The hypotheses on $A$ ensure that $A$ is isomorphic to a (finite) product of $\ZZ$'s and $\ZZ / n \ZZ$'s for various positive integers $n$ not divisible by the characteristic $p$ of $k$.  Any two distinct elements of $A$ have at least one different coordinate, so by projecting on that coordinate, we reduce to proving the lemma in the case $A=\ZZ$ and the case $A = \ZZ / n \ZZ$, $n$ not divisible by $p$.  The key point is that if $n$ is a positive integer prime to $p$, then the polynomial $x^n - 1 \in k[x]$ has $n$ \emph{distinct} roots in $k$ because its derivative $n$ is nonzero.  Any finite subgroup of the multiplicative group $k^*$ of a field $k$ is cyclic \cite[Lemma~7.6]{Herstein}, so we can find a primitive $n^{\rm th}$ root of unity $\zeta \in k^*$.  When $A = \ZZ$, choose an integer $n$ larger than $|a|+|b|$ and not divisible by $p$ and adjoin a primite $n^{\rm th}$ root of unity $\zeta$ to $k$ to obtain $K$.  The choice of $n$ ensures that the map $\ZZ \to K^*$ sending $1$ to $\zeta$ separates $a$ and $b$.  When $A = \ZZ / n \ZZ$, the same procedure yields an injective group homomorphism $A \to K^* = k(\zeta)^*$.  For the variant where $k$ is infinite, it is enough to show that in the case $A = \ZZ$, we already have enough group homomorphisms $A \to k^*$ to separate points.  Indeed, if $k$ is infinite, then either 1) $k$ contains primitive $n^{\rm th}$ roots of unity for arbitrarily large $n$, or 2) there is $u \in k^*$ not a root of unity.  In the first case we can separate points by the same procedure as above and in the second case we can separate points with the single injective group homomorphism $\ZZ \to k^*$ sending $1$ to $u$. \end{proof}

\begin{thm} \label{thm:torusinvariantideals}  Let $P$ be a fine monoid.  An ideal $J \subseteq \CC[P]$ is invariant under every automorphism \eqref{aut} of $\CC[P]$ iff $J = \CC[I]$ for some ideal $I \subseteq P$. \end{thm}

\begin{proof} The ideals $\CC[I]$ are clearly invariant under such automorphisms; the difficulty is to prove that an ideal $J \subseteq \CC[P]$ invariant under all such automorphisms is of the form $J = \CC[I]$ for an ideal $I \subseteq P$.  Given such an ideal $J$ (or any ideal of $\CC[P]$ at all), the subset \be I & := & \{ p \in P : [p] \in J \} \ee is clearly an ideal of $P$ and we clearly have a containment $\CC[I] \subseteq J$.  The issue is to prove the containment $J \subseteq \CC[I]$.  A typical element $j \in J$ can be uniquely written $j = \sum_{p \in S} a_p [p]$ for a finite subset $S \subseteq P$ and non-zero complex numbers $a_p$.  We prove by induction on $|S|$ that for any such $j$, $S \subseteq I$, which implies in particular that $j \in \CC[I]$.  This is trivial when $|S|=0$ or $|S|=1$.  When $|S|>1$, choose distinct $p,q \in S$.  Since $P$ is finitely generated, $P^{\rm gp}$ is a finitely generated abelian group, so by Lemma~\ref{lem:enoughmapstoCstar} there is a group homomorphism $h : P^{\rm gp} \to \CC^*$ with $h(q) \neq h(p)$.  (We are implicitly using integrality here to know that the images of $p$ and $q$ in $P^{\rm gp}$ are distinct.)  By invariance of $J$ under $h$, we have $h(j) \in J$ and hence $h(q)^{-1} h(j) \in J$ and $j' := j - h(q)^{-1} h(j) \in J$.  But \be j' & = & \sum_{p \in S} a_p( 1-h(q)^{-1} h(p)) [p] \\ & = & \sum_{p \in S \setminus \{ q \} } a_p( 1-h(q)^{-1} h(p))[p] , \ee so by induction, and the fact that the coefficient of $[p]$ in this sum is nonzero, we have $p \in I$.  But this argument works for any $p \in S$ so $S \subseteq I$ as desired. \end{proof}

\begin{thm} \label{thm:smoothifffree} Let $P$ be a fine monoid.  The $\CC$-algebra $\CC[P]$ is smooth iff $\ov{P}$ is free. \end{thm}

\begin{proof} First observe that for any finitely generated abelian group $A$, the $\CC$-algebra $\CC[A]$ is smooth.  Indeed, $\CC[A]$ is isomorphic to a finite tensor product of $\CC$-algebras of the form \be \CC[\ZZ] & = & \CC[t,t^{-1}] \\ \CC[\ZZ/n\ZZ] & = & \CC[t]/(t^n-1) \\ & \cong & \CC^n. \ee 

Now suppose $\ov{P} \cong \NN^r$ is free.  The $\ov{P}^{\rm gp} \cong \ZZ^r$ is also free and we obtain an isomorphism $P \cong P^* \oplus \ov{P}$ from Lemma~\ref{lem:splitting}.  Note that $P^*$ is a finitely generated abelian group as it is a subgroup of $P^{\rm gp}$, so $\CC[P]$ is the tensor product of the smooth $\CC$-algebra $\CC[\ov{P}] \cong \CC[x_1,\dots,x_n]$ and the smooth $\CC$-algebra $\CC[P^*]$.

Now suppose $\CC[P]$ is smooth.  It can be shown (c.f.\ \cite[1.13.4]{loggeometry}, \cite{Ogus}) that the integral closure of $\CC[P]$ in $\CC[P^{\rm gp}]$ is $\CC[P^{\rm sat}]$, so, since a smooth $\CC$-algebra is normal, $P$ must be saturated, hence fs, hence $\ov{P}^{\rm gp} \cong \ZZ^r$ is free by Lemma~\ref{lem:saturated} and we have $P \cong P^* \oplus \ov{P}$ by Lemma~\ref{lem:splitting}, and hence the smoothness of $\CC[P]$ implies that $\CC[\ov{P}]$ is also smooth.  

We are thus reduced to proving that $P \cong \NN^r$ when $P$ is a sharp fs monoid with $\CC[P]$ smooth and $P^{\rm gp} \cong \ZZ^r$.  Let $p_1,\dots,p_k$ be the irreducible elements of $P$ (\S\ref{section:monoidbasics}).  These $p_i$ generate $P$ and hence also $P^{\rm gp}$, so we must have $k \geq r$.  If $k=r$ then there cannot be any nontrivial relations among the $p_i$ because such a relation would imply $P^{\rm gp}$ has rank $< r$.  We thus reduce to showing $k=r$.  By standard finiteness results, we can lift the surjection $\NN^k \to P$ taking $e_i \to p_i$ to a presentation of $P$ as a coequalizer $\NN^m \rightrightarrows \NN^k \to P$ (the ``standard finiteness results" are only needed to ensure that $m$ is finite, which is not actually necessary in the present proof).  Pick a standard basis vector $e_j \in \NN^m$.  The images of $e_j$ under the parallel arrows are of the form $\sum_{i=1}^k a_i e_i$ and $\sum_{i=1}^k b_i e_i$ for $a_i,b_i \in \NN$, and the relation \be \sum_{i=1}^k a_i p_i & = & \sum_{i=1}^k b_i p_i \ee holds in $P$.  Since $P$ is sharp and the $p_i$ are irreducible, we must have $\sum_i a_i, \sum_i b_i > 1$.  Let $\Omega$ denote the module of K\"ahler differentiables of $\CC[P]$ relative to $\CC$; $\Omega$ is locally free since $\CC[P]$ is smooth.  Since $P$ is sharp, there is a unique monoid homomorphism $x : P \to \CC$ mapping $0 \in P$ to $1 \in \CC$ and mapping every non-zero element of $P$ to zero.  The corresponding $\CC$-algebra surjection $x : \CC[P] \to \CC$ is called the \emph{cone point} ($x$ is a $\CC$ point of $\Spec \CC[P]$).  The kernel of the $\CC$-algebra map $x$ is $\m_x = \CC[P \setminus \{ 0 \}] = \CC[\m_P]$.  Notice that $\Spec \CC[P^{\rm gp}] \cong \Spec \CC[\ZZ^r]$ is an open subscheme of $\Spec \CC[P]$ which is smooth, connected, and of constant dimension $r$.  The point $x$ is in the closure of this open subscheme because $x$ is the limit as $t \to 0$ of the $\CC$-points $y_t$ of $\Spec \CC[P^{\rm gp}]$ corresponding to the unique monoid homomorphism $y : P^{\rm gp} \cong \ZZ^r \to \CC^*$ taking each $e_i \in \ZZ^r$ to $t \in \CC^*$.  Since $\Omega$ is locally free and clearly of rank $r$ on $\Spec \CC[P^{\rm gp}]$, we conclude that $$\Omega|_x := \Omega \otimes_{\CC[P]} (\CC[P]/\m_x) = \Omega / \m_x \Omega$$ is isomorphic to $\CC^r$.  But on the other hand, the presentation of $P$ above gives rise to a presentation of $\CC[P]$ as a coequalizer and hence also to a presentation of $\Omega$ as the quotient of the free $\CC[P]$ module on $d[p_1], \dots, d[p_k]$ by $m$ relations of the form $d r_j = 0$, where \be r_j & = &  \prod_{i=1}^k [p_i]^{a_i} - \prod_{i=1}^k [p_i]^{b_i} . \ee  But $\sum_i a_i, \sum_i b_i > 1$ implies that $r_j \in \m_x^2$, hence $dr_j \in \m_x \Omega$ by Liebnitz, hence $\Omega|_x$ has $\CC$ vector space basis $d[p_1],\dots,d[p_k]$, hence $r=k$ and we're done. \end{proof}

\begin{thm} \label{thm:fsmonoidsarevarieties} Let $P$ be a fine monoid such that $\ov{P}^{\rm gp}$ is torsion-free.  Then $\Spec \CC[P]$ is a complex variety in the sense of Definition~\ref{defn:complexvariety}.  The variety $\Spec \CC[P]$ is connected iff $P^*$ is torsion-free and normal iff $\ov{P}$ is saturated. \end{thm}

\begin{proof} Certainly $\Spec \CC[P]$ is separated and of finite type over $\CC$ since it is affine and $P$ is finitely generated.  Since $\ov{P}^{\rm gp}$ is torsion-free, we can find an isomorphism of monoids $P = P^* \oplus \ov{P}$ and hence an isomorphism of complex varieties \bne{productdecomp} \Spec \CC[P] & = & \Spec \CC[P^*] \times \Spec \CC[\ov{P}] . \ene  Since $P$ is fine, $P^*$ is a finitely generated abelian group, so we can pick an isomorphism $P \cong \ZZ^r \oplus T$ with $T$ finite.  Since $\CC$ is algebraically closed of characteristic zero, $\Spec \CC[T]$ is isomorphic as a $\CC$ scheme to a disjoint union of $|T|$ copies of $\Spec \CC$.  Since $\ov{P}^{\rm gp} \cong \ZZ^s$ is torsion-free, \be \Spec \CC[\ov{P}^{\rm gp}] & \cong & \Spec \CC[x_1,x_1^{-1}, \dots , x_s,x_x^{-1}] \ee is an integral domain, hence $\CC[\ov{P}]$ is an integral domain because it is a subring of $\CC[\ov{P}^{\rm gp}]$ since $\ov{P}$ is integral by Lemma~\ref{lem:splitting}.  The first two statements now follow from the product decomposition \eqref{productdecomp}, which also shows that $\Spec \CC[P]$ is normal iff $\Spec \CC[\ov{P}]$ is normal iff $\CC[\ov{P}]$ is integrally closed in $\CC[\ov{P}^{\rm gp}]$.  By \cite[1.13.4]{loggeometry} the latter is equivalent to saturation of $\ov{P}$. \end{proof}

\begin{example} If $P$ is a fine monoid, then $\Spec \CC[P]$ need not be a complex variety in the sense of Definition~\ref{defn:complexvariety}.  The ring $\CC[P]$ is always reduced because it is contained in the reduced ring $\CC[P^{\rm gp}]$, but it need not be irreducible.  In Example~\ref{example:nocharacteristicchart} we encountered the submonoid $P$ of $\ZZ \oplus \ZZ / 2 \ZZ$ generated by $x = (1,0)$ and $y = (1,1)$.  This monoid $P$ is presented as the free monoid on $x,y$ subject to the single relation $2x=2y$, so \be \CC[P] & \cong & \CC[x,y]/(x^2-y^2) \\ & \cong & \CC[u,v]/(uv), \ee which is reducible. \end{example} 

\begin{rem}  Theorem~\ref{thm:torusinvariantideals}, Theorem~\ref{thm:smoothifffree}, and Theorem~\ref{thm:fsmonoidsarevarieties} continue to hold (by essentially the same proofs) when $\CC$ is replaced by an arbitrary algebraically closed field $k$, provided one always assumes that the torsion subgroup of $P^{\rm gp}$ (and hence also the torsion subgroup of $P^* \subseteq P^{\rm gp}$) has order prime to the characteristic of $k$. \end{rem}

\section{Monoidal spaces III} \label{section:monoidalspacesIII}

\subsection{Refinements and group isomorphisms} \label{section:LMSrefinements}  Here we define various types of \emph{refinements}; these refinements are analogous to refinements of fans in toric geometry and will play a central role in the geometric realization constructions of \S\ref{section:geometricrealization}.

\begin{defn} \label{defn:groupisomorphism} An $\LMS$ (or $\MS$) morphism $f : X \to Y$ is called a \emph{group isomorphism} iff $(f^\dagger)^{\rm gp} : f^{-1} \M_Y^{\rm gp} \to \M_X^{\rm gp}$ is an isomorphism (of sheaves of abelian groups) on $X$. \end{defn}

\begin{lem} \label{lem:groupisomorphism} Group isomorphisms are closed under composition and base change (in $\MS$, or, equivalently, in $\LMS$).  A map $f : X \to Y$ of locally finite type (resp.\ fine) fans is a group isomorphism iff $f$ is locally isomorphic to $\Spec h$ for a map $h : Q \to P$ of finitely generated (resp.\ fine) monoids such that $h^{\rm gp}$ is an isomorphism. \end{lem}

\begin{proof} It is trivial to see that group isomorphisms are closed under composition.  Group isomorphisms are closed under base change in $\MS$ because groupification commutes with direct limits of (sheaves of) monoids.

For the ``if" in the second statement, the question of whether $(f^\dagger)^{\rm gp}$ is an isomorphism is local on $f$, so we reduce to proving that it is an isomorphism when $f = \Spec h$ with $h : Q \to P$ a map of monoids such that $h^{\rm gp}$ is an isomorphism (we don't need the finiteness or integrality for this implication).  For $\p \in \Spec P$ with corresponding face $F$, the stalk of $f^\dagger$ at $\p$ is just the map $(h^{-1}(F))^{-1} Q \to F^{-1} P$ obtained by localizing $h$, which induces the isomorphism $h^{\rm gp}$ on groupifications, so that stalks of $(f^\dagger)^{\rm gp}$ are isomorphisms.

The ``only if" follows easily from the commutative diagram \eqref{fanmapdiagram} of \S\ref{section:fans}---one may take $h = f^\dagger_x$ to produce a neighborhood of $x$ in $f$ with the desired property. \end{proof}

\begin{defn} \label{defn:LMSrefinement} \emph{Good refinements} are the smallest class of $\LMS$ morphisms which are local, closed under composition and base change, and contain the maps $\Spec h : \Spec P \to \Spec Q$ for each good refinement of fine monoids (Definition~\ref{defn:monoidrefinement}) $h : Q \to P$. \emph{Refinements} (resp.\ \emph{strong refinements}) are the smallest class of $\SMS$ morphisms which are local, closed under composition and base change, and contain the maps $\Spec h : (\Spec P,\ov{\M}_P) \to (\Spec Q,\ov{\M}_Q)$ for each refinement (resp.\ strong refinement) of fine monoids (Definition~\ref{defn:monoidrefinement}) $h : Q \to P$. \end{defn}

Since $\LMS \to \SMS$ preserves fibered products and the notion of ``local" is ``the same" in both categories it is easy to see that $\sms{f}$ is a strong refinement for each good refinement $f$ in $\LRS$ (c.f.\ Lemma~\ref{lem:refinements}\eqref{sharpening}).

\begin{lem} \label{lem:goodrefinement} A group isomorphism of fine fans is a good refinement. \end{lem}

\begin{proof} The question is local by definition, so by the local description of such group isomorphisms in Lemma~\ref{lem:groupisomorphism} we reduce to proving that $\Spec h$ is a good refinement whenever $h$ is a map of fine monoids such that $h^{\rm gp}$ is an isomorphism---but then $h$ is a good refinement of monoids by Lemma~\ref{lem:refinements}\eqref{groupiso}. \end{proof}

\subsection{Blowup revisited} \label{section:blowup2} Let $P$ be a fine monoid, $I \subseteq P$ an ideal.  Recall (\S\ref{section:blowup}) that the blowup $\Bl_I P \to P$ is a map of fine fans.  In fact:

\begin{lem} \label{lem:blowup} For a fine monoid $P$ and an ideal $I \subseteq P$, the blowup $\Bl_I P \to P$ is a group isomorphism of fine fans.  In particular it is a good refinement (Lemma~\ref{lem:goodrefinement}). \end{lem}

\begin{proof} This is clear from the description of the blowup in Proposition~\ref{prop:blowup} and the characterization of group isomorphisms in Lemma~\ref{lem:groupisomorphism}. \end{proof}

Let $X$ be a locally monoidal space, $J \subseteq \M_X$ an ideal sheaf.  We define the \emph{blowup} $\Bl_J X \to X$ of $X$ along $J$ by \be \Bl_J X & := & \Proj_X \left ( \M_X \coprod J \coprod J^2 \coprod \cdots \right ) \ee (recall the abusive notation $J^n$ from \S\ref{section:blowup}).  Assume now that $X$ is fine (resp.\ a fine fan) and $J$ is coherent.  Then, locally on $X$, there is a fine monoid $P$, an ideal $I \subseteq P$, and a strict map (resp.\ an isomorphism---c.f.\ Lemma~\ref{lem:coherentsheaves2}) $f : X \to \Spec P$ with $f^{-1}I^\sim \cong J$ as $f^{-1} \M_P \cong \M_X$ modules.  By Lemmas~\ref{lem:relativeProj} and \ref{lem:pullbackProj}, we have (locally on $X$) a cartesian diagram \bne{blowupdiagram} &  \xym{ \Bl_J X \ar[d] \ar[r] & \Bl_I P \ar[d] \\ X \ar[r] & \Spec P } \ene in $\LMS$ (the horizontal arrows are strict so the diagram is also cartesian in $\MS$).  The map $\Bl_I P \to \Spec P$ is a good refinement (resp.\ a group isomorphism of fine fans) by Lemma~\ref{lem:blowup}.  Good refinements are local and closed under base change (by definition).  Similarly, group isomorphisms of fine fans are also local, so we have: 

\begin{thm} \label{thm:blowup} Let $X$ be a fine locally monoidal space (resp.\ fine fan), $J \subseteq \M_X$ a coherent ideal.  Then $\Bl_J X \to X$ is a good refinement of fine locally monoidal spaces (resp.\ a group isomorphism of fine fans).  \end{thm}

There are analogous results for the blowup of a fine sharp monoidal space $X$ along a coherent ideal $J$.  Recall that we let $\ov{\Bl}_J X$ denote the sharpening of $\Bl_J X$.  By definition of coherent ideal, we can find, locally on $X$, a fine monoid $P$, an ideal $I \subseteq P$, and a strict morphism $f : X \to (\Spec P,\ov{\M}_P)$ such that $J \cong f^* \ov{I}^{\sim}$.  By the same lemmas used above, we have a cartesian diagram \bne{sharpenedblowupdiagram} &  \xym{ \Bl_J X \ar[d] \ar[r] & \Bl_{\ov{I}^\sim} (\Spec P,\ov{\M}_P) \ar[d] \ar[r] & \Bl_I P \ar[d] \\ X \ar[r]^-f & (\Spec P,\ov{\M}_P) \ar[r] & (\Spec P,\M_P) } \ene in $\LMS$.  The sharpening functor $\LMS \to \SMS$ preserves inverse limits so the big square from the above diagram sharpens to a cartesian square: $$ \xym{ \ov{\Bl}_J X \ar[r] \ar[d] & \ov{\Bl}_I P \ar[d] \\ X \ar[r] & (\Spec P, \ov{\M}_P) } $$ is cartesian in both $\LMS$ and $\SMS$.  The right vertical arrow is just the sharpening of the blowup map discussed above, hence it is a strong refinement.  Since strong refinements are local and stable under base change, we have:

\begin{thm} \label{thm:sharpenedblowup} Let $X$ be a fine sharp monoidal space (resp.\ sharp fine fan), $J \subseteq \M_X$ a coherent ideal.  Then $\ov{\Bl}_J X \to X$ is a strong refinement of $\SMS$ (resp.\ of sharp fine fans). \end{thm}

\subsection{Differential realization of blowups}  We will need some results about realization of blowups in the differential setting.

\begin{lem} \label{lem:blowuprealizationlogsmooth} Let $f : X \to Y$ be a map of fine fans locally isomorphic to $\Spec h$ for monic $h$ (for example, the blowup of a fine fan along a coherent ideal---Proposition~\ref{prop:blowup}).  Then the realization $\AA(f)$ of $f$ in $\LDS$ or $\PLDS$ (or $\Sch_{\QQ}$) is log smooth. \end{lem}

\begin{proof} Log smoothness is local (Proposition~\ref{prop:logsmooth}) and the realization preserves open embeddings, so we reduce to the situation discussed in Example~\ref{example:RPlogsmooth}. \end{proof}

\begin{lem} \label{lem:blowuprealizationproper} Let $P$ be a fine monoid, $I \subseteq P$ an ideal.  The realization $\AA(\Bl_I P) \to \AA(P)$ of the blowup $\Bl_I P \to \Spec P$ in $\LDS$ is projective (that is, the underlying $\DS$ morphism is projective) and the realization of this blowup in $\PLDS$ is proper Euclidean (Definition~\ref{defn:properEuclidean}). \end{lem}

\begin{proof} Since $I$ is finitely generated (\S\ref{section:idealsandfaces}), we can find a finite subset $S \subseteq I$ so that $\coprod_S P \to I$ is a surjection of $P$ modules.  This induces a surjection \be h : \Sym^*_P ( \coprod_S P) & \to & P \coprod I \coprod I^2 \coprod \cdots \ee of graded monoids (under $P$).  Applying $\Proj$ yields a map of fans \be \Proj h : \Bl_I P & \to & (\Spec P) \times \PP^n \ee over $\Spec P$, where $n := |S|-1$.  It is enough to prove that the realization $\AA(\Proj h)$ is a closed embedding because \be \AA( \Spec P \times \PP^n) & = & \AA(P) \times \AA(\PP^n) \ee (since \eqref{FanstoLogEsp} preserves products) is projective / proper Euclidean over $\AA(P)$ (since $\AA(\PP^n)$ is a compact differentiable space, as we saw in Example~\ref{example:realizationoffans}).  Being a closed embedding is local on the base and $\AA$ preserves inverse limits, so it is enough to prove that, locally on the base, $\Proj h$ is a map of fans whose realization is a closed embedding.  But when $h$ is surjective, we saw in \S\ref{section:blowup} that, locally on the base, $\Proj h$ is $\Spec$ of a surjection of monoids, hence its realization is a closed embedding. \end{proof}

\begin{lem} \label{lem:surjectivity2} Suppose $f : X \to Y$ is a map of fine fans whose $\CC$ scheme realization is surjective.  Then the $\PLDS$ realization of $f$ is also surjective.  Suppose, furthermore, that $f$ is a group isomorphism and $Y$ is fs.  Then the $\LDS$ realization of $f$ is also surjective. \end{lem}

\begin{proof} By working locally, the result follows from Lemma~\ref{lem:surjectivity}. \end{proof}

\subsection{Resolution of singularities} \label{section:functorialresolution}  In this section we use functorial resolution for complex varieties (c.f.\ Definition~\ref{defn:complexvariety}) to obtain a functorial resolution for fs locally monoidal spaces.

\begin{defn} \label{defn:freelocus} Let $X$ be a monoidal space.  The \emph{free locus} in $X$ is \be X^{\rm free} & := & \{ x \in X : \ov{\M}_{X,x} \; {\rm is \; free} \, \}. \ee  If $X^{\rm free}=X$, then we say that $X$ is \emph{free}.  \end{defn}

Note that the free locus of $X$ is defined in terms of the underlying sharp monoidal space $\sms{X}$.

\begin{lem} For any monoid $P$, the free locus in $\Spec P$ is open.  If $f : X \to Y$ is a strict $\LMS$ morphism, then $X^{\rm free} = f^{-1}(Y^{\rm free})$.  For any quasi-coherent locally monoidal space $X$, the free locus in $X$ is open. \end{lem}

\begin{proof} To prove the free locus of $\Spec P$ is open, it suffices to prove that is is stable under generalization, so let $\p \subseteq \q$ be primes of $P$ with complementary faces $F \supseteq H$.  We need to prove that $P/F$ is free when $P/H$ is free.  It is straightforward to see that $P/H \to P/F$ is the quotient of $P/H$ by the \emph{face} $F/H$, so it suffices to prove that the quotient of a free monoid by a face is again free, which is easy: the faces of $\oplus_S \NN$ are the submonoids $\oplus_T \NN$, where $T \subseteq S$, and the corresponding quotient is the free monoid $\oplus_{S \setminus T} \NN$ (this is clear from the proof of Lemma~\ref{lem:faces}).  The second statement is clear from the definitions and the third statement follows from the first two because the question is local on $X$. \end{proof}

\begin{defn} \label{defn:idealsequence} Let $X$ be a (fine) locally monoidal space, sharp monoidal space, or a (locally noetherian) scheme.  An \emph{ideal sequence} in $X$ is a sequence $I=(I_1,I_2,\dots,I_m)$ consisting of a coherent ideal $I_1$ of $X_1 := X$, a coherent ideal $I_2$ of $X_2 := \Bl_{I_1} X_1$, a coherent ideal $I_3$ of $X_3 := \Bl_{I_2} X_2$, and so forth.  We often write $\Bl_I X$ for $\Bl_{I_m} X_m$.\footnote{In the context of sharp monoidal spaces, the blowup denoted $\Bl$ is replaced everywhere by the sharpened blowup $\ov{\Bl}$ of \S\ref{section:blowup2}.  Accordingly, we write $\ov{\Bl}_I X$ instead of $\Bl_I X$.}  Two ideal sequences are called \emph{equivalent} iff they become equal after possibly inserting the unit ideal at several points in each sequence. \end{defn}

If $f : X \to Y$ is a map of such spaces and $I$ is an ideal sequence in $Y$, then we define an ideal sequence $f^{-1}I =: J$ in $Y$, called the \emph{inverse image ideal sequence} as follows.  We first set $J_1 := f^{-1}I_1$, so there is a natural map $f_2 : Y_2 := \Bl_{J_1} Y \to X_2$; we then set $J_2 := f^{-1}_2 I_2$, and so forth.  It is clear that the inverse images of equivalent ideal sequences are equivalent and that formation of inverse images is compatible with composition in the usual sense.  We will be mostly interested in this construction when $f : U \into X$ is the inclusion of an open subspace, in which case we write $I|U$ instead of $f^{-1} I$; in this situation the aforementioned natural maps $f_2,f_3,\dots$ are also inclusions of open subspaces.

If $I$ and $J$ are two equivalent ideal sequences, then it is clear that $\Bl_I X = \Bl_J X$.  In fact, it is clear that one also has such a natural isomorphism as long as $X$ can be covered by opens $U_i$ such that $I|U_i$ and $J|U_i$ are equivalent for every $i$.  It is natural then to consider the sheaf on $X$ associated to the presheaf taking an open subset $U \subseteq X$ to the set of equivalence classes of ideal sequences on $U$.

\begin{defn} \label{defn:generalizedidealsequence} A global section $I$ of the sheaf defined immediately above will be called a \emph{generalized ideal sequence} on $X$.  The \emph{blowup} $\Bl_I X$ of a generalized ideal sequence is defined by taking local representative ideal sequences $I_i$ for $I$ on the opens $U_i$ in a cover $\{ U_i \}$ of $X$ and gluing the blowups $\Bl_{I_i} U_i$. \end{defn}

We now recall the main results on functorial resolution of complex varieties (see \cite{BM}, \cite{W}, \cite{Kollar}, and the references therein).  There is a way of assigning, to every complex variety $X$, an ideal sequence $I=I(X)$ in $X$ in a manner satisfying the following properties: \begin{enumerate}[label=FR\theenumi., ref=FR\theenumi] \item \label{normalblowups} The blowups $$X_2 = \Bl_{I_1} X, \; X_3 := \Bl_{I_2} X_2, \; \dots, \; \Bl_I X $$ are again complex varieties, normal if $X$ is normal, and the maps $$\cdots \to X_3 \to X_2 \to X$$ are surjective.  \item \label{resolution} For any complex variety $X$, the blowup $\Bl_{I(X)} X$ is smooth and $\Bl_{I(X)} X \to X$ is an isomorphism over the smooth locus of $X$.  \item \label{automorphisminvariance} For any isomorphism $f : X \to Y$ of complex varieties, $f^{-1}I(Y) = I(X)$.  \item \label{varietiessmoothinvariance} For any smooth map $f : X \to Y$ of complex varieties, $f^{-1} I(Y)$ is equivalent to $I(X)$. \end{enumerate}  Note that smooth maps are flat, so the inverse image ideal $f^{-1} I(Y)$ is identified with $f^* I(Y)$ as an $\O_X$ module.

We are now going to propagate the above functorial resolution into the world of monoids.

\begin{thm} \label{thm:resolutionofmonoids} There is a way of assigning an ideal sequence $I(P)$ in $F := \Spec P$ for each fs monoid $P$ satisfying the following properties: \begin{enumerate}[label=MR\theenumi., ref=MR\theenumi] \item \label{saturatedblowups} The monoidal spaces $$F := \Spec P, \; F_2 := \Bl_{I(P)_1} F, \; F_3 := \Bl_{I(P)_2} F_2, \; \dots, \; \Bl_{I(P)} F $$ are fs fans and the $\CC$-scheme realizations of the maps $$\cdots \to F_3 \to F_2 \to F$$ are surjective.  \item \label{freeresolution} $\Bl_{I(P)} F$ is a free fs fan and $\Bl_{I(P)} F \to F$ is an isomorphism over the free locus of $F$.  \item \label{isomorphisminvariance} If $h : Q \to P$ is an isomorphism of fs monoids, then $I(P) = (\Spec h)^{-1} I(Q)$.  \item \label{smoothinvariance} If $h : Q \to P$ is a map of fs monoids such that $\Spec \CC[P] \to \Spec \CC[Q]$ is a smooth morphism of complex varieties, then $I(P)$ is equivalent to $(\Spec h)^{-1} I(Q)$. \end{enumerate} \end{thm}

\begin{proof}  Let $P$ be an fs monoid, $F := \Spec P$ the associated fs fan.  By Theorem~\ref{thm:fsmonoidsarevarieties}, $X := \Spec \CC[P]$ is a normal complex variety, so the functorial resolution for complex varieties yields an ideal sequence $J=(J_1,J_2,\dots,J_m)$ on $X$ such that $\Bl_J X$ is a smooth complex variety and $$X_1 := X, \; X_2 := \Bl_{J_1} X_1, \; \dots, \; \Bl_J X $$ are normal complex varieties.  By \eqref{automorphisminvariance}, the ideal $J_1$ of $X=X_1$ is invariant under all $\CC$ algebra automorphisms of $\CC[P]$.  In particular, $J_1$ is invariant under the automorphisms \eqref{aut}, hence $J_1 = \CC[I_1]$ for an ideal $I_1 \subseteq P$ by Theorem~\ref{thm:torusinvariantideals}.  Then $F_2 = \Bl_{I_1} P$ is a fine fan by Theorem~\ref{thm:blowup} and $X_2 = \AA( \Bl_{I_1} P)$ is the $\CC$ scheme realization of $F_2$ (its image under the functor from fine fans to locally finite type $\CC$ schemes discussed in \S\ref{section:realizationoffans}) by Lemma~\ref{lem:blowuprealization}.  Cover $F_2$ by $\Spec Q_i$ for various fine monoids $Q_i$.  Then we have a corresponding cover $\Spec \CC[Q_i]$ of $X_2$.  Since $X_2$ is normal, $Q_i$ is saturated, thus we see that $F_2$ is in fact an fs fan.  Repeating the same argument, we see that $X_3 = \AA(F_3)$ and so forth, and that the $F_i$ are fs fans.  Since $\Bl_J X = \AA( \Bl_I P)$ is smooth, it follows from Theorem~\ref{thm:smoothifffree} that the fs fan $\Bl_{I(P)} P$ is free.  Similarly, the map $\Bl_I P \to \Spec P$ is an isomorphism over $(\Spec P)^{\rm free}$ because $\Bl_J X \to X$ is an isomorphism over the smooth locus of $X$.  The properties \eqref{isomorphisminvariance} and \eqref{smoothinvariance} are clearly inherited from the properties \eqref{automorphisminvariance} and \eqref{varietiessmoothinvariance} of the functorial resolution for complex varieties. \end{proof} 

\begin{defn} \label{defn:canonicalideal} An assignment of an ideal sequence $I(P)$ to each fine monoid $P$ satisfying the properties of the above theorem will be called a \emph{canonical resolution} of fs monoids.  The ideal sequence $I(P)$ on $\Spec P$ will be called the \emph{canonical ideal sequence}. \end{defn}

\begin{lem} \label{lem:resolutionofmonoids} Let $P$ be an fs monoid, $F$ a face of $P$, $l : P \to F^{-1} P$ the localization map, $\pi : F^{-1} P \to P/F$ the sharpening map.  Then: \begin{enumerate} \item \label{localizationinvariance} $I(F^{-1} P)$ is equivalent to $(\Spec l)^{-1} I(P)$.  \item \label{sharpeninginvariance} $I(P/F)$ is equivalent to $(\Spec \pi)^{-1} I(F^{-1} P)$.  \end{enumerate} \end{lem}

\begin{proof} For \eqref{localizationinvariance}, just note that $\Spec \CC[l]$ is an open embedding, hence smooth, so this is immediate from \eqref{smoothinvariance}.  For \eqref{sharpeninginvariance}, Lemma~\ref{lem:fssplitting} yields an isomorphism $F^{-1}P \cong F^{\rm gp} \oplus P/F$ compatible with the projection $\pi$, so that $\pi$ has a section $s : P/F \to F^{-1} P$ yielding the above splitting.  Then $\Spec \CC[s]$ is a smooth morphism of complex varieties because $\Spec \CC[F^{\rm gp}]$ is a smooth complex variety, hence $(\Spec s)^{-1} I(P/F)$ is equivalent to $I(F^{-1}P)$ by \eqref{smoothinvariance}.  But \be I(P/F) & = & (\Spec \pi)^{-1} (\Spec s)^{-1} I(P/F) \ee because the composition $(\Spec s)(\Spec \pi)$ is the identity, so the result follows.  \end{proof}

\begin{thm} \label{thm:LMSresolution}  Let $X$ be an fs locally monoidal space.  Fix an open cover $\{ U_i \}$ of $X$, fs monoids $P_i$, and strict $\LMS$ morphisms $a_i : U_i \to \Spec P_i$.  Let $I_i := a_i^{-1} I(P_i)$.  Then the ideal sequences $I_i$ glue to a generalized ideal sequence $I$ on $X$ and we have cartesian $\LMS$ diagrams $$ \xym{ \Bl_I X \ar[d] & \ar[l] \Bl_{I_i} U_i \ar[r] \ar[d] & \Bl_{I(P_i)} \Spec P_i \ar[d] \\ X & \ar[l] U_i \ar[r]^-{a_i} & \Spec P_i } $$ with strict horizontal arrows for every $i$.  In particular, $\Bl_I X$ is a free fs locally monoidal space and $\Bl_I X \to X$ is a strong refinement which is an isomorphism over the free locus of $X$. \end{thm}

\begin{proof} Consider a point $x \in U_{ij}$.  We want to show that $I_i$ and $I_j$ are equivalent on a neighborhood of $x$.  Let $\p_i := a_i(x) \in \Spec P_i$ and let $F_i \subseteq P_i$ be the complementary face.  The stalk of $\M_{P_i}$ at $\p_i$ is $F_i^{-1} P_i$ (c.f.\ \eqref{stalkformula}), so, since $a_i$ is strict, we have an isomorphism $a_{i,x} : F_i^{-1} P_i  \to  \M_{X,x}. $  Let $V_i := a_i^{-1}(\Spec (F_i^{-1} P_i))$ so we have a commutative $\LMS$ diagram $$ \xym{ U_i \ar[r]^-{a_i} & \Spec P_i \\ V_i \ar[u] \ar[r] & \Spec F_i^{-1} P_i \ar[u] } $$ where the vertical arrows are open embeddings (it is important here that $F_i$ is finitely generated---\S\ref{section:Spec}) and the horizontal arrows are strict.  Make the same definitions with $i$ replaced by $j$.  Then we have an isomorphism $a_{j,x} : F_j^{-1} P_j \to \M_{X,x}$ and hence an isomorphism $b := a_{j,x}^{-1} a_{i,x} : F_i^{-1}P_i \to F_j^{-1}P_j$.  The diagram $$ \xym{ F_i^{-1}P_i \ar[r]^-{a_i} \ar[rd]_b & \M_X(V_{ij}) \\ & F_j^{-1}P_j \ar[u]_{a_j} } $$ may not commute, but it does commute after composing with $\M_X(V_{ij}) \to \M_{X,x}$, so, since $F_i^{-1}P_i$ is finitely generated, it will commute after replacing $V_{ij}$ with some smaller neighborhood $W$ of $x$.  We thus obtain a commutative $\LMS$ diagram $$ \xym@C+20pt@R-10pt{ U_i \ar[r]^{a_i} & \Spec P_i \\ & \Spec F_i^{-1}P_i \ar[u] \\ W \ar[ru]^{a_i|W} \ar[rd]_{a_j|W} \ar[dd] \ar[uu] \\ & \Spec F_j^{-1}P_j \ar[d] \ar[uu]^{\cong}_{\Spec b} \\ U_j \ar[r]^{a_j} & \Spec P_j } $$ from which we compute \be I_i|W & = & (a_i^{-1} I(P_i))|W  \\ & = & (a_i|W)^{-1}(I(P_i)| \Spec F_i^{-1}P_i) \\ & \sim & (a_i|W)^{-1}(I(F_i^{-1}P_i)) \\ & = & (a_j|W)^{-1}(\Spec b)^{-1}(I(F_i^{-1}P_i)) \\ & = & (a_j|W)^{-1}(I(F_j^{-1}P_j)) \\ & \sim & I_j|W, \ee using Lemma~\ref{lem:resolutionofmonoids}. 

The left square in the diagram is cartesian by construction of $\Bl_I X$.  The right square is an instance of the cartesian square \eqref{blowupdiagram} of \S\ref{section:blowup2}. \end{proof}

We leave it to the reader to state the analog of Theorem~\ref{thm:LMSresolution} for fs \emph{fans}.  Our next result is the analog of Theorem~\ref{thm:LMSresolution} for \emph{sharp} monoidal spaces.  It is proved in the same way as Theorem~\ref{thm:LMSresolution}.

\begin{thm} \label{thm:SMSresolution} Let $X$ be an fs sharp monoidal space.  Fix an open cover $\{ U_i \}$ of $X$, fs monoids $P_i$ (which may as well be sharp), and strict $\LMS$ morphisms $a_i : U_i \to (\Spec P_i,\ov{\M}_{P_i})$.  Let $I_i := a_i^{-1} \ov{I(P_i)}$.  Then the ideal sequences $I_i$ glue to a generalized ideal sequence $I$ on $X$ and we have diagrams (cartesian in $\LMS$ and $\SMS$) $$ \xym{ \ov{\Bl}_I X \ar[d] & \ar[l] \ov{\Bl}_{I_i} U_i \ar[r] \ar[d] & \ov{\Bl}_{\ov{I(P_i)}} (\Spec P_i,\ov{\M}_{P_i}) \ar[d] \\ X & \ar[l] U_i \ar[r]^-{a_i} & (\Spec P_i,\ov{\M}_{P_i}) } $$ with strict horizontal arrows for every $i$.  In particular, $\ov{\Bl}_I X$ is a free fs sharp monoidal space and $\ov{\Bl}_I X \to X$ is a strong refinement which is an isomorphism over the free locus of $X$. \end{thm}

The right vertical $\SMS$ morphism in the diagram of Theorem~\ref{thm:SMSresolution} is the sharpening of the $\LMS$ morphism \be \Bl_{I(P_i)} P_i & \to & \Spec P_i. \ee  Indeed, the canonical ideal (sequence) $\ov{I(P_i)}$ on $\Spec P_i,\ov{\M}_{P_i}$ is the pullback of the canonical ideal (sequence) $I(P_i)$ on $\Spec P_i$ under the sharpening map (Lemma~\ref{lem:resolutionofmonoids}), so the diagram $$ \xym{ \Bl_{\ov{I(P_i)}} (\Spec P_i,\ov{\M}_{P_i}) \ar[d] \ar[r] & \Bl_{I(P_i)} (\Spec P_i, \M_{P_i}) \ar[d] \\ (\Spec P_i,\ov{\M}_{P_i}) \ar[r] & (\Spec P_i,\M_{P_i}) } $$ is cartesian (Lemma~\ref{lem:pullbackProj}).  This diagram stays cartesian after sharpening since sharpening preserves inverse limits, thus we obtain the desired identification.  (Actually one makes a similar argument at each step in the sequence of blowups defining the blowup of an ideal sequence.)

\section{Geometric realization} \label{section:geometricrealizationmain} In \cite{KM}, Kottke and Melrose introduced a kind of ``generalized blowup" which---in our language---associates a $\PLDS$-morphism $X' \to X$ to a positive log differentiable spaces $X$ (arising as fibered products of manifolds with corners) equipped with certain ``refinement data" closely related to the usual notion of refinement of fans in toric geometry.  In \cite[\S9]{Kat2}, Kato introduced a similar ``subdivision" procedure and explained how it could be used to resolve certain log smooth schemes \cite[\S10]{Kat2}.  The construction was revisited by Nizio\l \, in \cite{Niz} who used it to prove resolution of singularities for \emph{all} log regular schemes \cite[5.3]{Niz}.  In fact, she only uses a special case of Kato's subdivision construction which she calls ``log blowup."  She explains how this is related to Kato's construction in between Remark~4.6 and Theorem~4.7 of \cite{Niz}.

The purpose of this section is to adapt the constructions and results of Kato and Nizio\l \, to the setting of log differentiable spaces, and in fact to the setting of general log spaces (\S\ref{section:logspaces}).  This will yield a general version of the Kottke-Melrose blowup which can be used to resolve the singularities of any log smooth space (\S\ref{section:resolution}).

Here is a sketch of our construction.  We will work in the categories $\fLogEsp$ and $\fSMS$ of \emph{fine} log spaces (\S\ref{section:logspaces}) and \emph{fine} sharp monoidal spaces (\S\ref{section:chartsandcoherence}), though many of our constructions use only integrality.  We make heavy use of the functor \bne{fLogEsptofSMS} \fLogEsp & \to & \fSMS \\ \nonumber X & \mapsto & \sms{X} = (|\u{X}|,\ov{\M}_X) \ene introduced in \S\ref{section:logspacestomonoidalspaces}.  We introduce the subcategories of $\fSMS$ consisting of \emph{refinable} and \emph{strongly refinable} morphisms and show that a composition of (coherent) sharpened blowups is a strongly refinable $\fSMS$ morphism.  For a fine log space $X$ and a refinable $\fSMS$ morphism $r : F \to \sms{X}$, we construct---following Kato--- a log smooth map of log spaces $f : X' \to X$ together with an $\SMS$ morphism $g : \sms{X}' \to F$.  If $r$ is strongly refinable, then $g$ is strict.  The construction can be applied, for example, to the strong refinement given by the saturation map $\sms{X}^{\rm sat} \to \sms{X}$ (\S\ref{section:integrationandsaturation}).  If $\sms{X}$ is saturated (i.e.\ $X$ is saturated), then our construction can be applied to a resolution of singularities $F \to \sms{X}$ as in Theorem~\ref{thm:SMSresolution} to yield a locally projective (hence ``proper" in any category of spaces where that word has a reasonable meaning), log smooth map of log spaces $X' \to X$ which is an isomorphism over the free locus of $X$.  In particular, if $X$ itself is log smooth, then $X'$ will be log smooth and free---i.e. $X'$ will be a manifold, manifold with corners, etc, depending on the category of spaces where we work.  Our construction is a mild generalization of Kato's subdivision construction.  

\subsection{Kato category} \label{section:Katocategory}  Fix a fine log space $X$ and an $\fSMS$ morphism $r : F \to \sms{X}$.  Then we have a functor \bne{Katofunctor1} K : \fLogEsp/X & \to & \fSMS / \sms{X} \\ \nonumber (f : Y \to X ) & \mapsto & (\sms{f} : \sms{Y} \to \sms{X}) \ene and an object $r : F \to \sms{X}$ in the codomain of $K$.  Following Kato, we consider the comma category \be \A(X,r) & := & (K \downarrow r : F \to \sms{X}) \ee in the sense of MacLane \cite[II.6]{Mac}.  Explicitly, an object of $\A(X,r)$ is a pair $(f,g)$ consisting of a $\LogEsp$ morphism $f : Y \to X$ and an $\SMS$ morphism $g : \sms{Y} \to F$ such that $rg = \sms{f}$.  A morphism from $(f,g)$ to $(f',g')$ in $\A(X,r)$ is an $\fLogEsp/X$ morphism $h : Y \to Y'$ such that $g = g' \sms{h}$.  We will be interested in the question of whether the category $\A(X,r)$ has a terminal object $(f,g)$ and, if so, whether the map $g$ is strict; the maps $f,g$ defining such a terminal object will then serve as the $f$, $g$ mentioned in the introductory sketch above.  The categorical framework introduced here will allow us to more easily address gluing issues, facilitating the construction of such terminal objects.

Although we will ultimately be interested in the situation above, it will be convenient to introduce a slightly more general version of $\A(X,r)$.  Fix a fine log space $X$ and $\fSMS$ morphisms $a : \sms{X} \to G$ and $r : F \to G$.  Then we have a functor \bne{Katofunctor2} K : \fLogEsp /X & \to & \fSMS /G \\ \nonumber (f : Y \to X) & \mapsto & (a\sms{f} : \sms{Y} \to G). \ene  By definition, the \emph{Kato category} $\A(X,a,r)$ is the comma category $(K \downarrow r : F \to G)$.  An object of $\A(X,a,r)$ is a pair $$(f : Y \to X, \, g : \sms{Y} \to F)$$ (which we often denote simply by $(f,g)$) consisting of a map of log spaces $f : Y \to X$ and a morphism $g : \sms{Y} \to F$ of sharp monoidal spaces over $G$---i.e.\ making the diagram $$ \xym{ \sms{Y} \ar[d]_{\sms{f}} \ar[r]^-g & F \ar[d]^r \\ \sms{X} \ar[r]^a & G} $$ commute.   A morphism $$( f :Y \to X, \, g : \sms{Y} \to F) \to (f' : Y' \to X, \, g' : \sms{Y}' \to F)$$ in $\A(X,a,h)$ is an $\fLogEsp/X$ morphism $h : Y \to Y'$ such that $g = g' \sms{h}$.  We recover the old $\A(X,r)$ via the formula \be \A(X,r) & = & \A(X,\Id : \sms{X} \to \sms{X},r).\ee  Again we will ask whether $\A(X,a,r)$ admits a terminal object $(f:X' \to X, g : \sms{X}' \to F)$ and, if so, whether the map \bne{terminalobjectmap} & \sms{f} \times g : \sms{X}' \to \sms{X} \times_G F \ene is strict.

Formation of the Kato category is functorial in the input data $X,a,r$ as follows:  Suppose $ z : X' \to X$ is an $\fLogEsp$ morphism and $$ \xym@C+10pt{ F' \ar[r]^-j \ar[d]_{r'} & F \ar[d]^r \\ G' \ar[r]^-i & G \\ \sms{X}' \ar[u]^{a'} \ar[r]^-{\sms{z}} & \sms{X} \ar[u]_a } $$ is a commutative $\fSMS$ diagram.  Then we have a functor \bne{firstfunctor} \A(X',a',r') & \to & \A(X,a,r) \\ \nonumber (f',g') & \mapsto & (zf',jg'). \ene  There are various circumstances under which \eqref{firstfunctor} admits a right adjoint.  First, if $z,i,j$ are all open embeddings, then it is easy to see that \bne{functor1} \A(X,a,r) & \to & \A(X',a',r') \\ \nonumber (f,g) & \mapsto & (f | (f^{-1}(X') \cap g^{-1}(F')), g| (f^{-1}(X') \cap g^{-1}(F')) ) \ene is right adjoint to \eqref{firstfunctor}.  Second, if $F' = F \times_G G'$, then it is easy to see that \bne{functor2} \A(X,a,r) & \mapsto & \A(X',a',r') \\ \nonumber (f,g) & \mapsto & (f \times_X X', (g \pi_1, a'  \pi_2)) \ene is right adjoint to \eqref{firstfunctor}.  Since right adjoints preserve inverse limits, we have:

\begin{lem} \label{lem:Katocategory1} The functors \eqref{functor1} and \eqref{functor2} preserve inverse limits, in particular terminal objects.\end{lem}

\begin{lem} \label{lem:Katocategory2} Let $X$ be a fine log space, $a : \sms{X} \to G$, $r : F \to G$ two $\fSMS$ morphisms, $\{ U_i \}$ an open cover of $X$.  Then $\A(X,a,r)$ has a terminal object $(f,g)$ iff $\A(U_i,a|\sms{U}_i,r)$ has a terminal object $(f_i,g_i)$ for each $i$.  When this is the case, \eqref{terminalobjectmap} is strict iff the analogous map is strict for each $i$. \end{lem}

\begin{proof}  Suppose $\A(X,a,r)$ has a terminal object $(f:X' \to X, g: \sms{X}' \to F)$.  Set $U_i' := f^{-1}(U_i)$, $f_i := f|U_i' : U_i' \to U_i$, $g_i := g | \sms{U}'_i$.  By Lemma~\ref{lem:Katocategory1} for \eqref{functor1} (with $i=\Id$, $j=\Id$, $X' = U_i$, $a' = a|\sms{U}_i$), $(f_i,g_i)$ is a terminal object of $\A(U_i,a|\sms{U}_i,r)$ for each $i$.  Conversely, suppose $\A(U_i,a|\sms{U}_i,r)$ has a terminal object $$(f_i : U_i' \to U_i, g_i : \sms{U}'_i \to F)$$ for each $i$.  Then by the same Lemma, for each $i,j$ the objects $$(f_i | U'_{ij} : U'_{ij} \to U_{ij}, g_i | \sms{U}'_{ij} : \sms{U}'_{ij} \to F)$$ and $$(f_j | U'_{ij} : U'_{ij} \to U_{ij}, g_j | \sms{U}'_{ij} : \sms{U}'_{ij} \to F)$$ are both terminal objects of $\A(U_{ij},a|\sms{U}_{ij},r)$, so they are identified via a unique isomorphism.  These isomorphisms trivially satisfy the cocycle condition on triple overlaps because there is again only one way to identify two terminal objects, so we can glue the $(f_i,g_i)$ to form an object $(f,g)$ of $\A(X,a,r)$.  This object is easily seen to be terminal because the unique map to it can be produced locally.  The statement about strictness is clear from the relationship between $(f,g)$ and the $(f_i,g_i)$ above and the local nature of strictness. \end{proof}

\subsection{Geometric realization} \label{section:geometricrealization}  We need to place hypotheses on the $\fSMS$ morphism $r : F \to G$ ensuring the existence of a terminal object in $\A(X,a,r)$.  It is natural to make the following

\begin{defn} \label{defn:realizable} A $\fSMS$ morphism $r : F \to G$ is called \emph{realizable} (resp.\ \emph{strongly realizable}) iff, for any fine log space $X$ and any $\fSMS$ morphism $a : \sms{X} \to G$, the category $\A(X,a,r)$ has a terminal object $(f:X' \to X, g : \ov{X}' \to F)$ (resp.\ and the morphism \eqref{terminalobjectmap} is strict). \end{defn}

Although this definition is not particularly concrete, it ensures that realizable morphisms are rather well behaved:

\begin{lem} \label{lem:realizablemorphisms}  Realizable morphisms (and strongly realizable morphisms) are closed under composition and base change in $\fSMS$ and ``realizable" and ``strongly realizable" are local properties of $\fSMS$ morphisms. \end{lem}

\begin{proof}  Suppose $r : F \to G$ and $s : G \to H$ are realizable and we want to prove $sr$ is realizable.  Let $X$ be a fine log space, $a : \sms{X} \to H$ an $\fSMS$ morphism.  Since $s$ is realizable, we have a terminal object $(f:X' \to X,g)$ of $\A(X,a,s)$.  Since $r$ is realizable, we have a terminal object $(k:X'' \to X', l)$ of $\A(X',g,r)$.  Then it is easy to see that $(fk,l)$ is terminal in $\A(X,a,sr)$.  If $r$ and $s$ are strongly terminal, then we draw a big cartesian diagram relating these terminal objects and use stability of strict $\fSMS$ morphisms under composition and base change to prove that $\sms{X}'' \to \sms{X} \times_H F$ is strict,\footnote{The argument is much like the argument for smoothness in the proof of Theorem~\ref{thm:logsmooth} using the ``big diagram" \eqref{bigcartesiandiagram}.} so that $sr$ is strongly realizable.

Stability under base change follows easily from Lemma~\ref{lem:Katocategory1} for \eqref{functor2}.  The local nature of realizable morphisms is proved by the same gluing arguments used to prove Lemma~\ref{lem:Katocategory2}. \end{proof}

\begin{lem} \label{lem:geometricrealization1} Let $X$ be a fine log space, $x$ a point of $X$, $Q$ a fine monoid, $a : \sms{X} \to (\Spec Q,\ov{\M}_Q)$ an $\fSMS$ morphism.  Then there exists a neighborhood $U$ of $x$ in $X$, a fine monoid $S$, a strict, surjective monoid homomorphism $h : S \to Q$, and a monoid homomorphism $b : S \to \M_X(U)$ such that the $\fSMS$ diagram $$ \xym{ \sms{X} \ar[r]^-a & (\Spec Q, \ov{\M}_Q) \ar[d]^h_{\cong} \\ \sms{U} \ar[u] \ar[r]^-{\sms{b}} & (\Spec S, \ov{\M}_S) } $$ commutes. \end{lem}

\begin{proof}  We can assume $Q$ is sharp and we can hence view the map $a$ as a monoid homomorphism $a : Q \to \ov{\M}_X(X)$.  Since $\M_X \to \ov{\M}_X$ is a surjection of sheaves and $Q$ is finitely generated, we can find a neighborhood $U$ of $x$ such that $a(q)|U \in \ov{\M}_X(U)$ lifts to $\M_X(U)$ for every $q \in Q$.  Define a monoid $R$ by making the right square in the diagram $$ \xym{ \O_X^*(U) \ar[r] \ar@{=}[d] & R \ar[d] \ar[r]^-h & Q \ar[d] \\ \O_X^*(U) \ar[r] & \M_X(U) \ar[r] & \ov{\M}_X(U) }$$ cartesian.  It is easy to see that $R^* = \O_X^*(U)$ and that $h$ is strict and surjective.  If we were not concerned about finiteness, we could simply take $S=R$.  Certainly $R$ is integral since it can be viewed as a submonoid of $\M_X(U) \times Q$.  By Lemma~\ref{lem:strictfiniteness} we can find a finitely generated (hence fine) submonoid $S \subseteq R$ such that $h|S :S \to Q$ is strict and surjective---this is as desired. \end{proof}

\begin{lem} \label{lem:geometricrealization2} Let $X$ be a fine log space, $h : Q \to P$ a refinement of fine monoids, $a : Q \to \M_X(X)$ a monoid homomorphism, $\sms{a} : \sms{X} \to (\Spec Q,\ov{\M}_Q)$ the induced map of sharp monoidal spaces.  The Kato category $\A(X,\sms{a},\Spec \ov{h})$ has a terminal object $$(f_T : T \to X, g_T : \sms{T} \to (\Spec P,\ov{\M}_P))$$ and the map $f_T$ is log smooth.\footnote{We have not defined log smoothness for general log spaces, so this statement is only meaningful for $\LDS$ and $\PLDS$, though it will be clear from the proof that $f_T$ is also log smooth in other contexts.}  If $h$ is a good refinement (resp.\ strong refinement), then $T \cong X \times_{\AA(Q)} \AA(P)$ (resp.\ \eqref{terminalobjectmap} is strict). \end{lem}

\begin{proof}  As in \S\ref{section:monoidrefinements}, we define a (fine) monoid $R$ by the \emph{cartesian} diagram \be \xym@C+10pt{ Q \ar@/_1pc/[rdd] \ar[rd]^i \ar@/^1pc/[rrd]^h \\ & R \ar[r]^-p \ar[d] & P \ar[d] \\ & Q^{\rm gp} \ar[r]^-{h^{\rm gp}} & P^{\rm gp} } \ee so the map $\ov{p}: \ov{R} \to \ov{P}$ has a section $\ov{s} : \ov{P} \to \ov{R}$ satisfying $\ov{i} = \ov{s} \ov{h}$ by definition of ``refinement" (Definition~\ref{defn:monoidrefinement}).  Set $T := X \times_{\AA(Q)} \AA(R)$.  The fibered product here is taken in \emph{integral} log spaces, so $T$ is a fine log space.  Let $f_T : T \to X$ be the projection.  There is a natural $\SMS$ morphism $g_T : \sms{T} \to (\Spec P, \ov{\M}_P)$ obtained by composing the projection $\sms{T} \to \sms{\AA(R)}$, the natural map $\sms{\AA(R)} \to (\Spec R, \ov{\M}_R)$, and the map $\Spec \ov{s}$, using the equalities \be (\Spec P, \ov{\M}_P) & = & (\Spec \ov{P}, \ov{\M}_{\ov{P}}) \\  (\Spec R, \ov{\M}_R) & = & (\Spec \ov{R}, \ov{\M}_{\ov{R}}). \ee    The situation is summed up by the commutative $\SMS$ diagram $$ \xym@C+10pt{ & (\Spec P,\ov{\M}_P) \ar[d]_{\Spec \ov{p}} \ar@{=}[rd] \\ \sms{T} \ar[r]^-{\pi} \ar[d]_{\sms{f}_T} & (\Spec R,\ov{\M}_R) \ar[rd]_{\Spec \ov{i}} \ar[r]^-{\Spec \ov{s}} & (\Spec P, \ov{\M}_P) \ar[d]^-{\Spec \ov{h}} \\ \sms{X} \ar[rr]^-{\sms{a}} & & (\Spec Q,\ov{\M}_Q) } $$ where $g_T = (\Spec \ov{s})\pi$.  In particular, the ``big" square on the bottom commutes, so we have constructed an object of $\A(X,a,h)$ as in the statement of the lemma; denote it by $T$ for simplicity.

We claim that this object $T$ is terminal.  The rest of the statements follow easily:  The map $f_T$ is log smooth because it is a base change of $\AA(i)$, which is log smooth because $i$ is an injective map of monoids (Example~\ref{example:RPlogsmooth}).\footnote{The map $i$ also induces an isomorphism $i^{\rm gp}$, which ensures log smoothness in other contexts where injectivity alone is insufficient.}  If $h$ is a strong refinement, then $\ov{p}$ is an isomorphism by definition and $\ov{s}$ must be its inverse (c.f.\ Lemma~\ref{lem:refinements}\eqref{sectionuniqueness}), so $\Spec \ov{s}$ is strict, hence so is \eqref{terminalobjectmap} because it is a composition of the map $\sms{T} \to \sms{X} \times_{\Spec \ov{Q}} \Spec \ov{R}$ (which is easily seen to be strict by using the obvious charts), and the map \bne{anotherstrictmap} \Id \times \Spec \ov{s} : \sms{X} \times_{\Spec \ov{Q}} \Spec \ov{R} & \to & \sms{X} \times_{\Spec \ov{Q}} \Spec \ov{P}, \ene which is a since it is a base change of the latter.

To see that $T$ is terminal, suppose $f : Y \to X$ is a map of log spaces and $g : \sms{Y} \to (\Spec P,\ov{\M}_P)$ is a map of monoidal spaces making \bne{commdiag1} & \xym{ \sms{Y} \ar[d]_{\sms{f}} \ar[r]^-g & (\Spec P,\ov{\M}_P) \ar[d]^{\Spec \ov{h}} \\ \sms{X} \ar[r]^-{\sms{a}} & (\Spec Q,\ov{\M}_Q) } \ene commute.  Commutativity of \eqref{commdiag1} is equivalent to commutativity of the diagram \bne{commdiag2} & \xym{ \ov{\M}_Y(Y) & \ar[l]_-g \ov{P} \\ \ov{\M}_X(X) \ar[u]^{\ov{f}^\dagger} & \ar[l]_-{\sms{a}} \ov{Q} \ar[u]_{\ov{h}} } \ene of monoids.   First we claim that the composition $$(f^\dagger)^{\gp} a^{\gp} : Q^{\gp} \to \M_Y^{\gp}(Y)$$ takes $R \subseteq Q^{\gp}$ into $\M_Y(Y) \subseteq \M_Y^{\gp}(Y).$  Indeed, this can be checked on the level of characteristics, where it is immediate from the definition of $R$ and commutativity of \eqref{commdiag2}.  So, we have a map of monoids $b := (f^\dagger)^{\gp} a^{\gp} : R \to \M_Y(Y)$ making the diagram of monoids \bne{commdiag3} & \xym{ \M_Y(Y) & \ar[l]_-b R \\ \M_X(X) \ar[u]^{f^\dagger} & \ar[l]_-a Q \ar[u]_i } \ene commute, hence making the diagram of log spaces $$ \xym{ Y \ar[d]_f \ar[r]^-b & \AA(R) \ar[d]^{\Spec i} \\ X \ar[r]^-a & \AA(Q) } $$ commute.  We thus obtain a $\LogEsp/X$ morphism \be k := (f,b) : Y \to T \ee using the fiber product description of $T$.  To show that $k$ determines a $\A(X,\sms{a},\Spec \ov{h})$ morphism to $T$, we must check that $g_T \sms{k} = g$.  By definition of $g_T$, this is equivalent to showing that $g=\ov{b} \ov{s} : \ov{P} \to \ov{\M}_Y(Y)$.  Using the equality $\ov{s}\ov{h} = \ov{i}$, \eqref{commdiag3}, and \eqref{commdiag2}, we see that \bne{mainequality} \ov{b}\ov{s} \ov{h} & = & \ov{b} \ov{i} \\ \nonumber & = & \ov{f}^\dagger \ov{a} \\ \nonumber & = & g \ov{h}.\ene  Since $h$ is a refinement, $\ov{h}^{\rm gp}$ is surjective (Lemma~\ref{lem:refinements}\eqref{groupsurjective}), so $\ov{h}$ is an epimorphism in the category of integral monoids (Lemma~\ref{lem:injectivepullback}), hence the equality \eqref{mainequality} implies the desired equality $g=\ov{b} \ov{s}$ because we assume $\ov{\M}_Y(Y)$ is integral.   

Finally, suppose $$k' = (f,b') : Y \to T = X \times_{\AA(Q)} \AA(R) $$ is another map of log spaces over $X$ with $g_T \sms{k}' = g$ (i.e.\ another $\A(X,\sms{a},\Spec h)$ morphism from $(f,g)$ to $T$).  The fact that $k' =(f,b')$ is a well-defined map to the fibered product means we have a commutative diagram $$ \xym{ \M_Y(Y) & \ar[l]_-{b'} R \\ \M_X(X) \ar[u]^{f^\dagger} & \ar[l]_-a Q \ar[u]_i } $$ of monoids.  Comparing with \eqref{commdiag3}, we see that $b'i = bi$.  But $i^{\rm gp}$ is an isomorphism (\S\ref{section:monoidrefinements}), so $i$ is an epimorphism in the category of integral monoids (Lemma~\ref{lem:injectivepullback}), so $b=b'$ (because $\M_Y(Y)$ is integral), so $k=k'$.  \end{proof}

\begin{thm} \label{thm:geometricrealization}  Refinements (resp.\ strong refinements) in $\fSMS$ (Definition~\ref{defn:LMSrefinement}) are realizable (resp.\ strongly realizable).  \end{thm}

\begin{proof}  From Definition~\ref{defn:LMSrefinement} and Lemma~\ref{lem:realizablemorphisms} we reduce to proving that \be \Spec \ov{h} : (\Spec P,\ov{\M}_P) & \to & (\Spec Q, \ov{\M}_Q) \ee is realizable (resp.\ strongly realizable) for each refinement (resp.\ strong refinement) of fine monoids $h : Q \to P$.  Fix such a map $h$, a fine log space $X$, and an $\SMS$ morphism $b : \sms{X} \to (\Spec Q,\ov{\M}_Q)$.  We must show that $\A(X,b,\Spec \ov{h})$ has a terminal object (and that \eqref{terminalobjectmap} is strict when $h$ is a strong refinement).  By Lemma~\ref{lem:Katocategory2} this is local on $X$, so by Lemma~\ref{lem:geometricrealization1} we can assume, after possibly shrinking $X$ and replacing $Q$ with a different fine monoid with the same sharpening, that $b=\ov{a}$ for a map $a : Q \to \M_X(X)$.  The result then follows from Lemma~\ref{lem:geometricrealization2}. \end{proof}

The geometric realization we have considered in this section is closely related to the realization of fans of \S\ref{section:realizationoffans}.  For example:

\begin{prop} \label{prop:geometricrealization} Let $X$ be a fine log space, $Q$ a fine monoid, $a : Q \to \M_X(X)$ a monoid homomorphism, $r : F \to \Spec Q$ a map of fine fans such that $F$ can be covered by affines $\Spec P$ with $Q \to P$ a good refinement (for example, a group isomorphism---in particular, $r$ can be the blowup of $Q$ at an ideal $I \subseteq Q$).  Then the (log differentiable space part of the) terminal object of $\A(X,\ov{a},r)$ is given by $X \times_{\AA(Q)} \AA(F)$. \end{prop}

\begin{proof} This is clear from the local nature of the construction of such a terminal object, and the explicit recipe for this terminal object in Lemma~\ref{lem:geometricrealization2}. \end{proof}

\subsection{Saturation} \label{section:saturation}  An integral log space $X$ is called \emph{saturated} iff $\M_{X,x}$ is a saturated monoid (\S\ref{section:monoidbasics}) for every $x \in X$.  It is easy to see that an integral monoid $P$ is saturated iff $\ov{P}$ is saturated, so $X$ is saturated iff $\ov{\M}_{X,x}$ is saturated for all $x \in X$.

Now suppose $X$ is a fine log space so that $\sms{X}$ is a fine sharp monoidal space.  Let $\sms{X}^{\rm sat}$ be the fine sharp monoidal space obtained by sharpening the locally monoidal space $\Spec_{\sms{X}} \ov{\M}_X^{\rm sat}$.  (This sharpening step is necessary---the saturation of a fine, sharp monoid need not be charp.)  The map $\sms{X}^{\rm sat} \to \sms{X}$ is terminal among $\SMS$ morphisms from a saturated sharp monoidal space to $\sms{X}$.  If $\sms{X} \to (\Spec P, \ov{\M}_P)$ is strict, then we have a cartesian $\SMS$ diagram $$ \xym{ \sms{X}^{\rm sat} \ar[r] \ar[d] & (\Spec P^{\rm sat}, \ov{\M}_{P^{\rm sat}}) \ar[d] \\ \sms{X} \ar[r] & (\Spec P,\ov{\M}_P) } $$ as in Corollary~\ref{cor:sharpeningSpec}.  It then follows from Lemma~\ref{lem:refinements}\eqref{saturationrefinement} that $\sms{X}^{\rm sat} \to \sms{X}$ is a strong refinement of sharp monoidal spaces.  By Theorem~\ref{thm:geometricrealization}, the Kato category $\A(X,\sms{X}^{\rm sat} \to \sms{X})$ has a terminal object, which we will denote $X^{\rm sat}$, and the map $\sms{X^{\rm sat}} \to \sms{X}^{\rm sat}$ is strict, hence $X^{\rm sat}$ is an fs log space.  Unravelling the universal property of this terminal object, we see that $X^{\rm sat} \to X$ is terminal among maps from a saturated log space to $X$.

\begin{lem} \label{lem:PLDSsaturation} Let $X$ be a fine positive log differential space.  Then the map $X^{\rm sat} \to X$ is log smooth and a homeomorphism on topological spaces. \end{lem}

\begin{proof} This is local on $X$ so we can assume there is a fine chart $P \to \M_X(X)$, in which case we have a cartesian diagram $$ \xym{ X^{\rm sat} \ar[d] \ar[r] & \RR_+(P^{\rm sat}) \ar[d]^f \\ X \ar[r] & \RR_+(P) } $$ (with strict horizontal arrows) where $f$ is the $\PLDS$ realization of the saturation $P \into P^{\rm sat}$.  Log smooth maps and homeomorphisms are closed under base change in $\PLDS$, so it is enough to prove that $f$ is log smooth (for this, see Example~\ref{example:RPlogsmooth}) and a homeomorphism (for this, see \S\ref{section:examples}).  \end{proof}

\subsection{Resolution of singularities} \label{section:resolution}  Our basic results on resolution of singularities for log differentiable spaces are obtained formally by combining the ``combinatorial" resolution result Theorem~\ref{thm:SMSresolution} and the ``realization" result Theorem~\ref{thm:geometricrealization}.

\begin{thm} Let $X$ be an fs log differentiable space (resp.\ fine positive log differentiable space).  Then there is a locally projective (hence Euclidean proper), surjective, log smooth morphism $r : X' \to X$ of fs log differentiable spaces (resp.\ fine positive log differentiable spaces) such that $X'$ is free and $r$ is an isomorphism over the free locus of $X$. \end{thm}

\begin{proof}  Suppose first that $X$ is an fs log differentiable space.  Let $r : \ov{\Bl}_I \sms{X} \to \sms{X}$ be the ``canonical resolution" of the fs sharp monoidal space $\sms{X}$ from Theorem~\ref{thm:SMSresolution}.  Since $r$ is a refinement, the Kato category $\A(X,r)$ has a terminal object $$(r : X' \to X, g : \sms{X}' \to \ov{\Bl}_I \sms{X}).$$  We claim that the $\LDS$ morphism $r : X' \to X$ for this terminal object is as desired.

Since $r$ is in fact a \emph{strong} refinement, it is \emph{strongly} realizable (Theorem~\ref{thm:geometricrealization}), so $g$ is strict, hence $X'$ is free because $\ov{\Bl}_I \sms{X}$ is free.

Choose an open cover $\{ U_i \}$ of $X$ and fs charts $a_i : P_i \to \M_X(U_i)$ for the $\M_X|U_i$.  This yields charts \be \ov{a}_i : \sms{U}_i & \to & (\Spec P_i, \ov{\M}_{P_i}) \ee for the fs (resp.\ fine) sharp monoidal spaces $\sms{U}_i$.  By construction of $X'$, the log differentiable space $U_i' := r^{-1}(U_i)$ is the terminal object of $\A(U_i,\ov{a}_i,\ov{r}_i)$, where \be \ov{r}_i : \ov{\Bl}_{ \ov{I}_i }  (\Spec P_i, \ov{\M}_{P_i}) & \to &  (\Spec P_i, \ov{\M}_{P_i}) \ee is the canonical resolution of the fs sharp monoidal space $(\Spec P_i, \ov{\M}_{P_i})$.  By the discussion at the end of \S\ref{section:functorialresolution}, the map $\ov{r}_i$ is the sharpening (as the notation suggests) of the canonical resolution \be r_i : \ov{\Bl}_{ I_i }  P_i & \to & \Spec P_i \ee of the fs fan $\Spec P_i$.  The map $r_i$ is a group isomorphism of fs fans (Lemma~\ref{lem:blowup}), hence, in particular, a good refinement (Lemma~\ref{lem:groupisomorphism}), so by Proposition~\ref{prop:geometricrealization} we have a cartesian $\LDS$ diagram $$ \xym{ U_i' \ar[d] \ar[r] & \AA(\Bl_{I_i}) P_i ) \ar[d]^{\AA(r_i)} \\ U_i \ar[r]^{a_i} & \AA(P_i) }$$ (with strict horizontal arrows) where $\AA(r_i)$ is the $\LDS$ realization of $r_i$.  

The properties of $r$ asserted in the theorem are local on $X$ and stable under base change (the part about the free locus is stable under strict base change), so we need only prove that they hold for $\AA(r_i)$.  The map $\AA(r_i)$ is Euclidean proper by Lemma~\ref{lem:blowuprealizationproper}, log smooth by Lemma~\ref{lem:blowuprealizationlogsmooth}, and surjective by Lemma~\ref{lem:surjectivity2}.  It is an isomorphism over the free locus of $\AA(P_i)$ because $r_i$ is an isomorphism over the free locus of $\Spec P_i$. 

Notice that the same argument goes through verbatim with $\LDS$ replaced by $\PLDS$.  In $\PLDS$ we can weaken ``fs" to ``fine" because we can first take the saturation of $X^{\rm sat} \to X$ (this map certainly has all the desired properties in light of Lemma~\ref{lem:PLDSsaturation}), then apply our fs $\PLDS$ resolution theorem to $X^{\rm sat}$. \end{proof}

\begin{cor} Suppose $X \in \PLDS$ is log smooth.  Then there is a manifold with corners $X'$ and a locally projective, surjective, log smooth $\PLDS$ morphism $r : X' \to X$ which is an isomorphism over the free locus of $X$. \end{cor}

\end{document}